%% file: EssHyperb-1110.tex
\documentclass[11pt]{article}
\usepackage{amsfonts,amssymb,amsmath,amsthm}
\usepackage{epsfig,color}
\usepackage[active]{srcltx}

\newtheorem*{palis-conjecture}{\bf Conjecture (Palis)}
\newtheorem*{palis-conjecture-recast}{\bf Recasting Palis's conjecture}

\newtheorem*{main-theorem}{\bf Main theorem}
\newtheorem*{theo*}{\bf Theorem}
\newtheorem*{theorem*}{\bf Theorem}
\newtheorem*{theorem-extremal}{\bf Theorem \ref{t.extremal}'}
\newtheorem*{proposition*}{\bf Proposition}
\newtheorem*{corollary*}{\bf Corollary}
\newtheorem*{claim*}{\bf Claim}

\newtheorem{teo}{\bf Theorem}
\newtheorem{theo}[teo]{\bf Theorem}
\newtheorem{theorem}[teo]{\bf Theorem}
\newtheorem{lema}{\bf Lemma}[section]
\newtheorem{lemma}[lema]{\bf Lemma}

\newtheorem{propo}[lema]{\bf Proposition}
\newtheorem{prop}[lema]{\bf Proposition}
\newtheorem{proposition}[lema]{\bf Proposition}

\newtheorem{coro}[lema]{\bf Corollary}
\newtheorem{corollary}[lema]{\bf Corollary}

\newtheorem{claim}[lema]{\bf Claim}
\newtheorem{addendum}[lema]{\bf Addendum}

\theoremstyle{definition}
\newtheorem*{definition*}{\bf Definition}
\newtheorem{defi}{\bf Definition}
\newtheorem{definition}[defi]{\bf Definition}

\theoremstyle{remark}
\newtheorem*{question*}{\bf Question}
\newtheorem{obs}{\bf Remark}[section]
\newtheorem{remark}[obs]{\bf Remark}
\newtheorem{remarks}[obs]{\bf Remarks}

\numberwithin{equation}{section}

 \def\NN{{\mathbb N}}  
 \def\RR{{\mathbb R}}  
   
 \def\ZZ{{\mathbb Z}}

\def\cA{\mathcal{A}}    
\def\cB{\mathcal{B}}    
\def\cC{\mathcal{C}}    \def\cU{\mathcal{U}}
\def\cD{\mathcal{D}}   \def\cP{\mathcal{P}} \def\cV{\mathcal{V}}
    \def\cW{\mathcal{W}}
\def\cF{\mathcal{F}}   \def\cR{\mathcal{R}}

\newcommand{\Tang}{\operatorname{Tang}}
\newcommand{\Cycl}{\operatorname{Cycl}}

\newcommand{\diff}{\operatorname{Diff}}
\newcommand{\id}{\operatorname{Id}}

\newcommand{\La}{\Lambda}
\newcommand{\trans}{\mbox{~$|$\hspace{ -.46em}$\cap$}~}
\newcommand{\Lip}{\operatorname{Lip}}

\newcommand{\la}{\lambda}
\newcommand{\al}{\alpha}

\newcommand{\ga}{\gamma}

\newcommand{\po}{\Pi^{ss}}

\newcommand{\noi}{\noindent}

\makeatletter
\renewcommand{\l@section}{\@dottedtocline{2}{3.8em}{3.2em}}
\renewcommand{\l@subsection}{\@dottedtocline{3}{3.8em}{3.2em}}
\newcommand{\subsectionruninhead}{\@startsection{subsection}{2}{0mm}{-\baselineskip}{-0mm}{\bf\large}}
\newcommand{\subsubsectionruninhead}{\@startsection{subsubsection}{3}{0mm}{-\baselineskip}{-0mm}{\bf\normalsize}}
\makeatother

\begin{document}

\title{Essential hyperbolicity and homoclinic bifurcations:\\
a dichotomy phenomenon/mechanism for diffeomorphisms}
\author{Sylvain Crovisier\,\,\,\,\,\,\,\,Enrique R. Pujals}
\maketitle

\emph{\begin{flushright}
To Jacob for his 70$^\text{th}$ birthday. \hspace{2cm}\mbox{}
\end{flushright}}

\begin{abstract}
We prove that any diffeomorphism of a compact manifold can be approximated in topology $C^1$
by another diffeomorphism exhibiting a homoclinic bifurcation (a homoclinic tangency or a heterodimensional cycle) or by one which is essentially hyperbolic (it has a finite number of transitive hyperbolic attractors with open and dense basin of attraction).
\end{abstract}

\tableofcontents

\input{intro-1110.tex}

\input{classes-1110.tex}


\input{weak-hyperbolicity-1110.tex}

\input{continuation-1110.tex}

\input{boundary-1110.tex}

\input{jointintegrable-1110.tex}

\input{unstableleaves-1110.tex}

\input{disconnectedness-1110.tex}

\input{codimension-1110.tex}

\vskip 5pt

\noi Sylvain Crovisier\footnote{Partially supported by the ANR project \emph{DynNonHyp} BLAN08-2 313375. However S.C. does not support the fact that the short term projects funded by the ANR replace the long term research funded by the CNRS.}

\noi CNRS - Laboratoire Analyse, G\'eom\'etrie et Applications, UMR 7539,

\noi Institut Galil\'ee, Universit\'e Paris 13, 99 Avenue J.-B. Cl\'ement, 93430 Villetaneuse, France.

\noi crovisie@math.univ-paris13.fr

\vskip 5pt

\noi Enrique R. Pujals

\noi IMPA-OS

\noi Estrada Dona Castorina 110, 22460-320, Rio de Janeiro.

\noi enrique@impa.br

\end{document}

%% file: intro-1110.tex
\section{Introduction}
\subsection{Mechanisms classifying the dynamics}
In the direction to describe the long range behavior of trajectories
for ``most" systems  (i.e. in a subset of the space of dynamics
which is residual, dense, etc.), a crucial goal is to identify  any generic dynamical behavior.
It was briefly thought in the sixties that this could be realized
by the property of \emph{uniform hyperbolicity}.  Under this
assumption, the limit set decomposes into a
finite number of disjoint (hyperbolic) transitive sets and the asymptotic
behavior of any orbit is described by the dynamics in those finitely many transitive sets (see \cite{S}). Moreover, under 
the assumption of hyperbolicity one obtains a satisfactory (complete)
description of the dynamics of the system from a topological and statistical point of view.

Hyperbolicity was soon realized to be a less universal
property than what one initially thought: the space of dynamics
contains open sets of non-hyperbolic systems. We are now aimed to understand
how the space of systems can be organized according to the different kinds of
dynamical behavior they exhibit.

\paragraph{a- Characterization of non-hyperbolic systems.}
Dynamicists were lead to look for obstructions to hyperbolicity.
For instance any non-hyperbolic diffeomorphism can be approximated in the $C^1$-topology
by a system having a non-hyperbolic periodic orbit (see~\cite{M}, \cite{Aoki}, \cite{H1}).
Since Poincar\'e we know that some very simple configurations (such that the existence of a homoclinic orbit) could be the source of a widely complex behavior.
One has identified two simple obstructions for hyperbolicity which generate rich dynamical phenomena and they have played a crucial role in the study of generic non-hyperbolic behavior:
\begin{enumerate}
\item  {\it heterodimensional cycle}: the presence of two
periodic orbits of different stable dimension linked through the
intersection of their stable and unstable manifolds (see \cite{AS}, \cite{Sh}, \cite{D1});
\item  {\it homoclinic tangency}: the existence of a non-transversal
intersection between the stable and unstable manifolds of a periodic orbit (see \cite{N1}, \cite{N2}, \cite{PT}, \cite{PV}, \cite{BD}).
\end{enumerate}
These obstructions are relevant due to several dynamical consequences that they involve:
the first one is related to the existence of {non-hyperbolic robustly transitive systems} (see \cite{D1}, \cite{BDPR}, \cite{BDP});
the second one generates cascade of bifurcations, is related to
the existence of residual subsets of diffeomorphisms displaying infinitely many periodic attractors (see \cite {N3}) and to the local variations of entropy for surface diffeomorphisms (see~\cite{PS2}).

Another important property is that these obstructions are not isolated in the $C^1-$topology, and 
sometimes, there are not isolated  in a strong way: $i)$ among $C^2$-surface diffeomorphisms, any system with a homoclinic tangency is limit of an open set of diffeomorphisms having homoclinic
tangencies associated to hyperbolic sets (see \cite{N3}); $ii)$ among $C^1$-diffeomorphisms, any system with a heterodimensional cycle  is limit of an open set of diffeomorphisms having heterodimensional cycles  associated to hyperbolic sets of different indexes (see \cite{BDKS} and section \ref{ss.consequences}).
\medskip

In the 80's Palis conjectured (see \cite{P}, \cite{PT}) that these two bifurcations are the main
obstructions to hyperbolicity:

\begin{palis-conjecture}
Every $C^r$ diffeomorphism of a compact manifold can be $C^r$
approximated by one which is hyperbolic or by one exhibiting a
heterodimensional cycle or a homoclinic tangency.
\end{palis-conjecture}

This conjecture may be considered as a starting point to obtain a generic description
of $C^r$-diffeomorphisms. If it turns out to be true, we may focus on the two bifurcations
mentioned above in order to understand the dynamics.

\paragraph{b- Mechanisms \emph{versus} phenomena.}
To elaborate the significance of this conjecture, we would like to recast it in terms of mechanisms and dynamical phenomena. 
\smallskip

By a \emph{mechanism}, we mean a simple dynamical configuration for one diffeomorphism
(involving for instance few periodic points and their invariant manifolds)
that has the following properties:
\begin{itemize}
\item[--] it {\em ``generates itself''}: the system exhibiting this configuration is not isolated.
In general the mechanism is a co-dimensional bifurcation, but it produces a cascade of diffeomorphisms sharing the same configuration;
\item[--] it {\em ``creates or destroys''} rich and different dynamics for nearby systems
(for instance horseshoes, cascade of bifurcations, entropy's variations).
\end{itemize}
Following this definition, homoclinic tangencies and heterodimensional cycles are mechanisms in any $C^r-$topology for $r\geq 1.$
\medskip

In our context a {\em dynamical phenomenon} is any dynamical property
which provides a good global description of the system
(like hyperbolicity, transitivity, minimality, zero entropy, spectral decomposition) and which occurs on a ``rather large" subset of systems. 

We relate these notions and say that
{\em a mechanism is a complete obstruction to a dynamical phenomenon} when:
\begin{itemize}
\item[--] it is an {\em obstruction}: the presence of the mechanism prevents the phenomenon to happen;
\item[--] it is {\em complete}: each system that does not exhibit the dynamical phenomenon is approximated by another displaying the mechanism. 
\end{itemize}
In other words, a mechanism (or a dynamical configuration) is a complete obstruction to a dynamical phenomena, if it not only prevents the phenomenon to happen but it also generates itself creating rich dynamics and it is common  in the complement of the prescribed dynamical phenomenon. Following this approach, Palis's conjeture can be recasted:

\begin{palis-conjecture-recast}
Heterodimensional cycles and homoclinic tangencies are a complete obstruction to hyperbolicity.
\end{palis-conjecture-recast}

Let us give some examples where a dichotomy mechanism / phenomenon has been
proved or conjectured.
\begin{itemize}
\item[--] \emph{Homoclinic bifurcations / hyperbolicity}. This corresponds to the previous
conjecture and is known in dimensions 1 and 2 for the $C^1$-topology, see~\cite{PS1}.
\item[--] \emph{Transverse homoclinic intersection / robust zero topological entropy}.
It has been proved in any dimension for the $C^1$-topology, see~\cite{BGW}, \cite{C1}.
\item[--] \emph{Trapping region / residual transitivity}. Any $C^1$-generic diffeomorphism
$f$ is either transitive or sends a compact set into its interior, see~\cite{BoCr}.
\item[--] \emph{Homoclinic tangency / global dominated splitting}.
After a $C^1$-perturbation any diffeomorphism exhibits a homoclinic tangency or its limit dynamics holds a (robust) dominated splitting with one-dimensional central bundles, see~\cite{CSY}.
\end{itemize}

\paragraph{c- Main result.}
In the present paper, we prove the mentioned conjecture in the $C^1-$topology for
a weaker notion of hyperbolicity.

\begin{definition*}
A diffeomorphism is {\em essentially hyperbolic} if it has
a finite number of transitive hyperbolic attractors and if the union of their basins of attraction is open and dense in the manifold.
\end{definition*}

The essential hyperbolicity recovers the notion of Axiom A: most of the trajectories (in the Baire category) converge to a finite number of transitive attractors that are well described from a both topological and statistical point of view. Moreover, the dynamics in those hyperbolic attractors, govern the  dynamics of the trajectories that converge to them.
In fact, in an open and dense subset the forward dynamics does not distinguish the system
to an Axiom A diffeomorphism.
\medskip

Now, we state our main theorem:
\begin{main-theorem}
Any diffeomorphism of a compact manifold can be $C^1-$approximated by another diffeomorphism which:
\begin{enumerate}
\item either has a homoclinic tangency,
\item or has a heterodimensional cycle,
\item or is essentially hyperbolic.
\end{enumerate}
\end{main-theorem}

Roughly speaking we proved that  {\em homoclinic tangencies and heterodimensional cycles are the $C^1-$complete obstructions for the essential hyperbolicity.}

\begin{remark}
The proof gives a more precise result: inside the open set of diffeomorphisms that are not
limit in $\diff^1(M)$ of diffeomorphisms exhibiting a homoclinic tangency or a heterodimensional cycle,
the essentially hyperbolic diffeomorphisms contain a G$_\delta$ dense subset.
As a consequence, one may also require that these diffeomorphisms are also essentially hyperbolic for $f^{-1}$.
\end{remark}

\paragraph{d- Mechanisms \emph{associated to} phenomena.}
In contrast to the previous dichotomies, a mechanism could also be the key for a rich (semi-global) dynamics.
We  say that {\em a mechanism is associated  to a dynamical phenomenon} if the following holds:
\begin{itemize}
\item[--] the systems exhibiting the dynamical phenomenon can be approximated
by ones displaying the mechanism; 
\item[--] the ones exhibiting the mechanism generate (at least locally) the dynamical phenomenon.
\end{itemize}
As in the notion of complete obstruction, a mechanism is associated to a dynamical phenomenon not only if it generates it but if any time that the phenomenon appears by small perturbations the mechanism is created. Thus a goal would be to establish a dictionary between mechanisms
and (semi-global) dynamical phenomena.
\medskip

Let us mention some known examples.
\begin{itemize}
\item[--] \emph{Transverse homoclinic intersections / non-trivial hyperbolicity.}
On one hand, systems exhibiting a transversal homoclinic point of a hyperbolic periodic point has horseshoes associated to them; on the other hand horseshoes displays transversal homoclinic points (see for instance \cite{B} and \cite{S}).

\item[--] \emph{Heterodimensional cycles / non-hyperbolic $C^1$-robust transitivity.}
On the one hand, systems displaying heterodimensional cycles are
$C^1-$dense in the interior of the set of non-hyperbolic transitive diffeomorphisms
(see for instance~\cite{GW});
on the other hand, the $C^r-$unfolding of a (co-index one) heterodimensional cycles creates maximal invariant robustly transitive non-hyperbolic sets (see \cite{D1}).

\item[--] \emph{Homoclinic tangencies / residual co-existence of infinitely many independent pieces.} On the one hand, the existence of a homoclinic tangency
for $C^2$ surface diffeomorphisms, sectionally dissipative tangencies in higher dimension or the existence of a homoclinic tangencies
combined with heterodimensional cycles for $C^1$ diffeomorphisms may imply
locally residually the co-existence of infinitely many attractors (Newhouse phenomenon),
see~\cite{N3},~\cite{PV} and~\cite{BD}. On the other hand, it is conjectured that
any diffeomorphism exhibiting infinitely many attractors can be approximated by a diffeomorphism
which exhibits a homoclinic tangency (see for instance \cite{Bo}). 
\end{itemize}
Related to the above conjecture in \cite{Bo},  it was proved in \cite{PS4}  that for smooth diffeomorphisms, the co-existence of infinitely many attractors in a ``sectionally dissipative region of the manifold'' implies the creation of sectionally dissipative tangencies  by $C^1$ perturbations (see corollary  1.1  in \cite{PS4} for details). In a  more general framework as a byproduct of the proof of the main theorem, we prove the following. 
\smallskip

\noi{\bf Theorem. }{\em The co-existence of infinitely many attractors implies that either heterodimensional cycles or homoclinic tangencies can be created by $C^1$ perturbations.}
\smallskip

\noi See item c- in section~\ref{itinerary} for details and proof.

\paragraph{e- Robust mechanisms}
The mechanisms we presented are simple configurations of the dynamics but
as bifurcations are also one-codimensional.
From the deep studies of the role of cycles  and tangencies, Bonatti and Diaz have proposed to enrich Palis's conjecture and introduced the notion of
{\em robust heterodimensional cycles} and \emph{robust homoclinic tangencies},
meaning that now the mechanisms involve non-trivial transitive hyperbolic sets
instead of periodic orbits so that the cycles and tangencies may occur on an open set of diffeomorphisms.

From~\cite{BD2} the  main theorem can be restated in the following way:
\medskip

\noi{\bf Main theorem revisited. }{\em Any diffeomorphism of a compact manifold
can be $C^1-$approximated by another diffeomorphism which either is essentially hyperbolic,
or has a homoclinic tangency, or has a robust heterodimensional cycle.}
\medskip

We also refer to~\cite{Bo} for a complementary program about the dynamics of $C^1$-diffeomorphisms.

\subsection{Itinerary of the proof}\label{itinerary}

The proof focuses on diffeomorphisms far from homoclinic bifurcations and
consists in three parts.
\begin{itemize}
\item We first conclude that the quasi attractors (the Lyapunov stable chain-recurrence classes) are ``topologically hyperbolic'': they are partially hyperbolic homoclinic classes with a one-dimensional ``stable'' center bundle and
the union of their basin of attraction is dense in the manifold.
\item We then develop a series of perturbation techniques which ensure that topologically hyperbolic quasi-attractors are uniformly hyperbolic attractors.
\item At the end we prove that the union of the quasi-attractors is closed. With the second point this gives the finiteness of the hyperbolic attractors.
\end{itemize}
A diffeomorphism which satisfies the first and the third property could be called ``essentially topologically hyperbolic''.

\paragraph{a- Topological hyperbolicity.} From the start, we concentrate the study on quasi-attractors. Following \cite{C1,C2} (see theorems \ref{t.homoclinic} and \ref{t.aperiodic} below), it is concluded that $C^1-$far from homoclinic bifurcations, the aperiodic chain-recurrent classes are partially hyperbolic with a one-dimensional central bundl, and 
the homoclinic classes are partially hyperbolic with their central bundles being at most two-dimensional (however the hyperbolic extremal subbundles may be degenerated).
Moreover, a special type of dynamics has to hold along the central manifolds:
the center stable is chain-stable and the center unstable is chain-unstable.
We define a weak notion of topological hyperbolicity that we call \emph{chain-hyperbolicity}:
this is suitable for our purpose since in some cases the chain-hyperbolicity is robust under perturbations.
(See definition \ref{d.chain-hyperbolic} for details and justification of the names topological hyperbolicity and chain-hyperbolicity).

From corollary \ref{c.aperiodic} it is concluded that aperiodic classes can not be attractors and therefore they are out of our picture. For homoclinic classes, whenever the partially hyperbolic splitting has two extremal hyperbolic subbundles, corollary \ref{c.homoclinic}  concludes that the central bundle is one-dimensional subbundle and chain-stable otherwise a heterodimensional cycle is created.

\paragraph{b- Uniform hyperbolicity.}
At this step, a first dichotomy is presented (see corollary \ref{c.whitney}): either the quasi-attractor is contained in a normally hyperbolic submanifold (and from there one concludes the hyperbolicity, see corollary \ref{c.codim-one}) or  the strong stable foliation is non-trivially involved in the dynamic, meaning that at least two different points $x,y$ in the class share the same local strong stable leaf.  In this second case (see theorem \ref{t.position}), we will perturb the diffeomorphism in order to
obtain a {\em strong connection} associated to a periodic point, i.e. a periodic point whose strong stable and unstable manifolds intersect, see definition \ref{strong-int}; in particular, assuming that the quasi-attractor is not hyperbolic, a heterodimensional cycle can be created (see proposition \ref{p.strong-connection}).

To perform the perturbations, one has to discuss the relative position between two unstable leaves after projection by the strong stable holonomy: the position types are introduced in definition~\ref{definition-cases}. In particular, by analyzing the geometry of quasi-attractors one can reduce to the case the points $x,y$ belong to stable or to unstable manifolds of some periodic orbits.
Improving~\cite{Pu1} and~\cite{Pu2},
three different kinds of perturbations may be performed. They correspond to the following cases:
\begin{itemize}
\item[--] $x,y$ belong to unstable manifolds and their forward orbits have fast returns close to
$x$ or $y$.
\item[--] $x,y$ belong to unstable manifolds and their forward orbits have slow returns close to $x$ or $y$.
\item[--] $x,y$ belong to a stable manifold.
\end{itemize}
The two first cases are covered by theorem \ref{t.unstable} and the last one by theorem \ref{t.stable}.
To perform these perturbations one needs to control how the geometry of the class changes for any perturbed map; we prove (see proposition~\ref{p.continuation}) that whenever the perturbation of the homoclinic class does not display strong connection associated to periodic points then it is possible to get a well defined continuation for the whole class.

\paragraph{c- Finiteness of the attractors.}
The delicate point is to exclude the existence of an infinite number of sinks.
This is done by proving that for any non-trivial chain-recurrence classes, the extremal subbundles
are hyperbolic. We thus consider the splittings $E^s\oplus E^{cu}$ or $E^s\oplus E^{cs}\oplus E^{cu}$,
where $E^{cs},E^{cu}$ are one-dimensional, and in both cases we prove  that $E^{cu}$ is hyperbolic.
The first case follows from results in \cite{PS4}. In the second case, the hyperbolicity of the center unstable subbundle  follows for a more detailed understanding of the topological and geometrical structure of the homoclinic class (see theorem \ref{t.2D-central}). In fact, from being far from heterodimensional cycles, it is concluded that the the class is totally disconnected along the center stable direction (see theorems \ref{t.tot-discontinuity}) and from there a type of geometrical Markov partition is constructed (see proposition \ref{p.box}); this allows to use $C^2-$distortion arguments to conclude hyperbolicity of $E^{cu}$ as in~\cite{PS1} and \cite{PS4}.

After it is concluded that the chain-recurrence classes are partially hyperbolic with non-trivial extremal hyperbolic subbundles, the finiteness follows quite easily (see section \ref{ss.finiteness}).

\paragraph{Structure of the paper.}
In section \ref{s.classes} it is proved that the chain-recurrence classes for systems far from homoclinic bifurcations are ``topologically hyperbolic''.
Moreover, we stated there all the theorems (proved in the other sections)
needed to conclude the main theorem, which is done in subsection \ref{ss.quasi-attractor}.
In section \ref{s.weak-hyperbolicity} we give a general study of the chain-hyperbolic classes and their topological and geometrical structures.
This allows to obtain the continuation of some partially hyperbolic classes (done in section \ref{s.continuation}), and to introduce the notion of boundary points for quasi-attractors (done in section \ref{s.boundary}). In sections \ref{proofjointint} and \ref{p-nontransversal} are stated and proved the new perturbations techniques that hold in the $C^{1+\al}-$topology. In sections \ref{s.2D-central} and \ref{s.2D-central2} are studied partially hyperbolic homoclinic classes with a two-codimensional strong stable bundle, first analyzing their topological and geometrical structure and latter their hyperbolic properties.

\subsection{Some remarks about new techniques and $C^r-$versions of the main theorem}

We would like to highlight many of the new techniques developed in the present paper and  that can be used in other context.

\paragraph{\rm\em 1- Chain-hyperbolicity.}
We introduce the notion of chain-hyperbolic homoclinic class which generalizes the locally maximal hyperbolic sets. It allows to include some homoclinic classes having hyperbolic periodic points
with different stable dimensions, provided that at some scale, a stable dimension is well-defined.
We recover some classical properties of hyperbolic sets: the local product structure, the stability under perturbation, the existence of (chain) stable and unstable manifolds. See section~\ref{s.weak-hyperbolicity}.

\paragraph{\rm\em 2- Continuation of (non necessarily hyperbolic) homoclinic classes.}
It is well known that isolated hyperbolic sets are stable under perturbation and have a well
defined and  unique continuation. We extend this approach to certain partially hyperbolic sets
which are far from strong connections.
This is done by extending the continuation of their hyperbolic periodic points to their closure,
a technique that resembles to the notion of holomorphic motion. See section~\ref{s.continuation}.

\paragraph{\rm\em 3- Geometrical and topological properties of partially hyperbolic attractors.}
We study the geometrical structure of partially hyperbolic attractors
with a one-dimensional central direction in terms of the dynamics of the strong stable foliation.
For instance:
\begin{itemize}
 \item[--] It is presented a dichotomy proving that a homoclinic class is
 either embedded in a submanifold of lower dimension of the ambient space or
 one can create a strong connection (maybe after a perturbation). See theorems~\ref{t.tot-discontinuity} and~\ref{t.position}.

\item[--] In  certain cases it is introduced the notion of  stable boundary points of a partially hyperbolic homoclinic class (extending a classical notion for hyperbolic surfaces maps) which permits us to control the bifurcations that holds after perturbations. See proposition \ref{p.boundary} and lemma \ref{l.boundary1}.

\item[--] If they are no (generalized) strong connection,
it is proved that the homoclinic class is totally disconnected along its stable leaves. See theorem~\ref{t.tot-discontinuity}.

\item[--] The total disconnectedness mentioned above, allows us to introduce kind of
Markov partitions for non-hyperbolic partially hyperbolic classes. See proposition~\ref{p.box}.
\end{itemize}

\paragraph{\rm\em 4- Hyperbolicity of the extremal subbundles.}
For invariant compact sets having a dominated splitting $E\oplus F$ with $\dim(F)=1$,
\cite{PS1} and \cite{PS4} have developed a technique which allows to prove that $F$ is hyperbolic
provided $E$ is either uniformly contracted or one-dimensional.
We extend this result for partially hyperbolic systems with a $2$-dimensional central bundle,
that is when $E$ is only ``topologically contracted". See section~\ref{s.2D-central2}.

\paragraph{\rm\em 5- New perturbation techniques.}  It is developed new perturbation techniques suitable for partially hyperbolic sets with one-dimensional central directions. See theorems~\ref{t.unstable} and~\ref{t.stable}. We want to point out, that these perturbations hold in the $C^{1+\al}-$topology. Those perturbation resemble the $C^1-$connecting lemma but since in the present context a better understanding of the dynamic is available, then the perturbation can be perform in the $C^{1+\al}-$topology.

\paragraph{\rm\em 6- Consequences for hyperbolic dynamics.} Previous highlighted techniques can be formulated for hyperbolic attractors and have consequences
in terms of topological and geometrical structure. See theorems~\ref{t.tot-discontinuity} and~\ref{t.position}.

\paragraph{\rm\em 7- Generic structure of partially hyperbolic quasi-attractors.}
A byproduct of the proof shows (see theorem~\ref{t.consequences}) that for $C^1$-generic diffeomorphisms,
any quasi-attractor which has a partially hyperbolic structure with a one-dimensional
central bundle contains periodic points of different stable dimension.
\bigskip

We want to emphasize that many of the results contained in the present paper work in the $C^r-$category
for any $r\geq 1$ or for $r=1+\al$ with $\al\geq 0$ small. 
For instance, theorems \ref{t.position}, \ref{t.unstable} and \ref{t.stable} hold in the $C^{1+\al}-$topology.
This allows to prove (see the remark~\ref{r.position}, item \ref{i.position}) a partial version of Palis conjecture in the $C^{1+\al}-$category
when one restricts to partially hyperbolic attractors with one-dimensional center direction).

\begin{theorem*}
For any $C^2$ diffeomorphism $f$ of a compact manifold and any ``topologically hyperbolic attractor" $H(p)$ (i.e. which satisfies the assumptions stated in theorem~\ref{t.position}), there exists $\al>0$ with the following property.
For any $\delta>0$, there exists $C^{1+\al}$-perturbations $g$ of $f$ such that
\begin{itemize}
\item[--] either the homoclinic class $H(p_g)$ associated to the continuation $p_g$
of $p$ is hyperbolic,
\item[--] or there exists a periodic orbit $O$ of $g$ which has a strong homoclinic intersection
and one of its Lyapunov exponents has a modulus smaller than $\delta$.
\end{itemize}
\end{theorem*}

We don't know however if under the conclusions of this theorem it is possible to create a heterodimensional cycle by a $C^{1+\al}$-perturbation of
the diffeomorphism.

%% file: classes-1110.tex
\section{Chain-recurrence classes far from homoclinic bifurcations}
\label{s.classes}

We introduce in sections~\ref{ss.trapped} and~\ref{ss.homoclinic}
the notion of trapped plaque families and chain-hyperbolic homoclinic classes.
Their basic properties will be studied systematically later in section~\ref{s.weak-hyperbolicity},
but we will derive before (sections~\ref{ss.homoclinic}, \ref{ss.aperiodic}
and~\ref{ss.finiteness}) important consequences for the generic dynamics far
from homoclinic bifurcations.
We also present (sections~\ref{ss.extremal} and~\ref{ss.quasi-attractor})
the main results of the paper that are proved in the next sections
and explain how they imply the main theorem.
In the last part (section~\ref{ss.consequences}) we give other consequences of our techniques.
We start this section by recalling some classical definitions.
\medskip

In all the paper $M$ denotes a compact boundaryless manifold.
\begin{defi}
We say that $f\in \diff^1(M)$ exhibits a \emph{homoclinic tangency} if there is a hyperbolic periodic orbit $O$
and a point $x\in W^s(O)\cap W^u(O)$ with $T_xW^s(O)+ T_xW^u(O)\neq T_xM$.
\end{defi}

\begin{defi}
We say that $f\in \diff^1(M)$ exhibits a \emph{heterodimensional cycle} if there are
two hyperbolic periodic orbits $O$ and $O'$ of different stable dimension,
such that $W^u(O)\cap W^s(O')\neq \emptyset$ and $W^u(O')\cap W^s(O)\neq \emptyset$. 
\end{defi}

\begin{defi} From now on, with $\overline{\Tang\cup\Cycl}$ it is denoted the set of diffeomorphisms
that can be $C^1-$approximated by one exhibiting either a homoclinic tangency or a heterodimensional cycle.
We say that a diffeomorphisms $f$ is \emph{$C^1-$far from cycles and tangencies}
if $f\in \diff^1(M)\setminus\overline{\Tang\cup\Cycl}$ 
\end{defi}
\medskip

The global dynamics of a diffeomorphism may be decomposed in the following way.
The \emph{chain-recurrent set} is the set of points that belong to a periodic $\varepsilon$-pseudo orbit
for any $\varepsilon>0$. This compact invariant set
breaks down into invariant compact disjoint pieces, called the \emph{chain-recurrence classes}:
two points belong to a same piece if they belong to a same periodic $\varepsilon$-pseudo orbit
for any $\varepsilon>0$.
An invariant set is \emph{chain-transitive} if it contains a $\varepsilon$-dense
$\varepsilon$-pseudo-orbit for any $\varepsilon>0$.

\begin{defi} A \emph{quasi-attractor} is a chain-recurrence class which is Lyapunov stable,
i.e. which admits a basis of neighborhoods $U$ satisfying $f(U)\subset U$.
\end{defi}
\medskip

For any diffeomorphism, we define another notion of ``piece of the dynamics''.
Associated to a hyperbolic periodic point $p$, one introduces
its \emph{homoclinic class} $H(p)$ which is the closure of the transverse
intersection points between the unstable and the stable manifolds $W^u(O),W^s(O)$
of the orbit $O$ of $p$.
It also coincides with the closure of the set of hyperbolic points $q$ that are
\emph{homoclinically related} to the orbit of $p$, i.e. such that
$W^u(q)$ and $W^s(q)$ have respectively a transverse intersection point
with the stable and the unstable manifolds of the orbit of $p$.
Note that for diffeomorphisms $g$ that are $C^1$-close to $f$, the periodic point $p$
has a hyperbolic continuation $p_g$. This allows to consider the homoclinic class $H(p_g)$.

For a $C^1$-generic diffeomorphism, the periodic points are hyperbolic and
\cite{BoCr} proved that a chain-recurrence
class that contain a periodic point $p$ coincides with the homoclinic class $H(p)$.
The other chain-recurrence classes are called the \emph{aperiodic classes}.
Those classes are treated in subsections \ref{ss.homoclinic} and \ref{ss.aperiodic}.

We state two other consequences of Hayashi's connecting lemma and~\cite{BoCr}.
\begin{lemma}\label{l.prelim} For any $C^1$-generic diffeomorphism $f$ and any homoclinic class $H(p)$,
\begin{itemize}
\item[--] if $H(p)$ contains periodic points with different stable dimensions, then
$f$ may be $C^1$-approxi\-mated by diffeomorphisms having a heterodimensional cycle;
\item[--] $H(p)$ is a quasi-attractor if and only if it contains the unstable manifold of $p$.
\end{itemize}
\end{lemma}
\medskip

Quasi-attractor always exist but for a $C^1$-generic diffeomorphism they attract most orbit.
\begin{theorem}[\cite{MP,BoCr}]
Let $f$ be a diffeomorphism in a dense G$_\delta$ subset of $\diff^1(M)$.
Then the $\omega$-limit set of any point $x$ in a dense G$_\delta$ subset of $M$
is a quasi-attractor.
\end{theorem}

According to this result, the main theorem is a consequence of two independant properties
of $C^1$-generic diffeomorphisms that are $C^1$-far from cycles and tangencies:
\begin{itemize}
\item[--] the union of the quasi-attractors is closed (see proposition~\ref{p.finiteness});
\item[--] each quasi-attractor is a hyperbolic set (see theorem~\ref{t.position}).
\end{itemize}
Indeed by the shadowing lemma, any quasi-attractor which is hyperbolic is transitive
and attracts any orbit in a neighborhood. In particular,
the quasi-attractors are isolated in the chain-recurrence set. Since their union is closed,
they are finite.

\subsection{Trapped tangent dynamics}\label{ss.trapped}
Let $f$ be a diffeomorphism and $K$ be an invariant compact set.
\medskip

A \emph{dominated splitting} on $K$ is a decomposition $T_KM=E\oplus F$ of its
tangent bundle into two invariant linear sub-bundles such that, for some integer
$N\geq 1$, any unitary vectors $u\in E_x,v\in F_x$ at points $x\in K$ satisfy
$$2\|Df^N.u_x\|\leq \|Df^N.v_x\|.$$
This definition does not depend on the choice of a Riemannian metric on $M$.
In the same way, one can define dominated splittings
$T_KM=E_1\oplus \dots \oplus E_s$  involving more than two bundles.

When the bundle $E$ is uniformly contracted (i.e. when there exists $N\geq 1$
such that for any unitary vector $u\in E$ one has $\|Df^N.u\|\leq 2^{-1}$),
the stable set of each point $x$ contains an injectively embedded sub-manifold
$W^{ss}(x)$ tangent to $E_x$ called the \emph{strong stable manifold of $x$},
which  is mapped by $f$ on the manifold $W^{ss}(f(x))$.

A \emph{partially hyperbolic splitting} on $K$
is a dominated splitting $T_KM=E^s\oplus E^c\oplus E^u$ such that
$E^s$ and $E^u$ are uniformly contracted by $f$ and $f^{-1}$ respectively.

\medskip
\begin{defi}
A \emph{plaque family tangent to $E$} is a continuous map $\cW$ from the linear bundle $E$
over $K$ into $M$ satisfying:
\begin{itemize}
\item[--] for each $x\in K$, the induced map $\cW_x\colon E_x\to M$ is a $C^1$-embedding
which satisfies $\cW_x(0)=x$ and whose image is tangent to $E_x$ at $x$;
\item[--] $(\cW_x)_{x\in K}$ is a continuous family of $C^1$-embeddings.
\end{itemize}
The plaque family $\cW$ is \emph{locally invariant} if there exists $\rho>0$
such that for each $x\in K$ the image of the ball $B(0,\rho)\subset E_{x}$ by $f\circ \cW_{x}$
is contained in the plaque $\cW_{f(x)}$.
\end{defi}
We often identify $\cW_x$ with its image.
The plaque family theorem~\cite[theorem 5.5]{HPS} asserts that a locally invariant plaque family
tangent to $E$ always exists (but is not unique in general).

\begin{defi}\label{d.tt}
The plaque family is \emph{trapped} if for each $x\in K$, one has
$$f(\overline{\cW_x})\subset \cW_{f(x)}.$$
It is \emph{thin trapped} if for any neighborhood $S$ of the section $0$ in $E$ there exist:
\begin{itemize}
\item[--] a continuous family $(\varphi_x)_{x\in K}$ of $C^1$-diffeomorphisms of the spaces $(E_x)_{x\in K}$ supported in $S$;
\item[--] a constant $\rho>0$ such that for any $x\in K$ one has
$$f(\overline{\cW_x\circ \varphi_x(B(0,\rho))})\subset \cW_{f(x)}\circ \varphi_{f(x)}(B(0,\rho)).$$
\end{itemize}

\end{defi}
If a plaque family $\cW$ is thin trapped, then it is also the case
for any other locally invariant plaque family $\cW'$ tangent to $E$
(moreover there exists $\rho>0$ such that for each $x\in K$,
the ball $B(0,\rho)\subset E_{x}$ is sent by $\cW'_{x}$ into $\cW_{x}$, see lemma~\ref{l.uniqueness-coherence}).
One thus say that \emph{$E$ is thin trapped}.

\begin{remark}\label{r.nested}
Note also hat when $E$ is thin trapped, there exist nested families of trapped plaques
whose diameter are arbitrarily small.
\end{remark}

The two following properties are classical (see for instance~\cite[Lemma 2.4]{C2}).
On a small neighborhood of $K$, we introduce a cone field $\cC^E$ which is a thin neighborhood of the bundle $E$.
\begin{lemma}\label{l.uniqueness-coherence}
Let $K$ be a compact invariant set endowed with a dominated decomposition $T_KM=E\oplus F$.
There exists $r>0$ such that if there exists a trapped plaque family $\cW^{cs}$ tangent to $\cC^E$
whose plaques have a diameter smaller than $r$, then the following properties hold.
\begin{itemize}
\item[--] If $\widehat{\cW}^{cs}$ is another locally invariant plaque family tangent to $E^{cs}$,
then there exists $\rho>0$ such that
for each $x\in H(p)$ the image of the ball $B(0,\rho)\subset E_{x}$ by $\cW^{cs}_{x}$
is contained in $\widehat \cW^{cs}_{x}$.
\item[--] There exists $\varepsilon>0$ such that for any points $x,x'\in H(p)$ that are
$\varepsilon$-close with $\cW^{cs}_x\cap \cW^{cs}_{x'}\neq \emptyset$, then $f(\overline{\cW^{cs}_{x'}})\subset \cW^{cs}_{f(x)}$.
\end{itemize}
\end{lemma}

\subsection{Homoclinic classes}\label{ss.homoclinic}
Far from homoclinic bifurcations, the homoclinic classes of a generic diffeomorphism
satisfy some weak form of hyperbolicity.

\begin{defi}\label{d.chain-hyperbolic}
A homoclinic class $H(p)$ is said to be \emph{chain-hyperbolic} if:
\begin{itemize}
\item[-] $H(p)$ has a dominated splitting $T_{H(p)} M= E^{cs}\oplus E^{cu}$ into
center stable and center unstable bundles;
\item[-] there exists a plaque family $(\cW^{cs}_x)_{x\in H(p)}$ tangent to $E^{cs}$
which is trapped by $f$ and a plaque family $(\cW^{cu}_x)_{x\in H(p)}$ tangent to $E^{cu}$
which is trapped by $f^{-1}$;
\item[-] there exists a hyperbolic periodic point $q_s$ (resp. $q_u$) homoclinically related to the orbit of $p$
whose stable manifold contains $\cW^{cs}_{q_s}$
(resp. whose unstable manifold contains $\cW^{cu}_{q_u}$).
\end{itemize}
Such a class is \emph{topologically hyperbolic} if its
center stable and center unstable plaques are thin trapped by $f$ and $f^{-1}$ respectively.
\end{defi}

One will see (lemma~\ref{l.chain-stable} below) that for any point $x\in H(p)$, the plaque $\cW^{cs}_{x}$
is contained in the chain-stable set of $H(p)$.
This justifies the name ``chain-hyperbolicity":
this definition generalizes the hyperbolic basic sets endowed with
families of stable and unstable plaques
(in this case the plaques $\cW^{cs}$ are the images of
local stable manifolds by a backward iterate $f^{-n}$).
With additional assumptions, the chain-hyperbolicity is a robust property:
if $H(p)$ is chain-hyperbolic for $f$, coincides with its chain-recurrence class and
if $E^{cs}, E^{cu}$ are thin trapped by $f$ and $f^{-1}$ respectively,
then for any $g$ that is $C^1$-close to $f$ the homoclinic class $H(p_g)$
associated to the continuation $p_g$ of $p$ is also chain-hyperbolic (see lemma~\ref{l.robustness}).

\begin{theo}[\cite{C2}]\label{t.homoclinic}
Let $f$ be a diffeomorphism in a dense G$_\delta$ subset of $\diff^1(M)\setminus \overline{\Tang\cup\Cycl}$.
Then, any homoclinic class of $f$ is chain-hyperbolic.
Moreover, the central stable bundle $E^{cs}$ is thin trapped.
If it is not uniformly contracted,
it decomposes as a dominated splitting $E^{cs}=E^s\oplus E^c$
where $dim(E^c)=1$ and $E^s$ is uniform; and there exist periodic orbits
homoclinically related to $p$ and whose Lyapunov exponents along $E^c$ are arbitrarily close to $0$.
The same holds for the central unstable bundle $E^{cu}$ and $f^{-1}$.
\end{theo}
\begin{proof}
The statement in~\cite{C2} is slightly different and we have to justify why the center stable bundle $E^{cs}$
is thin trapped. When $E^{cs}$ is uniformly contracted, this is very standard. When $E^{cs}$
is not uniformly contracted, \cite[section 6]{C2} asserts that there exists a dominated splitting $E^{cs}=E^s\oplus E^c$
such that $\dim(E^c)=1$, $E^s$ is uniformly contracted and that the bundle $E^c$ has ``type (H)-attracting":
there exists a locally invariant plaque family $\cD$ tangent to $E^c$ and arbitrarily small open neighborhoods
$I$ of the section $0$ in $E^c$ satisfying $f(\overline{\cD_x(I_x)})\subset \cD_{f(x)}(I_{f(x)})$ for each $x\in H(p)$.
The neighborhood $I$ may be chosen as a continuous family of open intervals $(I_x)_{x\in H(p)}$.

Let us now consider a locally invariant plaque family $\cW$ tangent to $E^{cs}$.
Since $I$ is small, one has
$\cD_x(I_x)\subset \cW_x$ for any $x\in H(p)$ (see~\cite[lemma 2.5]{C2}).
One then builds for each $x$ a small open neighborhood $V_x$ of
$\cD_x(I_x)$ in $\cW_x$ which depends continuously on $x$:
this can be obtained by modifying a tubular neighborhood of $\cD_x(I_x)$ in $\cW_x$.
Since $E^s$ is uniformly contracted one can still require the
trapping property $f(\overline{V_x})\subset V_x$.
Let $U_x\subset E^{cs}_x$ be the backward image of $V_x$ by $\cW^{cs}$.
Since $U_x$ can be obtained by modifying the tubular neighborhood of a $C^1$-curve,
it can be chosen diffeomorphic in $E^{cs}$ to an open ball
through a diffeomorphism as stated in definition~\ref{d.tt}.
\end{proof}

One deduces that the tangent bundle over
a non-hyperbolic homoclinic class as in theorem~\ref{t.homoclinic}
has a dominated splitting $TM=E^s\oplus E^c\oplus E^u$ or $E^s\oplus E^c_1\oplus E^c_2\oplus E^u$
where each bundle $E^c$ or $E^c_1,E^c_2$ is one-dimensional, $E^s$ is uniformly contracted and $E^u$
is uniformly expanded (however, one of them can be trivial).
Note that under perturbations the homoclinic class $H(p_g)$
is still chain-hyperbolic but its center stable bundle $E^{cs}$ is a priori not thin trapped.
\medskip

We will focus on the invariant compact sets $K$ that are \emph{Lyapunov stable},
i.e. that have a basis of neighborhoods $U$ that are invariant by $f$
(i.e. $f(U)\subset U$).
\begin{coro}\label{c.homoclinic}
Let $f$ be $C^1$-generic in $ \overline{\Tang\cup\Cycl}^c$.
Then, for any Lyapunov stable homoclinic class of $f$ the center unstable bundle is uniformly expanded.
\end{coro}
\begin{proof}
For any open set $U\subset M$ and any integer $d\geq 0$,
one considers the following property:
\begin{description}
\item[$P(U,d)$:] There exists a hyperbolic periodic orbit $O\subset U$
whose stable dimension equals $d$.
\end{description}
This property is open: if $P(U,d)$ is satisfied by $f$, then
so it is by any diffeomorphism $g$ that is $C^1$-close to $f$.
Let us fix a countable basis of open sets $\cB$, i.e. for any compact set
and any open set $V$ satisfying $K\subset V\subset M$, there exists $U\in \cB$
such that $K\subset U\subset V$.
Then, for any diffeomorphism $f$ in dense G$_\delta$ subset $\cR_0\subset \diff^1(M)$,
for any open set $U\in \cB$ and any $d\geq 0$,
if there exists a perturbation $g$ of $f$ such that $P(U,d)$ holds for $g$, then the same holds for $f$.

We denote by $\cR\subset \cU$ a dense G$_\delta$ subset of
$\diff^1(M)\setminus \overline{\Tang\cup\Cycl}$ whose elements satisfy theorem~\ref{t.homoclinic}
and have hyperbolic periodic orbits are.

Let us consider $f\in \cR$ and a homoclinic class $H(p)$ of $f$ whose
center unstable bundle $E^{cu}=E^c_2\oplus E^u$ is not uniformly expanded. Hence $\dim(E^{c}_2)$
is one-dimensional, $p$ is not a sink (and apriori $E^u$ could be degenerated).
By the theorem~\ref{t.homoclinic}, there exists a hyperbolic periodic orbit $O$
homoclinically related to $p$ having some Lyapunov exponent along $E^{cu}$ arbitrarily close to $0$.
By Franks lemma, one can find a perturbation $g$ of $f$ such that $O$ becomes a hyperbolic periodic orbit
whose stable space contains $E^{c}_2$.
Since $f\in \cR_0$, one deduces that
any neighborhood of $H(p)$ contains a periodic orbit whose stable dimension is $d^s+1$,
where $d^s$ denotes the stable dimension of $p$.

Let us consider a locally invariant plaque families $\cW$ tangent to $E^{cs}$
over the maximal invariant set in a neighborhood of $H(p)$.
Let us consider a periodic orbit $O$
contained in a small neighborhood of $K$,
with stable dimension equal to $d^s+1$.
As a consequence, using the domination $E^{cs}\oplus E^{cu}$,
the Lyapunov exponents along $E^{cs}$ of $O$ is smaller than some uniform constant $-C<0$.
If the plaques of the family $\cW$ are small enough,
the lemma~\ref{l.largestable} and the remark~\ref{r.large} below then ensure that at some $q\in O$
one has $\cW_{q}\subset W^s(q)$.
By lemma~\ref{l.contper} below, $q$ is close to a hyperbolic periodic point
$z$ homoclinically related to $p$
whose plaque $\cW^{cu}_z$ is contained in the unstable set of $z$.
The plaque
$\cW_{q}$ intersects transversally the plaque $\cW^{cu}_z$.
This proves that the stable manifold of $q$ also intersects transversally
the unstable manifold of the orbit of $p$.

Since $H(p)$ is Lyapunov stable, it contains $W^u(z), q$ and $W^u(q)$.
As for $H(p)$, the point $q$ is not a sink. This proves that $E^u$ is non trivial.
Let $y\in W^u(q)\setminus \{q\}$.
Since $y$ belongs to $H(p)$, the stable manifold of the orbit of $p$ accumulates on $y$,
hence by a $C^1$-small perturbation produced by Hayashi's connecting lemma, one can create an intersection between the unstable manifold of $q$ and
the stable manifold of the orbit of $p$. The intersection between $W^u(p)$ and $W^s(q)$ persists hence
we have built a heterodimensional cycle, contradicting our assumptions.
We have proved that if $H(p)$ is Lyapunov stable, the bundle $E^{cu}$ is uniformly expanded.
\end{proof}

\subsection{Aperiodic classes}\label{ss.aperiodic}

Far from homoclinic bifurcations, the aperiodic classes have also a partially hyperbolic structure.
\begin{theo}[\cite{C2}]\label{t.aperiodic}
Let $f$ be a diffeomorphism in a dense G$_\delta$ subset of $\diff^1(M)\setminus \overline{\Tang\cup\Cycl}$.
Then, any aperiodic class of $f$ is a minimal set
and holds a partially hyperbolic structure $E^s\oplus E^c \oplus E^u$.
Moreover, there exists a continuous familly of center stable plaques
$\cW^{cs}$ tangent to $E^{cs}=E^s\oplus E^c$
which are trapped by $f$.
Similarly, there exists a continuous family of center unstable plaques
$\cW^{cu}$ tangent to $E^{cu}=E^c\oplus E^u$ which are trapped by $f^{-1}$.
\end{theo}
\medskip

\begin{coro}\label{c.aperiodic}
Let $f$ be generic in $\diff^1(M)\setminus\overline{\Tang\cup\Cycl}$. Then,
for any aperiodic class, the bundles $E^u$ and $E^s$ are non-degenerated.

The strong unstable manifolds of points of the class are not contained in the class.
In particular, the class is not Lyapunov stable.
\end{coro}
\begin{proof}
Let us consider an aperiodic class $K$ and
a locally invariant plaque family $\cW$ tangent to
$ E^{cs}$ over the maximal invariant set in a small neighborhood of $K$.
There exists a sequence of periodic orbits that accumulate on $K$.
A trapped plaque family $\cW^{cs}$ over $K$ whose plaques have small diameters
are contained in the plaques $\cW$ by lemma~\ref{l.uniqueness-coherence} below.
One deduces that one can extend the plaque family $\cW^{cs}$
over the maximal invariant set in a small neighborhood of $K$
as a trapped plaque family.

Since $K$ is a minimal set and $f$ is $C^1$-generic, Pugh's closing lemma
(the general density theorem) implies that $K$ is the Hausdorff limit of a sequence of periodic orbits.
For any $\tau$-periodic point $p$ whose orbit is close to $K$,
the plaque $\cW^{cs}_p$ is mapped into itself by $f^\tau$.
Since the plaque $\cW^{cs}$ is tangent to the bundle $E^{cs}= E^s\oplus E^c$
where $E^c$ has dimension $1$ and $E^s$ is uniformly contracted,
the orbit of any point in $\cW^{cs}_p$ accumulates in the future on a periodic orbit.

If $E^u$ is degenerate, the union of the plaques $\cW^{cs}_p$
cover a neighborhood of $K$, hence the orbit of any point in $K$
converges towards a periodic orbit, which is a contradiction.

If $E^u$ is not degenerate, the strong unstable manifold $W^{uu}(x)$ tangent to $E^u$
of any point $x\in K$ intersects the plaque $\cW^{cs}_p$ of a periodic point $p$.
One deduces that theres exists an orbit that accumulates on $K$ in the past and on
a periodic orbit $O$ in the future. If $W^{uu}(x)$ is contained in $K$,
the periodic orbit $O$ is contained in $K$, contradicting the fact that $K$ is an aperiodic class.
\end{proof}

\begin{remark}
Actually,  a stronger result can be proved.
\noindent
\emph{For any $C^1$-generic diffeomorphism and any aperiodic class $K$
endowed with a partially hyperbolic structure
$T_KM=E^s\oplus E^c\oplus E^u$ with $\dim(E^u)=1$, 
the class is not contained in a locally invariant submanifold tangent to $E^s\oplus E^c$.}

\noindent
Indeed, otherwise, one could work in this submanifold and get a contradiction as in the previous proof.
See also section \ref{ss.reduction}.
\end{remark}

\subsection{Reduction of the ambient dimension}\label{ss.reduction}

Let us consider an invariant compact set $K$
with a dominated splitting $T_KM=E^s\oplus F$ such that $E^s$ is uniformly
contracted. The dynamics on $K$ may
behave like the dynamics inside a manifold of smaller dimension.
This motivates the following definition.

\begin{definition}
A $C^1$-submanifold $\Sigma$ containing $K$ and tangent to $F$ is \emph{locally invariant}
if there exists a neighborhood $U$ of $K$ in $\Sigma$ such that
$f(U)$ is contained in $\Sigma$.
\end{definition}
More generally,
when $K$ admits a partially hyperbolic splitting $T_KM=E^s\oplus E^c\oplus E^u$
one may define the notion of locally invariant submanifold tangent to $E^c$. The next proposition state that the property 
defined above is robust by $C^1-$perturbations.

\begin{proposition}[\cite{BC2}]\label{p.whitney}
Let $K$ be an invariant compact set endowed with a dominated splitting
$T_KM=E^s\oplus F$ such that $E^s$ is uniformly contracted.
If $K$ is contained in a locally invariant submanifold tangent to $F$,
then the same holds for any diffeomorphism
$C^1$-close to $f$ and any compact set $K'$ contained in a small neighborhood
of $K$.
\end{proposition}

There exists a simple criterion for the existence of a locally invariant submanifold.

\begin{theo}[\cite{BC2}]\label{t.whitney}
Let $K$ be an invariant compact set with a dominated splitting $E^s\oplus F$
such that $E^s$ is uniformly contracted.
Then $K$ is contained in a locally invariant submanifold tangent to $F$
if and only if the strong stable leaves for the bundle $E^s$
intersect the set $K$ in only one point.
\end{theo}

One can deduce a generic version of previous theorem. 

\begin{corollary}\label{c.whitney}
Let $f$ be $C^1$-generic and $H(p)$ be a homoclinic class having a dominated splitting $E^s\oplus F$ such that $E^s$ is uniformly contracted.

Then, either $H(p)$ is contained in a locally invariant submanifold
tangent to $F$ or for any diffeomorphism $g$ that is $C^1$-close to $f$,
there exist two different points $x\neq y$ in $H(p_g)$ such that $W^{ss}(x)=W^{ss}(y)$.
\end{corollary}
\begin{proof}
By~\cite{BoCr},
there exists a dense G$_\delta$ subset $\cR\subset \diff^1(M)$
of diffeomorphisms whose homoclinic classes are chain-recurrence classes.
In particular, for any $f\in \cR$ and any homoclinic class $H(p)$ for $f$,
the class $H(p_g)$ for $g$ $C^1$-close to $f$ is contained in a small neighborhood of $H(p)$. By proposition~\ref{p.whitney}, one deduces that
if $H(p)$ has a dominated splitting $E^s\oplus F$ and
is contained in a locally invariant submanifold tangent to $F$, then the same holds for the classes $H(p_g)$.

As a consequence, for any $f$ in a dense G$_\delta$ subset of $\diff^1(M)$,
and any homoclinic class $H(p)$ of $f$, either for any diffeomorphism $g$
close to $f$ the class $H(p_g)$ is contained in a locally invariant submanifold tangent to $F$
or for any diffeomorphism $g$ close to $f$ the class $H(p_g)$ is not contained in such a manifold.
The theorem~\ref{t.whitney} ends the proof.

\end{proof}
\medskip

The previous result raises an important question for us:

\begin{question*}
When $H(p)$ is not contained in a locally invariant submanifold tangent to $F$,
is it possible to find a periodic point $q$ homoclinically related to the orbit of $p$
whose strong stable manifold $W^{ss}(q)\setminus \{q\}$ intersects $H(p)$?
\end{question*}

Such an intersection is called a generalized strong homoclinic intersection in the next section. We will provide answers for this problem in some particular cases, see
theorems~\ref{t.tot-discontinuity} and~\ref{t.position} below.

\subsection{Strong homoclinic intersections}
Inside a homoclinic class, some periodic points exhibit
a transverse intersection between their stable and unstable manifolds.
If this intersection holds along strong stable and unstable manifolds of
the periodic orbit we say that there is a strong
homoclinic connection. More precisely, we introduce the following definition:

\begin{defi}\label{strong-int}
Given a  hyperbolic periodic orbit $O$ with a dominated splitting $T_OM=E\oplus F$
such that the stable dimension of $O$ is strictly larger (resp. strictly smaller) than $\dim(E)$
it is said that $O$ exhibits a \emph{strong stable homoclinic intersection}
(resp. a \emph{strong unstable homoclinic intersection}) if the invariant manifold of $O$
tangent to $E$ and the unstable manifold of $O$ (resp. the invariant manifold of $O$
tangent to $F$ and the stable manifold of $O$) have an intersection point outside the orbit $O$.
\end{defi}

This definition can be generalized for homoclinic  classes.
\begin{defi}
A homoclinic class $H(p)$ has a \emph{strong homoclinic intersection}
if there exists a hyperbolic periodic orbit orbit $O$ homoclinically related to $p$
which has a strong homoclinic intersection.
\end{defi}

The strong homoclinic intersections allow sometimes to create heterodimensional cycles.
The following statement generalizes~\cite[proposition 2.4]{Pu1}. The proof is similar and we only sketch it.
\begin{prop}\label{p.strong-connection}
Let $H(p)$ be a homoclinic class for a diffeomorphism $f$ such that:
\begin{itemize}
\item[--] $H(p)$ has a dominated splitting $E\oplus F$ and the stable dimension of $p$
is $\dim(E)+1$;
\item[--] there exist some hyperbolic periodic orbits homoclinically related to $p$
having some negative Lyapunov exponents arbitrarily close to $0$.
\end{itemize}
If there exist some diffeomorphisms $g$ $C^1$-close to $f$ such that $H(p_g)$
has a strong homoclinic intersection, then there exist some $C^1$-close perturbations
of $f$ that have an heterodimensional cycle
between a hyperbolic periodic orbit homoclinically related to $p$ and
a hyperbolic periodic orbit of stable dimension $\dim(E)$.
\end{prop}

Before proving this proposition, we explain how it is possible by a $C^r$-perturbation to
transport the strong homoclinic intersection to another periodic orbit.

\begin{lemma}\label{l.change-connection}
Let $H(p)$ be a homoclinic class for a $C^r$-diffeomorphism $f$ with $r\geq 1$ such that:
\begin{itemize}
\item[--] $H(p)$ has a dominated splitting $E\oplus F$ and the stable dimension of $p$
is $\dim(E)+1$;
\item[--] $H(p_g)$ has a strong homoclinic intersection.
\end{itemize}
Then for any periodic point $q$ homoclinically related to $p$
there exist some $C^r$-close perturbations of $f$ that have a periodic point $q'$
homoclinically related to the orbit of $p$ which exhibit a strong homoclinic intersection
and whose minimal Lyapunov exponents along $F$ are close to the one of $q$.
\end{lemma}
\begin{proof}
Let us consider a transverse intersection point
$z_s$ between $W^{s}(O)$ and $W^u(q)$ and a transverse intersection point $z_u$ between $W^u(O)$
and $W^s(q)$ where $O$ is the orbit of $p$. There exists a transitive hyperbolic set $K$ which contains $z_s,z_u,O$
and which is included in a small neighborhood $U$ of $\overline{\{f^n(z_s)\}_{n\in \ZZ}
\cup\{f^n(z_u)\}_{n\in \ZZ}}$.
One deduces that there exists a sequence of periodic points $(q_n)$ converging to $p$
and whose orbit is contained in $U$ and homoclinically related to $p$.
One may choose these orbits in such a way that they spend most of their iterates
close to the orbit of $q$. Note that $K$ has a dominated splitting of the form
$E\oplus E^c\oplus F'$ where $E^c$ is one-dimensional
and $E\oplus E^c, F'$ respectively coincide with the stable and the unstable bundle.
As a consequence the minimal Lyapunov exponents of $q_n$ along $E$ are arbitrarily close to
the corresponding exponent of $q$ when $n$ is large.

For a small $C^r$ perturbation $g$ supported in a small neighborhood of $\zeta$
(hence disjoint from $K$), one can first ensure that $T_\zeta W^u(O)\oplus E_\zeta$ is one-codimensional
and then consider a small arc of diffeomorphisms $(g_t)$ which coincides with $g$ when $t=0$ and
which unfolds the strong intersection:
in a neighborhood of $\zeta$ the strong homoclinic intersection has disapeared for $t\neq 0$.
The local unstable manifold and the local manifold tangent to $E$
for $q_n$ accumulate on the local unstable manifold and the local manifold tangent to $E$ for $O$
respectively. One thus deduces that for a diffeomorphism $C^r$ close to $g$
and $n$ large enough, the strong stable and the unstable manifolds of the orbit of $q_n$ intersect.
This gives the conclusion for $q'=q_n$.
\end{proof}

\begin{proof}[Sketch of the proof of proposition~\ref{p.strong-connection}]
Let us fix $\varepsilon>0$ and a periodic point $q$ homoclinically related to the orbit of $p$
and whose minimal Lyapunov exponent along $F$ belongs to $(-\varepsilon,0)$.
Let $g$ be a diffeomorphism $C^1$-close to $f$ and
$O$ be a periodic orbit homoclinically related to the continuation
$p_{g}$ of $p$ for $g$ which exhibits a strong homoclinic intersection $y$ between its unstable manifold
and its invariant manifold tangent to $E$.
By lemma~\ref{l.change-connection}, one can find a small $C^1$-perturbation $g_1$
having a periodic point $q_1$ homoclinically related to $p_{g_1}$, whose minimal
Lyapunov exponent along $F$ belongs to $(-\varepsilon,0)$ and which exhibits a strong homoclinic intersection.

Let us consider a local stable manifold $\cD$ of $q_1$.
Since $q_1$ has a stable exponent close to $0$, one can by $C^1$-perturbation $g'$
(as small as one wants if one chooses $\varepsilon$ and $q$ accordingly) create inside $\cD$ a hyperbolic periodic point
$q'$ of stable dimension $\dim(E)$. Since $\cD$ has dimension $\dim(E)+1$, one can also require
that $\cD$ contains finitely many periodic points of stable dimension $\dim(E)+1$,
close to $q_1$, whose stable sets cover a dense subset
of $\cD$. If the perturbation is realized in a small neighborhood of $q_1$,
the manifold $W^u(p_{g'})$ intersects transversally $\cD$, hence one can ensure that
the unstable manifold of $p_{g'}$ intersects transversally the stable manifold of a periodic
point $q''$, so that $q''$ and $p_{g'}$ are homoclinically related.
The stable manifold of $q''$ intersects the unstable manifold of $q'$ along an orbit
contained in $\cD$. Since the local invariant manifolds of $q',q''$ are close to those
of $q_1$, one can by a small perturbation close to the strong homoclinic intersection of $q_1$
create an intersection between $W^u(q')$ and $W^{s}(q'')$.
This gives a heterodimensional cycle associated to the periodic orbit $q''$
that is homoclinically related to $p_{g'}$.
\end{proof}
\medskip

If a homoclinic class $H(p)$ contains two hyperbolic periodic points $q,q'$ homoclinically related
to $p$ such that the strong stable manifold $W^{ss}(q)\setminus \{q\}$ and the unstable manifold $W^u(q')$ intersect, one can create
a strong homoclinic intersection by a $C^r$-perturbation, for any $r\geq 1$.
We have a more general result.

\begin{lemma} \label{joint-int-easy} Let $f$ be a $C^r$-diffeomorphism, $r\geq 1$ and
let $q, p_x, p_y$ be three periodic points whose
orbits are homoclinically related such that 
\begin{itemize}
\item[--] the homoclinic class $H(q)$ has a dominated splitting $T_{H(q)}M=E^s\oplus F$ and
$\dim(E^s)$ is strictly smaller than the stable dimension of $O$;
\item[--] there are two distinct transversal intersection points $x\in W^u(p_x)\trans W^s(q)$, $y\in  W^u(p_y)\trans W^s(q)$ sharing the same strong stable leaf.
\end{itemize}
Then for any $r\geq 1$, there is $g$ $C^r$-close to $f$ such that $H(q_g)$ has a strong homoclinic intersection.
 \end{lemma}
\begin{proof}
One can assume that $y$ is distinct from $q$.
There is a transitive hyperbolic set $\La$
that contains $p_x$, $p_y$, $x$ and $q$ but not $y$.
So, it follows that there is a periodic point $\hat q$ homoclinically related to $p$ arbitrarily close to $x$
and whose orbit is close to $\La$ in  the Hausdorff topology.
One deduces that the local strong stable manifold of $\hat q$ and the local
unstable unstable manifold of the orbit of $\hat q$ are close to $y$.
By a $C^r$-perturbation, one can thus create an intersection at $y$, hence a strong connection between these manifolds,
keeping the transverse homoclinic orbits with $p$.
This shows that $H(q_g)$ has a strong homoclinic intersection for this new
diffeomorphism $g$.
\end{proof}
\medskip

We generalize again the definition of strong homoclinic intersection.

\begin{defi}
A homoclinic class $H(p)$ has a \emph{generalized strong homoclinic intersection}
if there exists a hyperbolic periodic orbit orbit $O$ homoclinically related to $p$,
having a dominated splitting $T_OM=E\oplus F$
such that the stable dimension of $O$ is strictly larger (resp. strictly smaller) than $\dim(E)$,
and whose invariant manifold tangent to $E$ (resp. to $F$) contains a point $z\in H(p)\setminus O$.
\end{defi}

Using the $C^1-$connecting lemma due to Hayashi, the following result holds immediately. 
\begin{prop}\label{p.generalized-strong-connection}
Let $H(p)$ be a homoclinic class for a diffeomorphism $f$
which has a generalized strong homoclinic intersection.
Then, there exist some $C^1$-close diffeomorphisms $g$ such that
$H(p_g)$ has a strong homoclinic intersection.
\end{prop}

One may wonder if this last result still holds in $C^r$-topologies for $r>1$.
We have a result in this direction under stronger assumptions.
The proof is much less elementary than the previous ones and will be obtained
as a corollary of theorem~\ref{t.stable} at the end of section~\ref{proofjointint}.

\begin{prop}\label{p.generalized-strong-connectionCr}
For any diffeomorphism $f_0$ and any homoclinic class
$H(p)$ which is a chain-recurrence class
endowed with a partially hyperbolic structure $E^s\oplus E^c\oplus E^u$, $\dim(E^c)=1$,
such that $E^s\oplus E^c$ is thin trapped,
there exists $\al_0> 0$ and a $C^1$-neighborhood $\cU$ of $f_0$
with the following property.

For any $\al\in [0,\al_0]$ and any $C^{1+\al}$-diffeomorphism $f\in \cU$ such that $H(p_f)$
has a generalized strong homoclinic intersection,
there exists a diffeomorphism $g$ arbitrarily $C^{1+\al}$-close to $f$ such that
$H(p_g)$ has a strong homoclinic intersection.
\end{prop}

\subsection{Total disconnectedness along the center-stable plaques}
Let us consider a chain-hyperbolic homoclinic class $H(p)$. In certain part of the proof of the main theorem, we need a better understanding on the geometrical properties of the class in order, for instance,  
to build analogs of Markov partitions. To do that, we
need to ensure that the intersection of $H(p)$ with its center-stable plaques is totally
disconnected. By lemma~\ref{l.uniqueness-coherence} this property does not depend on the choice of a center-stable plaque family.
It is provided by the following result proved in section~\ref{s.2D-central}.

\begin{theo}\label{t.tot-discontinuity}
Let $f$ be a diffeomorphism and $H(p)$ be a chain-hyperbolic homoclinic class
with a dominated splitting $E^{cs}\oplus E^{cu}=(E^{ss}\oplus E^c_1)\oplus E^c_2$ such that $E^c_1,E^c_2$ are one-dimensional and $E^{cs}$ and $E^{cu}$ are thin trapped.
Then, one of the following cases holds.
\begin{itemize}
\item The strong stable manifolds (tangent to $E^s$) intersect the class in at most one point.
\item There exists a periodic point $q$ in $H(p)$ whose strong stable manifold
$W^{ss}(q)\setminus\{q\}$ intersects $H(p)$.
\item The class is totally disconnected along the center-stable plaques.
\end{itemize}
\end{theo}
Under this general setting the point $q$ is not necessarily homoclinically related to $p$.
Note that this theorem also applies and may be interesting for locally maximal hyperbolic sets $K$ having
a dominated splitting $T_KM=E^s\oplus E^u=(E^{s}\oplus E^c)\oplus E^u$ such that
$E^c,E^u$ are one-dimensional.

\subsection{Extremal bundles}\label{ss.extremal}
Theorems~\ref{t.homoclinic} and~\ref{t.aperiodic}
show that the chain-recurrence classes $K$ of a $C^1$-generic
diffeomorphism far from homoclinic bifurcations have a partially hyperbolic
splitting $T_KM=E^s\oplus E^c\oplus E^u$ with $\dim(E^c)\leq 2$.
We now prove that the extremal bundles are non-degenerated.
This will ensure that the diffeomorphisms considered in the main theorem have
only finitely many sinks.

For aperiodic classes this has already been obtained with corollary~\ref{c.aperiodic}.
For homoclinic classes one can apply the following result.

\begin{teo}\label{t.extremal}
Let $f$ be a diffeomorphism in a dense G$_\delta$ subset of $\diff^1(M)$
and let $H(p)$ be a homoclinic class endowed with a partially hyperbolic splitting
$T_{H(p)}M=E^s\oplus E^c_1\oplus E^c_2\oplus E^u$, with $\dim(E^c_1)\leq 1$ and $\dim(E^c_2)\leq 1$.
Assume moreover that the bundles $E^s\oplus E^c_1$ and $E^c_2\oplus E^u$
are thin trapped by $f$ and $f^{-1}$ respectively and that the class is contained in a locally invariant
submanifold tangent to $E^s\oplus E^c_1\oplus E^c_2$.

Then one of the two following cases occurs:
\begin{itemize}
\item[--] either $H(p)$ is a hyperbolic set,
\item[--] or there exists diffeomorphisms $g$
arbitrarily $C^1$-close to $f$ with a periodic point $q$ homoclinically related
to the orbit of $p_g$ and exhibiting a heterodimensional cycle.
\end{itemize}
\end{teo}
\begin{remark}
We will see in section~\ref{ss.consequences} that the result can be improved: the second
case of the theorem never appears.
\end{remark}

The proof relies on techniques developed in \cite{PS1, PS2, PS4}  for $C^2$-diffeomorphisms
that extend a result in \cite{mane-contribution-stabilite} for one-dimensional endomorphisms.
We list different settings that have been already studied.

\paragraph{ a) The surface case.}
For $C^2$-maps, the non-hyperbolic transitive sets which have a dominated splitting
contain either a non-hyperbolic periodic point
or a curve supporting the dynamics of an irrational rotation.

\begin{theo}[\cite{PS1}]\label{t.KS1}
Let $f$ be a $C^2$ diffeomorphism of surface and $K$ be a compact invariant set
having a dominated splitting $T_KM=E\oplus F$, $\dim(F)=1$ whose periodic orbits are all hyperbolic. Then, one of the following cases occur.
\begin{itemize}
\item[--] $K$ contains a sink or a compact invariant one-dimensional submanifold
tangent to $F$.
\item[--] $F$ is uniformly contracted by $f^{-1}$.
\end{itemize}
\end{theo}
One deduces the following generic result.

\begin{corollary}\label{c.PS1}
Let $f$ be a $C^1$-generic diffeomorphism and $K$ be a
partially hyperbolic set endowed with a dominated splitting $T_KM=E^s\oplus E^c_1\oplus E^c_2\oplus E^u$, with $\dim(E^c_1)=\dim(E^c_2)=1$.

If $K$ is contained in a locally
invariant surface tangent to $E^c_1\oplus E^c_2$
and does not contain a periodic orbit of stable dimension
$\dim(E^s)$ or $\dim(E^s)+2$, then $K$ is hyperbolic.
\end{corollary}
\noindent
Note that a periodic orbit of stable dimension $\dim(E^s)$ or $\dim(E^s)+2$
is a source or a sink in the surface.
If $K$ is transitive and non trivial, it does not contain such a periodic orbit.
\begin{proof}
By proposition~\ref{p.whitney} and theorem~\ref{t.whitney},
the property for a partially hyperbolic set to be contained in a locally invariant
surface tangent to $E^c_1\oplus E^c_2$ is robust.
It is thus enough to consider open sets $\cU\subset \diff^1(M)$,
$U\subset M$ and a (non necessarily invariant) compact set $\Lambda\subset U$
such that for each $f\in \cU$ any invariant compact set $K$ contained in $U$ has a dominated splitting $T_KM=E^s\oplus E^c_1\oplus E^c_2\oplus E^u$ and is contained in a locally invariant surface tangent to $E^c_1\oplus E^c_2$: we have to obtain the conclusion of the theorem for an open and dense subset of diffeomorphisms in $\cU$ and invariant compact sets contained in $\Lambda$.
A standard Baire argument then concludes that the theorem holds for $C^1$ generic diffeomorphisms.

Let us fix a diffeomorphism $f_0\in \cU$ and consider the maximal invariant set $K_0$ in a small closed
neighborhood of $\Lambda$.
By assumption it is contained in a locally invariant surface $\Sigma_0$
tangent to $E^c_1\oplus E^c_2$. One can conjugate $f_0$ by a diffeomorphism which sends $\Sigma_0$
on a smooth surface $\Sigma$ and approximate the obtained diffeomorphism $f_1$ by a smooth diffeomorphism.
By this new diffeomorphism, the smooth surface $\Sigma$ is mapped on a smooth surface $f_1(\Sigma)$
which is $C^1$-close to $\Sigma$. As a consequence, there exists a smooth diffeomorphism $f_2$
that is $C^1$-close to $f_1$ which preserves $\Sigma$.
One deduces that the maximal invariant set $K_2$ for $f_2$ in a small neighborhood of $\Lambda$ is contained in
$\Sigma$.
One can perturb the restriction of $f_2$ to a neighborhood of $K_2$
in $\Sigma$ and obtain a smooth Kupka-Smale diffeomorphism
without any invariant one-dimensional submanifold supporting the dynamics of an irrational rotation. 
This perturbation can be extended to a smooth diffeomorphism of $M$: indeed the compactly supported diffeomorphism close to the identity
in $\Sigma$ are isotopic to the identity and can be extended in a trivializing neighborhood of $\Sigma$
as a compactly supported diffeomorphism close to the identity.

At this point we have built a smooth diffeomorphism $f_3$ that is $C^1$-close to $f$
and an invariant smooth surface $\Sigma$ which contains the maximal invariant set $K_3$ of $f_3$
in a small neighborhood of $\Lambda$. Moreover all the periodic orbits in $K_3$ are hyperbolic
and the dynamics inside any invariant one-dimensional submanifold of $K_3$ is Morse-Smale.
Theorem~\ref{t.KS1} then shows that any orbit in $K_3$ accumulates on a hyperbolic set.
Now, for any diffeomorphism $C^1$-close to $f_3$, the dynamics contained in a small neighborhood of
$\Lambda$ is hyperbolic: it contains a hyperbolic set $L$ of stable dimension $\dim(E^s)+1$,
a finite collection of hyperbolic periodic orbits $O_1,\dots,O_s$ of stable dimension $\dim(E^s)$ or $\dim(E^s)+2$ and
any other orbit accumulates in the future and in the past on $L\cup O_1\cup\dots O_s$.
\end{proof}

\paragraph{ b) The one-codimensional case.}
This has been considered for homoclinic classes.

\begin{theo}[\cite{PS4}]\label{codimension-one}
Let $f$ be a $C^2$ diffeomorphism and $H(p)$ be a homoclinic class
endowed with a partially hyperbolic splitting $E^s\oplus E^c$ with $\dim(E^c)=1$
whose periodic orbits are hyperbolic. Then $H(p)$ is hyperbolic.
\end{theo}
As before, this gives the following generic result
(which is a particular case of theorem~\ref{t.extremal}).

\begin{corollary}\label{c.codim-one}
For any $C^1$-generic diffeomorphism,
any homoclinic class $H(p)$ that is
\begin{itemize}
\item[--] endowed with a partially hyperbolic splitting $E^s\oplus E^c\oplus E^u$, $\dim(E^c)=1$,
\item[--] contained in a locally invariant submanifold tangent to $E^s\oplus E^c$,
\end{itemize}
is hyperbolic.
\end{corollary}
\begin{proof}
Consider a $C^1$-generic diffeomorphism $f$
and a homoclinic class $H(p)$ as stated in the corollary
and $\Sigma$ the locally invariant submanifold tangent to $E^s\oplus E^c$
containing $H(p)$.
By genericity, one can suppose that
the class $H(p)$ is a chain-recurrence class and that for any diffeomorphism
$g$ close to $f$, the class $H(p_g)$ is contained in a small neighborhood of $H(p)$.
Moreover, if for some arbitrarily close diffeomorphisms $g$ the chain-recurrence class
containing $p_g$ is hyperbolic, then the class $H(p)$ for $f$ is also hyperbolic.

Let us consider a $C^2$-diffeomorphism
$g$ arbitrarily close to $f$ in $\diff^1(M)$ and whose periodic orbits
are hyperbolic. By proposition~\ref{p.whitney}, the chain recurrence class $\Lambda$ containing $p_g$ is still contained in a locally invariant submanifold $\Sigma_g$.
As in the proof of corollary~\ref{c.PS1}, one may have chosen $g$ so that
$\Sigma_g$ is a smooth submanifold.
Let us assume by contradiction that $\Lambda$ is not hyperbolic:
there exists an invariant compact set $K\subset \Lambda$
that is not hyperbolic and that is minimal for the inclusion.
Since $K$ coincides with the support of an ergodic measure
whose Lyapunov exponent along $E^c$ is non-positive,
the set $K$ is transitive.
The set $K$ cannot be a sink, nor contain
an invariant one-dimensional submanifold tangent to $ E^c$,
since by transitivity the set $K$ would be reduced to a sink or
a union of normally attracting curves in $\Sigma_g$, contradicting the fact that $\Lambda$ is chain-transitive and contains $p_g$.
One can thus apply~\cite[lemma 5.12]{PS4} and conclude that $K$ is contained in a homoclinic
class $H(q)$. Since $H(q)$ is contained in a small neighborhood of $H(p)$,
it is contained in $\Sigma_g$. By theorem~\ref{codimension-one} applied for $g$
inside $\Sigma_g$, one deduces that $H(q)$ is a hyperbolic set. This
contradicts the fact that $K$ is non hyperbolic.
As a consequence, the chain-recurrence class containing $p_g$ is hyperbolic,
hence coincides with $H(p)$.
This proves that the homoclinic class $H(p)$ is hyperbolic.
\smallskip
\end{proof}

\paragraph{ c) The $2$-codimensional case.} For homoclinic classes
with two-codimensional strong stable bundle, one can replace the uniformity of the center stable
bundle by the thin trapping property and the total disconnectedness along the center stable plaques.
This theorem is proved in section~\ref{s.2D-central2}. 

\begin{theo}\label{t.2D-central}
 Let $f_0$ be a diffeomorphism and $H(p_{f_0})$ be a chain-recurrence class which is
a chain-hyperbolic homoclinic class  endowed with a dominated splitting $E^{cs}\oplus E^{cu}$ such that $E^{cu}$ is one-dimensional and $E^{cs},E^{cu}$ are thin trapped (for $f_0$ and $f_0^{-1}$ respectively).
Assume moreover that the intersection of $H(p_{f_0})$ with its center-stable plaques is totally disconnected.

Then, for any $C^2$ diffeomorphism $f$ that is close to $f_0$ in $\diff^1(M)$
and for any $f-$invariant compact set $K$ contained in a small neighborhood of $H(p_{f_0})$ and whose
periodic orbit are hyperbolic, one of the following cases occurs.
\begin{itemize}
\item[--] $K$ contains a sink or a compact invariant one-dimensional submanifold
tangent to $E^{cu}$.
\item[--] $E^{cu}$ is uniformly contracted by $f^{-1}$.
\end{itemize}
\end{theo}
\medskip

We can now prove that for $C^1$-generic diffeomorphisms far from
homoclinic bifurcations, the extremal sub-bundles of the homoclinic classes
are non-degenerated.

\begin{proof}[Proof of theorem~\ref{t.extremal}]
As before, one can assume that, for $g$  close to $f$,
the class $H(p_g)$ is contained in a small neighborhood of $H(p)$.
Moreover, if for some arbitrarily close diffeomorphisms $g$ the chain-recurrence class
containing $p_g$ is hyperbolic, then the class $H(p)$ for $f$ is hyperbolic.
The following several cases have to be considered.
\smallskip

Note first that when the bundle $E^{c}_1$ or $E^c_2$ is degenerated,
corollary~\ref{c.codim-one} implies that $H(p)$ is a hyperbolic set.
\smallskip

When the strong stable leaves intersect the class in at most one point,
theorem~\ref{t.whitney} implies that the class is contained in a
locally invariant submanifold tangent to $E^c_1\oplus E^{c}_2$.
By corollary~\ref{c.PS1} the class is then hyperbolic.
\smallskip

When the intersection of the class with the center stable plaques
is totally disconnected, one can apply theorem~\ref{t.2D-central}.
For any $C^2$ diffeomorphisms $g$ $C^1$-close to $f$ in $\diff^1(M)$
with hyperbolic periodic orbits, the chain-recurrence class containing
$p_g$ is hyperbolic. As a consequence $H(p)$ is hyperbolic.
\smallskip

It remains the case that both bundles $E^c_1,E^{c}_2$ are one-dimensional,
$E^c_1$ is not uniformly contracted,
the class contains two different point in a same strong stable leaf
and the intersection of the class with the center stable plaques is not
totally disconnected.
One can then apply theorem~\ref{t.tot-discontinuity} when the dynamics is restricted to a
locally invariant submanifold tangent to $E^s\oplus E^c_1\oplus E^{c}_2$
and one deduces that the class has a generalized strong homoclinic intersection.

By lemma~\ref{l.weak} and remark~\ref{r.weak} the class contains hyperbolic periodic orbits homoclinically related to $p$
and whose Lyapunov exponent along $E^c_1$ is arbitrarily close to zero.
One concludes applying the propositions~\ref{p.generalized-strong-connection}
and~\ref{p.strong-connection} and creating a heterodimensional cycle 
associated to a periodic orbit homoclinically related to $p$.
\end{proof}

\subsection{Finiteness of quasi-attractors}\label{ss.finiteness}
We now consider the \emph{quasi-attractors} 
and prove one part of the main theorem.

\begin{prop}\label{p.finiteness}
For any $C^1$-generic diffeomorphism that is far from homoclinic tangencies and heterodimensional cycles, the union of all the quasi-attractors is closed.
\end{prop}
\begin{proof}
Consider a sequence of quasi-attractors $(A_n)$ which converges towards a (chain-transitive) set $L$.
By theorem~\ref{t.homoclinic}, they are homoclinic classes $A_n=H(p_n)$
and one can assume that all the periodic orbits $p_n$ have the same dimension.
\medskip

\noi{\it Claim 1.
$L$ is a contained in a homoclinic class $H(p)$.}
\begin{proof}
If $L$ is contained in an aperiodic
class, by theorem~\ref{t.aperiodic}
it has splitting $T_LM=E^s\oplus E^c\oplus E^u$ with $\dim(E^c)=1$.
So, this is the same for the classes $A_n$.
Since the classes $A_n$ are quasi-attractor,
they are saturated by strong unstable leaves,
and therefore the same holds for $L$.
This contradicts corollary~\ref{c.aperiodic}.
\end{proof}
\medskip

By theorem~\ref{t.homoclinic}, the class
$H(p)$ has a dominated splitting $E^s\oplus E^c_1\oplus E^c_2\oplus E^u$
where $E^c_1$ and $E^c_2$ have dimension $0$ or $1$.
We assume by contradiction that $H(p)$ is not a quasi-attractor.
\medskip

\noi{\it Claim 2.
The stable dimension of the periodic points $p_n$ is strictly
larger than the stable dimension of $p$.}
\begin{proof}
Let us consider some plaque families $\cW^{cs},\cW^{cu}$
over the maximal invariant set in a neighborhood of $L$
and tangent to $E^s\oplus E^c_1$ and $E^{c}_2\oplus E^u$ respectively,
as in the definition of chain-hyperbolic class.
Let us assume by contradiction that the stable dimension of $p_n$ is smaller or equal to
the stable dimension of $p$.

We claim that for any periodic point $q_n$ homoclinically related to the orbit of $p_n$,
one has $\cW^{cu}_{q_n}\subset W^u(q_n)$.
Indeed if it is not the case, using that $\cW^u$ is trapped by $f^{-1}$,
there would exists a periodic point $q'_n\in \cW^{cu}_{q_n}$,
in the closure of $W^u(q_n)$ and whose stable dimension
is $\dim(\cW^{cu}_{q_n})-1$. Since $H(p_n)$ is a quasi-attractor it contains $q'_n$
and by lemma~\ref{l.prelim}, there exist $C^1$-perturbations of $f$ which exhibit a
heterodimensional cycle. This is a contradiction.

In particular, one has $\cW^{cu}_{q_n}\subset H(p_n)$ and, passing to the limit,
the set $L$ contains all the plaques $\cW^{cu}_x$, $x\in L$.
Let us consider any periodic point $q$ homoclinically related to the orbit of $p$
and close to $L$ such that $\cW^{cs}_q\subset W^s(q)$.
This exists by lemma~\ref{l.contper}.
The plaques $\cW^{cs}_q$ and $\cW^{cu}_x$ for some $x\in L$ intersect transversally,
hence the forward iterates of $\cW^{cu}_x$ accumulate $W^u_q$.
One thus deduces that $H(p)$ contains $W^u(q)$.
By lemma~\ref{l.prelim} $H(p)$ is thus a quasi-attractor, contradicting our assumption.
\end{proof}
\medskip

We are thus reduced to consider the case that on the union of the $A_n$ and $H(p)$
there exists a dominated splitting $TM=E^{cs}\oplus E^c\oplus E^u$
such that $E^c$ is one-dimensional, $E^{cs}\oplus E^c$ is thin trapped by $f$ over each quasi-attractor
$A_n$ and $E^{c}\oplus E^{cu}$ is thin trapped by $f^{-1}$ over $H(p)$.
We also fix a point $z\in L$ and a small neighborhood $U$ of $z$.
\medskip

\noi{\it Claim 3.
In each set $A_n$ there exists  a periodic orbit $O_n$
which avoids $U$.}
\begin{proof}
By theorem~\ref{t.extremal}, the bundle $E^u$ is non-degenerated
and the set $A_n$ is saturated by strong unstable leaves.
By a standard argument (see for instance~\cite[lemma 5.2]{mane-contribution-stabilite}),
each class $A_n$
contains an invariant compact set $K_n$ which avoids $U$.
Then one can reduce $K_n$ and assume that it is minimal.

Let us consider two plaque families $\tilde \cW^{cs}, \tilde \cW^u$
tangent to $E^{cs}\oplus E^c$ and $E^u$
with arbitrarily small diameter, above the maximal invariant set
in a small neighborhood of $A_n$
and whose restriction to $A_n$ satisfy the definition of chain-hyperbolic classes.
By the closing lemma, there exists a periodic orbit $\tilde O_n$
arbitrarily close to $K_n$ in the Hausdorff topology.
By lemma~\ref{l.contper}, there exists a point $q$ homoclinically related to the orbit of $p$
such that $\tilde \cW^{u}_q$ is contained in $W^u(q)$ and
intersects $\tilde \cW^{cs}_y$ for some $y\in \tilde O_n$ at a point $\zeta$.
One deduces that $\zeta$ converges toward a periodic orbit $O_n$ contained in the plaques of the family $\tilde \cW^{cs}$ above $\tilde O_n$.
Since $A_n$ is a quasi-attractor, it contains $\zeta$ and $O_n$.
By construction $O_n$ is included in an arbitrarily small neighborhood of $K_n$,
as required.
\end{proof}
\medskip

\noi{\it Claim 4.
There exists $N\geq 1$ such that
$f^N(W^u_{loc}(p))$ intersects transversally $W^s_{loc}(O_n)$
for each $n$ large.
Moreover, this property is stable under $C^1$-perturbations
with supports avoiding a neighborhood of $O_n$.}
\begin{proof}
Since the stable space of $O_n$ is $E^{cs}\oplus E^c$
and since $E^c$ is non-degenerate,
all the exponents of $O_n$ along $E^{cs}$ are bounded away from zero.
By lemma~\ref{l.largestable} and remark~\ref{r.large}, the orbit
$O_n$ contains a point $q_n$ such that
$\cW^{cs}_{q_n}\subset W^s_{loc}(q_n)$.
For $N$ large, $f^N(W^u_{loc}(p))$
is close to any point of $L$.
For $n$ large, $O_n$ is contained in a small neighborhood of $L$.
One thus deduces that $f^N(W^u_{loc}(p))$ and $W^s_{loc}(q_n)$ intersect transversally
and this property is robust under perturbations
with supports avoiding a neighborhood of $O_n$.
\end{proof}
\medskip

\noi{\it Conclusion.}
Since $W^u(O_n)$ is dense in $A_n$,
and $A_n$ converges towards $L$, the unstable manifold of $O_n$
has a point close to $z$ for $n$ large.
Since $z$ is in $H(p)$, the stable manifold of $p$ has a point close to $z$. Observing that the orbits of $O_n$ are far from the neighborhood $U$ of $z$,
there exists small perturbations given by the connecting lemma
in a small neighborhood of a finite number of iterates of $z$, such that
$W^s(p)$ and $W^u(O_n)$ intersect. The orbit of $O_n$ has been preserved
and the intersection $W^s(O_n)\cap W^u(p)$ is still non empty.
This gives a heterodimensional cycle and therefore a contradiction.
As a consequence $H(p)$ is a quasi-attractor.
\end{proof}

\begin{remark}
In the case the quasi-attractors $A_n$ are non-degenerated, $E^u$ coincides with the
unstable dimension of the periodic points in the sets $A_n$; hence we already know that
$E^u$ is non-degenerated. Theorem~\ref{t.extremal} is thus needed only to guarantee that
the sinks of $f$ accumulate on quasi-attractors.
\end{remark}

\subsection{Hyperbolicity of quasi-attractors: proof of the main theorem}
\label{ss.quasi-attractor}
It remains now to prove that
for any $C^1$-generic diffeomorphism that is far from homoclinic tangencies and heterodimensional cycles,
the quasi-attractors are hyperbolic.
The proof is independent from proposition~\ref{p.finiteness}.

Let us consider a quasi-attractor and  let us assume by contradiction that it is not hyperbolic.
From sections~\ref{ss.homoclinic} and~\ref{ss.aperiodic}, the quasi-attractor  is a homoclinic class $H(p)$
with a splitting $E^s\oplus E^c\oplus E^u$ where $E^c$ is one-dimensional, $E^s\oplus E^c$
is thin-trapped and it contains arbitrarily weak periodic orbits homoclinically related to $p$.
From theorem~\ref{t.extremal} and corollary~\ref{c.whitney}, for any diffeomorphism $g$ that is $C^1$-close to $f$
the homoclinic class $H(p_g)$ associated to the continuation of $p$ for $g$
contains two different points $x,y$ such that $W^{ss}(x)=W^{ss}(y)$.
The end of the proof is based on the next theorem. The first case of the dichotomy
is not satisfied in our setting
and in the second case, one can create a heterodimensional cycle in $H(p)$, by proposition~\ref{p.strong-connection}.
This contradicts our assumptions on the diffeomorphism $f$ and
concludes the proof of the main theorem.
\begin{teo}\label{t.position}
Let $H(p)$ be a homoclinic class of a diffeomorphism $f$ which is a quasi-attractor
endowed with a partially hyperbolic structure $E^s\oplus E^c\oplus E^u$
such that $\dim(E^c)=1$ and $E^{cs}=E^s\oplus E^c$ is thin trapped.
Assume also that all the periodic orbits in $H(p)$ are hyperbolic.
Then, there exists $\al\geq 0$ (which is positive if $f$ is $C^r$ for some $r>0$)
and $C^{1+\al}-$small perturbations $g$ of $f$ such that the homoclinic class associated to the continuation
$p_g$ of $p$ satisfies one of the following cases.
\begin{itemize}
\item[--] Either one has $W^{ss}(x)\neq W^{ss}(y)$ for any $x\neq y$ in $H(p_g)$ and therefore
the class $H(p_g)$ is contained in a $C^1$-submanifold $N\subset M$ tangent to $E^c\oplus E^u$
which is locally invariant.
\item[--] Or one has $W^{ss}(x)= W^{ss}(y)$ for some hyperbolic periodic point $x$ homoclinically
related to the orbit of $p_g$ and some $y\neq x$ in $H(p_g)\cap W^u(x)$ and therefore
the class $H(p_g)$ has a strong homoclinic intersection.
\end{itemize}
\end{teo}

\begin{obs}\label{r.position} We want to emphasize some features of theorem~\ref{t.position}.

\begin{enumerate}

\item The result does not require any generic assumption.

\item It holds in the $C^{1+\al}-$category for $\alpha>0$ small.

\item The theorem can also be applied to the context of hyperbolic attractors
whose stable bundle has a dominated splitting $E^s=E^{ss}\oplus E^c$ such that $\dim(E^c)=1$.
This can have important consequences in terms of the Hausdorff dimension of the attractor: if the the attractor is contained in a submanifold, the Hausdorff dimension is smaller than $1+u$ (where $u=dim(E^u)$); if there is a strong connection, the dimension could jump close to $1+u+s$ (where $s=dim(E^{ss})$)
(see~\cite{BDV-hausdorff}).
Note that the proof in the hyperbolic case is simpler
since we can use the hyperbolic continuation of any point in the attractor.

\item\label{i.position} In the case the bundle $E^c$ is not uniformly contracted, one can assume that
the periodic point $x$ has an arbitrarily small Lyapunov exponent.
Indeed by lemma~\ref{l.weak} and remark~\ref{r.weak}, for any $\varepsilon>0$, there exists a
periodic point $q$ homoclinically related to the orbit of $p$ and whose central Lyapunov exponent is contained
in $(-\varepsilon,0)$. Let us consider a perturbation $g$ having a periodic point $x$ homoclinically related to $p_g$ and exhibiting a strong homoclinic intersection. By another $C^r$-small perturbation
(see lemma~\ref{l.change-connection}), one can obtain a periodic point $x'$ homoclinically related to the orbit of $p$, with a strong connection and a central Lyapunov exponent close to the exponent of $q$.
\end{enumerate}

\end{obs}

The proof of theorem~\ref{t.position} is based on the following proposition
whose proof is postponed to section~\ref{s.boundary}
and uses the notion of \emph{stable boundary point} (see section~\ref{s.weak-hyperbolicity})
and of \emph{continuation of a homoclinic class} (see section~\ref{s.continuation}).
As before $W^{ss}_{loc}(x)$ and $W^u_{loc}(x)$ denote local stable and unstable
manifolds tangent to $E^s_x$ and $E^u_x$ respectively for the points $x\in H(p_g)$.
Note that this result holds in any $C^r$-topology, $r\geq 1$.

\begin{prop}\label{p.position}
Given a  $C^r$ diffeomorphisms under the assumptions of theorem~\ref{t.position}, for any $\al\in[0,r-1]$
one of the following cases occurs.
\begin{enumerate}
\item There exists $g$, $C^{1+\al}$-close to $f$, such that
for any $x\neq y$ in $H(p_g)$, one has $W^{ss}(x)\neq W^{ss}(y)$.
\item There exists $g$, $C^{1+\al}$-close to $f$, such that $H(p_g)$
exhibits a strong homoclinic intersection.
\item There exist a neighborhood $\cV\subset \diff^{1+\al}(M)$ of $f$ and
some hyperbolic periodic points $q$ and $p^x_n, p^y_n$ for $n\in \NN$
such that:
\begin{itemize}
\item[--] the continuations $q_g,p^x_{n,g},p^y_{n,g}$ exist on $\cV$ and
are homoclinically related to the orbit of $p_g$;
\item[--] $(p^x_{n,g})$, $(p^y_{n,g})$ converge towards two distinct points $x_g,y_g$ in $H(p_g)\cap W^s_{loc}(q_g)$ for any $g\in \cV$;
\item[--] the map $g\mapsto x_g,y_g$ is continuous at $f$;
\item[--] $y_g\in W^{ss}_{loc}(x_g)$ for any $g\in \cV$.
\end{itemize}
\item There exist two hyperbolic periodic points $p_x,p_y$ homoclinically related to the orbit of $p$ and an open set $\cV\subset \diff^{1+\al}(M)$
whose closure contains $f$, such that for any $g\in \cV$ the class $H(p_g)$ contains
two different points $x\in W^u(p_{x,g})$ and $y\in W^u(p_{y,g})$
satisfying $W^{ss}(x)=W^{ss}(y)$.
\end{enumerate}
\end{prop}
\medskip

One concludes the proof of theorem~\ref{t.position} by discussing the two last cases of the proposition~\ref{p.position}.
The two following theorems, proved in sections~\ref{proofjointint} and~\ref{p-nontransversal}
give a strong homoclinic intersection.

In the first case, the points $x,y$ belong to the stable manifold of a periodic point $q$. 
\begin{theo}\label{t.stable}
For any diffeomorphism $f_0$ and any homoclinic class $H(p)$ which is a chain-recurrence class
endowed with a partially hyperbolic structure $E^s\oplus E^c\oplus E^u$, $\dim(E^c)=1$,
such that $E^s\oplus E^c$ is thin trapped, there exists $\al_0> 0$ and
a $C^1$-neighborhood $\cU$ of $f_0$ with the following property.

For any $\al\in [0,\al_0]$, any diffeomorphism $f\in \cU$
and any $C^{1+\al}$-neighborhood $\cV$ of $f$,
if there exist:
\begin{itemize}
\item[--] some hyperbolic periodic points $q_f$ and $p_{n,f}^x,p_{n,f}^y$ with $n\in \NN$ for $f$
whose hyperbolic continuations $q_g,p_{n,g}^x,p_{n,g}^y$
exist for $g\in\cV$ and are homoclinically related to the orbit of $p_g$,
\item[--] two maps $g\mapsto x_g,y_g$ defined on $\cV$, continuous at $f$,
such that for any $g\in \cV$ the points $x_g,y_g$ belong to $W^{s}(q_g)$,
are the limit of $(p_{n,g}^x)$ and $(p_{n,g}^y)$ respectively
and satisfy $y_g\in W^{ss}_{loc}(x_g)$,
\end{itemize}
then, there exist $C^{1+\al}$-small perturbations $g$ of $f$
such that the homoclinic class $H(p_g)$ exhibits a strong homoclinic intersection.
\end{theo}

In the second case, the points $x,y$ belong to the unstable manifold of periodic points
$p_x,p_y$.
\begin{theo}\label{t.unstable}
For any diffeomorphism $f_0$, for any homoclinic class $H(p)$ which is a chain-recurrence class
endowed with a partially hyperbolic structure $E^s\oplus E^c\oplus E^u$, $\dim(E^c)=1$,
such that $E^s\oplus E^c$ is thin trapped
there exists $\al_0> 0$ and for any hyperbolic periodic points $p_x$, $p_y$
homoclinically related to the orbit of $p$, 
there exists a $C^1$-neighborhood $\cU$ of $f$ with the following property.

Given any $\alpha\in [0,\alpha_0]$ and any $C^{1+\al}$-diffeomorphism $f\in \cU$, if
there exist two different points $x\in W^u(p_{x,f})$ and $y\in W^u(p_{y,f})$ in $H(p_{f})$
satisfying $W^{ss}(x)=W^{ss}(y)$,
then, there exist $C^{1+\al}$-small perturbations $g$ of $f$
such that the homoclinic class $H(p_g)$ exhibits a strong homoclinic intersection.
\end{theo}
\medskip

Some weaker results similar to theorems \ref{t.unstable} and \ref{t.stable}    were obtained in \cite{Pu2} for attracting homoclinic classes in dimension three and assuming strong dissipative properties.

\subsection{Other consequence on quasi-attractor. Main theorem revisited} \label{ss.consequences}
As it was mentioned in the introduction, for $C^1$-generic diffeomorphisms
one obtains a stronger version of theorem~\ref{t.position}.
We point out that what follows in this section  is not used in the proof of our main theorem.

\begin{theorem}\label{t.consequences} 
Let $f$ be a diffeomorphism in a dense G$_\delta$ subset of $\diff^1(M)$
and let $\Lambda$ be a quasi-attractor endowed with a partially hyperbolic splitting
$T_\Lambda M=E^s\oplus E^c\oplus E^u$ with $\dim(E^c)=1$.
If $E^c$ is not uniformly contracted and not uniformly expanded, then $\Lambda$ is a homoclinic class which contains hyperbolic periodic points of both stable dimensions $\dim(E^s)$ and $\dim(E^s)+1$.
\end{theorem}

The proof uses the following result from \cite{BDKS}.
\begin{theorem}[\cite{BDKS}]\label{t.BDK}
Let $f$ be a diffeomorphism that exhibits a heterodimensional cycle
between two hyperbolic periodic points $p,q$
whose stable dimensions differ by $1$.

Then, there exist a $C^1$-perturbation $g$ of $f$ and two transitive hyperbolic sets $K,L$ -
the first contains the hyperbolic continuation $p_g$, the second has same stable dimension as $q$ -
that form a robust cycle: for any diffeomorphism $h$ that is $C^1$-close to $f$,
there exists heteroclinic orbits that join the continuations $K_g$ to $L_g$
and $L_g$ to $K_g$.
\end{theorem}
A consequence of this result is that for any $C^1$-generic diffeomorphism
and any hyperbolic point $p$ of stable dimension $i\geq 2$,
if there exists some small perturbations $g$ of $f$ which exhibits a heterodimensional cycle between a periodic point homoclinically related to $p_g$
and a periodic orbit of stable dimension $i-1$, then the homoclinic class $H(p)$ for $f$
contains periodic points of indices $i-1$.

\begin{proof}[Proof of theorem~\ref{t.consequences}]
The existence of the dominated splitting implies that there is no diffeomorphism $C^1$-close to $f$
which exhibits a homoclinic tangency in a small neighborhood of $\Lambda$.
\medskip

\noindent
\emph{Step 1.} We first prove that $\Lambda$ is a homoclinic class $H(p)$ which contains periodic orbits whose central exponents are arbitrarily close to $0$.
This uses the following.
\smallskip

\noindent
{\it Claim. If $\Lambda$ contains an invariant compact set $K$
such that any invariant measure supported on $K$ has a Lyapunov exponent along $E^c$ equal to $0$,
then $\Lambda$ contains periodic orbits whose central Lyapunov exponent is arbitrarily close to $0$.}
\begin{proof}
The proof is similar to the proof of theorem  9.25 in \cite{C3} and uses proposition 9.23 also in \cite{C3}. See also~\cite{Y}.
\end{proof}
\medskip

Since $E^c\oplus E^u$ is not uniformly expanded, the trichotomy given by~\cite[theorem 1]{C2}
and the previous claim imply that
the class $\Lambda$ contains periodic orbits whose central exponent is negative or arbitrarily close to $0$.
Similarly, since $E^s\oplus E^c$ is not uniformly contracted,
the class $\Lambda$ contains periodic orbits whose central exponent is positive or arbitrarily close to $0$.
In any case $\Lambda$ is a homoclinic class $H(p)$
which contains for any $\delta>0$
some periodic orbits $O^-_{\delta},O^+_{\delta}$ whose central exponent
is respectively smaller than $\delta$ and larger than $-\delta$. From the results in 
\cite{ABCDW} follows that that $H(p)$ contains periodic orbits
whose central exponents are arbitrarily close to $0$.
\medskip

\noindent
\emph{Step 2.} We then show that one can find a diffeomorphism
$C^1$-close to $f$ and a periodic orbit homoclinically related to $p_g$
which exhibits a heterodimensional cycle.

Using the center models introduced in~\cite{C1}, the dynamics along the central bundle $E^c$ can be classified into {\em chain-recurrent, chain-neutral, chain-hyperblic and chain-parabolic}  (see~\cite[section 2.2]{C2} for details).
Since $H(p)$ contains hyperbolic periodic orbits, some types can not occur
(the neutral and the parabolic ones).
Note that since $H(p)$ contains periodic orbits whose central exponent is close to $0$ and since $f$ is $C^1$-generic,
the class $H(p)$ is the limit of periodic orbits of both indices $\dim(E^s)$ and $\dim(E^s)+1$
for the Hausdorff topology.
When the central dynamics has the chain-recurrent type, \cite[proposition 4.1]{C2},
this implies that these periodic orbits are contained in $H(p)$, hence both indices appear in the class.

It reminds to consider a central dynamics which has the chain-hyperbolic type:
equivalently two cases are possible: either $E^s\oplus E^c$ is thin trapped by $f$ or $E^c\oplus E^u$
is thin-trapped by $f^{-1}$. In any case it follows that there exists a diffeomorphism $g$ that is $C^1$-close to $f$
and a periodic point homoclinically related to the continuation $p_g$ which exhibits a heterodimensional cycle:
in the first case, this is a direct consequence of theorem~\ref{t.position}, corollary~\ref{c.whitney},
theorem~\ref{t.extremal} and proposition~\ref{p.strong-connection};
in the second case, one argues as on the proof of corollary~\ref{c.homoclinic}.
\medskip

\noindent
\emph{Step 3.} We then concludes with theorem~\ref{t.BDK} that the class
$H(p)$ contains hyperbolic periodic points of different stable dimension.
\end{proof}

\bigskip

Theorem~\ref{t.extremal} can be combined with theorem~\ref{t.BDK} to get the following improvement.

\begin{theorem-extremal}
Let $f$ be a diffeomorphism in a dense G$_\delta$ subset of $\diff^1(M)$
and let $H(p)$ be a homoclinic class endowed with a partially hyperbolic splitting
$T_{H(p)}M=E^s\oplus E^c_1\oplus E^c_2\oplus E^u$, with $\dim(E^c_1)\leq 1$ and $\dim(E^c_2)\leq 1$.
Assume moreover that the bundles $E^s\oplus E^c_1$ and $E^c_2\oplus E^u$
are thin trapped by $f$ and $f^{-1}$ respectively and that the class is contained in a locally invariant
submanifold tangent to $E^s\oplus E^c_1\oplus E^c_2$.
Then $H(p)$ is a hyperbolic set.
\end{theorem-extremal}

\begin{proof} Arguing by contradiction, from theorem 6 it would be possible to create a heterodimensional cycle involving points of different indexes and from 
theorem \ref{t.BDK} it is get a robust heterodimensional cycle, then for generic diffeomorphisms the center dynamics it is not trapped neither for $f$ nor for $f^{-1};$ a contradiction.

\end{proof}

%% file: weak-hyperbolicity-1110.tex
\section{Properties of chain-hyperbolic homoclinic classes}\label{s.weak-hyperbolicity}
Let $H(p)$ be a homoclinic class which is chain-hyperbolic for a diffeomorphism $f$.
We consider as in the definition~\ref{d.chain-hyperbolic}
the two  periodic points $q_s,q_u\in H(p)$
and the plaque families $\cW^{cs},\cW^{cu}$ respectively tangent to the bundles $E^{cs},E^{cu}$.

\subsection{Periodic points with large stable manifold}
We first give an immediate consequence of the trapping property.
\begin{lemma}\label{l.largemanifold}
Let $O$ be a periodic orbit in $H(p)$.
If there exists a point $q_0\in O$ such that $\cW^{cs}_{q_0}$ is contained in the stable manifold of $q_0$,
then this property holds for any point $q\in O$ and more generally for any point
$z\in W^s(q_0)\cap H(p)$.
\end{lemma}
\begin{proof}
Any point $q\in O$ can be written as $q=f^{-n}(q_0)$ with $n\geq 0$.
By the trapping property, $\cW^{cs}_q$ is contained in $f^{-n}(\cW^{cs}_{q_0})$, hence
in $f^{-n}(W^{s}(q_0))=W^s(q)$.
Any point $z\in W^s(q_0)$ has large forward iterates $f^n(z)$, $n\geq n_0$ which remain close to $O$.
By continuity and the coherence (lemma~\ref{l.uniqueness-coherence})
one deduces that $\cW^{cs}_{f^n(z)}$ is also contained in the stable manifold of $O$.
By the trapping property this also holds for $z$.
\end{proof}

The homoclinic class $H(p)$ contains a dense set of ``good'' periodic points, in the sense which is defined in the next lemma:
\begin{lemma}\label{l.contper}
For any $\delta>0$ small, there exists a dense set $\cP_0\subset H(p)$ of periodic points homoclinically
related to the orbit of $p$ with the following property.
\begin{itemize}
\item[--] The modulus of the Lyapunov exponents of any point $q\in \cP_0$ are larger than $\delta$.
\item[--] The plaques $\cW^{cs}_q$ and $\cW^{cu}_q$ for any point $q\in \cP_0$
are respectively contained in the stable and in the unstable manifolds of $q$.
\end{itemize}
\end{lemma}

\begin{proof}
Let us choose $\delta>0$ such that the modulus of the Lyapunov exponents of $q_s$ and $q_u$
are larger than $2\delta$.
Let $U_s$ and $U_u$ be some small disjoint neighborhoods
of the orbits of $q_s$ and $q_u$ respectively:
there exist some constant $j\geq 1$ such that for any segment of orbit $\{x,\dots, f^{jn}(x)\}$ contained in
$H(p)\cap U_s$ or in $H(p)\cap U_u$, one has for any $u\in E_x$ and $v\in F_x$,
$$\prod_{i=0}^{n-1}\|Df^j_{f^{ij}(x)}.u\|\leq e^{-2\delta nj}\|u\| \text{ and }
\prod_{i=0}^{n-1}\|Df^j_{ij}.v\|\geq e^{2\delta nj}\|v\|.$$

We fix $\varepsilon>0$ small and
consider the periodic orbits $O$ that are homoclinically related to the orbit of $p$ with the following
combinatorics:
there are at least $\frac 1 2(1- \varepsilon).\tau$ consecutive iterates in $U_s$
and at least $\frac 1 2(1- \varepsilon).\tau$ consecutive iterates in $U_u$, where $\tau$ is the period of $O$.
In particular, the maximal Lyapunov exponent of $O$
along $E^{cs}$ is smaller than $-\delta$ and the minimal Lyapunov exponent of $O$
along $E^{cu}$ is larger than $\delta$.
Let us write the orbit $O=\{z,\dots,f^{\tau-1}(z)\}$ as the concatenation
of a segment of orbit $\{z,\dots, f^{m-1}(z)\}$ in $U_s$,
a segment of orbit $\{f^{m+\ell_1}(z),\dots, f^{2m+\ell_1-1}(z)\}$ in $U_u$,
and two other segments of orbit
$\{f^{m}(z),\dots,f^{m+\ell_1-1}(z)\}$ and
$\{f^{2m+\ell_1}(z),\dots,f^{2m+\ell_1+\ell_2-1}(z)\}$,
such that $m\geq \frac 1 2(1- \varepsilon).\tau$,
and $\ell_1,\ell_2\leq \frac \varepsilon 2 \tau$.
Provided $\varepsilon$ is small, at any iterate $z_k=f^k(z)$ with
$0\leq k <m/2$, one has for any $u\in E_{f^k(z)}$ and any $n\geq 0$,
$$\prod_{i=0}^{n-1}\|Df^j_{f^{ij}(z_k)}.u\|\leq e^{-\delta nj}.\|u\|.$$
One deduces that there exists $\rho>0$ such that the ball centered at $z_k$ with radius $\rho$
in $\cW^{cs}_{z_k}$ is contracted by forward iterations so that it is contained in the stable set of $z_k$.

Since the stable set of $q_s$ contains $\cW^{cs}_{q_s}$, there exists $N\geq 2$
such that $f^N(\cW^{cs}_{q_s})$ has a radius smaller than $\rho/2$.
If $\tau$ is large enough, since $\{z,\dots,f^{m-1}(z)\}$ is contained in the neighborhood
$U_s$ of the hyperbolic orbit of $q_s$, there exists an iterate $z_k=f^{k}(z)$,
$0\leq k <\frac m 2 -N$ arbitrarily close to $q_s$.
By continuity of the plaque family $\cW^{cs}$, one deduces that
$f^{N}(\cW^{cs}_{z_k})$ has radius smaller than $\rho$, hence is contained in the stable set of
$f^{N}(z_k)$. Consequently the plaque $\cW^{cs}_{z_k}$ is contained in the stable manifold of $z_k$.
By lemma~\ref{l.largemanifold} for any point $q$ in the orbit $O$, the plaque $\cW^{cs}_q$
is contained in the stable manifold of $q$. Similarly the unstable manifold of $q$ contains the plaque
$\cW^{cu}_q$.
In order to prove the lemma, it remains to show that the union of the orbits $O$ we considered is dense
in $H(p)$:
Indeed any point $x$ in $H(p)$ can be approximated by a hyperbolic periodic point
$q$ whose orbit is homoclinically related to the orbit of $q_s$ and $q_u$. Then there exists a transitive
hyperbolic set which contains the points $q,q_s,q_u,p$. One deduces by shadowing
that there exists a hyperbolic periodic orbit $O$ having a point close to $x$
which is homoclinically related to the orbit of $p$ and has the required combinatorics.
\end{proof}
When the central bundles are one-dimensional, one can control the size of the
invariant manifolds of the periodic orbits whose Lyapunov exponents are far from $0$.

\begin{lemma}\label{l.largestable}
Let us assume that there is a dominated splitting $E^{cs}=E\oplus E^c$ such that
$E^c$ has dimension $1$.
For any $\delta>0$, there exists $\rho>0$ with the following property:
let $O\subset H(p)$ be a periodic orbit whose Lyapunov exponents along $E^{cs}$ are smaller than $-\delta$.
Then, there exists $q\in O$ whose stable set contains the ball centered at $q$ with radius $\rho$.
\end{lemma}
\begin{proof}
Let $O\subset H(p)$ be a hyperbolic periodic orbit whose Lyapunov exponents along $E^{cs}$ are smaller than $-\delta$: since $E^c$ is one-dimensional this implies that there exists $q_0\in O$ such that for each $n\geq 0$ one has
$\|Df^n_{|E^c}(q_0)\|=\prod_{i=0}^{n-1}\|Df_{|E^{c}}(f^{i}(q_0))\|
\leq e^{-n.\delta}$.
The domination $E\oplus E^{c}$ then implies that for each $n\geq 0$, one has
\begin{equation}\label{e.unif}
\prod_{i=0}^{n-1}\|Df^N_{|E^{cs}}(f^{i.N}(q_0))\|\leq C.e^{-n},
\end{equation}
where $C,N>0$ are some uniform constants given by the domination.
One deduces from~(\ref{e.unif}) that a uniform neighborhood of $q_0$ in $\cW^{cs}_{q_0}$
is contained in $W^s(q_0)$.
\end{proof}

\begin{remark}\label{r.large}
The previous lemma still holds if one replaces $g$ by a diffeomorphism $C^1$-close to $f$
and if one considers a periodic orbit $O$ of $g$ contained in a small neighborhood of $H(p)$ and
a locally invariant plaque family of $g$ over $O$ whose plaques are $C^1$-close to the plaques of $\cW^{cs}$.
\end{remark}
\medskip

\begin{lemma}\label{l.linked}
Let us assume that $E^{cs}$ and $E^{cu}$ are thin trapped by $f$ and $f^{-1}$ respectively.
Then, all the hyperbolic periodic orbits contained in $H(p)$ are homoclinically related together.
\end{lemma}
\begin{proof}
First, observe that all the hyperbolic periodic points in $H(p)$
have the same stable index.
Let us take a periodic point $q$ in the class.
By lemma~\ref{l.contper}, there exists a periodic orbit $O$
homoclinically related to $p$ and having a point $q'$
arbitrarily close to $q$ such that $\cW^{cs}_{q'}\subset W^s(q')$ and $\cW^{cu}_{q'}\subset W^u(q')$.
One deduces that the plaques $\cW^{cs}_{q'}$ intersects $W^u(q)$ and
$\cW^{cu}_{q'}$ intersects $W^s(q)$. As a consequence $O$ and $q$ are homoclinically related.

\end{proof}

\subsection{Local product stability}

For any invariant compact set $K$, we define the \emph{chain-stable set} of $K$
as the set of points $x\in M$ such that for any $\varepsilon>0$, there exists a $\varepsilon$-pseudo-orbit
that joints $x$ to $K$. The chain-unstable set of $K$ is the chain stable set of $K$ for the map $f^{-1}$.

\begin{lemma}\label{l.chain-stable}
For any point $x\in H(p)$, the plaque $\cW^{cs}_x$ (resp. $\cW^{cu}_y$)
belongs to the chain-stable set (resp. the chain-unstable set) of $H(p)$.
\end{lemma}
\begin{proof}
By lemma~\ref{l.contper}, the point $x$ is the limit of periodic points $p_n\in \cP_0$
such that $\cW^{cs}_{p_n}$ is contained in the stable set of $p_n$ for each $n\geq 0$.
By definition of a plaque family, any point of $\cW^{cs}_{x}$ is limit
of a sequence of points $x_n\in \cW^{cs}_{p_n}$, proving that $x$ is contained in the
chain-stable set of $H(p)$.
\end{proof}

\begin{lemma}\label{l.bracket0}
For any points $x,y\in H(p)$, any transverse intersection point
between $\cW^{cs}_{x}$ and $\cW^{cu}_{y}$ is contained in $H(p)$.
\end{lemma}
\begin{proof}
By lemma~\ref{l.contper}, there exist two periodic points $p_x$ and $p_y$
close to $x$ and $y$ respectively whose orbits are homoclinically related to $p$
such that $\cW^{cs}_{p_x}\subset W^s(p_x)$
and $\cW^{cu}_{p_y}\subset W^u(p_y)$.
By continuity of the plaque families $\cW^{cs}$ and $\cW^{cu}$, one deduces that
$\cW^{cs}_{p_x}$ and $\cW^{cu}_{p_y}$ intersect transversally at a point $z'\in H(p)$ close to $z$.
Hence $z$ belongs to $H(p)$.
\end{proof}

\subsection{Robustness}
Let us consider a compact set $K$ having a dominated splitting $T_KM=E\oplus F$ for $f$.
If $U\subset M$ and $\cU\subset \diff^1(M)$ are some small neighborhoods of $K$ and $f$,
then for each $g\in \cU$
the maximal invariant set $K_g=\bigcap_{n\in \ZZ} g^n(\overline U)$ has a dominated splitting
$E_g\oplus F_g$ such that $\dim(E_g)=\dim(E)$. Moreover the maps $(g,x)\mapsto E_{g,x}, F_{g,x}$ are continuous.
Hence one may look for a plaque family tangent to the continuation $E_g$ of $E$ for $g$.

A collection of plaque families $(\cW_g)_{g\in \cU}$ tangent to the bundles $(E_g)_{g\in \cU}$
over the sets $(K_g)_{g\in \cU}$ is \emph{continuous} if
$(\cW_{g,x})_{g\in\cU,x \in K_g}$ is a continuous family of $C^1$-embeddings.
It is \emph{uniformly locally invariant} if there exists $\rho>0$ such that for each $g\in \cU$ and $x\in K_g$,
the image of the ball $B(0,\rho)\subset E_{g,x}$ by $g\circ \cW_{g,x}$
is contained in the plaque $\cW_{g,g(x)}$.

\begin{lemma}\label{l.extension0}
Let $K$ be an invariant compact set for a diffeomorphism $f$
having a dominated splitting $E\oplus F$. Then,
there exist some neighborhoods $U$ of $K$ and $\cU\subset \diff^1(M)$ of $f$
and a continuous collection of plaque families $(\cW_g)_{g\in \cU}$
tangent to the bundles $(E_g)_{g\in \cU}$ over the maximal invariant sets $(K_g)_{g\in \cU}$ in $\overline{U}$,
which is uniformly locally invariant.
\end{lemma}
\begin{proof}
Let $\exp$ be the exponential map from a neighborhood of the section $0$ in $TM$ to $M$.
Each diffeomorphism $g$ close to $f$ induces a diffeomorphism $\hat g$ on $TM$,
which coincides for each $x\in K$ with the map $\exp^{-1}_{g(x)}\circ g\circ \exp_x$
on a small neighborhood of $0\in T_xM$ and with the linear map $T_xg$ outside another small
neighborhood of $0$; moreover,
$\hat g$ is arbitrarily close to the linear bundle automorphism $Tg$ over the map $g$.
The proof of the plaque family theorem~\cite[theorem 5.5]{HPS} associates to each $x\in K_g$
the graph $\Gamma_{g,x}$ in $T_xM$ of a $C^1$ map $\psi_{g,x}\colon E_{g,x}\to F_{g,x}$ tangent to $E_{g,x}$ at $0\in T_xM$ and satisfying
\begin{equation}\label{e.invariance}
\hat g(\Gamma_{g,x})=\Gamma_{g,g(x)}.
\end{equation}
The graphs $\Gamma_{g,x}$ are uniformly Lipschitz and are characterized for some constant $C>0$ by
$$\Gamma_{g,x}=\bigcap_{n\geq 0}
\hat g^{-n}
(\{(y_1,y_2)\in E_{g,g^{-n}(x)}\times F_{g,g^{-n}(x)}, C\|y_1\|\geq \|y_2\|\}).$$
One thus deduces that they depend continuously on $(g,x)$ for the $C^0$-metric.
On the other hand, each map $\hat g$ has a dominated splitting $\hat E^{cs}\oplus\hat E^{cu}$
inside the spaces $T_x M$ and each graph $\Gamma_{g,x}$ is tangent to the bundle $\hat E^{cs}$.
The bundle $\hat E^{cs}$ depends continuously on $(g,x)$, hence the graphs $\Gamma_{g,x}$
depend continuously on $(g,x)$ also for the $C^1$-metric.

The plaque $\cW_{g,x}$ is defined as the image by the exponential
$\exp_x$ of a uniform neighborhood of $0\in \Gamma_{g,x}$.
For instance, one may choose $\varepsilon>0$ small and define for any $z\in E_{g,x}$,
$$\cW_{g,x}(z)=\exp_x(y,\psi_{g,x}(y)),\text{ where } y=\varepsilon.\frac{\arctan\|z\|}{\|z\|}.z.$$
By construction and the invariance~(\ref{e.invariance}),
the plaque families $(\cW_{g})$ are uniformly locally invariant.
\end{proof}

\begin{lemma}\label{l.robustness}
Let us assume that $E^{cs}$ and $E^{cu}$ are thin trapped by $f$ and $f^{-1}$ respectively
and that $H(p)$ coincides with its chain-recurrence class.
Then, there exist some neighborhoods $U$ of $K$ and $\cU\subset \diff^1(M)$ of $f$
and two continuous collections of plaque families $(\cW^{cs}_g)_{g\in\cU},(\cW^{cu}_g)_{g\in \cU}$
tangent to the bundles $(E^{cs}_g),(E^{cu}_g)$ over the maximal invariant sets $(K_g)_{g\in \cU}$ in $U$,
which are trapped by $g$ and by $g^{-1}$ respectively.

The plaques may be chosen arbitrarily small. As a consequence, for any diffeomorphism $g$ that is $C^1$-close to $f$, the homoclinic class $H(p_g)$ of $g$ associated to the continuation $p_g$ of $p$ is still chain-hyperbolic.
\end{lemma}
\begin{proof}
Let us consider a continuous collection of plaque families $(\cW_{g})$ tangent to the bundles $(E^{cs}_g)$
over the sets $(K_g)$ as given by lemma~\ref{l.extension0}.
Since $E^{cs}$ is thin trapped for $f$ over $H(p)$,
there exists a constant $\rho>0$ and a continuous family of embedding $(\varphi^0_x)$ of $(E^{cs}_x)$ supported in a small neighborhood $S$
of the section $0$ in $E^{cs}$ and
satisfying for each $x\in H(p)$,
\begin{equation}\label{l.trapped}
f(\overline{\cW_{f,x}\circ\varphi^0_x(B(0,\rho)}))\subset \cW_{f,f(x)}\circ\varphi^0_{f(x)}(B(0,\rho)).
\end{equation}
One may find a continuous family of embeddings $(\varphi_x)$ that is close to $(\varphi^0_x)$
for the $C^1$-topology and that extends to any point $x$ in a neighborhood of $H(p)$:
one fixes a finite collection of points $x_i$ in $H(p)$
and using a partition of the unity one defines $\varphi_x$ as a barycenter between
$\varphi^0_{x_i}$ associated to points $x_i$ that are close to $x$.
One deduces that there exist a neighborhood $U$ of $H(p)$ in $M$,
and a continuous family of embeddings
$(\varphi_x)$ of $(E^{cs}_x)$ over $U$,
such that (\ref{l.trapped}) still holds for $g$, $x\in K_g$
and $(\varphi_x)$.
One can thus define $\cW^{cs}_{g,x}$ as the embedding
$$z\mapsto \cW_{g,x}\circ \varphi_{g,x} \left(\frac{\rho.\arctan\|z\|}{\|z\|}.z\right).$$
By construction the plaque family $\cW^{cs}_g$ is trapped by $g$
and the collection $(\cW^{cs}_g)_{g}$ is continuous.

The plaques $\cW^{cs}_{g,x}$ may have been chosen arbitrarily small
and in particular much smaller than the stable manifold $W^s(q_{s,g})$
of the continuation $q_{s,g}$ of $q_s$.
The trapping property thus implies that $\cW^{cs}_{q_{s,g}}$ is contained in the stable manifold of $q_{s,g}$.
One builds similarly the plaques $\cW^{cu}_{g,x}$ and proves that $\cW^{cu}_{q_{u,g}}$
is contained in the unstable manifold of $q_{u,g}$ for any $g$ close to $f$.
We have thus shown that $H(p_g)$ is chain-hyperbolic.
\end{proof}
\begin{remark}\label{r.coherence}
Under the setting of lemma~\ref{l.robustness}.
One can check that the numbers $r,\rho,\varepsilon$ that appear in lemma~\ref{l.uniqueness-coherence}
for the coherence and the uniqueness of the plaque families can be chosen uniform in $g$.
\end{remark}

\subsection{Quasi-attractors}

\begin{lemma}\label{l.cont-quasi-attractor}
If the chain hyperbolic class $H(p)$ is a quasi-attractor and if the bundle $E^{cu}$
is uniformly expanded, then
for any diffeomorphism $g$ $C^1-$close to $f$ and
any hyperbolic periodic point $q$ homoclinically related to the orbit of $p_g$,
the unstable manifold $W^u(q)$ is contained in the homoclinic class $H(p_g)$.
\end{lemma}
\begin{proof}
Since $H(p)$ is a homoclinic class, there exists a dense set of points $x\in H(p)$
that belong to the stable manifold of $p$. Moreover by the trapping property,
$\cW^{cs}_x$ contains $f^{-n}(\cW^{cs}_{f^n(x)})$ for any $n\geq 0$, hence is contained in the
stable manifold of the orbit of $p$.

If $H(p)$ is a quasi-attractor and $E^{cu}$ is uniformly expanded,
it is the union of the unstable manifolds $W^u(x)$
of the points $x\in H(p)$. If one fixes $\rho>0$ then any disk
$D$ of radius $\rho$ contained in an unstable manifold $W^u(x)$ intersects transversally
the stable manifold of $p$. Hence, by compactness there exists $N\geq 1$ uniform such that
$f^N(D)$ intersects transversally the local stable manifold $W^s_{loc}(O)$ of the orbit $O$ of $p$.
This property is open: since $H(p)$ is a chain-recurrence class, for any $g$ close to $f$,
the class $H(p_g)$ is contained in a small neighborhood of $H(p)$,
hence for any disk $D$ of radius $\rho$
contained in $W^u(x)$ for some $x\in H(p)$,
the iterate $f^N(D)$ intersects transversally $W^s_{loc}(O_g)$.

Moreover since $H(p)$ is a quasi-attractor, there exists an arbitrarily small open neighborhood $U$
of $H(p)$ such that $f(\overline U)\subset U$. Hence for $g$ close to $f$ one still
has $g(\overline U)\subset U$ and the unstable manifold $W^u(O_g)$ is contained in $U$.
Since $U$ is a small neighborhood of the set $H(p)$, the partially hyperbolic structure extends
to the closure of $W^u(O_g)$; in particular the dynamics of $g$ uniformly expands along the manifold
$W^u(O_g)$.

One deduces that for any $g$ close to $f$,
for any point $x\in W^u(O_g)$, for any neighborhood $V$ of $x$ inside $W^u(O_g)$,
there exists an iterate $g^n(V)$ with $n\geq 1$ which contains a disk of radius $\rho$,
so that $g^{n+N}(V)\subset W^u(O_g)$ intersects transversally $W^s_{loc}(O_g)$.
One deduces that $H(p_g)$ meets $g^{n+N}(V)$, hence $V$. Since $V$ can be chosen arbitrarily small
and $H(p_g)$ is closed, the point $x$ belongs to $H(p_g)$. We have proved that $\overline{W^u(O_g)}\subset H(p_g)$.

Let $q$ be any hyperbolic periodic point homoclinically related to $p_g$.
The unstable manifolds of the orbit of $p$ and $q$ have the same closure.
In particular $W^u(q)\subset H(p_g)$.
\end{proof}

\subsection{Stable boundary points}\label{ss.one-codim}
We now discuss the case the center stable bundle has a dominated decomposition
$E^{cs}=E^s\oplus E^c$ with $\dim(E^c)=1$ and $E^s$ is uniformly contracted.

\paragraph{Half center-stable plaques.}
Any point $x\in H(p)$ has a uniform strong stable manifold which is one-codimensional inside $\cW^{cs}_x$.
A neighborhood of $x$ intersects $\cW^{cs}_x\setminus W^{ss}_{loc}(x)$ into two connected components.
The choice of an orientation on $E^c_x$ allows to denote them by $\cW^{cs,+}_x$
and $\cW^{cs,-}_x$.
One can then consider if $x$ is accumulated inside $\cW^{cs}_x\setminus W^{ss}_{loc}(x)$
by points of $H(p)$ in one or in both components.
Note that this does not depend on the choice
of the plaque family $\cW^{cs}$. Note also that the same case will occur all along the orbit of $x$.
\medskip

If one considers a point $y\in H(p)\cap W^{cu}_x$ close to $x$,
one gets an orientation of $E^c_y$ that matches with the orientation of $E^c_x$.
The points of $H(p)\cap \cW^{cs}_x$ close to $x$ projects on $\cW^{cs}_y$ through the holonomy
along the center unstable plaques, but there is no reason that the projection of the points
in $H(p)\cap \cW^{cs,+}_x$ are contained inside $\cW^{cs,+}_y$. However the following can be proved.

\begin{lemma}\label{l.integrability}
Consider any periodic point $q$ homoclinically related to $p$ and any point $x\in H(p)$ close to $q$
such that $W^u_{loc}(q)$ intersects $W^{ss}_{loc}(x)$ at a point $z$.
If $q$ is accumulated by $H(p)\cap \cW^{cs,+}_q$ then $z$ is accumulated by $H(p)\cap \cW^{cs,+}_x$.
More precisely, there exists $y\in H(p)\cap \cW^{cs,+}_q$ arbitrarily close to $q$
such that $W^u_{loc}(y)$ intersects $H(p)\cap\cW^{cs,+}_{x}$ close to $z$.
\end{lemma}
\begin{proof}
Let us consider a point $y_0\in \cW^{cs,+}_q$ close to $q$; it belongs to $W^s(q)$.
Let $D$ be a neighborhood of $z$ in $\cW^{cs}_{x}$ and $D^+$ a neighborhood of $z$ in $\cW^{cs,+}_{x}$.
By the $\lambda$-lemma, the sequence $f^{-n}(D)$, $n\geq 0$ converges toward $W^s(q)$.
Observe that the strong stable manifolds of $x$ and $z$ coincide.
By continuity of the strong stable lamination, the sequence $f^{-n}(W^{ss}_{loc}(x))$ converges toward
$W^{ss}(q)$. Hence $\cW^{cu}_{y_0}$ intersects $f^{-n}(D^+)$ close to $q$ for $n$ large enough.
The intersection is transversal, hence belongs to $H(p)$ by lemma~\ref{l.bracket0}. One thus deduces that $D^+$ intersects $H(p)$.
By taking $D^+$ arbitrarily small, one has proved that $z$ is accumulated by $H(p)\cap \cW^{cs,+}_x$.
Also the local unstable manifold $W^u_{loc}(y)$ of the point $y=f^n(y_0)$ intersects $D^+$, giving the conclusion.
\end{proof}

\begin{lemma}\label{l.NGSHI}
Let us assume that $H(p)$ does not contains periodic points $q,q'$ homoclinically related
to the orbit of $p$ such that $W^{ss}(q)\setminus\{q\}$ and $W^u(q')$ intersect.
Then any point $x\in H(p)$ is accumulated by $H(p)$ in $\cW^{cs}_x\setminus W^{ss}_{loc}(x)$.
\end{lemma}
\begin{proof}
Let us assume by contradiction that there exists a point
$x\in H(p)$ which is not accumulated by points in $(\cW^{cs}_{x}\cap H(p))\setminus W^{ss}_{loc}(x)$.
Let $q\in H(p)$ be a periodic point close to $x$ and homoclinically related to the orbit of $p$.
Its unstable manifold intersects transversally
$\cW^{cs}_x$ at a point $z\in H(p)$. Since $z$ can be chosen arbitrarily close to $x$,
it belongs to $W^{ss}_{loc}(x)$ and
it is not accumulated by points in $(\cW^{cs}_{x}\cap H(p))\setminus W^{ss}_{loc}(x)$.
By lemma~\ref{l.integrability}, $W^s_{loc}(q)\setminus W^{ss}(q)$ is disjoint from $H(p)$.
In particular the point $q$ is not accumulated by points in $(\cW^{cs}_{q}\cap H(p))\setminus W^{ss}_{loc}(q)$.
One can thus repeat for $q$ the argument we have made for $x$ and find a periodic point $q'\neq q$
homoclinically related to the orbit of $p$ such that $W^u(q')$ intersects $W^{ss}(q)$.
This contradicts our assumption.
\end{proof}

We now introduce the definition of the  stable boundary points, generalizing the notion
of stable boundary points for uniformly hyperbolic set whose stable bundle is one-dimensional
(see~\cite[appendix 2]{PT}). This notion plays an important role
and it is extensively studied in section \ref{s.boundary}.
\begin{defi}\label{d.boundary}
A point $x\in H(p)$ is a \emph{ stable boundary point} if it is not accumulated inside
both components of $\cW^{cs}_x\setminus W^{ss}_{loc}(x)$ by points of $H(p)$.
\end{defi}
Observe that if $x$ is a  stable boundary point, then any iterate of $x$ is also.
Note that if $E^{cs}$ is one-dimensional, a stable boundary point $x\in H(p)$
is a point which is not accumulated by points of $H(p)$ in both components of $\cW^{cs}_x\setminus \{x\}$.

Naturally in the same way, if the center unstable subbundle split $E^{cu}=E^c_2\oplus E^u$,
where $E^{c}_2$ is one-dimensional and $E^u$ is uniformly expanded, it can be defined the notion of
{\em unstable boundary point}.
\medskip

The next lemma about stable boundary points is a version of a classical one for hyperbolic systems.
A more general proposition about stable boundary points is provided in section \ref{s.boundary}.

\begin{lemma}\label{p.boundary0} Let $f$ be a diffeomorphism and $H(p)$ be a chain-recurrence class which is
a chain-hyperbolic homoclinic class  endowed with a dominated splitting $E^{cs}\oplus E^{cu}=E^{cs}\oplus (E^c_2\oplus E^{u})$ such that $E^{cs},E^{c}_2$ are one-dimensional, $E^{cs},E^{cu}$ are thin trapped (for $f$ and $f^{-1}$ respectively) and $E^u$ is uniformly expanded.
Then any stable boundary point of $H(p)$ belong to the
unstable set of a periodic point.
\end{lemma}
\begin{proof}
Let $x$ be a stable boundary point of $H(p)$.
Let us introduce three backward iterates $x_1=f^{-k}(x)$,  $x_2=f^{-l}(x)$ and $x_3=f^{-m}(x)$ arbitrarily close
with $k<l<m$.
If the center-unstable plaques of two of those three points (for instance $x_1,x_2$) intersect,
from the coherence (lemma~\ref{l.uniqueness-coherence}) it follows that the center-unstable plaque $\cW^{cu}_{x_1}$
is mapped into itself by $f^{k-l}$. Since $E^{cu}$ splits as $E^{c}_2\oplus E^u$,
one deduces that the backward orbit of $x_1$ belongs to the unstable set of a periodic point
of $\cW^{cu}_{x_1}$ (this point is not necessarily hyperbolic).

If the center unstable plaques of the three points do not intersect, we can assume that the center stable plaque of namely $x_2$  intersects the center unstable plaques of the other two points in different connected components of $\cW_{x_2}^{cs}\setminus \{x_2\}.$ By lemma \ref{l.bracket0} those points of intersection belong to $H(p)$ and using that $E^{cs}$ is thin trapped, the forward orbits of those points remain arbitrarily close to $x$ (provided that the points $x_1, x_2, x_3$ were sufficiently close) and contained in  different components of $\cW_{x_2}^{cs}\setminus \{x_2\};$ a contradiction. 
\end{proof}

The following proposition is not needed in the context of the present paper, however we provide it since it helps to understand the notion of boundary point.
\begin{prop}
Using~\cite[proposition 3.2]{C2}, one can prove that if the homoclinic class $H(p)$ is endowed with a partially hyperbolic
structure $E^s\oplus E^c\oplus E^u$ with $\dim(E^c)=1$ such that $E^{cs}=E^s\oplus E^c$
is thin trapped, then,
\begin{itemize}
\item[--] either any  stable boundary point $x\in H(p)$ belongs to the unstable manifold of
a periodic point,
\item[-] or there exists a diffeomorphism $g$ that is $C^1$-close to $f$ and a periodic orbit contained in
a small neighborhood of $H(p)$ which has a strong homoclinic intersection.
\end{itemize}
\end{prop}
\noindent One will use instead a similar result for quasi-attractors, see section~\ref{ss.structure} below.

\begin{proof}[Sketch of the proof]
Let $x$ be a strong boundary point.
Let us take the sequence $\{x_n=f^{-n}(x)\}_{n>0}$. 
Since $E^{cs}$ is thin trapped, one may take a small plaque family $\cW^{cs}$ which is trapped
and such that for each $n\geq 0$, one connected component $U_n$ of $\cW^{cs}_{x_n}\setminus W^{ss}_{loc}(x_n)$ is disjoint from $H(p)$.
In particular:
\begin{description}
\item[(*)]\emph{For any two close iterates $x_n,x_m$, the unstable manifold $W^u_{loc}(x_n)$
does not meet $U_m$.}
\end{description}
We consider two cases:
either the orientation of the center manifolds of all close backward iterates is preserved or not.
Equivalently, the tangent map $Df$ preserves or not a continuous orientation of the bundle $E^c$
over $\alpha(x)$, the $\alpha$-limit set of $x$. One can assume that $\alpha(x)$ is not reduced to
a periodic orbit since otherwise, $x$ belongs to the unstable manifold of a periodic orbit
and the statement follows.

\noindent
\emph{- The orientation preserved case.}
From property (*), any two close iterates $x_n,x_m$ are in twisted position
(see~\cite[section 3]{C2}), implying that $\alpha(x)$ is twisted.
If $\alpha(x)$ contains a periodic orbit $O$, it contains points in $W^{ss}(O)\setminus O$
and in $W^u(O)\setminus O$;
as a consequence, one can apply the Hayashi connecting lemma and get a strong
homoclinic intersection for $O$ by an arbitrarily small $C^1$-perturbation.
Otherwise $\alpha(x)$ contains a non-periodic minimal set and from~\cite[proposition 3.2]{C2},
there exists a diffeomorphism $g$ that is $C^1$-close to $f$ and a periodic orbit contained in
a small neighborhood of $H(p)$ which has a strong homoclinic intersection.

\noindent
\emph{- The orientation reversed case.}
Let us consider a sequence of arbitrarily close points $x_{n_k},x_{m_k}$
such that $Df^{m_k-n_k}$ reverse the local orientation on $E^c$ at $x_{n_k}$.
One may assume that they converge toward  a point $y\in \alpha(x)$.
Property (*) now implies that $H(p)\cap\cW^{cs}_y$ is contained in $W^{ss}(y)$.
This contradicts lemma~\ref{l.NGSHI} above.
\end{proof}

\subsection{Non-uniformly hyperbolic bundles}
When the bundle $E^{cs}$ is not uniformly contracted, the class may contain
weak periodic orbits.
\begin{lemma}\label{l.weak}
Let us assume that $H(p)$ is a chain-recurrence class
and that there exists a dominated splitting $E^{cs}=E^s\oplus E^c$ where $E^c$
is one-dimensional, $E^{cs}$ is thin-trapped and $E^s$ is uniformly contracted.
Then, there exists some hyperbolic periodic orbits in $H(p)$
whose Lyapunov exponent along $E^c$ is arbitrarily close to zero.
\end{lemma}
\begin{remark}\label{r.weak}
If one assumes that all the periodic orbits in $H(p)$ are hyperbolic, then one
can ensure that the obtained periodic orbits are homoclinically related to $p$.
Indeed, since $E^{cs}$ is thin-trapped, all the periodic orbits in $H(p)$
have the same stable dimension and by lemma~\ref{l.linked} are homoclinically related to $p$.
\end{remark}
\begin{proof}
One can consider an invariant compact set $K\subset H(p)$ such that
the restriction of $E^c$ to $K$ is not uniformly contracted and $K$ is minimal
for the inclusion and these properties.
Since the bundle $E^c_{|K}$ is one-dimensional, thin trapped and not uniformly contracted, $K$ coincides
with the support of an ergodic measure $\mu$ whose Lyapunov exponent along $E^c$ is zero.
The exponent of any other measure supported on $K$ is non-positive.

In the case there exists ergodic measures $\mu$ supported on $K$ whose Lyapunov exponent
along $E^c$ is negative and arbitrarily close to zero, the domination $E^{cs}\oplus E^{cu}$
implies that these measures are hyperbolic and the $C^1$-version of Anosov closing lemma
(see~\cite[proposition 1.4]{C2}) ensures that the chain-recurrence class $H(p)$ contains hyperbolic periodic orbits
whose Lyapunov exponent along $E^c$ is arbitrarily close to zero.

In the case there exists ergodic measures supported on $K$ with negative Lyapunov exponent along $E^c$
but never contained in a small interval $(-\delta,0)$, one can argue as in the proof of
\cite[theorem 1]{C2} and apply Liao's selecting lemma. Once again, the chain-recurrence class $H(p)$ contains hyperbolic periodic orbits
whose Lyapunov exponent along $E^c$ is arbitrarily close to zero.

In the remaining case, all the measures supported on $K$ have a Lyapunov exponent along $E^c$ that is equal to zero.
In particular, $E^{cu}$ is uniformly expanded on $K$.
We have also assumed that $E^{cs}$ is thin trapped.
As a consequence, one can choose over the maximal invariant set in a neighborhood of $K$
some plaques $\cD^{cs}$ and $\cD^{cu}$ with arbitrarily small diameter and that are trapped by $f$
and $f^{-1}$ respectively.

For any $\varepsilon>0$ there exists a periodic $\varepsilon$-pseudo-orbit
$x_0,x_1,\dots,x_n=x_0$ contained in $K$ such that the quantity
$$\frac 1 n \sum_{k=0}^{n-1}\log \|Df_{|E^c}(x_k)\|$$
is arbitrarily close to zero. By the weak shadowing lemma~\cite[lemma 2.9]{C2},
there exists a periodic orbit $O_0$ contained in an arbitrarily small neighborhood of $K$
and whose Lyapunov exponent along $E^c$ is close to zero.

The unstable manifold of a point $x\in K$ close to $O_0$ intersects a center-stable plaque
of $O_0$. Since these plaques are trapped and $E^c$ is one-dimensional,
this implies that the center-stable plaques of $O_0$ contains a periodic orbit $O'$
whose stable manifold intersects $W^u(x)$.
On the other hand $W^u(O')$ intersects a center-unstable plaque
of a point of $H(p)$. As a conclusion $O'$ is contained in the chain-recurrence class of $p$.
Since the plaques $\cD^{cs}$ have a small diameter, the Lyapunov exponent of $O'$ along $E^c$
is close to the Lyapunov exponent of $O$, hence is close to zero.

The conclusion of the lemma has been obtained in al the cases.
\end{proof}

%% file: continuation-1110.tex
\section{Continuation of chain-hyperbolic homoclinic classes}\label{s.continuation}

Let $H(p)$ be a homoclinic class of a diffeomorphism $f$ and
assume that it is a chain-recurrence class endowed with
a partially hyperbolic structure $E^s\oplus E^c\oplus E^u$ such that $\dim(E^c)=1$ and
the bundle $E^{cs}=E^s\oplus E^c$ is thin trapped.
By lemma~\ref{l.robustness}, the homoclinic class $H(p_g)$ is still chain-hyperbolic for the
diffeomorphisms $g$ close to $f$. We explain here, how in certain sense, the points in $H(p)$
can be continued in $H(p_g)$. If $f$ is far from strong homoclinic intersections, proposition~\ref{p.continuation}
shows that the points of $H(p_g)$
are in correspondence with the continuation of the points of $H(p)$ up to some identifications
and blow-ups in the central direction (that can be compared with the blow-up of an Anosov diffeomorphism during
the construction of a ``derived from Anosov'' map).

\subsection{Preliminary constructions}
\paragraph{Local central orientation.}
The bundle $E^c$ on $H(p)$ is one-dimensional and locally trivial.
Moreover it depends continuously on the dynamics $f$.
One deduces that for any $g,g'$ close to $f$, the orientations of $E^c_{g,x}$ and $E^c_{g',x'}$
for two points $x\in H(p_g)$ and $x'\in H(p_{g'})$ can be compared provided $x$ and $x'$ are close
(say at distance less than $\varepsilon$).
To make this precise, one can cover a neighborhood of $H(p)$ by
a finite number of open sets $U_i$ endowed with non-singular one-forms $\alpha_i$
such that $\alpha_i$ never vanishes on the bundle $E^c$.
Two close points $x,x'$ belong to a same open set $U_i$.
Two orientations on $E^c_{g,x}$ and $E^c_{g',x'}$ match if they both coincide with the class
of $\alpha_i$ or the class of $-\alpha_i$. If $x,x'$ are close enough, this does not depend on the open set $U_i$
containing $\{x,x'\}$. If one considers another collection of pairs $(U'_i,\alpha'_i)$,
the orientations on $E^c_{g,x}$ and $E^c_{g',x'}$ still match if the distance between $x$ and $x'$
is small and $g$ is close enough to $f$.

\paragraph{Plaque families.}
In the following one fixes $\delta>0$ small which is a lower bound for the modulus of the Lyapunov exponents
of $p_g$ for $g$ close to $f$.
One chooses some continuous collections of plaque families $(\cW^{cs}_{g})$
for the diffeomorphisms $g$ close to $f$ as given by lemma~\ref{l.robustness}.
Since $E^{cs}$ is thin trapped, the plaques may be chosen with a small diameter
so that the properties stated in lemma~\ref{l.uniqueness-coherence} hold.
Also, by lemma~\ref{l.largestable}, for $g$ that is $C^1$-close to $f$
and for any periodic point $q\in H(p_g)$
whose Lyapunov exponents along $E^{cs}$ are smaller than $-\delta/2$,
the plaque $\cW^{cs}_{g,q}$ is contained in the stable set of $q$.

One will consider local manifolds $W^ {ss}_{g,loc}(x)$ and $W^ {u}_{g,loc}(x)$
for $x\in H(p_g)$ with a small diameter so that $W^ {u}_{g,loc}(x)$ intersects a plaque
$\cW^ {cs}_{g,y}$ in at most one point and the intersection is always transversal.

\paragraph{Shadowing.}
Then, one chooses $\varepsilon>3\varepsilon'>0$ so that the following lemma holds and
so that for any $g,g'$ close to $f$ and any $x\in H(p_g)$ and $y\in H(p_{g'})$
satisfying $d(x,y)<\varepsilon$
the local manifold $W^{u}_{g,loc}(x)$ intersects $\cW^{cs}_{g',y}$.
\begin{lemma}\label{l.expansivity}
There exists $\varepsilon>3\varepsilon'>0$ small such that any diffeomorphisms $g,g'$ close to $f$ satisfy:
\begin{itemize}
\item[--] if $x,y\in H(p_g)$ are two points such that the forward orbit of $x$ is $\varepsilon$-shadowed
by the forward orbit of $y$, then $y\in \cW^{cs}_{g,x}$;
\item[--] if $x,y\in H(p_g)$ are two points $\varepsilon'$-close
such that $y$ belongs to $\cW^{cs}_{g,x}$, then
the forward orbit of $x$ is $\frac \varepsilon 3$-shadowed by the forward orbit of $y$;
\item[--] for any periodic orbit $O\subset H(p_g)$ of $g$
whose central Lyapunov exponent is smaller that $-\delta$,
any periodic orbit of $g'$ that $\varepsilon$-shadows $O$
also $\varepsilon'$-shadows $O$, has a central Lyapunov exponent smaller than $\delta/2$
and is homoclinically related to $p_{g'}$; moreover any point $x\in H(p_g)$ whose backward orbit $\varepsilon$-shadows $O$ belongs to the unstable manifold of $O$.
\end{itemize}
\end{lemma}
\begin{proof}
We prove the first item. Let us consider the intersection point $z$ between $\cW^{cs}_{g,x}$ and $W^u_{g,loc}(y)$.
By uniform local invariance of $\cW^{cs}_{g}$, one checks inductively that
the point $g^n(z)$ is the intersection point between
$\cW^{cs}_{g,g^n(x)}$ and $W^u_{g,loc}(g^n(y))$ for $n\geq 0$.
If $z\neq y$, since $z$ and $y$ belong to the same unstable leaf, the distance
$d(g^n(z),g^n(y))$ increases exponentially and becomes much larger than $\varepsilon$, contradicting
that the distance between $g^n(x)$ and $g^n(y)$ is bounded by $\varepsilon$.
One deduces that $y=z$, hence $y$ belongs to $\cW^{cs}_{g,x}$.

Now we choose $\varepsilon'\ll \varepsilon$ and prove the second item.
Since $E^{cs}$ and $E^{u}$ are thin trapped by $f$ and $f^{-1}$,
lemma~\ref{l.robustness} associates some continuous trapped plaque families $\widehat \cW^{cs}_g$ and $\widehat \cW^{cu}_g$
over $H(p_g)$ for $g$ close to $f$ with diameter smaller than $\varepsilon/3$.
From lemma~\ref{l.uniqueness-coherence} if $\varepsilon'$ is small enough, then
for any $x,y\in H(p_g)$ such that $y\in \cW^{cs}_{g,x}$ and $d(x,y)<\varepsilon'$, the point $y$ belongs to $\widehat \cW^{cs}_{g,x}$. By the trapping property, $g^n(y)$ belongs to $\widehat \cW^{cs}_{g,g^n(x)}$
for any $n\geq 0$, hence $d(g^n(x),g^n(y))<\varepsilon/3$ as required.

We then prove the properties of the third item.
We first note that if $g,g'$ are close to $f$ and $\varepsilon$ is small enough, then
any periodic orbit $O'$ of $g'$ that $\varepsilon$-shadows a periodic orbit $O$ of $H(p_g)$
still has a partial hyperbolic structure and has Lyapunov exponents close to those of $O$.
This proves that the central Lyapunov exponent of $O'$ is smaller than $-\delta/2$.

One deduces from lemma~\ref{l.largestable} that for some point $q'\in O'$
the stable manifold of $q$ has  uniform size inside $\cW^{cs}_{g',q'}$.
From lemma~\ref{l.contper}, there exists a dense set of periodic points $x\in H(p_{g'})$
whose stable manifold has a uniform size.
If $\varepsilon$ is small enough
and $g,g'$ close enough to $f$, one thus deduces that
$q'$ is close to a point of $H(p_{g'})$.
From the uniformity of the invariant manifolds, we deduce that
the stable and unstable manifolds of $q'$ intersect the stable and unstable manifolds of a hyperbolic
periodic orbit homoclinically related to $p_{g'}$. In particular, $O'$ is homoclinically related to $p_{g'}$.

Let us consider again, as given by lemma~\ref{l.robustness},
some continuous plaque families $\widehat \cW^{cs}_h$ and $\widehat \cW^{cu}_h$
over $H(p_h)$ for $h$ close to $f$ with diameter much smaller than $\varepsilon'$.
From lemma~\ref{l.uniqueness-coherence} there exists $\rho>0$ such that
for any $g$ close to $f$ and any $x\in H(p_g)$, the ball $B(x,\rho)$
in $\cW^{cs}_{g,x}$ is contained in $\widehat \cW^{cs}_{g,x}$.
From the trapping property,
the following holds for $g,g'$ close to $f$:
if $x\in H(p_g)$ and $y\in H(p_{g'})$ such that $d(x,y)<\varepsilon$
satisfy that $W^u_{g',loc}(y)$ intersects $\widehat \cW^{cs}_{g,x}$, then the same holds for $g(x)$ and $g'(y)$.
Using the estimate~(\ref{e.unif}) in the proof of lemma~\ref{l.largestable},
there exists a uniform integer $N\geq 1$
and an iterate $q\in O$ such that $g^N(\cW^{cs}_{g,q})$ has radius smaller than $\rho$, hence is contained
in $\widehat \cW^{cs}_{g,g^N(q)}$. Let us choose $q'\in O'$ such that $d(g'{}^n(q'),g^n(q))<\varepsilon$
for each $n\in \ZZ$.
Provided that $g,g'$ have been chosen close enough to $f$, the intersection $z_n$
between $W^u_{g',loc}(g'{}^n(q'))$ and $\cW^{cs}_{g,g^n(q)}$ for $0\leq n\leq N$ are close
to the $N$ first iterates of $z_0$ under $g$, hence $z_N$ is contained in $\widehat \cW^{cs}_{g,g^N(q)}$.
By our construction, one deduces that $W^u_{g',loc}(g'{}^n(q'))$
intersects $\widehat \cW^{cs}_{g,g^n(q)}$ for any $n\geq N$, hence any $n\in \ZZ$. 
With the same argument, $\cW^{cs}_{g',g'{}^n(q')}$
intersects $\widehat \cW^{cu}_{g,g^n(q)}$, for any $n\in \ZZ$.
Since the diameter of the plaques $\widehat \cW^{cu}$ and
$\widehat \cW^{cs}$ is much smaller than $\varepsilon'$, one deduces that $g^{n}(q)$ and $g'{}^n(q')$
are at distance smaller than $\varepsilon'$. We have proved that $O$ is $\varepsilon'$-shadowed by $O'$.

Let us now consider a point $x\in H(p_g)$ whose backward orbits $\varepsilon$-shadows
the backward orbit of a point $q\in O$.
Let us introduce for each $n\geq 0$ the intersection point $z_n$
between $W^u_{g,loc}(g^{-n}(x))$ and $\cW^{cs}_{g,g^{-n}(q)}$.
By construction one has $g(z_{n+1})=z_n$ and in particular $z_0$
is contained in the intersection of the $g^n(\cW^{cs}_{g,g^{-n}(q)})$.
By assumption $\cW^{cs}_{g,g^{-n}(q)}$ is contained in the stable manifold of
$g^{-n}(q)$. This proves that $z_0$ coincides with $q$. As a consequence
$z_0$ belongs to $W^u_{g,loc}(q)$.
\end{proof}

\subsection{Continuation of uniform periodic points}
The periodic points with uniform Lyapunov exponents have a uniform hyperbolic continuation.
\begin{lemma}\label{l.cont-periodic}
There exists a simply connected open neighborhood $\cU\subset\diff^1(M)$ of $f$ such that:
\begin{itemize}
\item[--] The hyperbolic continuation of $p$ exists for any $g\in \cU$
and the class $H(p_g)$ is chain-hyperbolic.
\item[--] For any $g\in \cU$ and any periodic orbit $O\subset H(p_g)$ of $g$
whose central Lyapunov exponent is smaller than $-\delta$,
the hyperbolic continuation $O_{g'}$ of $O$ exists for any $g'\in \cU$
and is homoclinically related to $p_{g'}$.
Moreover its central Lyapunov exponent is still smaller than $-\delta/2$,
and $O_{g'}$ is $\frac \varepsilon 3$-shadowed by $O$.
\end{itemize}
\end{lemma}
\begin{proof}
Lemma~\ref{l.robustness} gives the existence of an open set $\cU$ satisfying the first item.

Let us consider a path $(\gamma_t)_{t\in [0,1]}$ in $\cU$ between $g$ and $g'$
and the maximal interval $I$ containing $0$ where the hyperbolic continuation $O_t$ of $O$
is defined and $\varepsilon/2$-shadows $O$.
If $I=[0,t_0)$ with $t_0<1$, one can consider a periodic orbit $O_{t_0}$ for $g_{t_0}$
that is the limit of a sequence of orbit $O_t$ for $t<t_0$. By construction $O_{t_0}$
$\varepsilon$-shadows $O$, hence $O_{t_0}$ has a central Lyapunov exponent smaller than $-\delta/2$
and also $\varepsilon'$-shadows $O$ by lemma~\ref{l.expansivity}.
Since $\varepsilon'<\varepsilon/3$, we have contradicted the definition of $t_0$.
Hence, the orbit $O$ has a hyperbolic continuation $O_{g'}$ for $g'$. Since $\cU$ is simply connected,
this continuation is unique. We have shown that $O$ is $\varepsilon'$-shadowed by $O_{g'}$,
hence by lemma~\ref{l.expansivity}, $O_{g'}$ is homoclinically related to $p_{g'}$, has
a central Lyapunov exponent  smaller than $-\delta/2$. Since $\varepsilon'<\varepsilon/3$,
all the properties stated in the second item are satisfied.
\end{proof}
\medskip

This justifies the following definition.
\begin{definition} Let us
 denotes with $\cP$, the set of hyperbolic periodic points $q\in H(p)$ homoclinically related to the orbit of $p$
whose continuation $q_g$ exists for any diffeomorphism
$g\in\cU$ and such that for some $g\in \cU$ the central Lyapunov exponent of $q_g$
is smaller than $-\delta$.
\end{definition}
\noindent
Since for any $g\in \cU$ the central Lyapunov exponents of $p_g$ is smaller than $-\delta$, there exists a dense set of periodic points in $H(p_g)$
whose central Lyapunov exponent is smaller than $-\delta$.
By lemma~\ref{l.cont-periodic}, one deduces that the continuations $q_g$ of points in $q\in\cP$ are dense in $H(p_g)$.

Note also that by lemma~\ref{l.cont-periodic} the central Lyapunov exponent of $q_g$ for $q\in \cP$
is smaller than $-\delta/2$; hence the plaque $\cW^{cs}_{g,q_g}$ is contained in $W^s_{g}(q_g)$.

\subsection{Pointwise continuation of $H(p)$}

\begin{definition}\label{d.continuation}
For any $g,g'\in \cU$, one says that two points $x\in H(p_g)$
and $x' \in H(p_{g'})$ \emph{have the same continuation} if there exists a sequence of
hyperbolic periodic points $(p_n)$ in $\cP$ such that
$(p_{n,g})$ and $(p_{n,g'})$ converge toward $x$ and $x'$ respectively.
\end{definition}
\noindent
This implies that
$g^k(x)$ and $g'^k(x')$ have the same continuation for each $k\in \ZZ$.
\medskip

By compactness and density of the points $q_g$ with $q\in \cP$,
one sees that, for any $g,g'\in \cU$, any point $x\in H(p_g)$
has the same continuation as some $x'\in H(p_{g'})$.
In general $x'$ is not unique.
The following implies that if $x'_1,x'_2\in H(p_{g'})$ have the same continuation as $x$,
then $x'_2$ belongs to $\cW^{cs}_{g',x'_1}$.

\begin{lemma}\label{l.cont-central}
For any $g,g'\in \cU$, let us consider $x\in H(p_g)$
and $x'\in H(p_{g'})$ such that $x$ and $x'$ have the same continuation.
Then, the orbits of $x$ by $g$ is $\frac \varepsilon 3$-shadowed by the orbit of $x'$ by $g'$.

As a consequence, if $x_1,x_2\in H(p_g)$ are $\varepsilon'$-close and satisfy
$x_2\in \cW^{cs}_{g,x_1}$, then for any $x'_1,x'_2\in H(p_{g'})$ such that $x_i,x'_i$ have the same continuation
for $i=1,2$, one still has $x'_2\in \cW^{cs}_{g',x'_1}$.
\end{lemma}
\begin{proof}
Let us consider a sequence $(p_n) \in \cP$ whose continuations $(p_{n,g})$, $(p_{n,g'})$
for $g$ and $g'$  converges toward $x$ and $x'$ respectively.
From lemma~\ref{l.cont-periodic}, the orbit of $(p_{n,g})$ by $g$ is
$\frac \varepsilon 3$-shadowed by the orbit of $(p_{n,g'})$ by $g'$. Taking the limit, one deduces that
the orbit of $x$ by $g$ is $\frac \varepsilon 3$-shadowed by the orbit of $x'$ by $g'$.

If $x_1,x_2\in H(p_g)$ are $\varepsilon'$-close and satisfy
$x_2\in \cW^{cs}_{g,x_1}$, by lemma~\ref{l.expansivity} the forward orbit of $x_2$ is
$\frac \varepsilon 3$-shadowed by the forward orbit of $x_1$.
By the first part of the lemma, one deduces that
for any $x'_1,x'_2\in H(p_{g'})$ such that $x_i,x'_i$ for $i=1,2$ have the same continuation, then
the forward orbit of $x'_1$ by $g'$
is $\varepsilon$-shadowed by the forward orbit of $x'_2$ by $g'$.
By lemma~\ref{l.expansivity}, this implies that $x'_2\in \cW^{cs}_{g',x'_1}$.
\end{proof}

One then shows that
if $x$ is a hyperbolic periodic point in $\cP$,
then $x'$ coincides with its hyperbolic continuation (hence is unique).
This is also true for the unstable manifold of points in $\cP$.

\begin{lemma}\label{l.cont-unstable}
For any $g\in \cU$, let $q_g$ be the hyperbolic continuation of some point $q\in \cP$
and let us consider some point $x\in W^u(q_g)\cap H(p_g)$.
Then, for any $g'\in \cU$, there exists a unique point $x'\in H(p_{g'})$
which has the same continuation as $x$; moreover $x'$ belongs to $W^u(q_{g'})$ and varies continuously with $g'$.
In particular the hyperbolic continuation $q_{g'}$ of $q_g$ is the unique point in $H(p_{g'})$
such that $q_g$ and $q_{g'}$ have the same continuation (in the sense of the definition~\ref{d.continuation}).
\end{lemma}
\begin{proof}
Let us consider any $x'\in H(p_{g'})$ which has the same continuation as $x$.
From lemma~\ref{l.cont-periodic}, the orbit of $q_{g'}$ by $g'$ is $\frac \varepsilon 3$-shadowed by
the orbit of $q_g$ by $g$ and from lemma~\ref{l.cont-central}, the orbit of $x$ by $g$ is $\frac \varepsilon 3$-shadowed by
the orbit of $x'$ by $g'$.
There exists $N\geq 1$ such that the backward orbit of $g^{-N}(x)$ is
$\frac \varepsilon 3$-shadowed by the backward orbit of $g^{-N}(q_g)$.
Hence the backward orbit of ${g'}^{-N}(x')$ is $\varepsilon$-shadowed by
the backward orbit of ${g'}^{-N}(q_{g'})$. By lemma~\ref{l.expansivity},
$x$ belongs to the unstable manifold of $g^{-N}(q_g)$.
It remains to prove that $x'$ is the only point in $H(p_{g'})$ which has the same
continuation as $x\in H(p_g)$.

Let $x'_1,x'_2\in H(p_{g'})$ be two points that have the same continuation as $x\in H(p_g)$.
By lemma~\ref{l.cont-central} their orbits under $g'$ both $\frac \varepsilon 3$-shadow the orbit of $x$ under $g$.
By lemma~\ref{l.expansivity},
${g'}^n(x'_2)$ belongs to $\cW^{cs}_{g',{g'}^n(x'_1)}$ for each $n\in \ZZ$.
When $n$ goes to $-\infty$, the points ${g'}^n(x'_2)$ and ${g'}^n(x'_1)$
are contained in a small local unstable manifold of the orbit of $q_{g'}$.
Since the plaques $W^u_{loc}$ and $\cW^{cs}$ intersect in at most one point,
this implies that $x'_1=x'_2$.

Let us denote by $x_{g'}$ the point which has the same continuation as $x$.
In order to prove the continuity of the map $g'\mapsto x_{g'}$,
one considers any limit point $x'$ of points $x_{g'}$ when $g'$ goes to $g$.
As before, the orbit of $x$ by $g$ is $\varepsilon$-shadowed by the orbit of $x'$,
so that ${g}^n(x')$ belongs to the unstable manifold of the orbit of $q$ and to
$\cW^{cs}_{g,{g}^n(x)}$ for each $n\in \ZZ$. This implies $x=x'$.
\end{proof}

\smallskip

\begin{remark}\label{r.cont-unstable}
Lemma~\ref{l.cont-unstable} also implies that definition~\ref{d.continuation} does not depend on the choice
of $\delta$ and $\cU$. Indeed, if one considers $\widetilde \delta\in (0,\delta)$ and $\widetilde \cU\subset \cU$
another neighborhood of $f$, then one gets two sets of periodic points
$\cP\subset \widetilde \cP$.
Let us consider $g,g'\in \widetilde{\cU}$ and two points $x\in H(p_g)$,
$x' \in H(p_{g'})$ which have the same continuation on $\widetilde \cU$ with respect to $\widetilde \cP$;
we claim that they also have the same continuation with respect to $\cP$.
Indeed one considers a sequence $(\widetilde {p}_n)$ in $\widetilde \cP$
such that $(\widetilde {p}_{n,g})$ converges toward $x$.
Then, for each $n$ there exists $p_n\in \cP$ such that
$p_{n,g}$ is close to $\widetilde {p}_{n,g}$. By lemma~\ref{l.cont-unstable},
$p_{n,g'}$ is close to $\widetilde {p}_{n,g'}$, hence one can obtain a sequence
$(p_n)$ in $\cP$ such that $(p_{n,g})$ converges toward $x$ and $(p_{n,g'})$ converges toward $x'$,
as wanted.
\end{remark}

\subsection{Continuations far from strong homoclinic intersections}
For $g\in \cU$ we define $\widetilde{H(p_g)}$
as the set of pairs $\tilde x=(x,\sigma)$ where $x\in H(p_g)$ and $\sigma$
is an orientation of $E^c_{g,x}$, such that
$x$ is accumulated in $H(p_g)\cap \cW^{cs,+}_{g,x}$ where $\cW^{cs,+}_{g,x}$
is the component of $\cW^{cs}_{g,x}\setminus W^{ss}_{loc}(x)$ determined by the orientation $\sigma$
as introduced in section~\ref{ss.one-codim}.

One can view $\widetilde{H(p_g)}$ as a subset of the unitary bundle associated to $E^c_g$ over $H(p_g)$.
The dynamics of $g$ can thus be lifted to $\widetilde{H(p_g)}$ and defines a map $\tilde g$.
One also defines the projection $\pi_g\colon \widetilde{H(p_g)} \to H(p_g)$ such that $\pi_g(x,\sigma)=x$.

\begin{proposition}\label{p.continuation}
Let $H(p)$ be a homoclinic class of a diffeomorphism $f\in \diff^r(M)$ such that
\begin{itemize}
\item[--] it is not a periodic orbit,
\item[--] is a chain-recurrence class endowed with
a partially hyperbolic structure $E^s\oplus E^c\oplus E^u$ such that $\dim(E^c)=1$ and
$E^{cs}=E^s\oplus E^c$ is thin trapped.
\end{itemize}
In a $C^1$-small neighborhood $\cU$ of $f$ in $\diff^1(M)$ we consider
a $C^r$-open connected set $\cV\subset \cU$ such that there is no diffeomorphism $g\in \cV$
whose homoclinic class $H(p_g)$ has a strong homoclinic intersection.

Then, for each $g,g'\in \cV$, the following holds:
\begin{description}
\item[a)] \emph{(Lifting).} The map $\pi_g\colon \widetilde{H(p_g)}\to H(p_g)$ is surjective
and semi-conjugates $\tilde g$ to $g$.
\item[b)] \emph{(Continuation of the lifting).}
For any $\tilde x_g=(x_g,\sigma)\in \widetilde{H(p_g)}$,
there is a unique $\tilde x_{g'}=(x_{g'},\sigma')\in \widetilde{H(p_{g'})}$
such that $x_g=\pi_g(\tilde x_g)$ and $x_{g'}=\pi_{g'}(\tilde x_{g'})$ have the same continuation
and such that the orientations $\sigma$ on $E^c_{g,x_g}$ and $\sigma'$ on $E^c_{g',x_{g'}}$ match;
this defines a bijection $\Phi_{g,g'}\colon \widetilde{H(p_g)}\to \widetilde{H(p_{g'})}$.
We denote $\Phi_g:=\Phi_{f,g}$. 
\item[c)] \emph{(Continuation of the projection).}
For any $x_g\in H(p_g)$ and $x_{g'}\in H(p_{g'})$ having the same continuation,
there exists $\tilde x\in \widetilde{H(p)}$ such that $\pi_g(\Phi_g(\tilde x))=x_g$
and $\pi_{g'}(\Phi_{g'}(\tilde x))=x_{g'}$.
\end{description}
\end{proposition}

\begin{remarks}\label{r.continuity}
One may consider on $\widetilde{H(p_g)}$ the topology induced by $E^c_g$.
This set is in general not compact since a sequence of points $x_n\in H(p_g)$ that are accumulated
in $H(p_g)\cap \cW^{cs,+}_{g,x_n}$ may converge toward a point $x\in H(p_g)$ which is not accumulated
in $\cW^{cs,+}_{g,x}$.
One can show however that the map $(g,\tilde x)\mapsto \Phi_g(\tilde x)$ is semi-continuous.
\end{remarks}

The next lemma is used in the proof of the proposition~\ref{p.continuation} and of lemma~\ref{l.continuite}.
\begin{lemma}\label{l.ordering}
Let us consider $q_1,q_2\in \cP$ and $g,g'\in \cV$ such that
$d(q_{1,g},q_{2,g})<\varepsilon/3$.
If $W^u_{g,loc}(q_{1,g})$ intersects $\cW^{cs,+}_{g,q_{2,g}}$, then
$W^u_{g',loc}(q_{1,g'})$ does not intersect $\cW^{cs,-}_{g',q_{2,g'}}$.
\end{lemma}
\begin{proof}
By lemma~\ref{l.cont-periodic} and our choice of $\varepsilon$,
one has $d(q_{1,h},q_{2,h})<\varepsilon$ for any $h\in \cU$, hence
$W^u_{h,loc}(q_{1,h})$ intersects $\cW^{cs}_{h,q_{2,h}}$.
Now, $W^{ss}_{h,loc}(q_{2,h})$ is one-codimensional in $\cW^{cs}_{h,q_{2,h}}$
and varies continuously with $h$.
Let us assume that $W^u_{g,loc}(q_{1,g})$ intersects $\cW^{cs,+}_{g,q_{2,g}}$ and that
$W^u_{g',loc}(q_{1,g'})$ intersects $\cW^{cs,-}_{g',q_{2,g'}}$.
By connectedness of $\cV$,
one deduces that for some $h_0\in \cV$ the local  manifolds $W^u_{h_0,loc}(q_{1,h_0})$
and $W^{ss}_{h_0,loc}(q_{2,h_0})\setminus \{q_{1,h_0}\}$ intersect.
By using lemma~\ref{joint-int-easy} one gets a diffeomorphism $h\in \cV$ having a strong
homoclinic intersection in $H(p_h)$, giving a contradiction.
\end{proof}

\begin{lemma}\label{l.continuite}
Under the setting of proposition~\ref{p.continuation},
if $(g_n)$ converges in $\cV$ toward $g$ and
$(\tilde x_n)$ toward $\tilde x$ in $\widetilde {H(p)}$, then
any limit $\bar x$ of $(\Phi_{g_n}(\tilde x_n))$ satisfies
$\pi_g(\bar x)\in \cW^{cs}_{g,x_g}\setminus \cW^{cs,+}_{g,x_g}$ where $x_g=\pi_g(\tilde x)$.
\end{lemma}
\begin{proof}
Let us assume by contradiction that $\bar x$ belongs to $\cW^{cs,+}_{g,x_g}$.
There exists a sequence $(p_n)$ in $\cP$ that converges toward $x_f=\pi_f(\tilde x)$
such that $W^u_{loc}(p_n)$ intersects $\cW^{cs,+}_{x_f}$
and $(p_{n,g_n})$ converges towards $\bar x$.
By proposition~\ref{p.continuation}, the sequence $(p_{n,g})$ converges toward $x_g$.
Hence, one can consider $n$ large such that $p_{n,g}$ is close to $x_g$.
By continuity of the map $h\mapsto p_{n,h}$, the point $p_{n,h}$ is still close to $x_h$ for a diffeomorphisms $h$ nearby. For $m$ large enough,
$p_{n,g_m}$ is close to $x_g$ and $p_{m,g_m}$ is close to $\bar x$.
One deduces that $W^u_{loc,g_m}(p_{n,g_m})$ meets $\cW^{cs,-}_{g_m,p_{m,g_m}}$.
On the other hand, since $W^u_{loc}(p_n)$ meets $\cW^{cs,+}_{x_f}$,
the local manifold $W^u_{loc}(p_{n,g})$ meets $\cW^{cs,+}_{g,p_{m}}$.
The lemma~\ref{l.ordering} below contradicts our assumption that $f$ is far from homoclinic intersections.
\end{proof}

\begin{proof}[Proof of proposition~\ref{p.continuation}]
We introduce the open set $\cU$ and the collection of periodic points $\cP$ as
in the previous sections.

The item a) of the proposition is a direct consequence of lemmas~\ref{l.NGSHI}
and~\ref{joint-int-easy}.
The item b) is first proved in the case $x_g$ is the hyperbolic continuation $q_g$
of a periodic point $q\in \cP$. In this case there is only one possible continuation $x_{g'}$.
We are thus reduced to prove.
\medskip

\noi{\it Claim 1.
Consider any periodic point $q\in \cP$ and an orientation $\sigma$ on $E^c_q$. If
$q_g$ is accumulated by $H(p_g)\cap \cW^{cs,+}_{g,q_g}$ for some $g\in \cV$, then the same
holds for any $g$.}
\begin{proof}
Let us consider $g\in \cV$ such that $q_g$ is accumulated by $H(p_g)\cap \cW^{cs,+}_{g,q_g}$.
In particular, there exists a sequence $(p_n)$ in $\cP$ such that $(p_{n,g})$ converges toward $q_g$
and $W^u_{g,loc}(p_{n,g})$ intersects $\cW^{cs,+}_{g,q_g}$.
By lemma~\ref{l.cont-unstable}, the sequence $(p_{n,g'})$ converges toward $q_{g'}$.
Moreover $W^u_{g',loc}(p_{n,g'})$ does not intersect $W^{ss}_{g',loc}(q_{g'})$
since this would contradict our assumptions
by lemma~\ref{joint-int-easy}.
Also by lemma~\ref{l.ordering}, $W^u_{g',loc}(p_{n,g'})$ does not intersects $\cW^{cs,-}_{g',q_{g'}}$.
One thus deduces that $W^u_{g',loc}(p_{n,g'})$ intersects $\cW^{cs,+}_{g',q_{g'}}$.
The intersection point belongs to $H(p_{g'})$ by lemma~\ref{l.bracket0}, hence $q_{g'}$ is accumulated
by $H(p_{g'})\cap \cW^{cs,+}_{g',q_{g'}}$.
\end{proof}
\medskip

We now prove the item b) in the general case.
\medskip

\noi{\it Claim 2.
Let us consider $x_g\in H(p_g)$ and $x_{g'}\in H(p_{g'})$ and a sequence $(p_n)$ in $\cP$
such that $(p_{n,g})$ converges toward $x_g$ and $(p_{n,g'})$ converges toward $x_{g'}$.
If the local unstable manifolds $W^u_{g,loc}(p_{n,g})$ intersect $\cW^{cs,+}_{g,x_g}$,
then there exists another sequence $(\bar p_n)$ in $\cP$ having the same properties as $(p_n)$
and which satisfies furthermore that
the local unstable manifolds $W^u_{g',loc}(\bar p_{n,g'})$ intersect $\cW^{cs,+}_{g',x_{g'}}$.}
\begin{proof}
We first remark that each point $(p_{n,g})$, with $n$ large enough,
is accumulated by $H(p_{g})\cap \cW^{cs,+}_{g,p_{n,g}}$.
Indeed, $W^u_{g,loc}(p_{1,g})$ intersects $\cW^{cs,+}_{g,p_{n,g}}$ for $n$ large at some point
$y_n$ which belongs to $H(p_g)$ by lemma~\ref{l.bracket0}.
By lemma~\ref{l.largestable}, $\cW^{cs,+}_{g,p_{n,g}}$ is contained in the stable manifold of $p_{n,g}$,
hence the forward orbit of $y_n$ accumulates the orbit of $p_{n,g}$, proving the announced property.
From claim 1, the points $p_{n,g'}$ are also accumulated by 
$H(p_{g'})\cap \cW^{cs,+}_{g',p_{n,g'}}$.

Now we note that $W^u_{g',loc}(p_{n,g'})$ does not intersect $\cW^{cs,-}_{g',x_{g'}}$.
Indeed, if this occurs, one would deduce that for $m\gg n$
the manifold $W^u_{g',loc}(p_{n,g'})$ intersects $\cW^{cs,-}_{g',p_{m,g'}}$
and that $W^u_{g,loc}(p_{n,g})$ intersects $\cW^{cs,+}_{g,p_{m,g}}$.
By lemma~\ref{l.ordering} this would contradict our assumptions.
If $W^u_{g',loc}(p_{n,g'})$ intersects $\cW^{cs,+}_{g',x_{g'}}$ for a subsequence $(\bar p_n)$ of $(p_n)$,
the claim holds. We thus reduced to consider the case
$W^u_{g',loc}(p_{n,g'})$ intersect $W^{ss}_{g',loc}(x_{g'})$. We denote by $z_n$ the intersection.
Since $p_{n,g'}$ is accumulated by $H(p_{g'})\cap \cW^{cs,+}_{g',p_{n,g'}}$,
lemma~\ref{l.integrability} implies that there exists $\bar p_n\in \cP$ such that
\begin{itemize}
\item[--] $\bar p_{n,g'}$ is close to $p_{n,g'}$ (hence $(\bar p_n)$ has the same properties as $(p_n)$),
\item[--] $W^u_{g',loc}(\bar p_{n,g'})$ intersects $\cW^{cs,+}_{g',x_{g'}}$ as announced.
\end{itemize}
\end{proof}

The last claim implies the existence statement of the item b): if $\tilde x_g$ belongs to $\widetilde{H(p_g)}$,
one may approximate the points of $H(p_g)\cap \cW^{cs,+}_{g,x_g}$ by periodic points that are the continuations
for $g$ of points in $\cP$. Hence, there exists a sequence $(p_n)$ in $\cP$ such that
$W^u_{g,loc}(p_{n,g})$ intersects $\cW^{cs,+}_{g,x_g}$ for each $n$. Taking a subsequence, one may also assume that
the points $p_{n,g'}$ converge toward a point $x_{g'}\in H(p_{g'})$. One defines $\tilde x_{g'}=(x_{g'},\sigma)$
such that $\sigma$ is the orientation with matches with the orientation of $\tilde x_g$.
By the previous claim, one can replace the sequence $(p_n)$ by another one $(\bar p_n)$ such that
$(\bar p_{n,g'})$ still converges toward $x_{g'}$ and furthermore
$W^u_{g',loc}(\bar p_{n,g'})$ intersects $\cW^{cs,+}_{g',x_{g'}}$ for each $n$.
The intersection point belongs to $H(p_{g'})$ by lemma~\ref{l.bracket0}, hence $\tilde x_{g'}$ belongs to
$\widetilde {H(p_{g'})}$, as required.
\medskip

\noi{\it Claim 3.
For any $g_1,g_2\in \cV$,
let us consider $x_1\in H(p_{g_1})$ and $x_2\in H(p_{g_2})$ having the same continuation.
Then, there exist two matching orientations on $E^c_{g_1,x_1}$, $E^c_{g_2,x_2}$
and a sequence $(p_n)$ in $\cP$ such that $(p_{n,g_i})$ converges toward $x_i$ and
the local unstable manifolds $W^u_{g_i,loc}(p_{n,g_i})$ intersects $\cW^{cs,+}_{g_i,x_i}$
for $i=1,2$.}
\begin{proof}
By assumption, there exists a sequence $(p_n^0)$ in $\cP$ such that
$(p^0_{n,g_i})$ converges toward $x_i$ for $i=1,2$.
We first replace $(p_n^0)$ by a sequence $(p_n^1)$ so that $W^u_{g_1,loc}(p_{n,g_1}^1)$ does not intersect $W^{ss}_{g_1,loc}(x_1)$:
if there exists a subsequence of $(p_n^0)$ which has this property, we get the subsequences $(p_n^1)$;
otherwise, one can assume that $W^u_{g_1,loc}(p_{n,g_1}^0)$ intersects $W^{ss}_{g_1,loc}(x_1)$ for each $n\geq 0$.
From lemma~\ref{l.integrability},
there exists $y\in \cW^{cs}_{g,p_{n,g_1}^0}$ arbitrarily close to $p_{n,g_1}^0$
such that its unstable manifold intersects $\cW^{cs}_{x_1}\setminus W^{ss}_{g_1,loc}(x_1)$.
One can approximate $y$ by a point $p_{n,g_1}^1$ with $p_n^1\in \cP$.
Doing this for each $n$,
one gets a required sequence $(p_n^1)$ such that $(p^1_{n,g_1})$ still converges toward $x_1$.
By lemma~\ref{l.cont-unstable}, one can ensure that the sequence $(p^1_{n,g_2})$ converges toward $x_2$.
By choosing the orientations on $E^c_{g_i,x_i}$, one can now assume that
$W^u_{g_1,loc}(p^1_{n,g_1})$ intersects $\cW^{cs,+}_{g_1,x_1}$.
By the claim~2, one can modify again the sequence $(p_n^1)$ and replace it by a sequence $(p_n)$
having the required properties.
\end{proof}

One can now conclude the uniqueness part of the item b).
Let us assume by contradiction that $\tilde x_g\in \widetilde{H(p_g)}$ has two
distinct continuations $\tilde x_{g'}^1,\tilde x_{g'}^2$ in $\widetilde{H(p_{g'})}$ as stated in item b).
By lemma~\ref{l.cont-central}, one may assume that $x_{g'}^1$ belongs to $\cW^{cs,-}_{g',x^2_{g'}}$.
Claim~3 provides us with two sequence $(p_n^i)$, $i=1,2$.
On the one hand $W^u_{g,loc}(p^i_{n,g})$ intersects $\cW^{cs,+}_{g,x_g}$,
hence for $n\geq 1$ and $m\gg n$, $W^u_{g,loc}(p_{n,g}^1)$ intersects
$\cW^{cs,+}_{g,p^2_{m,g'}}$.
On the other hand $x_{g'}^1\in \cW^{cs,-}_{g',x^2_{g'}}$,
hence $W^u_{g',loc}(p^1_{n,g'})$ intersects $\cW^{cs,-}_{g',p^2_{m,g'}}$.
By lemma~\ref{l.ordering}, this contradicts our assumptions.
\medskip

The item c) is a direct consequence from the claim~3.
\end{proof}
\bigskip

\begin{corollary}\label{c.continuation}
Under the assumptions of proposition~\ref{p.continuation},
let us consider $g\in \cV$ and $\tilde x,\tilde y\in \widetilde{H(p_{g})}$ such that
the projections $x=\pi_{g}(\tilde x)$ and $y=\pi_{g}(\tilde y)$ are $\varepsilon'$-close and satisfy $y\in W^{ss}_{g,loc}(x)$.

Then, for any $g'\in \cV$ the projections $x'=\pi_{g'}(\tilde x')$ and $y'=\pi_{g'}(\tilde y')$,
associated to the continuations $\tilde x',\tilde y'\in \widetilde{H(p_{g'})}$ of $\tilde x,\tilde y$, still satisfy $y'\in \cW^{cs}_{g',x'}$ and the open region in $\cW^{cs}_{g',x'}$ bounded by
$W^{ss}_{loc}(x)\cup W^{ss}_{loc}(y')$ does not meet $H(p)$.
When the orientations of $\tilde x$ and $\tilde y$ match, one also has $y'\in W^{ss}_{g',loc}(x')$.
\end{corollary}
\begin{proof}
Let us consider two points $\tilde x,\tilde y$ whose projections are $\varepsilon'$-close
and satisfy $y\in \cW^{cs}_{g,x}$.
Then, the same holds for $g'$ and the continuations $x',y'$ by lemma~\ref{l.cont-central}.

The point $x$ is the limit of a sequence $(p_{n,g})$ with $p_n\in\cP$ such that $W^u_{loc}(p_{n,g})$
intersects $\cW^{cs,+}_x$. We claim that $y'$
does not belong to $\cW^{cs,+}_{x'}$.
Let us assume by contradiction that this is not the case.
On the one hand $y$ does not meet $\cW^{cs,+}_{g,x}$
whereas $W^u_{loc}(p_{n,g})$ intersects $\cW^{cs,+}_x$: this implies that
$W^u_{loc}(p_{n,g})$ intersects the component of $\cW^{cs}_{g,y}\setminus W^{ss}_{g,loc}(y)$
corresponding to the orientation of $\tilde x$.
On the other hand for $m$ large $p_{m,g'}$ is close to $x'$, hence
$W^u_{loc}(p_{m,g'})$ intersects the component of $\cW^{cs}_{g',y'}\setminus W^{ss}_{g',loc}(y')$
corresponding to the reversed orientation of $\tilde x$.
There exists $q\in \cP$ such that $q_g$ and $q_{g'}$ are arbitrarily close to $y$ and $y'$ respectively.
Hence, $W^{u}_{g,loc}(p_{n,g})$ intersects one component of $W^{cs}_{g,q}\setminus W^{ss}_{g,loc}(q_g)$
and $W^{u}_{g',loc}(p_{n,g'})$ intersects the component of $W^{cs}_{g',q}\setminus W^{ss}_{g',loc}(q_{g'})$
which corresponds to the other orientation.
From lemma~\ref{l.ordering}, this implies that there exists $h\in \cU$ such that $H(p_h)$
has a strong homoclinic intersection, contradicting our assumptions.

Similarly, $x'$ does not belong to $\cW^{cs,+}_{g',y'}$ for the orientation
on $E^c_{y'}$ induced by $\tilde y$.

When the orientations of $\tilde x$ and $\tilde y$ match, this
implies that $y'$ belongs to $W^{ss}_{g',loc}(x')$.
When the orientations differ, $W^{ss}_{g',loc}(x')$ and $W^{ss}_{g',loc}(y')$
bound the open region $\cW^{cs,+}_{g',x'}\cap \cW^{cs,+}_{g',y'}$.
If there exists a point $\tilde z\in \widetilde{H(p)}$ whose projection by $\pi_{g'}$
belongs to this region, the discussion above proves that
the projection of its continuation for $g$ also belongs to
$\cW^{cs,+}_{g,x}\cap \cW^{cs,+}_{g,y}$. But for $g$ this open region
is empty since $y\in W^{ss}_{g,loc}(x)$.
This is a contradiction. Hence the open region bounded by $W^{ss}_{g',loc}(x')$ and $W^{ss}_{g',loc}(y')$
in $\cW^{cs}_{g',x'}$ does not meet $H(p_{g'})$.
\end{proof}

\begin{corollary}\label{c.continuation2}
Under the assumptions of proposition~\ref{p.continuation},
let us consider a diffeomorphism $g\in \cV$ and a hyperbolic periodic point $q_g$
whose hyperbolic continuation $q_{g'}$ is defined and homoclinically related to the orbit of
$p_{g'}$ for each $g'\in \cV$.

Then, for $g'\in \cV$, $q_{g'}$ is the unique point in $H(p_{g'})$ which has the same
continuation as $q_g$.
\end{corollary}
\begin{proof}
It is enough to prove that in a small neighborhood of $g$, the point $q_{g'}$ is the unique point in $H(p_{g'})$ which has the same continuation as $q_g$.
Let us consider a sequence $(p_n)$ in $\cP$ such that $p_{n,g}$
accumulates on $q_{g}$ and $W^{u}_{g,loc}(p_{n,g})$ intersects $\cW^{cs,+}_{g,q_{g}}$.
One may also choose the $p_n$ such that
$W^{u}_{g,loc}(q_{g})$ intersects $\cW^{cs,-}_{g,p_{n,g}}$.
By lemma~\ref{r.continuity}, for any $g'\in \cV$,
the limit $\bar q_{g'}$ of $(p_{n,g'})$ is a periodic point in $\cW^{cs}_{g,q_{g'}}\setminus
\cW^{cs,-}_{g,q_{g'}}$.
Also $W^{u}_{g',loc}(\bar q_{g'})$ intersects $\cW^{cs,-}_{g',p_{n,g'}}$.

For $n$ large and $g'$ close to $g$, the points $p_{n,g'}$ and $q_{g'}$ are close:
this implies that $\bar q_{g'}$ is contained in a small neighborhood of $q_{g'}$.
Since $q_{g'}$ is uniformly hyperbolic for any $g'$ close to $g$, this implies that
$\bar q_{g'}$ and $q_{g'}$ coincide, as claimed.
\end{proof}

%% file: boundary-1110.tex
\section{Boundary points of quasi-attractors}\label{s.boundary}
We discuss the properties of chain-hyperbolic homoclinic classes as in the previous section
that are furthermore quasi-attractors. In particular, we conclude the proof of proposition~\ref{p.position}. The following slightly more general setting will be considered.
\begin{itemize}
\item[--] Let $V\subset M$ be an invariant open set which is a trapping region $f(\overline V)\subset V$.
\item[--] Assume that the maximal invariant set in $V$ is endowed with a partially
hyperbolic splitting $E^s\oplus E^c\oplus E^u$ such that $\dim(E^c)=1$.
\item[--] Let $H(p)\subset V$ be a chain-hyperbolic homoclinic class with the splitting
$E^{cs}\oplus E^{cu}=(E^s\oplus E^c)\oplus E^u$ and containing the
unstable manifold of $p$.
\end{itemize}
In particular, $H(p)$ is saturated by the unstable leaves, tangent to $E^{u}$, and
$U$ is foliated by a forward invariant foliation which extends the strong stable lamination tangent to $E^s$.

\subsection{Comparison of unstable leaves through the strong stable holonomy}\label{su-intersections}
Let us assume that $H(p)$  satisfies the following property.
\begin{description}
\item[Strong intersection property:] \emph{there exist $x,y\in H(p)$ with $y\in W^{ss}(x)\setminus \{x\}$.}
\end{description}
As explained in section~\ref{ss.reduction}, this property prevents the class
to be contained in  a lower dimensional submanifold tangent to $E^c\oplus E^u$.
  
For any point $x\in H(p)$, we fix arbitrarily some plaque $\cD$
transverse to $W^{ss}_{loc}(x)$ and define for any $z$
close to $W^{ss}_{loc}(x)$ the projection $\Pi^{ss}(z)\in \cD$
through the strong stable holonomy.
When $z$ belongs to $H(p)$, the map $\Pi^{ss}$ is a homeomorphism from a neighborhood of $z$
in $\cW^{cu}_z$ to a neighborhood of $\Pi^{ss}(z)$ in $\cD$.
Hence, the projection $\Pi^{ss}(W^u_{loc}(z))$
is a one-codimensional topological submanifold of $\cD$.
In particular, in a neighborhood of $z$, the set
$\cD\setminus \Pi^{ss}(W^u_{loc}(z))$  has locally two connected components.

\begin{definition}\label{definition-cases}
Let us fix $\varepsilon_0>0$ small. The following situations can occur.
\begin{description}
\item[-- The transversal case.] There exists $x,y\in H(p)$ with $y\in W^{ss}_{loc}(x)\setminus \{x\}$ such that
$\Pi^{ss}(W^u_{loc}(y))$ intersects both components of
$\Pi^{ss}(B(x,\varepsilon_0))\setminus \Pi^{ss}(W^u_{loc}(x))$.

\item[-- The jointly integrable case.] There exists $x,y\in H(p)$ with $y\in W^{ss}_{loc}(x)\setminus \{x\}$
such that \\ $\Pi^{ss}(W^u_{loc}(x))$ and $\Pi^{ss}(W^u_{loc}(y))$ coincide in $\Pi^{ss}(B(x,\varepsilon_0))$.

\item[-- The strictly non-transversal case.]
For any $x,y\in H(p)$ with $y\in W^{ss}_{loc}(x)\setminus \{x\}$,
the projection $\Pi^{ss}(W^u_{loc}(y))$
intersects one of the components of
$\Pi^{ss}(B(x,\varepsilon_0))\setminus \Pi^{ss}(W^u_{loc}(x))$ and is disjoint from the other.
\end{description}
\end{definition}
\noindent Note that these definitions do not depend on the choice of the plaque $\cD$.
Clearly one of these three cases happen.
The transversal and the jointly integrable cases may occur at the same time.
The strictly non-transversal case is quite particular.
\begin{lemma}\label{l.boundary}
Let us assume that $H(p)$ does not satisfy the transversal case and consider
two points $x,y\in H(p)$ with $y\in W^{ss}_{loc}(x)\setminus \{x\}$.
For  $\varepsilon$ small, if $\Pi^{ss}(W^u_{loc}(y))$ intersects
$\Pi^{ss}(B(x,\varepsilon))\setminus \Pi^{ss}(W^u_{loc}(x))$,
then $x$ and $y$ are not accumulated by $H(p)$
in the same component of $\cW^{cs}_x\setminus W^{ss}_{loc}(x)$.
\end{lemma}
\begin{proof}
Note that if $\varepsilon$ is small enough
and if $\Pi^{ss}(W^u_{loc}(y))$ intersects
$\Pi^{ss}(B(x,\varepsilon))\setminus \Pi^{ss}(W^u_{loc}(x))$,
then $\Pi^{ss}(W^u_{loc}(x))$ intersects
$\Pi^{ss}(B(y,\varepsilon_0))\setminus \Pi^{ss}(W^u_{loc}(y))$.

We denote by $U^+_x,U^-_x$ the local
connected components of $\Pi^{ss}(B(x,\varepsilon))\setminus \Pi^{ss}(W^u_{loc}(x))$
such that $\Pi^{ss}(W^u_{loc}(y))$ meets $U^-_x$ and is disjoint from $U^+_x$.
We also denote by $U^+_y,U^-_y$ the local
connected components of $\Pi^{ss}(B(y,\varepsilon_0))\setminus \Pi^{ss}(W^u_{loc}(y))$
such that $\Pi^{ss}(W^u_{loc}(x))$ meets $U^+_y$ and is disjoint from $U^-_y$.
In particular, $U^+_x\subset U^+_y$.

Let us assume by contradiction that $y$ is accumulated by $H(p)$
from the side of $\cW^{cs}_x\setminus W^{ss}_{loc}(x)$ which projects in $U^+_x$.
Let us consider a point $z\in H(p)$ close to
$y$ and which projects inside $U^+_x$.
Its local unstable manifold is close to the unstable manifold of $y$, hence
$\Pi^{ss}(W^u_{loc}(z))$ meets $U^-_x$ also.
This implies that we are in the transversal case which is a contradiction.

Similarly if $x$ is accumulated by $H(p)$ from the side of
$\cW^{cs}_x\setminus W^{ss}_{loc}(x)$ which projects in $U^-_y$, we find a contradiction.
One deduces that $x$ and $y$ can not be accumulated by $H(p)$
on the same side of $\cW^{cs}_x\setminus W^{ss}_{loc}(x)$.
\end{proof}

\subsection{Structure of the  stable boundary points}\label{ss.structure}

For quasi-attractors not in the transversal case, we prove 
that the  stable boundary points (see section~\ref{ss.one-codim})
belong to the unstable manifold of a periodic orbit.

\begin{proposition}\label{p.boundary}
Let $H(p)$ be a homoclinic class such that
\begin{itemize}
\item[--] $H(p)$ is a quasi-attractor endowed with a partially hyperbolic structure $E^s\oplus E^c\oplus E^u$ such that 
$E^c$ is one-dimensional and $E^{cs}=E^s\oplus E^c$ is thin trapped,
\item[--] for any periodic points $q,q'\in H(p)$ homoclinically related to the orbit of $p$, the manifolds
$W^{ss}(q)\setminus \{q\}$ and $W^u(q')$ are disjoint,
\item[--] the transversal case does not hold.
\end{itemize}
Then any  stable boundary point of $H(p)$ belongs to the unstable manifold of a periodic point.
\end{proposition}

\begin{proof}
Let $x$ be a  stable boundary point of $H(p)$.
Let us assume by contradiction that the point $x$ does not belong to the unstable manifold of a periodic point.
In particular, the unstable manifolds $W^u(f^n(x))$ for $n\in \ZZ$
are all distinct.

Let us consider a point $\zeta$ in the $\alpha$-limit set of $x$.
By considering a plaque transverse to $W^{ss}_{loc}(\zeta)$,
the holonomy $\Pi^{ss}$ is well defined in a neighborhood of $\zeta$.
Since $E^{cs}$ is thin trapped, the plaques of the family $\cW^{cs}$
can be chosen small and one may thus assume that one of the components of
$\cW^{cs}_{x}\setminus W^{ss}_{loc}(x)$ is disjoint from $H(p)$.
Let us introduce two backward iterates $x_1=f^{-n}(x)$ and $x_2=f^{-m}(x)$,
of $x$ close to $\zeta$. By the trapping property, one of the components of
$\cW^{cs}_{x_i}\setminus W^{ss}_{loc}(x_i)$ is also disjoint from $H(p)$ for $i=1$ and $i=2$.
Since $x_1$ and $x_2$ are close, it makes sense to compare the orientations of $E^c_1$ and $E^c_2$.
Choosing different iterates $x_1$ and $x_2$ if necessary, one may assume
that the tangent map $Df^{n-m}\colon E^c_{x_1}\to E^c_{x_2}$ preserves the orientation.

\begin{claim}
Exchanging $x_1$ and $x_2$ if necessary, $W^{ss}_{loc}(x_2)$
meets $W^u_{loc}(x_1)$.
\end{claim}
\begin{proof}
Observe that the plaque $\cW^{cs}_{x_2}$ meets $W^u_{loc}(x_1)$ at a point $x'_1\in H(p)$.
One chooses a small path $t\mapsto x_1(t)$ inside $W^u_{loc}(x_1)$
between $x_1=x_1(0)$ and $x'_1=x_1(1)$. Since $H(p)$ is a quasi-attractor
this path is contained in $H(p)$.
Each plaque $\cW^{cs}_{x_1(t)}$ meets $W^u_{loc}(x_2)$ at a point $x_2(t)$,
defining a path $t\mapsto x_2(t)$ inside $W^u_{loc}(x_2)\cap H(p)$.

For any $t\in [0,1]$, the plaques $\cW^{cs}_{x_1(t)}$ and $\cW^{cs}_{x_2(t)}$
projects by $\Pi^{ss}$ on a $C^1$ curve $\gamma(t)$ which is
topologically transverse to $\Pi^{ss}(W^u_{loc}(x_1))$ and $\Pi^{ss}(W^u_{loc}(x_2))$.
The set $\cD\setminus \Pi^{ss}(W^u_{loc}(x_1))$ has locally two connected components $U^+,U^-$. Hence, $\gamma(t)\setminus \Pi^{ss}(x_1)$ has two connected components $\gamma^+(t)\subset U^+$ and $\gamma^-(t)\subset U^-$ for each $t$.

Let us consider the components $\gamma_1^\pm:=\gamma^\pm(0)$.
By lemma~\ref{l.NGSHI} and since $x_1$ is a  stable boundary point,
$\Pi^{ss}(H(p)\cap \cW^{cs}_{x_1})$ meets one of them, $\gamma_1^-$,
and is disjoint from the other one, $\gamma_1^+$.
Similarly, we define $\gamma^-_2,\gamma^+_2$ the connected components of
$\gamma(1)\setminus \Pi^{ss}(x_2)$, such that
$\Pi^{ss}(H(p)\cap \cW^{cs}_{x_2})$ meets the first
and is disjoint from the second.
One deduces that $\gamma^+_2$ is contained in $U^+$ or in $U^-$.
Recall that $\gamma^+_1\subset U^+$.
Since $Df^{n-m}$ preserves the local orientation of $E^c$,
the orientations on $\gamma^+_1$ and $\gamma^+_2$ match and $\gamma_2^+$
is contained in $U^+$.

As a consequence $\Pi^{ss}(W^u_{loc}(x_2))$ is disjoint from
$\gamma^+_1:=\gamma^+(0)$ and from $\gamma^-_2:=\gamma^-(1)$. Since we are not in the transversal case,
one deduces that $\Pi^{ss}(W^u_{loc}(x_2))$ contains $\Pi^{ss}(x_1)$ or
$\Pi^{ss}(x_1')$. Exchanging $x_1$ and $x_2$ if necessary, one has
$W^{ss}(x_1')=W^{ss}(x_2)$.
\end{proof}
\smallskip

Let us denote $x'_1$ the intersection point between
$W^{ss}_{loc}(x_2)$ and $W^u_{loc}(x_1)$.
Since $x_2$ is a boundary point, one connected component of
$\cW^{cs}_{x_2}\setminus W^{ss}_{loc}(x_2)$ is disjoint from $H(p)$.
By lemma~\ref{l.NGSHI} the other component contains sequences of points
of $H(p)$ that accumulate on $x_2$ and $x_1'$.
One deduces from the lemma~\ref{l.boundary} that
the projections of $W^u_{loc}(x_1)$ and $W^u_{loc}(x_2)$ through
the strong stable holonomy match.
Consequently, there exists a periodic point $q\in H(p)$ such that
$W^u_{loc}(x_1)$ and $W^u_{loc}(x_2)$
project on $W^u(q)$ by the strong stable holonomy.
Note that when $x_1,x_2$ are arbitrarily close to $\zeta$,
the point $q$ is  also close.

If $q$ and $\zeta$ are distinct, one may consider
backward iterates $x'_1,x'_2$ closer to $\zeta$.
One builds another periodic point $q'\in H(p)$.
All the local unstable manifolds of $x_1,x_2,x'_1,x'_2,q,q'$
have the same projection through the strong stable holonomy.
By lemma~\ref{l.linked}, $q$ and $q'$ are homoclinically related to the orbit of $p$.
This proves that $W^{ss}_{loc}(q)$ and $ W^{u}_{loc}(q')$ intersect,
contradicting our assumption.

If $q$ and $\zeta$ coincide, one can consider higher backward iterates
$f^{-n}(x)$ in a neighborhood of $\zeta$. They all have distinct local unstable plaques
whose projection by the strong stable holonomy coincide.
One deduces that one can find a sequence of such backward iterates which accumulates on
a point $\zeta'\in W^{ss}_{loc}(\zeta)$ different from $\zeta$.
Repeating the construction near $\zeta'$,
one builds a periodic point $q'\in H(p)$
distinct from $q$ and as before $W^{ss}_{loc}(q)$ and $ W^{u}_{loc}(q')$ intersect,
giving again a contradiction.
This ends the proof of the proposition.
\end{proof}

\subsection{The transversal case}
When $H(p)$ is a quasi-attractor, the lemma~\ref{l.cont-quasi-attractor} ensures that
for diffeomorphisms $g$ close to $f$  the unstable manifold
$W^u(p_g)$ is still contained in $H(p_g)$.

\begin{lemma}\label{l.transversal}
Let us assume that $H(p)$ is a quasi attractor and consider $f'$, $C^1$-close to $f$,
such that the transversal case holds for a pair of points $x\neq y$ in $H(p_{f'})$.
Then, for any two different hyperbolic periodic points $p_x,p_y$ homoclinically related
to the orbit of $p_{f'}$ and close to $x$ and $y$ respectively,
and for any diffeomorphism $g$ that is $C^1$-close to $f'$
there exist $x'\in W^u(p_{x,g})$ and $y'\in W^u(p_{y,g})$ in $H(p_g)$ satisfying $W^{ss}(x')=W^{ss}(y')$.
\end{lemma}
\begin{proof}
Let $x, y\in H(p_{f'})$ with $y\in W^{ss}_{loc}(x)\setminus \{x\}$
such that the intersection between $\Pi^{ss}(W^{u}_{loc}(x))$
and $\Pi^{ss}(W^{u}_{loc}(y))$ is topologically transversal.
Consider two periodic points $p_x,p_y$ homoclinically related to $p_{f'}$
and close to $x$ and $y$ respectively, so that
the local unstable manifolds of $p_x$ and $p_y$ are close to the local
unstable manifold of $x$ and $y$.
This implies that $\Pi^{ss}(W^u_{loc}(p_x))$ and $\Pi^{ss}(W^u_{loc}(p_y))$
intersect topologically transversally. By continuity of the local unstable manifolds and the local strong stable holonomy this property still holds for any
$g$ close to $f'$:
there are points $x'\in W^u_{loc}(p_{x,g}), y'\in W^u_{loc}(p_{y,g})$ such that 
$W^{ss}(x')=W^{ss}(y')$. By lemma~\ref{l.cont-quasi-attractor}, the local unstable manifolds of $p_{x,g},p_{y,g}$ remain in $H(p_g)$ and therefore the points $x', y'$ are in $H(p_g).$
\end{proof}

\subsection{The jointly integrable case}
The next lemma states that in the jointly integrable case either a heterodimensional cycle is created by a $C^r-$perturbation or for any point in the class there is a well defined continuation.
\begin{lemma}\label{joint.int.continuation}
Let us assume that $H(p)$ is a quasi-attractor whose periodic orbits
are hyperbolic, that $E^{cs}$ is thin trapped
and that the jointly integrable case holds.
Then for any $r\geq 1$ such that $f\in \diff^r(M)$, one of the following cases occurs.
\begin{itemize}
\item[--] There exists $g$ that is $C^r$-close to $f$
such that $H(p_g)$ exhibits a strong homoclinic intersection.
\item[--] There exists a hyperbolic periodic point $q$ homoclinically related to the orbit of $p$, two maps $g\mapsto x_g,y_g$ defined on a neighborhood $\cV$ of $f$ in $\diff^r(M)$ and continuous at $f$ such that for any diffeomorphism $g\in \cV$ the points $x_g,y_g$ belong to $H(p_g)\cap W^{s}(q_g)$ and are continuations of $x_f,y_f$. Moreover
$y_g$ belongs to $W^{ss}_{loc}(x_g)$.
\end{itemize}
\end{lemma}
\begin{proof}
Note that by our assumptions the results of sections~\ref{s.weak-hyperbolicity} and~\ref{s.continuation} apply. In particular
for $g$ $C^1$-close to $f$ the class $H(p_g)$ is still chain-hyperbolic and
contains $W^u(p)$. Let us assume that the first item of the proposition does not hold:
on a $C^r$-neighborhood $\cV$ of $f$, there is no diffeomorphism whose homoclinic
class $H(p_g)$ has a strong homoclinic intersection.

Recall that all the periodic orbits are hyperbolic. Since $E^s\oplus E^c$ is thin trapped,
they have the same index and by lemma~\ref{l.linked}, they are all homoclinically related.
There is no periodic points $q,q'\in H(p)$ such that
$W^{ss}(q)\setminus \{q\}$ and $W^u(q')$ intersect: otherwise, one gets a strong homoclinic intersection by using lemma~\ref{joint-int-easy}.
In particular, the proposition~\ref{p.boundary} can be applied.

As in definition~\ref{definition-cases}, let $x,y\in H(p)$ be two close points with disjoint local unstable manifolds
such that for any $z\in W^u_{loc}(x)\cap B(x,\varepsilon_0)$
we have $W^{ss}_{loc}(z)\cap W^u_{loc}(y)\neq \emptyset.$
Observe that there exists a periodic point $q\in H(p)$ close to $x$
whose local stable manifold intersects both the local unstable manifold of $x$ and $y$. Without lose of generality, we can assume that $x,y$ belong to $W^s_{loc}(q)$.

The point $x, y$ do not belong both to the unstable manifold of some periodic points $p_x, p_y$:
otherwise, we would get a strong connection by applying lemma~\ref{joint-int-easy}.
We can thus now assume that $x$ does not belong to the unstable manifold of a periodic point.
In particular, by proposition~\ref{p.boundary} it is not a  stable boundary point
and it is accumulated by points in $H(p)$ from both connected components of $\cW^{cs}_x\setminus W^{ss}_{loc}(x).$
The corollary~\ref{c.continuation} (in the orientation preserving case) implies that there exist two maps
$g\mapsto x_g,y_g$ on $\cV$ satisfying $(x_f,y_f)=(x,y)$ and
for any $g$ close to $f$, the points $x_g,y_g$ belong to $H(p_g)$ and have the same strong stable manifold. The points $x_g,y_g$ are accumulated by $H(p_g)$ in the same component
of $\cW^{cs}_{x_g}\setminus W^{ss}_{loc}(x_g)$.

Let us prove the continuity.
Since the point $x$ is accumulated from both sides, it has two continuations
$g\mapsto x_g,x'_g$. By lemma~\ref{l.cont-central},
for any $g$ one has $x'_g\in \cW^{cs}_{x_g}$.
One can choose an orientation of $E^c_x$ and by lemma~\ref{l.ordering}
assume that for any $g$, the point $x'_g$ does not meet
$\cW^{cs,+}_{x_g}$.
By lemma~\ref{l.continuite}, the map $g\mapsto x'_g$ is semi-continuous at $f$:
when $(g_n)$ is a sequence that converges to $f$, then
any limit $\bar x'$ of $(x'_{g_n})$ does not meet $\cW^{cs,-}_{x'_f}=\cW^{cs,-}_{x}$.
One deduces that any limit $\bar x$ of $(x_{g_n})$ does no meet
$\cW^{cs,-}_{x}$ either.
Since the map $g\mapsto x_g$ is also semi-continuous,
the limit $\bar x$ does not meet $\cW^{cs,+}_{x}$.
One deduces that $\bar x$ belongs to $W^{ss}_{loc}(x)$.
The orbit of $\bar x$ is shadowed by the orbit of $x$, hence one
deduces that $\bar x=x$.
Let us now consider any limit point $\bar y$
of $(y_{g_n})$. By construction it has to belong to $W^{ss}_{loc}(x)$
and $W^{ss}_{loc}(y)$,
and so $\bar y=y$. We have thus proved that the maps
$g\mapsto x_g,y_g$ are continuous at $f$.
\end{proof}

\subsection{The strictly non-transversal case}
In the strictly non-transversal case, roughly speaking is proved that either by perturbation is created a strong homoclinic connection, or for a diffeomorphisms nearby the strong stable leaves contains at most one point in the class or there are two periodic points such that for any diffeomorphisms nearby their unstable manifolds intersects some strong stable leaves (see lemma \ref{l.strictly-transversal}). 
\begin{lemma}\label{l.boundary1}
Let us assume that $H(p)$ satisfies the strictly non-transversal case.
Then, any close points $x\neq y$ in $H(p)$ satisfying $y\in W^{ss}_{loc}(x)$
are  stable boundary points. Moreover they are not accumulated by
$H(p)$ in the same component of $\cW^{cs}_{x}\setminus W^{ss}_{loc}(x)$.
\end{lemma}
\begin{proof}
Since $H(p)$ satisfies the strictly non-transversal case and $x,y$ are close,
there exists $y'\in W^u_{loc}(y)$ and $x'\in W^u_{loc}(x)$ such that
$y'\in W^{ss}_{loc}(x')$ and for any $\varepsilon>0$, the manifolds
$\Pi^{ss}(W^{u}_{loc}(y'))$ intersects
$\Pi^{ss}(B(x',\varepsilon))\setminus \Pi^{ss}(W^u_{loc}(x'))$.
By lemma~\ref{l.boundary}, they are not accumulated by $H(p)$
in the same component and in particular both are  stable boundary points.
\end{proof}
\smallskip

For the points $(x,y)$ as in the previous lemma
the following property obviously holds
(the open region considered below is then empty):
\begin{itemize}
\item[(**)]
\it $\cW^{cs}_x$ contains $y$. The open region in $\cW^{cs}_x$ bounded by $W^{ss}_{loc}(x)\cup W^{ss}_{loc}(y)$
does not meet $H(p)$.
\end{itemize}
Note that this property already appeared in corollary~\ref{c.continuation}.
The next lemma states that the set of such pairs $(x,y)$ is quite small.

\begin{lemma}\label{l.boundary2}
Let $H(p)$ be a quasi-attractor such that $E^{cs}$ is thin trapped,
the strictly non-transversal case holds and for any periodic points
$q,q'\in H(p)$ the manifolds $W^{ss}(q)\setminus \{q\}$ and $W^{u}(q)$ are disjoint.
Let us fix $\delta>0$.
Then, there exist $N\geq 1$ and finitely many periodic points $p_1,\dots,p_s$ such that
any points $x\neq y$ in $H(p)$ satisfying (**) and $d(x,y)\geq \delta$
belong to the union of the $f^N(W^u_{loc}(p_i))$, $i\in\{1,\dots,s\}$.
\end{lemma}
\begin{proof}
We fix $\delta>0$ small.
We first note that by lemma~\ref{l.boundary1} and proposition~\ref{p.boundary},
any $x,y$ as in the statement of the lemma are  stable boundary points and there exists
some periodic points $p_x,p_y\in H(p)$ such that
$x$ belongs to $W^u(p_x)$ and $y$ to $W^u(p_y)$.

Let $P$ be the (closed) set of pairs $(x,y)\in H(p)^2$ satisfying~(**) and $d(x,y)\geq \delta$.
We have to prove that if two pairs $(x,y)$ and $(x',y')$ in $P$ are close, then
$x'\in W^u_{loc}(x)$ and $y'\in W^u_{loc}(y)$. This is done by contradiction:
we consider a sequence $(x_n,y_n)_{n\geq 0}$ in $P$ that converges toward $(x,y)$
and assume that all the leaves $W^{u}_{loc}(x_n)$ are distinct.
One may assume that $x$ is accumulated by $H(p)$ inside $\cW^{cs,+}_{x}$.

First we claim that $W^{u}_{loc}(x_n)$ does not cut $W^{ss}_{loc}(x)$.
Otherwise, we denote by $z_n$ the intersection point. The plaque $\cW^{cs}_{z_n}$ coincides
with $\cW^{cs}_{x}$ in a neighborhood of $z_n$ by lemma~\ref{l.uniqueness-coherence}, hence
$z_n$ is not accumulated by $H(p)\cap \cW^{cs,-}_{z_n}$ for $n$ large.
One deduces that $z_n$ and $x$ belongs to the same local strong stable leaf and are accumulated
by points of $H(p)\cap \cW^{cs,+}_{z_n}$ and $H(p)\cap \cW^{cs,+}_{x}$ respectively,
contradicting the definition of the strictly non-transversal case.

Let $\Pi^{ss}$ be the projection along the strong stable holonomy on a disk $\cD$ transverse
to $W^{ss}_{loc}(x)$. The projections $\Pi^{ss}(W^u_{loc}(x_n)), \Pi^{ss}(W^u_{loc}(x)), \Pi^{ss}(W^u_{loc}(y)), \Pi^{ss}(W^u_{loc}(y_n))$ are one codimensional manifolds of $\cD$:
by our assumptions, the one-dimensional curve $\gamma=\Pi^{ss}(\cW^{cs}_x)$ meets them in this order.
Since we are in the strictly non-transversal case, the order is the same on any other curve
$\gamma'=\Pi^{ss}(\cW^{cs}_{x'})$ where $x'\in W^u_{loc}(x)$ is close to $x$.
In particular, when $x'$ is the intersection point between $W^u_{loc}(x)$ and
$\cW^{cs}_{x_n}$, one finds a contradiction since $W^{u}_{loc}(x)$ and $W^{u}_{loc}(y)$
cannot intersect the open region of $\cW^{cs}_x$ bounded by $W^{ss}_{loc}(x)\cup W^{ss}_{loc}(y)$
and by the same argument as above, $W^{u}_{loc}(x)\cap W^{ss}_{loc}(x_n)$ and
$W^{u}_{loc}(y)\cap W^{ss}_{loc}(y_n)$ are empty.
This concludes the proof of the lemma.
\end{proof}

\begin{lemma}\label{l.strictly-transversal}
Let us assume that $H(p)$ is a quasi-attractor whose periodic orbits
are hyperbolic, that $E^{cs}$ is thin trapped
and that the strictly non-transversal integrable case holds.
Then for any $r\geq 1$ such that $f\in \diff^r(M)$, one of the following cases occurs.
\begin{itemize}
\item[--] There exists $g$, $C^r$-close to $f$
such that $H(p_g)$ exhibits a strong homoclinic intersection.
\item[--] There exists $g$, $C^r$-close to $f$
such that for any $x\neq y$ in $H(p_g)$ one has $W^{ss}(x)\neq W^{ss}(y)$.
\item[--]
There exist two hyperbolic periodic points $p_x,p_y$ homoclinically related
to the orbit of $p$ and an open set $\cV\subset \diff^r(M)$
whose closure contains $f$, such that for any $g\in \cV$ the class $H(p_g)$ contains
two different points $x\in W^u(p_{x,g})$ and $y\in W^u(p_{y,g})$
satisfying $W^{ss}(x)=W^{ss}(y)$.
\end{itemize}
\end{lemma}
\begin{proof}
As in the proof of lemma~\ref{joint.int.continuation},
for $g$ that is $C^1$-close to $f$ the class $H(p_g)$ is still chain-hyperbolic and contains $W^u(p)$. Moreover, one can assume that for any periodic points $q,q'\in H(p)$
the manifolds $W^{ss}(q)\setminus \{q\}$ and $W^u(q)$ do not intersect.
Let us fix $\delta<0$ small.
One can consider the periodic points $p_1,\dots,p_s$ and the integer $N\geq 1$
provided by the lemma~\ref{l.boundary2}.
These points are hyperbolic, homoclinically related to $p$ by lemma~\ref{l.linked}
and have a continuation for any $g$ that is $C^1$-close to $f$.
One may also assume there is no $g$ in a $C^r$-neighborhood of $f$
such that $H(p_g)$ has a strong homoclinic intersection.
One can then consider the continuation given by proposition~\ref{p.continuation}.
We also introduce the period $\tau_i$ of each periodic point $p_i$.

In a small neighborhood of $f$ in $\diff^r(M)$,
consider for each pair $(p_i,p_j)$ the (closed) subset $D_{i,j}$ of diffeomorphisms $g$ such that
the class $H(p_g)$ contains some distinct points $x\in f^{N+\tau_i}(\overline{W^u_{loc}(p_{i,g})})$
and $y\in f^{N+\tau_j}(\overline{W^u_{loc}(p_{j,g})})$ with $y\in \overline{W^{ss}_{loc}(x)}$.
The diffeomorphisms in the interior of $D_{i,j}$ are in the third case of the lemma.

If the sets $D_{i,j}$ have empty interior,
there exists an open set $\cV$ in $\diff^r(M)$
whose closure contains $f$ such that for any $g\in \cU$,
any $p_i,p_j$ and any distinct points
$x\in f^{N+\tau_i}(W^u_{loc}(p_{i,g}))$
and $y\in f^{N+\tau_j}(W^u_{loc}(p_{j,g}))$ one has $y\not\in W^{ss}_{loc}(x)$.
To conclude, we have to prove that for $g\in \cU$ close to $f$
and any distinct points $x, y\in H(p_g)$ one has $W^{ss}(x)\neq W^{ss}(y)$,
giving the second case of the lemma.
This is done by contradiction: one considers a pair $(x,y)$ such that
$y\in W^{ss}_{loc}(x)$ and
up to consider a backward iterate, one can require that
the points $x,y$ satisfy $d(x,y)>2\delta$.
Having chosen $g$ close enough to $f$, one deduces (lemma~\ref{l.cont-central})
that any continuations $x_f,y_f$ for $f$ still satisfy
$d(x_f,y_f)>\delta$.

If $x,y$ are accumulated in the same component
of $\cW^{cs}_{x}\setminus W^{ss}_{loc}(x)$, then by corollary~\ref{c.continuation}
(in the orientation preserving case) the same holds for the continuations $x_f,y_f$ for $f$.
This contradicts lemma~\ref{l.boundary1}.

If $x,y$ are accumulated in different components
of $\cW^{cs}_{x}\setminus W^{ss}_{loc}(x)$,
then by corollary~\ref{c.continuation} (in the orientation reversing case)
the continuations $x_f,y_f$ for $f$
satisfy (**). Since their distance is bounded from below by $\delta$,
lemma~\ref{l.boundary2} implies that $x_f,y_f$
belong to $f^N(W^u_{loc}(p_i))$ and $f^N(W^u_{loc}(p_j))$ respectively.
By lemma~\ref{l.cont-unstable}, one deduces that for the diffeomorphism $g$ close, the points $x,y$
belong to $g^{N+\tau_i}(W^u_{loc}(p_{i,g}))$ and $g^{N+\tau_j}(W^u_{loc}(p_{j,g}))$ respectively.
This contradicts our assumption on $g$.
\end{proof}

\subsection{Proof of proposition~\ref{p.position}}

Let us consider a diffeomorphism $f\in \diff^{1+\al}(M)$, $\al\geq 0$, and a homoclinic class $H(p)$ as in the statement of theorem~\ref{t.position}
and assume that the two first cases of the proposition do not occur.
If the jointly integrable case holds,
the lemma~\ref{joint.int.continuation} gives the third case of the proposition.
If the transversal or the strictly non-transversal case holds,
the lemmas~\ref{l.transversal} and~\ref{l.strictly-transversal}
give the fourth case of the proposition.

%% file: jointintegrable-1110.tex
\section{Periodic stable leaves: proof of theorem \ref{t.stable}}
\label{proofjointint}
In this section we prove theorem~\ref{t.stable} and proposition~\ref{p.generalized-strong-connectionCr}.
Let us consider:
\begin{itemize}
\item[1)] A diffeomorphism $f_0$
and a homoclinic class $H(p_{f_0})$ which is a chain-recurrence class
endowed with a partially hyperbolic splitting $E^s\oplus E^c\oplus E^u$ where
$E^c$ is one-dimensional and $E^s\oplus E^c$ is thin-trapped.
\item[2)] Some $\al\in [0,1)$, a $C^{1+\al}$-diffeomorphism $f$ that is $C^1$-close to $f_0$,
an open neighborhood $\cV\subset \diff^{1+\al}(M)$ of $f$ and some collections of hyperbolic periodic points $q_{f}$, $\{p_{n,f}^x\}_{n\in \NN}$ and $\{p^{y}_{n,f}\}_{n\in \NN}$ for $f$
such that the following properties hold.
\begin{itemize}
\item[--] For $g\in \cV$, the continuations $q_{g}$, $p_{n,g}^x$, $p_{n,g}^y$
exist and are homoclinically related to $p_{g}$.
\item[--] For each $g\in \cV$, the sequences $(p_{n,g}^x)$ and $(q_{n,g}^y)$ converge
towards two distinct points $x_g,y_g$ in $H(p_g)\cap W^s_{loc}(q_g)$
such that $y_g$ belongs to $W^{ss}_{loc}(x_g)$.
\item[--] The maps $g\mapsto x_g,y_g$ are continuous at $f$.
\end{itemize}
\end{itemize}

We will show that if $\al\geq 0$ is small, then
there exists a diffeomorphism $g\in \cV$ whose homoclinic class $H(p_g)$ has a strong homoclinic intersection.

\begin{propo}\label{p.jointintegrable}
For any diffeomorphism $f_0$ and any homoclinic class $H(p_{f_0})$ satisfying the assumption 1) above, there exists $\alpha_0\in (0,1)$ and a $C^1$-neighborhood $\cU$ of $f$
with the following property.

\noindent
For any $\al\in [0,\al_0]$, any diffeomorphism $f$,
any neighborhood $\cV\subset \diff^{1+\al}(M)$ and any maps $g\mapsto x_g,y_g$ satisfying the assumption 2), there exists a transverse intersection $z\in W^s(q_{f})\cap W^{u}_{loc}(q_{f})\setminus \{q_{f}\}$
and an arc of diffeomorphisms $(g_t)_{t\in [-1,1]}$ in $\cV$ such that
\begin{itemize}
\item[--] for each $t\in [-1,1]$, considering the
(unique) continuation $z_t$ of $z$ for $g_t$, 
the center stable plaque $\cD^{cs}_{z_t}$ intersects $W^u_{loc}(x_{g_t})$ 
and $W^u_{loc}(y_{g_t})$ at some points $\hat x_t$ and $\hat y_t$;
\item[--] considering an orientation of the central bundle in a neighborhood of $q$,
one has $$\hat y_{-1}\in \cD^{cs,-}_{\hat x_{-1}} \text{ and } \hat y_{1}\in \cD^{cs,+}_{\hat x_{1}}.$$
\end{itemize}
\end{propo}
\medskip

Let us conclude the proof of theorem~\ref{t.stable}.
By construction and lemma~\ref{l.bracket0}, for each $t\in [-1,1]$ the points
$z_t,x_t,y_t$ belong to the homoclinic class $H(p_{g_t})$.
Moreover one can find for each $n\in \NN$ two hyperbolic periodic points
$\hat p_{n,g}^x$ and $\hat p_{n,g}^y$ whose continuations exists for every $g\in \cV$,
are homoclinically related to $p_{g}$ and are arbitrarily close to the intersections
$\hat x_g,\hat y_g$
between $W^{s}_{loc}(z_{g})$ and $W^u_{loc}(x_g)$ or $W^u_{loc}(y_g)$ respectively.
By corollary~\ref{c.continuation2}, one can assume that the hyperbolic points
$\hat p_{n,g}^x$ and $\hat p_{n,g}^y$ are the hyperbolic continuations of points of $\cP$.
For $n$ large, $W^u_{loc}(\hat p^y_{n})$ intersects $\cW^{cs,-}_{\hat p^x_{n}}$ for $g_{-1}$
and $\cW^{cs,+}_{\hat p^x_{n}}$ for $g_{1}$.
One can thus apply lemma~\ref{l.ordering} and obtain a diffeomorphism $g\in \cV$ which has a strong
homoclinic intersection.
Note that the neighborhood $\cV$ of $f$ can be taken arbitrarily small. As a consequence
the perturbation $g$ is arbitrarily $C^{1+\al}$-close to $f$.
Hence the proposition implies theorem \ref{t.stable}.

\subsection{An elementary $C^{1+\al}$-perturbation lemma}
\label{elementary}
The perturbations in sections~\ref{proofjointint} and~\ref{p-nontransversal} will be realized through the following lemma.
\begin{lemma}\label{l.perturbation}
Let us consider a $C^{1+\al}$ map $v_0\colon \RR^d\to \RR^\ell$,
and two numbers $\widehat D>2 D>0$.
Then, there exists a $C^{1+\al}$-map $v\colon \RR^d\to \RR^\ell$
which coincides with $v_0$ on the ball $B(0,D)$ and with $0$
outside the ball $B(0,\widehat D)$
and whose $C^{1+\al}$-size is arbitrarily small if
the $C^{1+\al}$-size of $v_0$ and the quantity 
$\widehat D^{-(1+\alpha)} \sup_{B(0,\widehat D)}\|v_0\|$ are small.
\end{lemma}
\begin{proof}
One chooses a smooth bump map $\rho\colon \RR^d\to [0,1]$
which coincides with $0$ outside $B(0,\frac 2 3 \widehat D)$ and with $1$ inside
$B(0,D)$. The map $v$ is then defined by $v=\rho.v_0$.

When $\al>0$, we define $\Lip_\al(h)$ the $\al$-H\"older size of a map $h$, that is $$\Lip_\al(h)=\sup_{x\neq y}\frac{\|h(x)-h(y)\|}{\|x-y\|^\al}.$$
We then denote by $A, A'$ the $C^0$ norm of $v_0,Dv_0$ and by
$A_\al,A'_{\al}$ the $\al$-H\"older sizes of $v_0,Dv_0$ on $B(0,\widehat D)$. There exists a universal constant $C>0$ such that for any $\al\in (0,1]$ one has
$$\Lip_\al(\rho)\leq C.\widehat D^{-\al},$$
$$\Lip_\al(D\rho)\leq C.\widehat D^{-(1+\al)}.$$

From inequalities above, it is easily to check that  when the $C^{1+\al}$ size of $v_0$ is small,
the $C^{1+\al}$-size of $v$ is controlled by $A \widehat D^{-(1+\al)}$:
\begin{itemize}
\item[--] The $C^0$ norm of $v$ is smaller than $A$.
\item[--] The $C^0$-norm of $Dv$ is bounded by $A\Lip_1(\rho)+A'\leq CA\widehat D^{-1}+A'$.
\item[--] When $\al>0$, the $\al$-H\"older constant of $Dv$ is bounded by
\begin{equation}\label{e.control}
A_\al\Lip_1(\rho)+A\Lip_\al(D\rho)+A'_\al+\sup_{B(0,\frac 2 3 \widehat D)}\|Dv_0\|\;\Lip_\al(\rho). 
\end{equation}
\end{itemize}
Observe that the three first terms in~(\ref{e.control}) are small when $A'_\al$
and $A\widehat D^{-(1+\al)}$ are small.
Indeed the usual convexity estimate gives
$$A_\al\leq C A^{1/(1+\al)}{A'_\al}^{\al/(1+\al)}.$$
For any $x\in B(0,\frac 2 3 \widehat D)$ one has
\begin{equation*}
\begin{split}
\|Dv_0(x)\|&\leq C.\left[ \sup_{\|u\|=\widehat D/3}
\frac {\|v_0(x+u)-v_0(x)\|}{\|u\|}+\sup_{y\in B(x,\widehat D/3)}
\|Dv_0(y)-Dv_0(x)\|\right]\\
&\leq 3C.(A\widehat D^{-1}+A'_\alpha\widehat D^\alpha).
\end{split}
\end{equation*}
The last term in~(\ref{e.control}) is thus smaller than $A\widehat D^{-(1+\al)}+A'_\al$.

When the $C^{1+\al}$-size of $v_0$ is small, $A'_\al$ is small and the lemma follows.
\end{proof}

\begin{remark}\label{r.perturbation}
When $v_0(0)=0$, for proving that the quantity 
$\widehat D^{-(1+\alpha)} \sup_{B(0,\widehat D)}\|v_0\|$ is small
it is enough to show that
$\widehat D^{-\alpha} \sup_{B(0,\widehat D)}\|Dv_0\|$ is small.
\end{remark}

\subsection{Preliminary constructions}\label{ss.prel}
To simplify the presentation, one will assume that $q_0$ coincides with $p_0$
and is fixed by $f_0$.

\paragraph{The smoothness bound $\al_0$.}
We denote also by $\lambda_c\in (0,1)$ an upper bound for the contraction along $E^c$ and by $\lambda_u>1$ a lower bound for the expansion along the bundle $E^u$.

We choose $\alpha_0>0$ small so that
$$\lambda_u^{\alpha_0}\max(\lambda,\lambda_c)<1,$$
$$\|Df_0^{-1}\|^{\al_0}\lambda<1.$$
In particular, one can consider $\rho\in (0,1)$ such that
$$\lambda^{1/\alpha_0}<\rho<\|Df_0^{-1}\|^{-1}.$$

\paragraph{The neighborhoods $V_1,V_2$ of $x$.}
Once the smoothness $\al\in[0,\al_0]$ and the neighborhood $\cV$
have been fixed, one introduces a continuity point $f'\in \cV$
for both maps $g\mapsto x_g,y_g$.
Let $z_{f'}\in W^s(p_{f'})\cap W^u_{loc}(p_{f'})$ be a transverse homoclinic point of the orbit of $p_{f'}$ that does not belong to the orbit of $x_{f'}$ or $y_{f'}$.
We choose two small open neighborhoods $V_1,V_2$
of $x_{f'}$, such that $\bar V_2\subset V_1$.
Choosing them small enough,
the orbit of the intersection $\overline{V_1}\cap W^s_{loc}(p_{f'})$
is disjoint from the orbit of $y_{f'}$ and $z_{f'}$.

Since $f'$ is a continuity point of $g\mapsto x_g,y_g$,
for any diffeomorphism $g\in \cV$ close to $f'$,
the point $x_g$ still belongs to $V_2$ and
the orbit of the intersection $\overline{V_1}\cap W^s_{loc}(p_{g})$
is still disjoint from the orbit of the continuations $y_g,z_g$.

\paragraph{The diffeomorphism $f$.}
We choose a diffeomorphism $f\in \cV$ arbitrarily close to $f'$.
One can require that $f$ is of class $C^\infty$ and that there is no resonance
between the eigenvalues of the linear part associated to the orbit of $p_f$.
As a consequence of Sternberg linearization theorem, the dynamics in a neighborhood
of the orbit of $p_f$ can be linearized by a smooth conjugacy map.

In order to simplify the notations we will denote $p=p_f$, $q=q_f$, $x=x_f$, $y=y_f$.

\paragraph{Local coordinates.}
One can find a small neighborhood $B$ of $p$ and a $C^r$-chart $B\to \RR^d$
which linearizes the dynamics and maps $p$ on $0$ and
the local manifolds $W^{ss}_{loc}(p), W^s_{loc}(p), W^u_{loc}(p)$ inside the coordinate planes $\RR^s\times \{0\}^{u+1}$, $\{0\}^s\times \RR\times \{0\}^u$
and $\{0\}^{s+1}\times \RR^u$, where $s,u,d$ denotes the dimension of $E^{ss}, E^u$
and $M$ respectively. The coordinates in the chart are written $(\bar x,\bar y, \bar z)\in \RR^s\times \RR\times \RR^u$.

The map $f$ viewed in the chart
is thus a linear map $A=A_s\times A_c\times A_u$ of $\RR^d$ which preserves these coordinate planes.
Replacing $x,y$ by iterates, one can assume that their forward orbits
are contained in $B$.

\paragraph{The local stable disk $D$.}
Let $z_0$ be the transverse homoclinic point of the orbit of $p$ for $f$
that is the continuations of $z_{f'}$. For $n\geq 0$ we also define $z_{-n}=f^{-n}(z_0)$.

We can thus choose a small neighborhood $D$ of $z_0$ in $W^s(p)$ whose orbit
is disjoint from the orbits of $x$ and $y$.
Replacing $z_0$ by an iterate, one can assume that its backward orbit belongs to $B$.
The disk $D$ (or one of its backward iterates) endowed with its strong stable foliation
can then be linearized.

\begin{lemma}
By a $C^{1+\al}$-small perturbation of $f$ one may assume
furthermore that in the chart at $p$,
\begin{itemize}
\item[--] $D$ is contained in an affine plane parallel to
the local stable manifold $W^{s}_{loc}(p)$,
\item[--] the strong stable manifolds inside $D$ coincide with the affine planes
parallel to $W^{ss}_{loc}(p)$.
\end{itemize}

\end{lemma}

\begin{proof}
We choose a large integer $n\geq 1$.
The ball $W$ centered a $z_{-n}$ of radius $r=\lambda_u^{-n}$
does not intersect the local stable manifold of the orbit of $p$,
neither the iterates $z_{-k}$ for $k\neq n$.

We first rectify the disc $D$:
in the chart, the disc $f^{-n}(D)$ can be seen as the graph
of a map whose derivative has norm smaller than $\lambda^n$.
By the $\lambda$-lemma,
this graph is arbitrarily $C^{1+\alpha}$-close to
the linear plane $W^{s}_{loc}(p)$.
One can thus apply lemma~\ref{l.perturbation}: by a diffeomorphism supported inside
$W$ which fixes $z_{-n}$,
one can send a neighborhood of $z_{-n}$ inside $D$ in an affine plane
parallel to $W^s_{loc}(p)$.
By remark~\ref{r.perturbation}, this diffeomorphism is $C^{1+\al}$-close to the identity provided
that $\lambda^nr^{-\al}=(\lambda\lambda_u^{\alpha})^n$ is small, which is the case if $\al<\al_0$ and our choice of $\al_0$.

Assuming now that $D$ is contained in an affine plane parallel to $W^{s}_{loc}(p)$,
we denote by $W^{c}_{loc}(z_0)$ the affine space containing $z_0$ parallel to $\{0\}^s\times \RR\times\{0\}^u$.
We rectify the strong stable foliation inside $D$:
this is the image of the affine foliation parallel to $W^{ss}_{loc}(p)$
by a diffeomorphism $\Phi$ of the form
$$\Phi\colon(\bar x,\bar y,\bar z)\mapsto (\bar x, \varphi(\bar x,\bar y),\bar z),$$
which fixes $z_0$ and $W^{c}_{loc}(z_0)$.
Let us again consider $n\geq 1$ large.

Inside $f^{-n}(D)$, the strong stable foliation is the image of the affine foliation by
the map $\Phi_n=A^{-n}\circ \Phi\circ A$ where $A=(A_s,A_c,A_u)$ is the linear map of $\RR^d$
which coincides with $D_pf$.
The components $\Phi_{n,\bar x},\Phi_{n,\bar z}$ of $\Phi$ along the coordinates
$\bar x,\bar z$ coincide with the identity of the planes $\RR^s\times \{0\}$
and $\{0\}\times \RR^u$.
The derivative of the component $\Phi_{n,\bar y}$ at a point $\zeta$ is
$$D\Phi_{n,\bar y}(\zeta)= A^{-n}_c\;\partial_{\bar x} \varphi(A^n.\zeta) \; A^n_s+ \partial_{\bar y}\varphi(A^n.\zeta).$$
When $n$ goes to infinity, the first term $A^{-n}_c\;\partial_{\bar x} \varphi \; A^n_s$
goes to zero as $\lambda^n$ since the contraction $A_s$ is stronger than $A_c$.
Since $f$ is assumed to be smooth, $\partial_{\bar x}\varphi(\zeta), \partial_{\bar y}\varphi(\zeta)$
are Lipschitz in $\zeta$.
The map $A^n$ sends a uniform neighborhood of $z_{-n}$ in $f^{-n}(D)$ inside
a ball of radius $\lambda_c^n$ of $D$; hence
if one restricts $D\Phi_{n,\bar y}$ to a small neighborhood of $p$,
the second term $\partial_{\bar y}\varphi(A^n.\zeta)$ is $\lambda_c^n$-close to $\partial_{\bar y}\varphi(z_0)$.
One deduces that $D\Phi_{n,\bar y}$ converges uniformly to the identity and that
$$\|D\Phi_{n,\bar y}-\id\|\leq \lambda^n+\lambda_c^n.$$
The same argument shows that the Lipschitz constant
of $D\Phi_{n,\bar y}$ goes to zero as $n$ goes to infinity.
One can thus apply lemma~\ref{l.perturbation}, in order to rectify the strong stable foliation
on a small neighborhood of $z_{-n}$ in $f^{-n}(D)$, by a map supported on the ball
$B(z_{-n},\lambda_u^{-n})$. The perturbation is small in topology $C^{1+\alpha}$,
provided that
$$\|D\Phi_{n,\bar y}-\id\|\lambda_u^{n\alpha}\leq  (\lambda^n+\lambda_c^n)\lambda_u^{n\alpha}$$
is small, which is the case  when $n$ is large since   $\alpha<\alpha_0$ by the choice of $\al_0$.
\end{proof}

\paragraph{The perturbation support}
Let us denote by $D^m$ the connected component of $f^{-m}(D)\cap B$ which contains $z_{-m}$.
We choose two small open neighborhoods $U_1,U_2$
of $x$ in $W^s_{loc}(p)$, such that $\bar U_2\subset U_1$:
they are obtained as the intersection of $V_1,V_2$ with $W^s_{loc}(p)$.
By construction, their orbit is disjoint from the orbit of $z_0$ and $y$.
For each $n\geq 0$ and $s>0$, we introduce $R_1^n(s)$ the product (in the coordinates of the chart at $p$)
$$R_1^n(s)=f^{n}(U_1)\times \{|\bar z|<s\},$$
and similarly we define $R_2^n(s)$. See figures~\ref{rectji} and~\ref{staji}.

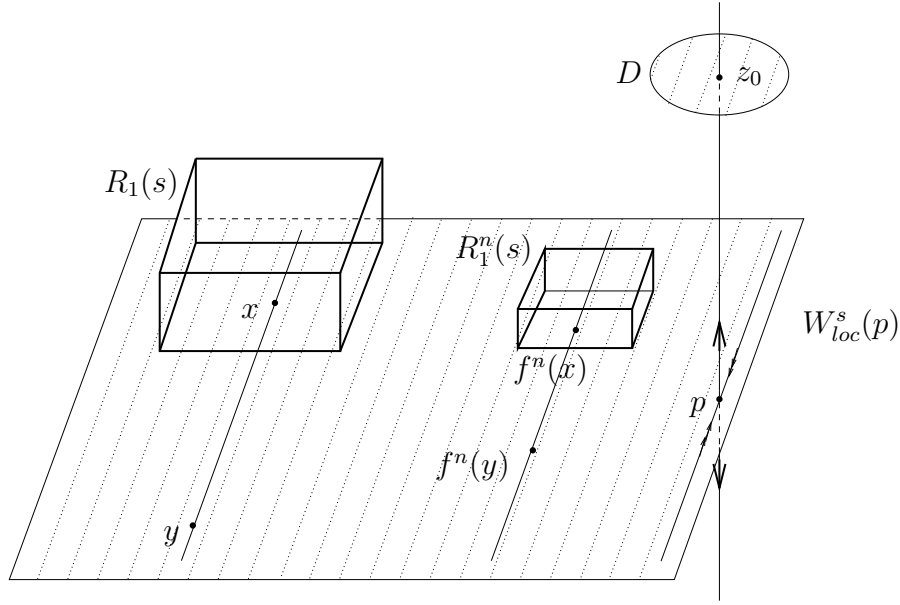
\begin{figure}[subsection]

\begin{center}
\input{pert1-new.pstex_t}
\end{center}

\caption{The perturbation support. \label{rectji}}
\end{figure}

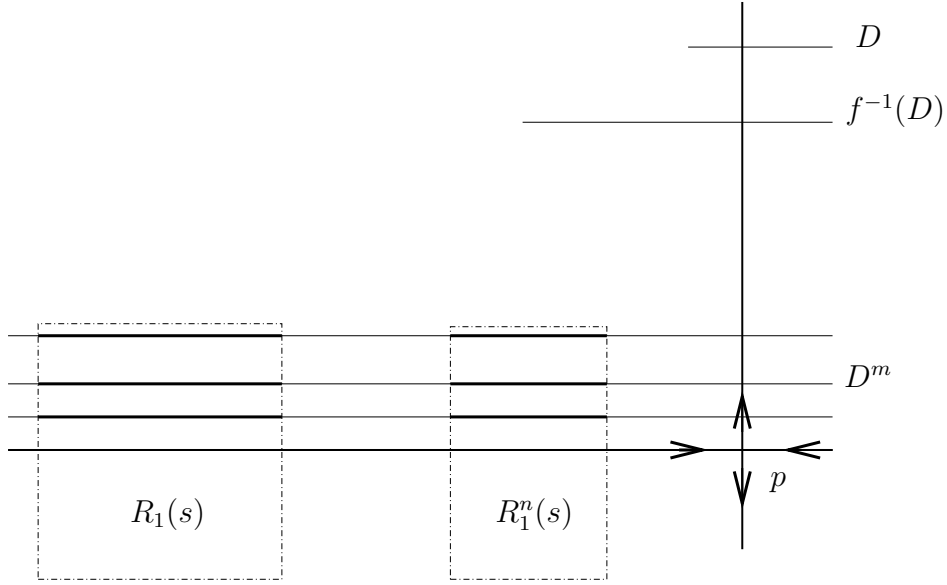
\begin{figure}[subsection]

\begin{center}
\input{pert2-new.pstex_t}
\end{center}

\caption{The local stable disks $D^m$. \label{staji}}
\end{figure}

\subsection{The perturbation}\label{ss.twist}

Let us choose a linear form $L$ on $\RR^u$ and recall that $\rho\in (0,1)$ has been chosen smaller than $\|Df^{-1}\|^{-1}$.
The perturbation $g$ of $f$ will be obtained as the composition $T\circ f$ where $T$
in the chart around $p$ coincides with a map $T_n$, for $n$ large, given by the following lemma.
See figure~\ref{pertji}.
\begin{figure}[subsection]

\begin{center}
\input{pert3-new.pstex_t}
\end{center}

\caption{The perturbation. \label{pertji}}
\end{figure}
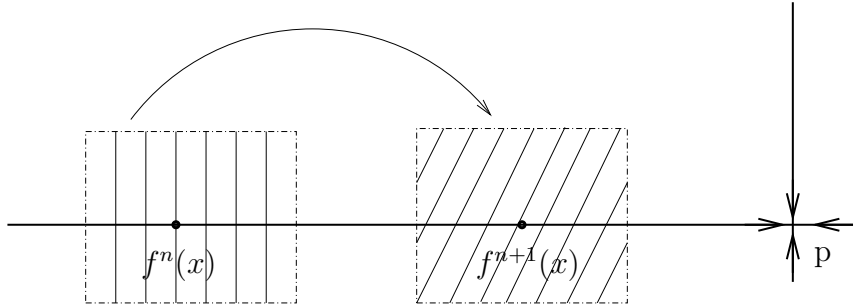

\begin{lemma}\label{l.twist} There exists a sequence of smooth diffeomorphisms $T_n$ of $\RR^d$ such that
\begin{itemize}
\item[--] $T_n$ coincides with the identity outside $R_1^n(\rho^n)$ and on $W^{s}_{loc}(p)$,
\item[--] $DT_n$ coincides on $R_2^n(\rho^{n+1})$ with the linear map
$$B:(\bar x, \bar y, \bar z)\mapsto (\bar x,\bar y + \rho^{\al_0 n}.L(\bar z), \bar z),$$
\item[--] $(T_n)$ converges to the identity in topology $C^{1+\alpha}$.
\end{itemize}
\end{lemma}

\begin{proof}
Let us choose a smooth map $\varphi\colon \RR^{s+1}\to [0,1]$
supported on $U_1$ which takes the value $1$ on $\overline{U_2}$
and a smooth map $\psi\colon \RR^u\to [0,1]$ supported on the unit ball
and which coincides with $1$ on the ball $B(0,\rho)$.
We then define
$$T_n\colon (\bar x,\bar y,\bar z)\mapsto (\bar x,\bar y+t_n(\bar x,\bar y,\bar z), \bar z),$$
$$t_n(\bar x,\bar y,\bar z)=\rho^{\al_0 n} \; \varphi\circ f^{-n}(\bar x,\bar y,0)\; \psi(\rho^{-n}.\bar z)\; L(\bar z).$$
The two first properties are clearly satisfied.
On $R_1^{n}(\rho^n)$, the factor $L(\bar z)$ is bounded
(up to a constant) by $\rho^n$.
Since $f^{-n}$ is linear
and (by our choice of $\rho$) has a norm smaller than $\rho^{-n}$, as before the $C^{1+\al}$
size of the perturbation $T$ can be easily computed: it is smaller
than $(\rho^{\al_0-\al})^n$ and goes to zero as $n$ gets larger.
\end{proof}

\begin{remark}
After the perturbation, the orbits of $z_0$ and $p$ are unchanged.
The local manifold $W^s_{loc}(p)$ and its strong stable foliation
are also the same. For $m$ large and $s>0$ small, the forward orbit of $D^m\cap R_1^{n+1}(s)$
does not intersect the support of the perturbation,
hence the strong stable foliation on $D^m\cap R_1^{n+1}(s)$ still
coincides with the linear one.
\end{remark}

\subsection{Proof of proposition~\ref{p.jointintegrable}}
Recall that $z_{-m}$ is the image of $z_0$ by the linear map $A^{-m}$.
Let us choose a linear form $L$ on $\RR^u$ and a constant $c>0$ such that
for infinitely many values of $m\geq 0$ one has
\begin{equation}\label{e.choixL} L(z_{-m})>c\|z_{-m}\|.\end{equation}
We define $L_t=-tL$ for any $t\in [-1,1]$.
The construction of section~\ref{ss.twist}
associates to $n\geq 1$ large, a perturbation $g_t=T_{n,t}\circ f$.
We also consider a large integer $m\geq 1$ so that the distance of $z_{-m}$
to $p$ is smaller than $\rho^{n+1}$ and~\eqref{e.choixL} is satisfied.
The point $z$ announced in the statement of the proposition wil be $z_{-m}$.

We introduce the continuations $x_t=x_{g_t},y_t=y_{g_t}$ of $x,y$ for $g_t$
and the intersection $\hat x_t,\hat y_t$
of the local unstable manifold at $g_t^{n+1}(x_t),g_t^{n+1}(y_t)$ with the disc $D^m$.
By lemma~\ref{l.largemanifold}, the points $\hat x_t,\hat y_t$ belong to $H(p_t)$.
Since for each considered perturbation, the disc $D^{m}$ is still contained in $W^s_{loc}(p)$
and is endowed with the same linear strong stable foliation,
it is enough to introduce the projection $\pi_c$ on the central coordinate $\bar y$
and to show that
\begin{equation}\label{e.trans}
\pi_c(\hat x_{1})<\pi_c(\hat y_{1}) \text{ and } \pi_c(\hat x_{-1})>\pi_c(\hat y_{-1}).
\end{equation}

\medskip

First we notice that
since $g_t$ is close to $f$ and $f'$ in $\cV$, the continuations $x_t$ and $y_t$
of $x,y$ are still contained in $U_2$ and in $W^s_{loc}(p)\setminus \overline{U_1}$.
Since $g_t$ coincides with $A$ outside $R_1^n(\rho^n)$,
the local unstable manifolds of $g_t^n(x_t)$ and $g_t^{n+1}(y_t)$ are tangent to the cone
$$\cC^u_n=\{(v^{cs},v^u)\in \RR^{s+1}\times \RR^u, \; \|v^{cs}\|\leq \lambda^n\|v^u\|\}.$$
By construction the local unstable manifold of $g_t^{n+1}(x_t)$ in $f(R_2^n(\rho^{n+1}))$ is
tangent to the cone $B_t(\cC^u_{n+1})$, where $B_t$ is the linear map
associated to $L_t$ as in lemma~\ref{l.twist}.

The points $\hat x_t,\hat y_t$ are contained in the intersection of
these cones with the affine plane parallel to $\RR^{s+1}\times \{0\}$ containing $A^{-m}(z_0)$.
One deduces that
$$\pi_c(\hat x_t)\in B(\pi_c(x_t)-t\rho^{\alpha_0 n}L(z_{-m}),\lambda^n\|z_{-m}\|),$$
$$\pi_c(\hat y_t)\in B(\pi_c(y_t),\lambda^n\|z_{-m}\|).$$

By assumption we have $\pi_c(x_t)=\pi_c(y_t)$ and by our choice of $\rho$ one has $\lambda<\rho^{\alpha_0}$.
In particular, by~\eqref{e.choixL}, for $n$ large enough and $t=-1$ or $t=1$, these two balls are disjoint.
One also controls the sign of $\pi_c(\hat y_t)-\pi_c(\hat x_t)$ and gets~\eqref{e.trans} as wanted.

\subsection{Proof of proposition~\ref{p.generalized-strong-connectionCr}}
The number $\al_0>0$ is given by theorem~\ref{t.stable}.
The open set $\cU$ is chosen to satisfy theorem~\ref{t.stable} and
proposition~\ref{p.continuation}.

We then consider $\al\in [0,\al_0]$
and a diffeomorphism $f$ as in the statement of the proposition.
Let us assume by contradiction that in a $C^{1+\al}$-neighborhood $\cV$ of $f$,
there is no diffeomorphism $g$ such that $H(p_g)$ has a strong homoclinic intersection.
The proposition~\ref{p.continuation} applies.

By assumption there exists a hyperbolic periodic point $q_f$ homoclinically related to the
orbit of $p$ and a point $x_{f}\in H(q_{f})\cap W^{ss}(q_{f})\setminus \{q_{f}\}$.
By lemma~\ref{l.NGSHI}, $x_{f}$ is accumulated by points of the class $H(p_{f})$ in
$\cW^{cs}_{x_{f}}\setminus W^{ss}_{loc}(x_{f})$.
Considering the forward orbit of these points,
one deduces that $x_{f}$ and $q_{f}$ are accumulated by points of $H(q_{f})$
inside the same component of $\cW^{cs}_{x_{f}}\setminus W^{ss}_{loc}(x_{f})$.
By corollary~\ref{c.continuation}, there exists a continuation $g\mapsto x_g$ such that
$x_g$ belongs to $W^{ss}_{loc}(q_g)$ for each $g\in \cV$. Since $q_g$ and $W^{ss}_{loc}(q_g)$
vary continuously, one can argue as in the proof of lemma~\ref{l.cont-unstable} and conclude that
$g\mapsto x_g$ is continuous.

Now the theorem~\ref{t.stable} applies to the diffeomorphisms $f_0,f$ and to the points $q=y=p$ and $x$. One gets a strong homoclinic intersection for some $g\in \cV$ and the class $H(p_g)$.
This is a contradiction, concluding the proof of the proposition.

%% file: pert1-new.pstex_t
\begin{picture}(0,0)%
\includegraphics{pert1-new.pstex}%
\end{picture}%
\setlength{\unitlength}{1657sp}%
\begingroup\makeatletter\ifx\SetFigFontNFSS\undefined%
\gdef\SetFigFontNFSS#1#2#3#4#5{%
  \reset@font\fontsize{#1}{#2pt}%
  \fontfamily{#3}\fontseries{#4}\fontshape{#5}%
  \selectfont}%
\fi\endgroup%
\begin{picture}(12637,8999)(2049,-8498)
\put(12376,-5641){\makebox(0,0)[b]{\smash{{\SetFigFontNFSS{12}{14.4}{\rmdefault}{\mddefault}{\updefault}{\color[rgb]{0,0,0}$p$}%
}}}}
\put(5671,-4246){\makebox(0,0)[b]{\smash{{\SetFigFontNFSS{12}{14.4}{\rmdefault}{\mddefault}{\updefault}{\color[rgb]{0,0,0}$x$}%
}}}}
\put(4501,-7576){\makebox(0,0)[b]{\smash{{\SetFigFontNFSS{12}{14.4}{\rmdefault}{\mddefault}{\updefault}{\color[rgb]{0,0,0}$y$}%
}}}}
\put(9316,-3346){\makebox(0,0)[b]{\smash{{\SetFigFontNFSS{12}{14.4}{\rmdefault}{\mddefault}{\updefault}{\color[rgb]{0,0,0}$R_1^n(s)$}%
}}}}
\put(4051,-2356){\makebox(0,0)[b]{\smash{{\SetFigFontNFSS{12}{14.4}{\rmdefault}{\mddefault}{\updefault}{\color[rgb]{0,0,0}$R_1(s)$}%
}}}}
\put(9001,-6586){\makebox(0,0)[b]{\smash{{\SetFigFontNFSS{12}{14.4}{\rmdefault}{\mddefault}{\updefault}{\color[rgb]{0,0,0}$f^n(y)$}%
}}}}
\put(10171,-5146){\makebox(0,0)[b]{\smash{{\SetFigFontNFSS{12}{14.4}{\rmdefault}{\mddefault}{\updefault}{\color[rgb]{0,0,0}$f^ n(x)$}%
}}}}
\put(14671,-4471){\makebox(0,0)[b]{\smash{{\SetFigFontNFSS{12}{14.4}{\rmdefault}{\mddefault}{\updefault}{\color[rgb]{0,0,0}$W^s_{loc}(p)$}%
}}}}
\put(13141,-691){\makebox(0,0)[b]{\smash{{\SetFigFontNFSS{12}{14.4}{\rmdefault}{\mddefault}{\updefault}{\color[rgb]{0,0,0}$z_0$}%
}}}}
\put(11341,-736){\makebox(0,0)[b]{\smash{{\SetFigFontNFSS{12}{14.4}{\rmdefault}{\mddefault}{\updefault}{\color[rgb]{0,0,0}$D$}%
}}}}
\end{picture}%

%% file: pert2-new.pstex_t
\begin{picture}(0,0)%
\includegraphics{pert2-new.pstex}%
\end{picture}%
\setlength{\unitlength}{1657sp}%
\begingroup\makeatletter\ifx\SetFigFontNFSS\undefined%
\gdef\SetFigFontNFSS#1#2#3#4#5{%
  \reset@font\fontsize{#1}{#2pt}%
  \fontfamily{#3}\fontseries{#4}\fontshape{#5}%
  \selectfont}%
\fi\endgroup%
\begin{picture}(13323,8695)(-122,-8633)
\put(2296,-7756){\makebox(0,0)[b]{\smash{{\SetFigFontNFSS{12}{14.4}{\rmdefault}{\mddefault}{\updefault}{\color[rgb]{0,0,0}$R_1(s)$}%
}}}}
\put(7786,-7756){\makebox(0,0)[b]{\smash{{\SetFigFontNFSS{12}{14.4}{\rmdefault}{\mddefault}{\updefault}{\color[rgb]{0,0,0}$R_1^n(s)$}%
}}}}
\put(12781,-5731){\makebox(0,0)[b]{\smash{{\SetFigFontNFSS{12}{14.4}{\rmdefault}{\mddefault}{\updefault}{\color[rgb]{0,0,0}$D^{m}$}%
}}}}
\put(13186,-1726){\makebox(0,0)[b]{\smash{{\SetFigFontNFSS{12}{14.4}{\rmdefault}{\mddefault}{\updefault}{\color[rgb]{0,0,0}$f^{-1}(D)$}%
}}}}
\put(12781,-646){\makebox(0,0)[b]{\smash{{\SetFigFontNFSS{12}{14.4}{\rmdefault}{\mddefault}{\updefault}{\color[rgb]{0,0,0}$D$}%
}}}}
\put(11431,-7216){\makebox(0,0)[b]{\smash{{\SetFigFontNFSS{12}{14.4}{\rmdefault}{\mddefault}{\updefault}{\color[rgb]{0,0,0}$p$}%
}}}}
\end{picture}%

%% file: pert3-new.pstex_t
\begin{picture}(0,0)%
\includegraphics{pert3-new.pstex}%
\end{picture}%
\setlength{\unitlength}{1657sp}%
\begingroup\makeatletter\ifx\SetFigFontNFSS\undefined%
\gdef\SetFigFontNFSS#1#2#3#4#5{%
  \reset@font\fontsize{#1}{#2pt}%
  \fontfamily{#3}\fontseries{#4}\fontshape{#5}%
  \selectfont}%
\fi\endgroup%
\begin{picture}(12711,4555)(148,-7733)
\put(2746,-7261){\makebox(0,0)[b]{\smash{{\SetFigFontNFSS{12}{14.4}{\rmdefault}{\mddefault}{\updefault}{\color[rgb]{0,0,0}$f^n(x)$}%
}}}}
\put(7966,-7261){\makebox(0,0)[b]{\smash{{\SetFigFontNFSS{12}{14.4}{\rmdefault}{\mddefault}{\updefault}{\color[rgb]{0,0,0}$f^{n+1}(x)$}%
}}}}
\put(12376,-7081){\makebox(0,0)[b]{\smash{{\SetFigFontNFSS{12}{14.4}{\rmdefault}{\mddefault}{\updefault}{\color[rgb]{0,0,0}p}%
}}}}
\end{picture}%

%% file: unstableleaves-1110.tex
\section{Periodic unstable leaves: proof of theorem \ref{t.unstable}}

\label{p-nontransversal}

Now we continue with the proof of theorem \ref{t.unstable}.
In this section we consider:

\begin{itemize}
\item[1)] A diffeomorphism $f_0$ and a homoclinic class $H(p_{f_0})$
which is a chain-recurrence class endowed with a partially hyperbolic splitting $E^s\oplus E^c\oplus E^u$ where $E^c$ is one-dimensional and $E^s\oplus E^c $ is thin-trapped.
\item[2)] Two hyperbolic periodic points $p_{x,f_{0}}$ and $p_{y,f_{0}}$
homoclinically related to the orbit of $p_{f_0}$ and a diffeomorphism
$f$ that is $C^1$-close to $f_0$ such that
there exists two points $x\in W^{u}(p_{x,f})$, $y\in W^{u}(p_{y,f})$ in $H(p_f)$ whose strong stable manifold coincide.
\end{itemize}

Note that by lemma~\ref{l.robustness}, the homoclinic class
associated to the hyperbolic continuation $p_f$ of $p_{f_0}$ is still chain-hyperbolic. Moreover the continuations of $x,y$ are well defined and unique (lemma~\ref{l.cont-unstable}).
We will show that $f$ is the limit of diffeomorphisms $g$ such that $H(p_g)$ has a strong homoclinic intersection.
The results of this section are sum up in the next proposition.

\medskip

\begin{propo}\label{p.unstable-nt}

For any diffeomorphism $f_0$ and any homoclinic class $H(p_{f_0})$
satisfying the assumption 1) above, there exists $\alpha_0\in (0,1)$
with the following property.

For $\alpha\in [0,\alpha_0]$ and any hyperbolic periodic points $p_{x,f_0},p_{y,f_0}$
homoclinically related to the orbit of $p_{f_0}$,
any $C^{1+\alpha}$-diffeomorphism $f$ that is $C^1$-close to $f_0$
and satisfies 2) can be $C^{1+\alpha}$-approximated
by a diffeomorphism $g$ such that:

\begin{itemize}
\item[--] Either there exists a periodic point $q$ homoclinically related to the orbit of $p_g$
such that $$W^{ss}_{loc}(q)\cap W^u(p_{y,g})\neq \emptyset.$$
\item[--] Or the continuations of $x,y$ satisfy
$x_g \notin W^{ss}_{loc}(y_g)$ and also $x_g$ belongs to an arbitrary  previously selected component of $\cW^{cs}_x\setminus W^{ss}_{loc}(x)$.
\end{itemize}
\end{propo}

\medskip

Note that by using corollary \ref{c.continuation} in both cases of the conclusion of the proposition,
$f$ is $C^{1+\al}$-approximated by a diffeomorphism whose homoclinic class
$H(p)$ exhibits a strong homoclinic intersection.
The proposition thus clearly implies theorem \ref{t.unstable}.

In what follows, in  subsection \ref{ss.holonomy} are introduced the fake holonomies and it is explained the H\"older regularity. In  subsection \ref{ss.localization} it is shown that the recurrences to the point $x$ in proposition \ref{p.unstable-nt} hold along the center direction and in subsection \ref{recurrence-dichotomy} it is a presented a dichotomy related to the recurrence time. Related to this dichotomy, two different perturbations are introduced in lemma \ref{nontransversal1} and \ref{nontransversal2-bis} proved in sections \ref{nont-proof} and \ref{nont2-bis-proof} respectively.

\subsection{Strong stable holonomy}\label{ss.holonomy}

\paragraph{Plaques.}

Using an adapted metric if needed, we can assume that there exist
constants $\lambda>1$ and $0<\lambda_s<1<\lambda_u$ such that for any $x\in H(p_{f_0})$
and any unitary vectors $u\in E^s_x$, $v\in E^c_x+E^u_x$ and $w\in E^u$, one has
$$\lambda.\|D_xf_0.u\|<\|D_xf_0.v\|, \quad \|D_xf_0.u\|\leq \lambda_s\quad
\text{and} \quad \|D_xf_0.w\|\geq \lambda_u.$$
Let us introduce a strong stable cone field $\cC^s$ above $H(p_{f_0})$:
one can choose $a>0$ small and define at each point $x$ the set
$$\cC^s_x=\{(u^s,u^c,u^u)\in E^s_x+E^c_x+E^u_x,\quad
\|u^s\|\geq a.\|u^c+u^u\|\}.$$
The cone field extends continuously to a neighborhood $U$ of $H(p_{f_0})$
such that at any $x\in U\cap g(U)$,
$$D_xf_0^{-1}.\cC^s_{g(x)}\subset \cC^s_{x}.$$
For some $r_0>0$, at any point $x\in U$ there exists a plaque of radius $r_0$
tangent to $\cC^s$.
Similarly, one can define a center-unstable cone field $\cC^{cu}$ and
an unstable cone-field $\cC^u$ on $U$ close to the bundles $E^{c}\oplus E^u$ and $E^u$ respectively. All these properties remain valid for any diffeomorphism that is $C^1$-close to $f_0$.

\paragraph{Strong stable holonomy.}

It is a classical fact that the strong stable holonomies are H\"older.
The proof extends to more general objects, that we call
\emph{fake holonomies}. For more references see~\cite{burns-wilkinson}.

Let us consider a small constant $\delta>0$ that is  used to measure how orbits separate.
For any diffeomorphism $f$ that is $C^1$-close to $f_0$, let us consider two different points $z\in H(p_f)$ and
$z'\in W^{u}_{loc}(z)$ close to each other. Note that there exists a smallest integer
$N=N(z,z')\geq 1$ such that $f^N(z)$ and $f^N(z')$ are at distance larger than $\delta$.

\begin{defi}\label{fake-holonomy}

Two points $\widehat{\Pi^{ss}}(z), \widehat{\Pi^{ss}}(z')$ are called
\emph{fake strong stable holonomies} of $z,z'$ if they satisfy the following properties.

\begin{itemize}

\item[--] There exists a center-unstable plaque of radius  $r_0$ containing  $\widehat{\Pi^{ss}}(z)$ and $\widehat{\Pi^{ss}}(z')$.
\item[--] There exists two plaques of radius $r_0$ at $f^N(z)$ and $f^N(z')$
that are tangent to $\cC^s$ and contain $f^N(\widehat{\Pi^{ss}}(z))$ and
$f^N(\widehat{\Pi^{ss}}(z'))$ respectively.

\item[--] For $0\leq k \leq N$, the distances $d(f^k(z), f^k(\widehat{\Pi^{ss}}(z)))$, $d(f^k(z),f^k(\widehat{\Pi^{ss}}(z)))$ are smaller than $r_0$.

\end{itemize}

\end{defi}

Note that by invariance of the cone field $\cC^s$ under backward iterations
the point $f^k(\widehat{\Pi^{ss}}(z))$ belongs to a plaque
at $f^k(z)$ tangent to $\cC^s$ and whose radius is smaller than $\lambda_s^k.r_0$.
\medskip

The choice for the plaques tangent to $\cC^s$ is of course not unique: one can consider for instance
the local strong stable manifold (in this case, the fake holonomies coincide with the usual
strong-stable holonomies) but one can also choose the local strong stable manifold of a diffeomorphism
$C^1$-close to $f$. In fact the fake holonomies  allow us to compare the holonomies when the diffeomorphism is changed.

\paragraph{H\"older regularity.}

We now sketch how the classical result about H\"older regularity adapts for the fake holonomies.
\begin{lemma}\label{l.holonomy}
If $\delta>0$ has been chosen small enough,
then there exists $\alpha_s>0$ such that for any diffeomorphism $f$
that is $C^1$-close to $f_0$,
for any $z\in H(p_f)$ and $z'\in W^u_{loc}(z)$ close,
and for any fake holonomies $\widehat{\Pi^{ss}}(z), \widehat{\Pi^{ss}}(z')$, one has
$$d(\widehat{\Pi^{ss}}(z), \widehat{\Pi^{ss}}(z'))\leq d(z,z')^{\alpha_s}.$$

\end{lemma}

\begin{proof}[Sketch of the proof]
Observe that if $N$ is sufficiently large (provided that $z'$ is close enough to $z$), the distances $d(f^N(\po(z)),f^N(z))$ and $d(f^N(\po(z')),f^N(z'))$
are exponentially small. Hence
$d(f^N(\widehat{\po}(z)), f^N(\widehat{\po}(z')))$
is of the same order than $d(f^N(z), f^N(z'))$ and close to $\delta$.

The distance $d(\widehat{\Pi^{ss}}(z), \widehat{\Pi^{ss}}(z'))$
is bounded by $\|Df^{-1}\|^Nd(f^N(\po(z)), f^N(\widehat{\Pi^{ss}}(z')))$
and the distance $d(z, z')$ is bounded from below by
$\|Df\|^{-N}d(f^N(\po(z)), f^N(\widehat{\Pi^{ss}}(z')))$.
This proves that there exists $\sigma>0$ (which only depends on $f_0$)
such that
$$d(\widehat{\Pi^{ss}}(z), \widehat{\Pi^{ss}}(z'))\leq \sigma^N. d(z,z').$$
On the other hand, since the distance along the unstable manifolds
growth uniformly, there exists another constant $C>0$ such that
$$N\leq C.\log d(z,z').$$
The result follows from these two last inequalities.
Observe that the exponent $\alpha_s$ only depends on $C$ and $\sigma$ which are uniform on
a $C^1$-neighborhood of $f_0$.
\end{proof}

\paragraph{Regularity of the strong stable bundle.}

The regularity of the strong stable bundle needs more smoothness on the diffeomorphism.
Note that the strong stable bundle is defined at any point whose forward orbit is contained in a
small neighborhood $U$ of $H(p_{f_0})$.
\begin{lemma}\label{l.bundle}
There exists $\alpha'_s$ such that for any diffeomorphism $f$ that is $C^1$-close to $f_0$
and of class $C^{1+\alpha}$ for some $\alpha\in (0,\alpha'_s)$, there exists a constant $C>0$
with the following property.

At any points $z,z'$ close having their forward orbit contained in $U$, one has
$$d(E^{ss}_z,E^{ss}_{z'})\leq C.d(z,z')^{\alpha}.$$

\end{lemma}

\begin{proof}[Sketch of the proof]

Let us choose $K>\|Df\|_\infty$ and as before denote by $\lambda\in (0,1)$ a bound for the domination between $E^{ss}$ and $E^c\oplus E^u$. We choose $\alpha'_s>0$ such that
$K^{\alpha'_s}\lambda<1$.
By working in charts, one has for some constant $C>0$,
\begin{equation*}
\begin{split}
d(E^{ss}_z,E^{ss}_{z'})&\leq d( Df^{-1}_{f(z)}(E^{ss}_{f(z)}), Df^{-1}_{f(z)}(E^{ss}_{f(z')}))+
d(Df^{-1}_{f(z)}(E^{ss}_{f(z')}),Df^{-1}_{f(z')}(E^{ss}_{f(z')}))\\
&\leq \lambda d(E^{ss}_{f(z)},E^{ss}_{f(z')}) + C.d(f(z),f(z'))^\alpha.
\end{split}
\end{equation*}

By induction one gets for any $k\geq 1$,
$$d(E^{ss}_z,E^{ss}_{z'})\leq C.\sum_{j=0}^{k-1} \lambda^j d(f^{j+1}(z),f^{j+1}(z'))^\alpha
+ \lambda^{k}d(E^{ss}_{f^k(z)},E^{ss}_{f^k(z')}).$$
One can bound $d(f^{j}(z),f^{j}(z'))$ by $K^jd(z,z')$.
Since $\lambda K^{\al'_s}<1$,
this gives 
$$d(E^{ss}_z,E^{ss}_{z'})\leq C (d(z,z')^\alpha + \lambda^k).$$
By choosing $k$ large enough, one gets the estimate.

\end{proof}

\subsection{Localization of returns to $p_x$}\label{ss.localization}

We now fix a diffeomorphism $f$ that is $C^1$-close to $f_0$.
We assume that  is $C^{1+\al}$ for some $\al\in [0,1)$.
In order to simplify the notations we will now set $p=p_f$, $p_x=p_{x,f}$ and $p_y=p_{y,f}$. Let $\tau_x$ be the period of $p_{x}$
and let us consider a local central manifold $W^c_{loc}(p_x)$.

\smallskip

We will use the following assumption:

\begin{itemize}

\item[(***)] \emph{The intersection between $W^{ss}(p_x)\setminus \{p_x\}$ and
$H(p)$ is empty.}

\end{itemize}
The orbit of any point $z\in W^s(p_{x})\setminus \{p_{x}\}$ meets the fundamental 
domain $f^{\tau_x}(W^s_{loc}(p_{x}))\setminus W^s_{loc}(p_{x})$.
The next lemma states, that if $H(p)$ and $p_x$ satisfy (***),
and if $z$ belongs to $H(p)\cap W^s_{loc}(p_x)$ then its orbit meets a kind of ``fundamental center domain"
of $p_x$.

\begin{lema}\label{l.localization1}

\label{local1} If (***) is satisfied,
there are points $z_0, z_1$ contained in $ W^c_{loc}(p_x)\setminus \{p_x\}$
such that if $z\in W^{s}_{loc}(p_x)\cap H(p)$ then there is $k\in \ZZ$  verifying that
$$f^{k}(z)\in W^{ss}_{loc}([f^{2\tau_x}(z_0),z_0)])\cup W^{ss}_{loc}([f^{2\tau_x}(z_1),z_1)])$$
where $[f^{2\tau_x}(z_i),z_i]$, for $i\in \{0,1\}$
is the connected arc of $W^c_{loc}(p_x)$ whose extremal points are $z_i,$ $f^{2\tau_x}(z_i)$
and (see figure~\ref{f.funda})
$$W^{ss}_{loc}([f^{2\tau_x}(z_i),z_i])= \bigcup_{\{z'\in [f^{2\tau_x}(z_i),z_i]\}} W^{ss}_{loc}(z').$$

\end{lema}

\begin{figure}[subsection]

\begin{center}



\input{per-new.pstex_t}

\end{center}

\caption{Fundamental center domains.  \label{f.funda}}

\label{center domain}

\end{figure}
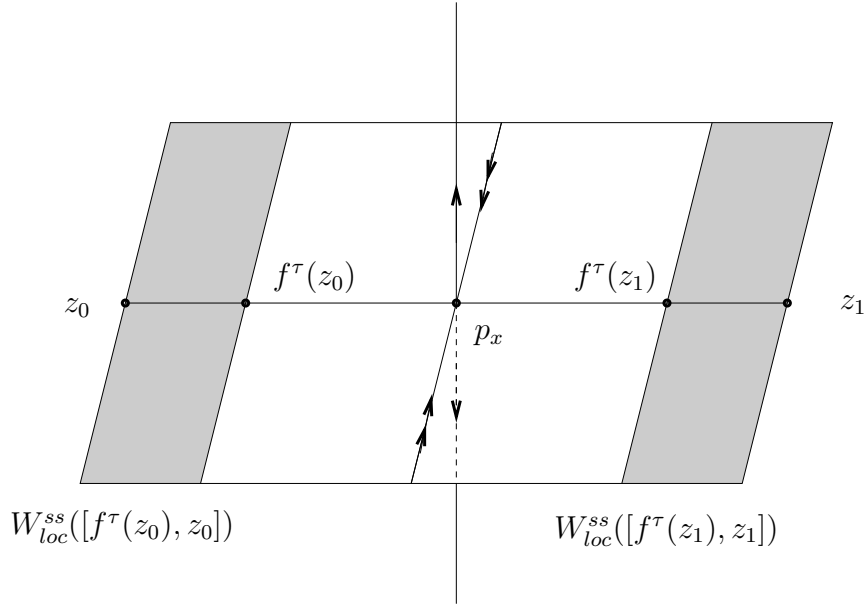

\begin{proof}
Let us consider two points $z_0^0$ and $z_1^0$ in two different connected
components of $W^c_{loc}(p_x)$ and set $z^n_i=f^{i\tau_x}(z_i^0)$.

 Note that the image of $W^{ss}_{loc}([f^{2\tau_x}(z_i^n),z_i^n])$ by $f^\tau$
is contained in $W^{ss}_{loc}([f^{2\tau_x}(z_i^{n+1}),z_i^{n+1}])$.
The union of the $W^{ss}_{loc}([f^{2\tau_x}(z_i^n),z_i^n])$ over $n\geq 0$
and of $W^{ss}_{loc}(p_x)$ contains a neighborhood of $p_x$.

If the thesis of the lemma does not hold, it follows that
for arbitrarily large $n\geq 0$, there exists a point $\zeta_n\in H(p_x)\cap W^s_{loc}(p_x)$
which belongs to $W^{ss}_{loc}([f^{2\tau_x}(z_i^{n+1}),z_i^{n+1}])$ and whose
preimage by $f^{\tau}$ does not belong to $W^{ss}_{loc}([f^{2\tau_x}(z_i^n),z_i^n])$.
An accumulation point $\zeta$ of $\{\zeta_n\}$ belongs to
$W^{ss}_{loc}(p_x)\setminus\{p_x\}]\cap H(p_x)$,
contradicting the assumption of the lemma.
\end{proof}
\bigskip

We now describe the returns of the forward orbit of $x$ in the neighborhood $W$ of the orbit of $p_x$. We need to take into account the orbits
that follow the orbit of $x$ during some time.
For that we let $\lambda_s\in (0,1)$ be an upper bound for the
contraction along $E^s$, we let $\lambda>1$ be a lower bound for the
domination between $E^s$ and $E^c\oplus E^u$ as in section~\ref{ss.holonomy}
and we let $\mu_c>\mu_s$ in $(0,1)$ be the modulus of the center eigenvalue at $p_x$
and the maximal modulus of the strong stable eigenvalues at $p_x$.
We also choose $\rho>1$ such that
$$\rho<\min (\lambda,\lambda_s^{-1/2},\mu_c/\mu_s).$$
We then introduce some ``forward dynamical balls" centered at $x$:
we fix $k_0\geq 1$ and for $n\geq 0$ we define the set
$$B_n(x)=\left\{z\in M,\; \forall\,  0\leq k\leq n, \, d(f^k(z),f^k(x))<\rho^{k-k_0}.\prod_{\ell=0}^{k-1}\|Df_{|E^s}(f^\ell(x))\|\right\}.$$
Note that:
\begin{itemize}
\item[(i)] By our choice of $\rho$,
the intersection of all the balls $B_n(x)$ coincides with a local strong stable manifold of $x$
and the image $f^n(B_n(x))$ has diameter smaller than $\sqrt \lambda_s^{n-k_0}$.
\item[(ii)] By taking $k_0$ large enough, the point $y$ belongs to the balls $B_n(x)$ and its forward iterates
satisfy the stronger estimate
\begin{equation}\label{i.2}
d(f^k(y),f^k(x))<\frac 1 3\rho^{k-k_0}.\prod_{\ell=0}^{k-1}\|Df_{|E^s}(f^\ell(x))\|.
\end{equation}
\end{itemize}
\noindent Let us now assume that (***) holds.

\begin{itemize}
\item[(iii)] For $n$ large enough, $f^n(B_n(x))$ does not intersect $W^u_{loc}(p_x)$. Otherwise
$B_n(x)$ would intersect a large backward iterate of $W^u_{loc}(p_x)$: this would imply that the strong stable manifold of the orbit of $p_x$ contains $x$
and contradicts our assumptions that $W^{ss}(p_x)\cap H(p)=\{p_x\}$. In fact, by first item if for $n$ large enough, $f^n(B_n(x))$ intersects $W^u_{loc}(p_x)$ it follows that $p_x\in W^{ss}_{loc}(x)$.
\item[(iv)] One can choose the neighborhood $W$ of the orbit of $p_x$ so that
the backward orbit of $x$ is contained in $W$ and $x\not \in \overline W$.
The lemma~\ref{l.localization1} above implies that the forward iterates of $x$ close to $p_x$
are close to the central manifold of $p_x$. Consequently, their distance to the local unstable manifold of the orbit of $p_x$ decreases by iteration by a factor close to the central eigenvalue of $p_x$. One thus gets the following.
\end{itemize}

\begin{lemma}\label{l.localization2}
Let us fix $\eta>0$ small.
If (***) holds,
any large iterate $f^n(B_n(x))$ which intersects $B(x,\delta)$
has the following property.

Let $m$ be the largest integer such that $m<n$ and
$f^{m}(B_n(x))$ is not contained in $W$. Then the distance between the points of $f^n(B_n(x))$
and $W^u_{loc}(p_x)$ belongs to $[\mu_c^{(1+\eta).(n-m)},\mu_c^{(1-\eta).(n-m)}]$, where
$\mu_c$ denotes the modulus of the center eigenvalue associated to $p_{x}$.
\end{lemma}

\begin{itemize}
\item[(v)] For any forward iterate $f^\ell(x)$ close to $p_x$, the quantity
$\rho \|Df_{|E^s}(f^\ell(x))\|$ is smaller than $\mu_c$ by our choice of $\rho$.
We thus obtain another version of the estimate of item (i).
\end{itemize}

\begin{lemma}\label{l.localization3}
If $f^n(B_n(x))$ intersects $B(x,\mu_c^N)$ for some $N\geq 1$ large,
then the diameter of $f^n(B_n(x))$ is smaller than
$\sqrt\lambda_s^{n-N}\mu_c^N$.
\end{lemma}

\subsection{Recurrence time dichotomy}\label{recurrence-dichotomy}

As before we denote by $\la_u,\la>1$ the lower bounds for the expansion along $E^u$ and the domination between $E^s$ and $E^c\oplus E^u$.
By lemma~\ref{l.holonomy}, there exists $\alpha_s\in (0,1)$
such that the strong stable fake holonomies are $\alpha_s$-H\"older.
The lemma~\ref{l.bundle} gives $\al'_s\in (0,1)$ which control the smoothness of the strong stable bundle.
Recall that by $\mu_c\in (0,1)$ we denote the modulus of the center eigenvalue associated to $p_{x}$ for $f$.
We also denote by $\bar \al_0$ the bound on the smoothness associated to
$H(p_{f_0})$ in proposition~\ref{p.generalized-strong-connectionCr}.

Let $\chi,K_1,K_2$  be some positive constants defined by 
$$\chi=\frac{\log \lambda_u}{\log\lambda_u+\|Df^{-1}_0\|},$$
$$K_1= \frac{|\log\mu_c|}{\chi\;\log\la},\,\,\,\,\,K_2=\frac{(1-\al_s)|\log\mu_c|}{\al_s\log\la_u}.$$

\medskip

Let us consider again the $C^{1+\al}$-diffeomorphism $f$ that is $C^1$-close to $f$.
If $\al$ belongs to $[0,\bar \al_0]$ and condition (***) does not hold, then
the proposition~\ref{p.generalized-strong-connectionCr} implies that there exist
$C^{1+\al}$-perturbation $g$ of $f$ such that $H(p_g)$ has a strong homoclinic intersection,
concluding the proof of the proposition.
\medskip

In the following we assume that condition (***) holds for $f$ and as in the statement of
proposition~\ref{p.unstable-nt}, that there exist two different points $x\in W^u(p_x)$
and $y\in W^u(p_y)$ whose strong stable manifolds coincide.
For any $N$ large, we take $V$ a neighborhood of size $\mu_c^N$ around $f^{-1}(x).$
We define $n= n(N)$, the smallest element of $\NN\cup \{\infty\}$
such that $f^n(B_n(x))$ intersects $V$.
By the property (iii) of section~\ref{ss.localization}, the sequence $\{n(N)\}$ increases and goes to $+\infty$ as $N$ increases.

\medskip

We  fix a constant $K>\max(1,K_0,K_1,K_2)$ and we are going to consider two cases:
\begin{enumerate}
\item {\it Fast returns.} There exists arbitrarily large $N$ such that 
\begin{eqnarray}\label{no-opttr}
n(N)\leq  K. N. \end{eqnarray}
\item {\it Slow returns.} There exists arbitrarily large $N$ such that 
\begin{eqnarray}\label{opttr} n(N)> K.N . \end{eqnarray}
\end{enumerate}
One of these two conditions (maybe both) occur.
If the first option holds, we prove the following.
\begin{lemma}\label{nontransversal1}
Assume that (***) and (\ref{no-opttr})
hold for some $K>0$, and that $\al<\inf(\frac 1 {K-1},\al'_s)$.
Then there exists a diffeomorphism $\varphi\in \diff^{1+\al}(M)$
that is $C^{1+\al}$-close to the identity such that $g=\varphi\circ f$ has a hyperbolic periodic point $q$
homoclinically related to the orbit of $p_{x,g}$ and whose strong stable manifold
$W^{ss}(q)\setminus \{q\}$ intersects $W^{u}(p_{y,g})$.
\end{lemma}

\noindent If the second option holds, we prove the following.

\begin{lemma}\label{nontransversal2-bis}
Assume that~(\ref{opttr})
holds for some $K>\max(K_1,K_2)$,
and that $1+\al<\frac{K}{\max(K_1,K_2)}$.
Then, there exists a diffeomorphism $\varphi\in \diff^{1+\al}(M)$ that is $C^{1+\alpha}$-close to the identity such that $g=\varphi\circ f$ satisfies the second option of
proposition \ref{p.unstable-nt}: if one fixes an orientation on $E^c_{y}$,
there exist two such diffeomorphism $g^+,g^-$ such that
$x_{g^+}$ (resp. $x_{g^-}$) belongs to $\cW^{cs,+}_{g^+,y_{g^+}}$
(resp. $\cW^{cs,-}_{g^-,y_{g^-}}$).

\end{lemma}

\noindent Both lemmas and the proposition~\ref{p.generalized-strong-connectionCr} conclude the proof of proposition~\ref{p.unstable-nt}.

\medskip

Note that for proving proposition~\ref{p.unstable-nt}
one can choose $K$ independently from $\mu_c$, for instance any
$$K=\|Df_0\|\max\left(\frac 3{\log\lambda_s},\frac {2\;\chi} {\log \lambda}, \frac{2\;(1-\alpha_s)}{\alpha_s\log \lambda_u}\right).$$
In this way we obtain a bound
$$\al_0=\inf\left(\bar \al_0, \frac 1 {K-1},\al'_s, \frac{K}{\max(K_1,K_2)}\right)$$
for the smoothness exponent $\alpha$ in proposition~\ref{p.unstable-nt},
which only depends on $f_0$ as announced.

\subsection{Fast returns: proof of lemma \ref{nontransversal1}}\label{nont-proof}

Let us assume that condition~(\ref{no-opttr}) holds for some large values of $N$
and some $K>0$ such that $\al<\inf\left(\frac{1}{K-1},\al'\right)$.
We also assume that (***) holds so that the lemma~\ref{l.localization2} applies.

\begin{lemma}\label{l.ab}
There are $a>b$ in $(K^{-1},1)$ such that some arbitrarily large $N$ and $n=n(N)$ satisfy:

\begin{enumerate}
\item $f^n(B_n(x))\cap B(x,\mu_c^{a\,n})\neq \emptyset$ and
\item $f^m(B_m(x))\cap B(x,\mu_c^{b\,n})= \emptyset$ for any $k_0<m<n.$
\end{enumerate}
Moreover $\frac a b$ can be chosen arbitrarily close to $\frac K {K-1}$.

\end{lemma}

\begin{proof}
We introduce the integers $N_i$ and $n_i=n(N_i)$ satisfying
for any $i$,
$$ N_i<N_{i+1},\;\; n_i<n_{i+1}, \text{ and }\,
\forall\, N_{i-1} < N \leq N_{i},\;\; n(N)=n_i.$$
We will prove that there are positive constants $b'<a'$ in $(K^{-1},1)$
and there is $n_i$ sufficiently large such that $N_i> a'.n_i$ and $N_j< b'.n_i$ for $0\leq j<i$. We  then choose any $b<a$ in $(b',a')$.
One can check easily that for these large $n=n_i$ the result holds:
\begin{itemize}
\item[--] We have $f^n(B_n(x))\cap B(x,\mu_c^{N_i})\neq\emptyset$ with $N_i>a'.n$, hence
$f^n(B_n(x))\cap B(x,\mu_c^{a'.n})$ is non-empty.
By lemma~\ref{l.localization3}, the diameter of $f^n(B_n(x))$ is bounded
by $\sqrt \lambda_s^{n-N_i}\mu_c^{N_i}$ which is much smaller than $\mu_c^{a'.n}$.
As a consequence $f^n(B_n(x))$ is contained in $B(x,\mu_c^{a.n})$.
\item[--] By definition of the sequence $(N_j)$, for any $m<n=n_i$, one has $f^m(B_m(x))\cap B(x,\mu_c^{N_{i-1}+1})=\emptyset$ and
$b.n>b'.n+1>N_{i-1}+1$, implying the second condition of the lemma.
\end{itemize}

\medskip

Let us now prove the existence of the constants $a'<b'$.
We denote by $m_i$ the smallest integer such that
$$\forall m_i\leq m < n_i,\;\; f^m(B_{n_i}(x))\subset W.$$
By lemma~\ref{l.localization2} if one chooses $\varepsilon>0$ small
and if $N_i$ is large enough, one has
$$(1+\varepsilon).(n_i-m_i)\geq N_i.$$
Let us define
$$R=\limsup_{j\to +\infty} \frac{N_j}{n_j}.$$
By~(\ref{no-opttr}), $R$ belongs to $[K^{-1},1]$.

For any $j$ larger than a constant $j_0$
we have $\frac{N_j}{n_j}< (1+\varepsilon)R.$
For some $i$ sufficiently large we also have $\frac{N_i}{n_i}> (1-\varepsilon)R$.
If $j<i$ we have $n_j\leq m_i\leq n_i-(1+\varepsilon)^{-1}N_i$ and so for $j_0<j<i$ we have
$$N_j \leq (1+\varepsilon)R n_j\leq (1+\varepsilon)R (n_i-(1+\varepsilon)^{-1}N_i)\leq R[1-(1-\varepsilon)R+\varepsilon] n_i.$$
Since $R$ belongs to $[K^{-1},1]$, then $[1-(1-\varepsilon)R+\varepsilon] <(1-\varepsilon)$ for $\varepsilon$ small and therefore taking $a'= (1-\varepsilon)R$ and $b'= R[1-(1-\varepsilon)R+\varepsilon] $ the result holds. To check that it also holds for $j<j_0$ it is enough to take $i$ sufficiently large.

Observe that the quantity $\frac a b$ is close to $\frac {1-\varepsilon}{1-(1-\varepsilon)R+\varepsilon}$. Since $R\geq K^{-1}$,
when $\varepsilon$ goes to $0$
the limit is larger or equal to $\frac{K}{K-1}$.

\end{proof}

\bigskip

We can now conclude the proof of lemma~\ref{nontransversal1}.
\begin{proof}[\it Proof of lemma~\ref{nontransversal1}]
We fix $a,b$ and a large integer $n$ as in lemma~\ref{l.ab}.
By assumption $\al<(K-1)^{-1}$ and $\frac a b$ can be chosen close
to $\frac K {K-1}$. One can thus ensure that $1+\al$ is smaller than $a/b$.

Let $D\subset W^{ss}_{loc}(x)$ be the smallest disc containing $y$.
By construction it is contained in the ball $B_n(x)$, hence its image
by $f^n$ is contained in $B(x,\mu_c^{an})$.
We consider a $C^{1+\al}$-diffeomorphism $\varphi$ supported in
$B(x,\mu_c^{bn})$ which sends $f^n(D)$ into $D$
and define $g=\varphi\circ f$.
By construction the support of the perturbation $g$
is disjoint from $D$ and its $n-1$ first iterates.

\begin{claim}

If $1+\al<\inf(\frac a b,\al'_s)$, by choosing $n$ large the
diffeomorphism $\varphi$ can be taken arbitrarily close to the identity
in $\diff^{1+\al}(M)$.
\end{claim}

\begin{proof}
Let us consider a $C^{1+\al}$-chart $U\to \RR^d$ of a neighborhood $U$ of $x$ such that
$x$ coincides with $0$ and $W^{ss}_{loc}(x)$ coincides with the plane
$\RR^k\times \{0\}$, where $k=\dim(E^{ss})$.
For $n$ large, the plaque $W^{ss}_{loc}(f^n(x))$ is close to $W^{ss}_{loc}(x)$
and coincides in the chart with the graph of a map $\chi_0\colon \RR^k\to \RR^{d-k}$.
We introduce the map
$$(z_1,z_2)\mapsto (0,-\chi_0(z_1))$$
which is close to $0$ in the $C^{1+\al}$ topology and satisfies $\|v_0(0)\|\leq e^{-an}$
by construction.

One can thus apply the lemma~\ref{l.perturbation} in order to build
a map $v\colon \RR^d\to \RR^{d-k}$ which
coincides with $v_0$ on the ball $B(0,e^{-an})$ and with $0$ outside $B(0,e^{-bn})$.
The map $\varphi\colon (z_1,z_2)\mapsto (z_1,z_2-v(z_1,z_2))$ is the announced diffeomorphism.
In order to prove that $\varphi$ is close to the identity in $\diff^{1+\al}(M)$,
one has to check that $e^{(1+\al)bn}\sup_{B(0,e^{-bn})}\|v_0\|$ is small.

Since $\al<\al'_s$, by lemma~\ref{l.bundle} we have
$$\|Dv_0(0)\|\leq C\|v_0(0)\|^\al\leq Ce^{-\al an}.$$
Since $v_0$ is close to the identity in $\diff^{1+\al}$,
there exists an arbitrarily small constant $\varepsilon>0$ such that
$$\sup_{B(0,e^{-bn})}\|Dv_0\|\leq \|Dv_0(0)\|+ \varepsilon e^{-\al bn}\leq 2\varepsilon e^{-\al bn}.$$
This gives
$$\sup_{B(0,e^{-bn})}\|v_0\|\leq \|v_0(0)\|+e^{-bn}\sup_{B(0,e^{-bn})}\|Dv_0\|\leq e^{-an}+2\varepsilon e^{-(1+\al)bn}.$$
Since $a>(1+\al)b$, this shows that $e^{(1+\al)bn}\sup_{B(0,e^{-bn})}\|v_0\|$ is small
when $n$ is large.

\end{proof}

\medskip

To continue with the proof of lemma~\ref{nontransversal1}, we note that
the map $\varphi\circ f^n$ is a contraction on $D$, hence
the diffeomorphism $g$ has a $n$-periodic point $q$
whose strong stable manifold contains $D$.
Since the backward orbit of $W^u_{loc}(p_y)$
is disjoint from the support of the perturbation,
the manifolds $W^{ss}(q)$ and $W^u_{loc}(p_{y,g})$
intersect.

In particular $W^{s}(q)$ and $W^u_{loc}(p_{y,g})$
have a transversal intersection.
On the other hand the orbit of $q$ has a point close to $p_x$,
hence $W^{s}(p_x)$ and $W^u(q)$ have a transversal intersection.
One deduces that $q$ is homoclinically related to the orbits
of $p_{x,g}$ and $p_{y,g}$.
This concludes the proof of lemma~\ref{nontransversal1}.
\end{proof}

\subsection{Slow returns: proof of lemma \ref{nontransversal2-bis}}\label{nont2-bis-proof}

Let us fix a center-unstable plaque $\cD$ at $x$ and for diffeomorphisms $g$ close to $f$
we consider the strong stable holonomy $\Pi^{ss}_g$ to $\cD$.
Since the map $f^\tau_x$ is linear in a neighborhood of $p_x$, one can choose $\cD$
in the linear plane corresponding to the sum of the central and unstable eigenspaces.
Observe that it contains the manifold $W^{u}_{loc}(p_x)$ for $f.$

Under condition~(\ref{opttr}), we are going to get a perturbation $g$ of $f$ such that
$\po_g(x_g)\neq \po_g(y_g)$, proving that $W^{ss}_{loc}(x_g)$ and $W^{ss}_{loc}(y_g)$ are disjoint.
Since $x_g,y_g$ belong to a same center-stable plaque $\cW^{cs}_{g,y_g}$, the projections
$\po_g(x_g),\po_g(y_g)$ are contained in a $C^1$-curve of $\cD$
that is tangent to a central cone field.
Moreover, one will be able to choose the perturbation to satisfy either $x_g\in\cW^{cs,+}_{g,y_g}$ or $x_g\in\cW^{cs,-}_{g,y_g}$.

\paragraph{Description of the perturbation.}

We recall that we have fixed a large integer $N\geq 1$ and that $V$ denotes the ball $B(x,\mu^N_c)$.

Let us fix two small constants $\widehat r=\frac 1 2 \mu_c^N$ and $r<\widehat r$ in $(0,1)$.
We perform the perturbation $g$ of $f$ in the ball $B(f^{-1}(x),\widehat r)$, in such a way that
$W^{u}_{loc}(p_x)$ is still contained in the coordinate subspace and the distance
between $x=f(f^{-1}(x))$ and $g(f^{-1}(x))$ is $r$ along the central coordinate.
This can be realized by a small perturbation of $f$ in $\diff^1(M)$ provided
$\widehat r$ and $r/\widehat r$ are large enough. Moreover one can require that the $C^0$ size of the perturbation
is equal to $r$.

Later, in item $7$, we explain how the perturbation can be adapted to be $C^{1+\al}$-small.
Note that the point $x$ can be pushed to $g(f^{-1}(x))$ along $E^c_x$ in any of the two
central directions at $x$.

To get the conclusion, we choose a small constant $\varepsilon>0$
(independent from $N$) and show that
the distances $d(\po_g(y_g), x)$ and $d(x_g, g(f^{-1}(x)))$ are smaller than $\varepsilon.r$,
which is much smaller than $d(x,g(f^{-1}(x)))$.

\bigskip

\noi {\bf 1- Estimating $d(y_g, y)$.}
Observe that $y_g$ does not necessarily coincide with $y$ since the forward orbit of $y$ may intersect the region of
perturbation. However by lemma~\ref{l.cont-unstable} the point $y_{g}$ belongs to the local unstable manifold of $p_{y,g}=p_{y,f}$
which coincides for $f$ and $g$.
We will consider the distance $dist$ along the unstable plaques (which is locally comparable in a uniform way to the distance in the ambient space). 
We also introduce a constant $C\gg\frac{1}{\la_u-1}$
independent from $N$.

\begin{lemma}\label{l.upper}

If for some positive integer $m$ the two points $f^m(y), g^m(y_g)$ belong to a same
unstable plaque, then their distance satisfies $dist(f^m(y), g^m(y_g))< C.r$.
\end{lemma}

\begin{proof}

Let us assume by contradiction that the estimate does not hold.
Observe that the distance by the action of $f$ growth by a factor $\la_u$ and the $C^0$ distance between $f$ and $g$ is at most $r$,
which is much smaller than $C.r$. One deduces that the points $f^{m+1}(y), g^{m+1}(y_g)$ still belong to a same unstable plaque.
Denoting $\gamma=\frac{\lambda_uC-1}C>1$, their distance now satisfy
\begin{equation*}
\begin{split}
dist(f^{m+1}(y),g^{m+1}(y_g))&>  \la_u \; dist(f^{m}(y),g^{m}(y_g))-r\\
&>(\la_u-C^{-1}) \;dist(f^{m}(y),g^{m}(y_g)) =\ga.dist(f^{m}(y),g^{m}(y_g)).
\end{split}
\end{equation*}
Therefore after $k$ iterates the distance become larger than
$ \ga^k.C.r$ and so increasing to  infinity. This is a contradiction with the fact $y_g$ is a continuation of $y$.
\end{proof}

\begin{lemma}\label{l.compare1}
The $n(N)$ first iterates of $y$ and $y_{g}$ coincide for $f$ and for $g$.
\end{lemma}

\begin{proof}
Since $y$ belongs to the dynamical balls $B_n(x)$,
the segment of orbit $(y,\dots,f^{n(N)}(y))$ is also a segment of orbit for $g$.
Let us consider the first integer $m\geq 1$ such that $g^m(y_g)=f^m(y_g)$ enters in the region of perturbation and let us assume by contradiction
that $m< n(N)$.

As for $y,y_g$, one knows that $f^m(y_g)$ and $f^m(y)$ belong to a same unstable plaque: by lemma~\ref{l.upper} they are at distance smaller than $C.r$.
If $r$ has been chosen small enough one has $C.r<\frac 1 2 \mu_c^N=\widehat r$.
By definition of $m$ one also has
$d(f^m(y_g),x)<\hat r=\frac 1 2 \mu_c^N$.
As a consequence $f^m(y)$ belongs to $V$, hence $m\geq n(N)$.
This contradicts our assumption.
This shows that the orbit $(y_g,\dots,g^{n(N)}(y_g))$
coincides for $f$ and for $g$.

\end{proof}

Since $y,y_g$ belong to an unstable plaque,
and since by lemma~\ref{l.compare1}
their $n(N)$ first iterates are the same by $f$ and by $g$,
the points $f^{n(N)}(y)$ and $g^{n(N)}(y_g)$ belong to
a same unstable plaque and by lemma~\ref{l.upper}, their distance
is smaller than $C.r$. For any $0\leq m\leq n(N)$ we obtain
\begin{equation}\label{l.dy}
d(g^m(y_g),f^m(y))<\la_u^{m-n(N)}C.r.
\end{equation}

\bigskip

\noi {\bf 2- Estimating $d(x_g, g(f^{-1}(x)))$.}
Arguing as in lemma~\ref{l.upper}, one shows that
for any positive integer $m$, if the two points $f^m(x), g^m(x_g)$ belong to a same
unstable plaque, then their distance satisfy $dist(f^m(x), g^m(x_g))<C r$.

Let us denote by $\lambda'>1$ a lower bound for the domination between the bundles $E^c$ and $E^u$
and consider two large constants $k\ll\ell$ (independent from $N$)
such that $\lambda_{u}^{\ell}.(\lambda')^{-k}>C$.
If $N$ has been chosen large, the $\ell$ first iterates of $x,x_g,g(f^{-1}(x))$ are the same by $f$ and by $g$.
Let us assume by contradiction that the distance $dist(g(f^{-1}(x)),x_{g})$ inside $W^u_{loc}(p_{x,g})$ is larger than
$(\lambda')^{-k}.r$. Since the distance between $x$ and $g(f^{-1}(x))$ in the central direction is
equal to $r$, one deduces that the distance from $f^{\ell}(g(f^{-1}(x)))$ to $f^{\ell}(x_{g})$ is much larger than its distance
to $f^{\ell}(x)$. In particular $f^{\ell}(x)$ and $f^{\ell}(x_{g})$ are contained in a same unstable plaque and by our choice of $k,\ell$, their distance
is larger than $C.r$, which is a contradiction. Consequently
$$dist(g(f^{-1}(x)),x_{g})<(\lambda')^{-k}.r.$$
Taking $k$ large enough, one has $d(g(f^{-1}(x),x_g)<\varepsilon. r$ as wanted.

\bigskip

\noi {\bf 3- Estimating $d(\widehat{\Pi^{ss}_f}(y_g), \Pi^{ss}_f(y))$.}
Since $y_{g}$ belongs to the unstable manifold $W^u_{loc}(y)$ for $f$, one can
introduce some fake holonomies $\widehat{\Pi^{ss}_f}(y_g), \widehat{\Pi^{ss}_f}(y)=\Pi^{ss}_f(y)$ for $f$.
By~(\ref{l.dy}) and lemma~\ref{l.holonomy}, one gets 
$$d(\widehat{\Pi^{ss}_f}(y_g), \Pi^{ss}_f(y))< d(y, y_{g})^{\al_s}<[\la_u^{-n(N)}C.r]^{\al_s}.$$
\bigskip

\noi {\bf 5- Estimating $d(\po_g(y_g), \widehat{\Pi^{ss}_f}(y_g))$.}
As before we first compare the iterates of $f$ and $g$.
\begin{lemma}\label{l.compare2}
The $\chi.n(N)$ first iterates of $y_{g}$, $\po_g(y_g)$ and $\widehat{\Pi^{ss}_f}(y_g)$ coincide for $f$ and for $g$, where $\chi=\frac{\log\lambda_u}{\log \lambda_u+\log\|Df_0^{-1}\|}.$
\end{lemma}

\begin{proof}
By lemma~\ref{l.compare1} we already know that
the $n(N)$ first iterates of $y_g$ under $f$ and $g$ coincide.
Since $\chi\in (0,1)$ and from the estimate~(\ref{l.dy}),
the points $y$ and $y_g$ do not separate by $f$ during the time $\chi.n(N)$
and by definition of the fake holonomies, the $\chi.n(N)$ first iterates of the points $\widehat{\Pi^{ss}_f}(y_g)$ and $y_g$ remain in a same strong stable plaque.

From~(\ref{l.dy}) and the definition of $\chi$,
we also have that for $0\leq m\leq \chi.n(N)$,

\begin{equation}\label{e.sep}
d(f^{m}(y_g),f^{m}(y))<\lambda_u^{-n(N)+m}.C.r<\|Df_0^{-1}\|^{-m}<\frac 1 3 \rho^{m-k_0}.
\prod^{m-1}_{\ell=0}\|Df_{|E^s}(f^\ell(x))\|.
\end{equation}
With~(\ref{i.2}), this shows that $y_g$ belongs to the dynamical ball $B_{\chi.n(N)}(x)$.

We will prove by induction on $m\leq \chi.n(N)$ that
$\po_g(y_g)$ and $\widehat{\Pi^{ss}_f}(y_g)$ by $f$ also belong to the dynamical ball
$B_m(x)$. This will imply that their $m^\text{th}$ iterates by $f$ and $g$ coincide
and conclude the proof of the lemma.

Let us choose $\eta>0$ small and $m_0\geq 0$ large.
If $N$ has been chosen large enough, the point $y_g$ is close to $y$
and the points $\po_g(y_g)$ and $\widehat{\Pi^{ss}_f}(y_g)$ are close to $x$;
as a consequence, the three points belong to the dynamical balls $B_m(x)$
with $0\leq m\leq m_0$. When $m$ is larger than $m_0$, the diameter of
$f^{m-1}(B_{m-1}(x))$ is small, hence
$$d(f^{m}(\Pi^{ss}_g(y_g)),f^{m}(y_g))
\leq e^\eta.\|Df_{|E^s}(f^{m-1}(x))\|.d(f^{m-1}(\Pi^{ss}_g(y_g)),f^{m-1}(y_g)),$$
$$d(f^{m}(\widehat{\Pi^{ss}_f}(y_g)),f^{m}(y_g))
\leq e^\eta.\|Df_{|E^s}(f^{m-1}(x))\|. 
d(f^{m-1}(\widehat{\Pi^{ss}_f}(y_g)),f^{m-1}(y_g)).$$
With~(\ref{e.sep}), (\ref{i.2}), this gives the required estimate and gives the conclusion.
\end{proof}

Since the points $\po_g(y_g), \widehat{\Pi^{ss}_f}(y_g)$
belong to a same center-unstable plaque and since their $\chi.n(N)$ first iterates by $f$ remain close,
one deduces that for any $0\leq m\leq \chi.n(N)$, the points
$f^m(\po_g(y_g))$ and $f^m(\widehat{\Pi^{ss}_f}(y_g))$ are still contained in a center-unstable plaque,
whereas the pairs of point $f^m(\po_g(y_g))$, $f^m(y_g)$ and $f^m(\widehat{\Pi^{ss}_f}(y_g))$, $f^m(y_g)$
are contained in strong-stable plaques.
This shows that
$$d(\po_g(y_g), \widehat{\Pi^{ss}_f}(y_g))<\lambda^{-\chi.n(N)},$$
where $\lambda>1$ is the lower bound for the domination between the bundles $E^{s}$
and $E^c\oplus E^u$.

\bigskip

\noi {\bf 6-  Estimating $d(\po_g(y_g), x)$.}
From the estimates we obtained, we get
$$dist(\po_g(y_g), x)< d(\po_g(y_g), \widehat{\Pi^{ss}_f}(y_g)) + d(\widehat{\Pi^{ss}_f}(y_g), \Pi^{ss}_f(y))
< \lambda^{-\chi.n(N)}+[\lambda_u^{-n(N)}C.r]^{\alpha_s}.$$
In order to conclude, the perturbation should thus satisfies:
$$ \lambda^{-\chi.n(N)}+[\lambda_u^{-n(N)}C.r]^{\alpha_s}<\varepsilon.r.$$

Since $\chi,\alpha_s, C, \varepsilon$ are constants independent from $N$,
this inequality holds if $N$ large enough and the following are satisfied:
$$\alpha_s(n(N)\log \lambda_u+|\log r|)>|\log r|+c,$$
$$n(N)\log\lambda>|\log r|+c,$$
where $c>0$ is independent from $N$.

From the definition of $\widehat r$ and since $n(N)> K.N$, one gets the following condition
\begin{equation}\label{e.est1}
|\log r|< B.|\log \widehat r|-c,
\end{equation}
where
$$B= \inf\left(\chi\log\lambda,\frac{\alpha_s}{1-\alpha_s}\log\lambda_u\right)\frac{K}{|\log\mu_c|}.$$
Note that by our choice of $K$, the factor $B$
is larger than $1$.

\bigskip

\noi {\bf 7- Realization of the $C^{1+\al}$ perturbation.}
By lemma~\ref{l.perturbation},
in order to be able to realize a $C^{1+\alpha}$ perturbation
supported on a ball of radius $\widehat r$ such that $d(g(f^{-1}(x)),x)=r$,
one has to check that for some $A>\al$
one can choose $r,\widehat r$ arbitrarily small satisfying
\begin{equation}\label{e.perturb}
|\log r|>(1+A)|\log \widehat r|.
\end{equation}
Note that this also implies the estimate
$C.r<\widehat r=\frac 1 2 \mu_c^N$ that we used in paragraph 1.

By our choice of $K$, both conditions~(\ref{e.est1}) and~(\ref{e.perturb}) can be realized
simultaneously provided $1+\alpha$ is smaller than $B$.

%% file: per-new.pstex_t
\begin{picture}(0,0)%
\includegraphics{per-new.pstex}%
\end{picture}%
\setlength{\unitlength}{1657sp}%
\begingroup\makeatletter\ifx\SetFigFontNFSS\undefined%
\gdef\SetFigFontNFSS#1#2#3#4#5{%
  \reset@font\fontsize{#1}{#2pt}%
  \fontfamily{#3}\fontseries{#4}\fontshape{#5}%
  \selectfont}%
\fi\endgroup%
\begin{picture}(11640,9044)(2191,-9533)
\put(13816,-5101){\makebox(0,0)[b]{\smash{{\SetFigFontNFSS{12}{14.4}{\rmdefault}{\mddefault}{\updefault}{\color[rgb]{0,0,0}$z_1$}%
}}}}
\put(5761,-4741){\makebox(0,0)[b]{\smash{{\SetFigFontNFSS{12}{14.4}{\rmdefault}{\mddefault}{\updefault}{\color[rgb]{0,0,0}$f^\tau(z_0)$}%
}}}}
\put(2206,-5146){\makebox(0,0)[b]{\smash{{\SetFigFontNFSS{12}{14.4}{\rmdefault}{\mddefault}{\updefault}{\color[rgb]{0,0,0}$z_0$}%
}}}}
\put(8371,-5551){\makebox(0,0)[b]{\smash{{\SetFigFontNFSS{12}{14.4}{\rmdefault}{\mddefault}{\updefault}{\color[rgb]{0,0,0}$p_x$}%
}}}}
\put(2881,-8476){\makebox(0,0)[b]{\smash{{\SetFigFontNFSS{12}{14.4}{\rmdefault}{\mddefault}{\updefault}{\color[rgb]{0,0,0}$W^{ss}_{loc}([f^\tau(z_0),z_0])$}%
}}}}
\put(10261,-4741){\makebox(0,0)[b]{\smash{{\SetFigFontNFSS{12}{14.4}{\rmdefault}{\mddefault}{\updefault}{\color[rgb]{0,0,0}$f^\tau(z_1)$}%
}}}}
\put(11026,-8521){\makebox(0,0)[b]{\smash{{\SetFigFontNFSS{12}{14.4}{\rmdefault}{\mddefault}{\updefault}{\color[rgb]{0,0,0}$W^{ss}_{loc}([f^\tau(z_1),z_1])$}%
}}}}
\end{picture}%

%% file: disconnectedness-1110.tex
\section{Structure in the center-stable leaves}\label{s.2D-central}

In this section we prove theorem~\ref{t.tot-discontinuity}
on the geometry of chain-hyperbolic classes.
It is used in the proof of theorems~\ref{t.extremal} and~\ref{t.2D-central}.
As a consequence (see proposition~\ref{p.box}),
for some chain-hyperbolic classes, one can
replace the plaques $\cW^{cs}_x$ by submanifolds $V_x$ whose
boundaries are disjoint from $H(p)$.

In the whole section, $H(p)$ is a chain-recurrence class with a dominated splitting $E^{cs}\oplus E^{cu}=(E^{s}\oplus E^{c}_1)\oplus E^c_2$ such that
$E^c_1,E^c_2$ are one-dimensional and $E^{cs},E^{cu}$ are thin-trapped.
We assume moreover that for each periodic point $q\in H(p)$, the set
$W^{ss}(q)\setminus \{q\}$ is disjoint from $H(p)$.

\subsection{Geometry of connected compact sets}

One can obtain connected compact sets as limit of $\varepsilon$-chains, i.e.
finite sets $\{x_0,\dots,x_m\}$ such that $d(x_i,x_{i+1})<\varepsilon$ for each $0\leq i<m$.
This idea is used to prove the following lemma.

\begin{lema}\label{l.connex}
For any $n\geq 1$, any distance on $\RR^n$ which induces the standard topology,
any closed connected set $K\subset \RR^n$, any point $x\in K$
and any $0\leq D\leq\mbox{Diam}(K)$, there exists a compact connected set $K(D)\subset K$ containing $x$ and whose diameter is equal to $D$.
\end{lema}
\begin{proof}
For $\varepsilon>0$, one can choose a finite set
$X_\varepsilon=\{x_0,x_1,\dots,x_m\}\subset K$
such that
\begin{itemize}
\item[--] $x$ belongs to $X_\varepsilon$;
\item[--] for each $0\leq i< m$, the open balls $B(x_i,\varepsilon)$
and $B(x_{i+1},\varepsilon)$ intersect;
\item[--] the diameter of $X_\varepsilon$ belongs to $[D,D+2\varepsilon]$.
\end{itemize}
Let $K_\varepsilon$ be the closed $\varepsilon$-neighborhood of $X_\varepsilon$.
It is a connected compact set contained in the $\varepsilon$-neighborhood of $K$.
Up to considering an extracted sequence, $(K_\varepsilon)$
converges for the Hausdorff topology towards a compact set $K(D)$
which contains $x$, is connected and has diameter $D$ as required.
\end{proof}
\medskip

Recall that for $x\in H(p)$, the submanifold
$W^{ss}(x)$ is diffeomorphic to $\RR^{d-2}$, where $d=\dim(M)$.

\begin{lema}\label{l.converge}
Consider a sequence $(z_n)$ in $H(p)$ which converges to a point
$z$ and for each $n$ a compact connected set $C_n\subset W^{ss}(z_n) \cap H(p)$
which converges for the Hausdorff topology in $M$ towards a (compact connected) set
$C\subset W^{ss}(z) \cap H(p)$. Then the restriction of the
intrinsic distance of $W^{ss}(z_n)$ to the set
$C_n$ converges towards the intrinsic distance of $W^{ss}_z$ to $C$.
\end{lema}
\begin{proof}
Let $U$ be a bounded neighborhood of $K$ inside $W^{ss}(z)$
which is diffeomorphic to $\RR^{d-2}$.

For $z_n$ close to $z$, there exists an open set $U_n\subset W^{ss}_{z_n}$,
containing $z_n$, diffeomorphic to $\RR^{d-2}$ and which is close to $U$ for the $C^ 1$-topology on immersions
of $\RR^{d-2}$. In particular, $U$ and $U_n$ are diffeomorphic by
a map whose Lipschitz constant is arbitrarily close to $1$.
Since $K_n$ is connected and contains $z_n$, it is included in $U_n$.
This gives the conclusion.
\end{proof}

\subsection{Structure in the strong stable leaves}
We are aimed first to  prove total discontinuity in the strong stable leaves.

\begin{prop}\label{ss-tot-discontinuity}
Let $f$ be a diffeomorphism and $H(p)$ be a chain-hyperbolic class 
satisfying the assumptions of theorem~\ref{t.tot-discontinuity}.
If for any periodic point $q\in H(p)$ the set $W^{ss}(q)\setminus \{q\}$ is disjoint from $H(p)$, then, for each $x\in H(p)$, the set $W^{ss}_{loc}(x)\cap H(p)$ is totally disconnected.
\end{prop}

At any points, one considers the plaques $\cW^{cu}_x\subset f(\cW^{cu}_{f^{-1}(x)})$.
We choose the plaques $\cW^{cs},\cW^{cu}$ with a diameter small enough so that
for each $x,y\in H(p)$ the intersection
$\cW^{cs}_x\cap f(\cW^{cu}_y)$ is transversal and contains at most one point
(which belongs to $H(p)$ by lemma~\ref{l.bracket0}).

For this proof we will endow $H(p)$ with the \emph{center-stable topology}:
two points $x,y\in H(p)$ are close if the distance $d(x,y)$ is small
and $x\in \cW^{cs}_y$ (or equivalently $y\in \cW^{cs}_x$ by lemma~\ref{l.uniqueness-coherence}).
The \emph{center-stable distance} on $H(p)$
is the distance between $x$ and $y$ inside $\cW^{cs}_x$.

Since $\cW^{cs}$ is trapped, $W^{ss}(x)\cap \cW^{cs}_x$ is contained inside $W_{loc}^{ss}(x)$ and the center-stable topology induces on $W^{ss}(x)\cap H(p)$ the intrinsic topology of $W^{ss}(x)$.

\paragraph{Local holonomy.} We fix $\rho>0$ and define the ball $B^{cs}(x)$
centered at $x\in H(p)$ of radius $\rho$ contained in $\cW^{cs}_x$.
If $\rho$ is small, for any points $x_0\in H(p)$ and $y_0,z_{0}\in \overline{\cW^{cu}_{x_0}}\cap H(p)$
the \emph{local holonomy} $\Pi^{cu}$ along the center-unstable plaques
$f(\cW^{cu}_{f^{-1}(x)})$, ${x\in \cW^{cs}(x_0)\cap H(p)}$,
is defined from $B^{cs}(z_0)\cap H(p)\subset \cW^{cs}_{z_0}$ to $\cW^{cs}_{y_0}$.

\paragraph{Global holonomy.}
We now try to extend globally the holonomy. A strong stable leaf may intersect a plaque of $\cW^{cu}$
in several points, hence the global holonomy may be multivalued.
A \emph{global holonomy along the plaques $\cW^{cu}$} is a closed connected set
$\Delta\subset H(p)\times H(p)$ (endowed with the product center-stable topologies)
such that for any $(x,y)\in \Delta$ one has
$y\in \overline{\cW^{cu}_{x}}$ and $x\in \overline{\cW^{cu}_{y}}$.
The sets $\pi_1(\Delta)$ and $\pi_2(\Delta)$ denote the projections on the first and the second factors.

\medskip

One can obtain global holonomies from connected sets contained in a strong stable leaf.
\begin{lema}\label{l.extension}
Let $\Delta_0$ be a global holonomy along the center-unstable plaques,
and $C\subset H(p)$ be a set which is closed and connected
for the center-stable topology and which contains $\pi_1(\Delta_0)$.

Then, there exists a global holonomy $\Delta$ along the center-unstable plaques
containing $\Delta_0$, such that $\pi_1(\Delta)\subset C$ and satisfying one of the following cases.
\begin{enumerate}
\item\label{i.extension1} $\pi_1(\Delta)=C$;
\item\label{i.extension2} $\Delta$ is non-compact;
\item\label{i.extension3} there exists $(x,y)\in \Delta$ such that $y\in \overline{\cW^{cu}_{x}}\setminus \cW^{cu}_x$ or
$x\in \overline{\cW^{cu}_{y}}\setminus \cW^{cu}_y$.
\end{enumerate}
\end{lema}
\begin{proof}
If $\{\Delta_n\}$ is a family of global holonomies along the center-unstable plaques
that is totally ordered by the inclusion, then the closure of the union $\overline{\cup_n \Delta_n}$
is also a global holonomy. By Zorn's lemma one deduces that there exists a global holonomy
$\Delta$ containing $\Delta_0$, satisfying $\pi_1(\Delta)\subset C$ and maximal with these properties for the inclusion. We prove by contradiction that $\Delta$ satisfies one of the properties above.
We fix a pair $(x_0,y_0)\in \Delta_0$.

If $\pi_1(\Delta)\neq C$, then there exists $r_1>0$ and for each $\varepsilon_1>0$
there exists a sequence $(x_0,\dots,x_s)$ in $C$ such that
\begin{itemize}
\item[--] for each $0 < i \leq s$, the points $x_{i-1},x_{i}$ are at distance less than $\varepsilon_1$
and $x_{i}\in B^{cs}(x_{i-1})$;
\item[--] the point $x_s$ and the set $\pi_1(\Delta)$ are at distance exactly $r_1$ inside $\cW^{cs}_{x_s}$.
\end{itemize}

If $\Delta$ does not satisfies the items~\ref{i.extension2}) or \ref{i.extension3}), then
for any $(x,y)\in H(p)\times H(p)$ close to $\Delta$
and any $x'\in H(p)$ close $x$  (for the center-stable topology),
$B^{cs}(y)$ meets $\cW^{cu}_{x'}$ at a point $y'\in H(p)$ which also satisfies
$x'\in \cW^{cu}_{y'}$.

This allows to build inductively a sequence $(y_0,\dots,y_{\ell})$ for some $0\leq \ell\leq s$
and associated to $(x_0,\dots,x_\ell)$ such that, for each $i$, the pair $(x_i,y_i)$ is at
a small distance from $(x_{i-1},y_{i-1})$ for the center-stable distance.

More precisely, there exists $r>0$ and for each $\varepsilon>0$ there exists
a sequence $(x_0,y_0),\dots,(x_\ell,y_\ell)$ such that for the product center-stable distance
on $H(p)\times H(p)$ the following holds:
\begin{itemize}
\item[--] for each $0 < i \leq \ell$, one has $x_i\in \cW^{cu}_{y_i}$ and
$y_i\in \cW^{cu}_{x_i}$;
\item[--] for each $0 < i \leq \ell$, the pairs $(x_{i-1},y_{i-1})$ and $(x_{i},y_i)$ are at distance less than $\varepsilon$;
\item[--] the pair $(x_\ell,y_\ell)$ and the set $\Delta$ are at distance exactly $r$.
\end{itemize}
When $\varepsilon$ goes to $0$ and up to consider a subsequence,
the set $\Delta\cup \{(x_0,y_0),\dots,(x_\ell,y_\ell)\}$ converges for the Hausdorff distance
towards a compact connected set $\Delta'$ which is a global holonomy, strictly contains $\Delta$
and satisfies $\pi_1(\Delta')\subset C$. This contradicts the maximality of $\Delta$ and proves the lemma.
\end{proof}
\medskip

The strong stable leaves are preserved under global holonomies along center-unstable plaques.

\begin{addendum}\label{a.extension}
In the case each set $C$ and $\pi_2(\Delta_0)$ is contained in a strong stable leaf,
one can ensure furthermore that $\pi_2(\Delta)$ is also contained in a strong stable leaf.
\end{addendum}
\begin{proof}

We repeat the proof of lemma~\ref{l.extension} requiring furthermore
that the projection $\pi_2(\Delta)$ of the global holonomies are contained in the strong stable leaf $W^{ss}(y_0)$.
Indeed if $\{\Delta_n\}$ is totally ordered family of such global holonomies, then
the closure of the union $\overline{\cup_n\Delta_n}$ projects in $W^{ss}(y_0)$ by $\pi_2$: this is due to the choice of
the center-stable topology.

Let us consider a maximal global holonomy $\Delta$ satisfying
$\pi_2(\Delta)\subset W^{ss}(y_0)$ and given by Zorn's lemma.
Assume by contradiction that $\Delta$ does not satisfies the three items of
lemma~\ref{l.extension}. In particular, it is compact and 
one may fix $(x,y)\in \Delta$ and an $n\geq 1$ such that
$f^n(\pi_2(\Delta))$ is contained in $\cW^{cs}_{f^n(y_0)}$.
One repeats the same construction as above and builds a global holonomy
$\Delta'$ that contains strictly $\Delta$. If $\pi_2(\Delta')$ is contained in $W^{ss}(y_0)$, one has contradicted the maximality of $\Delta$. One will thus assume that the set $f^n(\pi_2(\Delta'))\subset \cW^{cs}_{f^n(y_0)}$ is not contained in a strong stable leaf. Since it is connected, it contains a point $z$
such that both local components of $\cW^{cs}_z\setminus W^{ss}_{loc}(z)$ at $z$
meet $\Pi^{cu}(C)$. If one considers a hyperbolic periodic orbit $O$ homoclinically related to $p$ having a point $q_0$ close to $z$, the local holonomy $\Pi^{cu}$ along the plaques of $\cW^{cu}$ allows to project $f^n(\pi_2(\Delta'))$ on a connected compact subset of
$\cW^{cs}_{q_0}$ which meets $W^{ss}(q_0)$.
Since $W^{ss}(q_0)\setminus\{q_0\}$ is disjoint from $H(p)$,
one deduces that the projection contains $q_0$.
Consequently the unstable manifold of some point $q\in O$ meets $C$ at some point $x$.

By lemma~\ref{l.connex} and since $E^{ss}$ is uniformly contracting,
there exists $\varepsilon>0$ such that any backward iterate
$x_{-n}=f^{-n}(x)$ is contained in a connected compact set $C_{-n}\subset W^{ss}_{loc}(x_{-n})\cap H(p)$
which has a radius equal to $\varepsilon$. Since $x$ belongs to the unstable set of some point $f^k(q)$ in the orbit of $q$,
the backward iterates of $x$ and $q$ become arbitrarily close.
Let $\tau$ be the period of $q$.
One gets that the projection $\Pi^{cu}(C_{-n\tau})$ by holonomy on $\cW^{cs}_q$
converges to a compact connected set contained in $W^{ss}_{loc}(q)$
with diameter equal to $\varepsilon$.
This contradicts our assumption that $W^{ss}(q)\setminus\{q\}$ is disjoint from $H(p)$.
In all the cases we have found a contradiction and the lemma is proved.
\end{proof}

\paragraph{Triple holonomy.} The previous results on holonomies extend to
connected set of triples.

\begin{lema}\label{l.triple}
Let $\Delta$ be a global holonomy along the center-unstable plaques,
$(x_0,y_0)$ be a pair in $\Delta$ and $z_0\in H(p)$ be a point 
which belong to the connected component of $\overline{\cW^{cu}_{x_0}}\cap \overline{\cW^{cu}_{y_0}}$ bounded by $x_0$ and $y_0$.
Then there exists a set $X\subset H(p)\times H(p)\times H(p)$ containing $(x_0,y_0,z_0)$
such that
\begin{itemize}
\item[--] $X$ is closed and connected for the center-stable topology,
\item[--] for each triple $(x,y,z)\in X$ one has $(x,y)\in \Delta$ and
$z\in \overline{\cW^{cu}_{x_0}}\cap \overline{\cW^{cu}_{y_0}}$,
\item[--] one of the two following cases holds:
\begin{enumerate}
\item the set of pairs $(x,y)$ for $(x,y,z)\in X$ coincides with $\Delta$,
\item $X$ is non-compact.
\end{enumerate}
\end{itemize}
Moreover if $\pi_1(\Delta)$ and $\pi_2(\Delta)$ are contained in strong stable leaves, then the same holds for $\pi_3(X)$.
\end{lema}

\begin{proof}
The proof is the same as for lemma~\ref{l.extension} and addendum~\ref{a.extension} but the third case of lemma~\ref{l.extension} has not to be considered since
for all the triples $(x,y,z)\in X$, the point $z$ belongs to the connected component of $\overline{\cW^{cu}_{x}}\cap \overline{\cW^{cu}_{y}}$ bounded by $x$ and $y$
and its distance to $x$ and $z$ is thus controlled.
\end{proof}

\begin{remark}\label{r.triple}
If one projects the set $X$ obtained in lemma~\ref{l.triple} on any
pair of coordinates, for instance as $\pi_{1,3}(X)=\{(x,z),(x,y,z)\in X\}$,
one gets a set which is connected. Hence the closure of $\pi_{1,3}(X)$
for the center-stable topology is a global holonomy.
\end{remark}

\paragraph{Non compact holonomy.}
We now build non bounded holonomies.
\begin{lema}\label{l.unbounded}
If for some $x\in H(p)$ the set $W^{ss}(x)\cap H(p)$ is not totally disconnected,
then there exists a global holonomy $\Delta$ along the center-unstable plaques which is non-compact, non trivial (i.e. there exists $(x_0,y_0)\in \Delta$ such that $x_0\neq y_0$)
and such that both $\pi_1(\Delta)$ and $\pi_2(\Delta)$ are contained in strong stables leaves.
\end{lema}
\begin{proof}
One considers a non trivial compact connected set $\Gamma\subset H(p)$ contained in some strong stable leaf
and the accumulation set $\Lambda$ of the backward iterates $f^{-n}(\Gamma)$ (which is invariant by $f$).
The uniform expansion along $E^{ss}$ and the lemmas~\ref{l.connex} and~\ref{l.converge}
above imply that for any $x_0\in \Lambda$ the strong stable leaf $W^{ss}(x_0)$,
contains a closed connected set $C_0\subset \Lambda$ which is not compact and contains $x_0$.

There exist some points $y_0\in H(p)$ distinct from $x_0$ such that $x_0\in \cW^{cs}_{y_0}$
and $y_0\in \cW^{cs}_{x_0}$ hold.
Indeed, $x_0$ is accumulated by periodic points $q\in H(p)$ whose period is arbitrarily large.
Consequently the sets $\cW^{cs}_q$ are pairwise disjoint. Hence, there exists $q$ close to $x_0$
whose plaque $\cW^{cs}_q$ intersects $\cW^{cu}_{x_0}$ at a point $y_0$ which belongs to
$H(p)\setminus \{x_0\}$ from lemma~\ref{l.bracket0}.

Assuming that the conclusion of the lemma does not hold one builds a sequence of compact holonomies
$(\Delta_n)$ such that $\pi_1(\Delta_n)$ is contained in $\Lambda$, both $\pi_1(\Delta_n), \pi_2(\Delta_n)$ are contained in strong stable leaves, and the diameter of
$\pi_1(\Delta)$ in the strong stable leaf goes to infinity with $n$.
The holonomy $\Delta_0$ is just the initial pair $(x_0,y_0)$.
One constructs $\Delta_{n+1}$ from $\Delta_n$ in the following way.

In the strong stable leaf that contains $\pi_1(\Delta_n)$,
one considers a closed non-compact connected set $C_n\subset \Lambda$.
One then applies lemma~\ref{l.extension} and its addendum~\ref{a.extension}
and finds a global holonomy $\Delta_n'\supset \Delta_n$ such that again $\pi_1(\Delta_n')$ is contained in $C_n$ and both $\pi_1(\Delta_n'), \pi_2(\Delta_n')$ are contained in strong stable leaves. 
By assumption $\Delta_n'$ is compact and in particular $\pi_1(\Delta'_n)$ is strictly contained
inside $C_n$. As a consequence there exists $(x'_n,y'_n)\in \Delta_n'$ such that
$x'_n\in \overline{\cW^{cu}_{y'_n}}\setminus \cW^{cu}_{y'_n}$
or $y_n\in \overline{\cW^{cu}_{x'_n}}\setminus \cW^{cu}_{x'_n}$.
Using the fact that for each $x\in H(p)$ we have
$$f^{-1}(\overline{\cW^{cu}_x})\subset \cW_{f^{-1}(x)}^{cu},$$
the set of images $(f^{-1}(x),f^{-1}(y))$ for $(x,y)\in \Delta'_n$ is still a compact global holonomy:
this is $\Delta_{n+1}$. We also define $(x_{n+1},y_{n+1})=(f^{-1}(x_n),f^{-1}(y_n))$.

By construction $\pi_1(\Delta_1)$ is a non-trivial compact connected set.
Since $E^{ss}$ is uniformly contracted, the projection
$\pi_1(\Delta_n)$, which contains $f^{-n}(\pi_1(\Delta_1))$, has a diameter
(for the distance inside $W^{ss}(x_n)$) which increases exponentially.
This ends the construction of the sequence $(\Delta_n)$.

Up to considering a subsequence, one can assume that
the sequence $(x_n,y_n)$ converges towards a pair $(x,y)\in H(p)\times H(p)$
for the classical topology on $M$. By construction $x_n,y_n$ are at a bounded distance,
hence $x$ and $y$ are distinct.

For each $n$, one endows $W^{ss}(x_n)\times W^{ss}(y_n)$ with the supremum distance
between the intrinsic distances inside $W^{ss}(x_n)$ and $W^{ss}(y_n)$.
Let us fix $D>0$.
By lemma~\ref{l.connex}, for each $n$ large one can find a compact connected set
$\Delta_n^D$ contained in $\Delta_n$ of diameter $D$ and containing $(x_n,y_n)$.
One can assume that the sequence $(\Delta_n^D)$ converges for the Hausdorff topology
towards a compact connected set $\Delta^D\subset W^{ss}(x)\times W^{ss}(y)$.
By lemma~\ref{l.converge}, this set has diameter $D$.
Now the closure of the union of the $\Delta^D$ over $D$ is a global holonomy which is non-compact and
whose projections by $\pi_1,\pi_2$ are both contained in strong stable leaves.
\end{proof}

\paragraph{Unbounded projections of holonomies}
Non-compact holonomies allow to obtain non-compact connected sets inside strong stable leaves.

\begin{lema}\label{l.projection}
Let $\Delta$ be a non-compact holonomy such that $\pi_1(\Delta),\pi_2(\Delta)$ are contained in strong stable leaves. Then the closure of $\pi_1(\Delta)$ for the center-stable topology
is non-compact.
\end{lema}
\begin{proof}
First notice that one can replace $\Delta$ by $f^{-1}(\Delta)$.
By the trapping of the center-unstable plaques this allows to have
$x\in \cW^{cu}_y$ and $y\in \cW^{cu}_x$ for each $(x,y)\in \Delta$
and to work with the plaques of the family $\cW^{cu}$.
The set of pairs $(x,y)\in \Delta$ such that $x=y$ is closed.
By the choice of the central-stable topology it is also open. Hence two cases occurs:
either $x=y$ for each $(x,y)\in \Delta$ and $\pi_1(\Delta)=\pi_2(\Delta)$ is non-compact;
or for each $(x,y)\in \Delta$ one has $x\neq y$ and this is the case one considers now.
For any pair $(x,y)\in \Delta$, we denote by $[x,y]$ the closed segment of $\cW^{cu}_{x}$ bounded by $x,y$.
\medskip

Let us assume by contradiction that the closure of $\pi_1(\Delta)$ is compact.
One can find a finite
collection of points $\{x_j\}\subset \pi_1(\Delta)$ which satisfies that for any $x\in \pi_1(\Delta)$ there exists $x_j$
such that
\begin{itemize}
\item[--] $x$ belongs to $B^{cs}(x_j)$;
\item[--] for any $y,z\in H(p)\cap \cW^{cu}_x$ such that $(x,y)\in\Delta$ and $z\in [x,y]$,
the plaque $\cW^{cu}_{x_j}$ intersects $B^{cs}(z)$.
\end{itemize}

In the following we will consider holonomies $D$ with $\pi_1(D)\subset \pi_1(\Delta)$
and we introduce the set of points $x_{j}$ that are ``avoided'' by $D$:
$$\cP(D)=\{x_j, \;  \forall (x,y)\in D,\; \forall z\in [x,y]\cap H(p), \;  x\notin\cW^{cs}_{x_j}
\text{ or } B^{cs}(z)\cap \cW^{cu}_{x_j}=\emptyset\}.$$

Since the closure of $\pi_1(\Delta)$ is compact and $\Delta$ is not,
one can find $x_i$ with the following property.
\begin{description}
\item[(****)] \it There exists $(x',y'),(x'',y'')\in\Delta$ with $x',x''\in\cW^{cs}_{x_i}$
such that
\begin{itemize}
\item[--] for each $z\in ([x',y']\cup [x'',y''])\cap H(p)$, the plaque $\cW^{cs}_z$ intersects
$\cW^{cu}_{x_i},\cW^{cu}_{x'},\cW^{cu}_{x''}$;
\item[--] $\cW^{cs}_{y'}$ and $\cW^{cs}_{y''}$ intersect $\cW^{cu}_{x_i}$ in two distinct points.
\end{itemize}
\end{description}
Note that in particular the plaques $\cW^{cs}_{y'}$ and $\cW^{cs}_{y''}$ are disjoint.
This allows us to build a compact holonomy $D\subset \Delta$ which ``almost fails'' to be a graph above its first projection. 

\begin{claim}\label{c.holonomy}
There exists a compact holonomy $D$ having the following properties:
\begin{enumerate}
\item\label{i.zero} $\pi_1(D)\subset \pi_1(\Delta)$;
$\pi_2(D)$ is contained in a strong stable leaf;
\item\label{i.one} $D$ is a continuous graph over its first factor;
\item\label{i.two} there is $x_i\in \cP(D)$ satisfying (****).
\end{enumerate}
\end{claim}
\begin{proof}
Let us first notice that since $\Delta$ is non-compact it contains compact holonomies
$\Delta'$ with arbitrarily large diameter by lemma~\ref{l.connex}.
One can thus assume that for such a compact holonomy $\Delta'$, there
exists $x_{i}$ and two pairs $(x',y'),(x'',y'')\in \Delta'$ satisfying (****).
Working with $\varepsilon$-chains as in the proof of lemma~\ref{l.connex}, one can build a compact connected set $D_0\subset \Delta'$ such that~\ref{i.two}) is satisfied for $x_i$. More precizely for any $\varepsilon>0$ one builds a finite set
$X_{\varepsilon}=\{(x(0),y(0)),\dots,(x(s),y(s))\}$ contained in $\Delta'$ such that
\begin{itemize}
\item[--] $(x(k),y(k))$ and $(x(k+1),y(k+1))$ are $\varepsilon$-close for each $0\leq k<s$;
\item[--] the pairs $(x',y')=(x(0),y(0))$ and $(x'',y'')=(x(s),y(s))$ and the point $x_i$
satisfy (****);
\item[--] for any pair $(x(k),y(k))$ with $x(k)\in \cW^{cs}_{x_i}$,
and for any point $z\in [x(k),y(k)]\cap H(p)$ the intersection $B^{cs}(z)\cap \cW^{cu}_{x_i}$ is empty.
\end{itemize}
The compact holonomy $D_0$ is obtained as limit of the sets $X_\varepsilon$.
Repeating the construction with the other points $x_j$, one gets a new compact global holonomy
$D\subset D_0$ such that~\ref{i.one}) is satisfied. Note that~\ref{i.two}) is still satisfied but for a new point $x_i$. Since $D\subset \Delta$, the condition~\ref{i.zero}) holds also.
\end{proof}
\medskip

We now fix a compact holonomy $D$ satisfying the properties~\ref{i.zero}), ~\ref{i.one}) and \ref{i.two}) above.
We do not assume that it is contained in $\Delta$. However we choose it so that the cardinal of $\cP(D)$ is maximal.
\medskip

Let us consider the points $x_i,x',x''$ in property~\ref{i.two}) and (****) and consider the plaques $\cW^{cs}_{x_i},\cW^{cs}_{x'},$ $\cW^{cs}_{x''}$
and the ordering of their intersection on $\cW^{cu}_{x_i}$.
Then $\cW^{cs}_{x_i}$ is not ``in the middle'' of $\cW^{cs}_{x'}$ and $\cW^{cs}_{x''}$.
\begin{claim}\label{c.ordering}
The point $x_i$ does not belong to the connected component
of $\cW^{cu}_{x_i}\setminus (\cW^{cs}_{x'}\cup \cW^{cs}_{x''})$
bounded by $\cW^{cs}_{x'}$ and $\cW^{cs}_{x''}$.
\end{claim}
\begin{proof}
Let us define the compact connected set $C:=\pi_1(D)$.
For each $x\in C$, there exists a unique pair $(x,y)\in D$; moreover
$x\neq y$. One can thus consider
the orientation on $E^{cu}_x$ determined by the component of $\cW^{cu}_x\setminus \{x\}$
which contains $y$. This defines a continuous orientation of the bundle 
$E^{cu}_{|C}$.

One can compare the orientations of $E^{cu}_{x'}$ and $E^{cu}_{x''}$
as transverse spaces to the one-codimensio\-nal plaque $\cW^{cs}_{x_i}$. By the trapping property, for any $k\geq 0$ the forward iterates $f^{k}(x')$
and $f^{k}(x'')$ still belong to the same plaque $\cW^{cs}_{f^k(x_i)}$,
hence the orientations comparison will be the same for $k=0$ or $k$ large.
Since $C$ is a compact subset of a strong stable leaf,
for $k\geq 1$ large $f^k(C)$ is contained in $\cW^{cs}_{f^k(x_i)}$; so for any continuous orientation of $E^{cu}_{|f^k(C)}$, the orientations on $E^{cu}_{f^k(x')}$ and $E^{cu}_{f^k(x'')}$ match.

One deduces that for the orientation on $E^{cu}_{|C}$ considered above,
the orientations on $E^{cu}_{x'}$ and $E^{cu}_{x''}$ match.
By definition of the orientation on $E^{cu}_{|C}$, this implies the claim.
\end{proof}
\medskip

Let $\gamma'=[x',y']$ and $\gamma''=[x'',y'']$.
One now defines a homeomorphism $\varphi\colon\gamma'\cap H(p)\to\gamma''\cap H(p)$.
For $z'\in \gamma'\cap H(p)$, one can use lemma~\ref{l.triple}
and find a closed connected set $X_{z'}\subset H(p)\times H(p)\times H(p)$
containing $(x',y',z')$ and such that for all $(x,y,z)\in X_{z'}$ one has
$z\in \cW^{cu}_x\cap \cW^{cu}_y$ and $(x,y)\in D$.

\begin{claim}
There exists a unique map $\chi\colon D\to H(p)$ which is continuous for the center-stable topology, sends $(x',y')$ on $z'$
and satisfies $\chi(x,y)\in [x,y]$ for each $(x,y)\in D$. Its graph coincides with $X_{z'}$,
which is thus compact.
\end{claim}
\begin{proof}
By remark~\ref{r.triple}, the closure $\widetilde \Delta$
of $\pi_{1,3}(X_{z'})$ is a gobal holonomy satisfying property 1).

Let us assume by contradiction that the projection map
$\pi_{1,2}\colon X_{z'}\to D$ is not injective: in particular
$\widetilde \Delta$ contains two different pairs $(x,z)$ and $(x,\zeta)$,
having the same projection by $\pi_1$. Let us choose $x_j$ such that $x\in B^{cs}(x_j)$
and $\cW^{cu}_{x_j}$ intersects both $B^{cs}(z)$ and $B^{cs}(\zeta)$.
Repeating the argument of the proof of claim~\ref{c.holonomy},
there exists a compact holonomy $\widetilde D\subset \widetilde \Delta$
satisfying the properties~\ref{i.zero}), \ref{i.one}), \ref{i.two}) above
such that $x_j$ belongs to $\cP(\widetilde D)$.
By construction  for each $(x,z)\in \widetilde \Delta$, there exists $(x,y)\in D$ such that
$z$ belongs to $[x,y]$. The definition of the set $\{x_{j}\}$ and the fact that for each $(x,z)\in \widetilde D$ there exists
$(x,y)\in D$ such that $z\in [x,y]$ imply that $\cP(D)\subset \cP(\widetilde D)$.
Since $x_{j}$ belongs to $\cP(\widetilde D)\setminus\cP(D)$, we have contradicted the maximality of $D$.
Hence the map $\pi_{1,2}\colon X_{z'}\to D$ is injective.

Since $D$ is compact, one deduces that $X_{z'}$ is also compact and the first case of lemma~\ref{l.triple} holds.
Consequently, the projection $\pi_{1,2}$ is also surjective $X_{z'}$.
This proves that $X_{z'}$ is the graph of a map $\chi\colon D\to H(p)$. Since $X_{z'}$ is compact, this map is continuous.
The connectedness of $D$ implies that the map $\chi$ is unique.
\end{proof}
\medskip

One deduces that $X_{z'}$ contains a unique triple of the form $(x'',y'',z'')$ and
one sets $\varphi(z')=z''$. The claim implies that $\varphi$ is monotonous
for the ordering on $\gamma',\gamma''$.
One can build similarly a map from $\gamma''$ to $\gamma'$, which is an inverse of $\varphi$.
Consequently $\varphi$ is a homeomorphism which is monotonous for the ordering on $\gamma',\gamma''$.
\medskip

Let $y'_i$ be the intersection between $\cW^{cs}_{y'}$ and $\cW^{cu}_{x_i}$
and $y''_i$ be the intersection between $\cW^{cs}_{y''}$ and $\cW^{cu}_{x_i}$.
Let $\gamma'_i,\gamma_{i}''$ be the segments contained in $\cW^{cu}_{x_i}$ and bounded by $\{x_i,y'_i\}$ and $\{x_i,y''_i\}$ respectively.
One defines two monotonous homeomorphisms
$\psi'\colon \gamma'\cap H(p)\to \gamma_{i}\cap H(p)$ and $\psi''\colon \gamma''\cap H(p)\to \gamma_{i}\cap H(p)$ which send $x'$ and $x''$ on $x_{i}$.
There are obtained by considering local projection throught the center-stable holonomy:
one has $\psi'(z')=z$ when $z\in \cW^{cs}_{z'}$ (and equivalently
when $z'\in \cW^{cs}_{z}$).
One thus obtains a monotonous homeomorphism $\varphi_{i}=\psi'\circ\varphi\circ{\psi''}^{-1}$
from $\gamma_i'\cap H(p)$ to $\gamma_i''\cap H(p)$.

From the claim~\ref{c.ordering} and exchanging $(x',y')$ and $(x'',y'')$ if necessary,
one can assume that $y''_i$ is between $x_i$ and $y'_i$ inside $\cW^{cu}_{x_i}$.
Consequently $\varphi_i$ maps monotonously
$H(p)\cap \gamma_i'$ into itself.
The sequence $z_n=\varphi^n_i(y'_i)$ thus converges to a point $z$
which is fixed by $\varphi_i$ but all the $z_n$ are distinct since by assumption
$z_0=y'_i$ and $z_1=y''_i$ are distinct.

By construction, for each $n$ one associates a compact connected set
$X_n=X_{{\psi'}^{-1}(z_n)}\subset H(p)\times H(p)\times H(p)$ which contains the
triples $(x',y',{\psi'}^{-1}(z_n))$ and $(x'',y'',{\psi''}^{-1}(z_n))$.
Its projection on its third factor is a compact connected set $C_n\subset H(p)$
containing ${\psi'}^{-1}(z_n)$ and ${\psi''}^{-1}(z_n)$ and contained
in a strong stable leaf.
Similarly, let $X=X_{{\psi'}^{-1}(z)}$ and $C$ be its projection on the third factor.
Then, $C_n$ converges towards $C$ for the Hausdorff topology on compact sets of $M$,
whereas the points ${\psi'}^{-1}(z_n), {\psi''}^{-1}(z_{n+1})\in C_n$
converge towards ${\psi'}^{-1}(z), {\psi''}^{-1}(z)\in C$.

Since $z$ is fixed by $\varphi_i$, the center-stable plaques
of the points ${\psi'}^{-1}(z), {\psi''}^{-1}(z)$ intersect,
whereas since $z_n,z_{n+1}$ are distinct, the center-stable plaques
of the points ${\psi'}^{-1}(z_n), {\psi''}^{-1}(z_{n+1})$ are disjoint.
Thus the intrinsic distances between ${\psi'}^{-1}(z), {\psi''}^{-1}(z)$
and ${\psi'}^{-1}(z_n), {\psi''}^{-1}(z_{n+1})$ are bounded away,
contradicting lemma~\ref{l.converge}.
The proof of lemma~\ref{l.projection} is now complete.
\end{proof}
\bigskip

We now finish the proof of the proposition.
\begin{proof}[\bf Proof of proposition~\ref{ss-tot-discontinuity}]
Let us assume by contradiction that for some point $x\in H(p)$
the set $H(p)\cap W^{ss}(x)$ is not totally disconnected.
We will build a periodic point $q\in H(p)$, a point $z_0\in W^s(q)\cap H(p)$
and a set $C\subset W^{ss}(z_0)$ which is closed connected and non-compact
for the intrinsic topology on $W^{ss}(z_0)$.
In the stable manifold of the orbit of $q$,
the iterates $f^n(C)$ accumulate a non-trivial subset of $W^{ss}(q)$,
contradicting the assumption that $W^{ss}(q)\cap H(p)= \{q\}$.

In order to build $q$ and $C$, we apply lemma~\ref{l.unbounded}
and consider a non-compact holonomy $\Delta$
and a pair $(x_0,y_0)\in \Delta$ such that $x_0\neq y_0$.
The sets $\pi_1(\Delta),\pi_2(\Delta)$ are contained in strong stable leaves
and by lemma~\ref{l.projection} their closures in the leaves are not compact.
Let us remind that $\cW^{cu}_{x_0}$ is a one-dimensional curve
and consider the open connected subset $U\subset\cW^{cu}_{x_0}$
bounded by $\{x_0,y_0\}$.
Two cases have to be studied.

If $H(p)$ does not meet the set $U$,
then $x_0$ is an unstable boundary point of the
chain-hyperbolic class $H(p)$ (see definition~\ref{d.boundary}).
By lemma~\ref{p.boundary0},
there exists a periodic point $q$ in $H(p)$
whose stable set contains $\pi_1(\Delta)$.
We define $z_0=x_0$ and the set $C$ as the closure of $\pi_1(\Delta)$ in $W^s(q)$,
finishing the proof in this case.

Let us assume now that there exists a point $\zeta\in U\cap H(p)$.
We introduce a hyperbolic periodic point $q$ homoclinically related to $p$
and close to $\zeta$ such that $\cW^{cs}_q\subset W^{s}(q)$ as given by lemma~\ref{l.contper}.
The plaques $\cW^{cs}_q$ and $\cW^{cu}_{x_0}$ intersect at a point $z_0\in U\cap H(p)$.
By lemma~\ref{l.triple}, there is a closed connected set $X\subset H(p)\times H(p)\times H(p)$
which contains $(x_0,y_0,z_0)$, such that for each $(x,y,z)\in X$ one has
$z\in\cW^{cu}_x\cap \cW^{cu}_y$ and $(x,y)\in \Delta$.
Moreover the projection $\pi_3(X)$ is contained in a strong stable leaf of $W^s(q)$
and $X$ is non-compact. We want to show that the closure of $\pi_3(X)$ in
$W^s(q)$ is non-compact.

We know that the closure of
one of the three projections $\pi_1(X),\pi_2(X),\pi_3(X)$ is non-compact.
If for instance this happens for $\pi_1(X)$, the closure of
$\pi_{1,3}(X)$ is a non-compact holonomy by remark~\ref{r.triple}.
Hence by lemma~\ref{l.projection}, the closure of $\pi_3(X)$ is non-compact also.
One concludes that in any case the closure $C$ of
$\pi_3(X)$ is non-compact: we have found a non-compact connected
set contained in $H(p)\cap W^{ss}(z_0)$ as claimed, concluding the proof of the proposition in the second case.
\end{proof}

\subsection{Structure in the center-stable leaves: proof of theorem~\ref{t.tot-discontinuity}}

By the trapping property, the iterates of each plaque $\cW^{cs}_x$, $x\in H(p)$,
remain in a small neighborhood of $H(p)$, hence is covered by a strong stable foliation.
We call \emph{strong stable plaques} the connected components of the strong stables leaves of
$\cW^{cs}_x$.

\begin{lema}\label{l.graph}
For any $x\in H(p)$,
let us consider a connected compact set $\Gamma\subset H(p)\cap \cW^{cs}_x$.
Then $\Gamma$ intersects each strong stable plaque of $\cW^{cs}_x$ in at most one point.
In particular this is a curve.
\end{lema}
\begin{proof}
Let us assume by contradiction that $\Gamma$ intersects some strong stable leaf $L$ of
$\cW^{cs}_x$ in at least two distinct points $z,z'$.
Let us consider two small closed neighborhoods $U$ and $U'$ of $z,z'$ in $\cW^{cs}_x$,
such that that $U\setminus L$ and $U'\setminus L$ have two connected components.

We introduce the connected components $\Gamma_z,\Gamma_{z'}$ of $\Gamma\cap U$ and $\Gamma\cap U'$
containing $z$ and $z'$ respectively.
These two sets are not reduced to $z$ and $z'$ and, by proposition~\ref{ss-tot-discontinuity}
$\Gamma_{z}\cap L$ and $\Gamma_{z'}\cap L$ are totally disconnected.
In one of the connected components $V$ of $U\setminus L$, all the strong stable plaques
close to $z$ are met by $\Gamma_z$. The same holds for $\Gamma_{z'}$ and a component
$V'$ of $U'\setminus L$.

We claim that one can reduce to the case both components $V,V'$ are on the same side
of $L$. Indeed if this is not the case, the connected set $\Gamma$ intersects $L$
at another point $z''$. One can thus define three sets $V,V',V''$; among them, two
are on the same side of $L$.

Let $\tilde L$ be a strong stable plaque close to $z$ and $z'$ which intersects
$V$ and $V'$: all the plaques close to $\tilde L$ meet both $\Gamma_z$ and $\Gamma_{z'}$.

Let $q$ be a periodic point homoclinically related to $p$
and close to a point in $\Gamma_{z'}\cap \tilde L$.
The local strong stable manifold $W^{ss}_{loc}(q)$ is close to $\tilde L$
and the projection of $\Gamma_{z}$ by the center-unstable holonomy
on $\cW^{cs}_q$ is a connected compact set that intersects both sides
of $W^{ss}_{loc}(q)$. One deduces that this projection meets $\Gamma_{z}$
at a point $y\in H(p)\cap W^{ss}(q)$ which is distinct from $q$.
This contradicts our assumption.
\end{proof}
\medskip

Let us call~\emph{graph} of a plaque $\cW^{cs}_x$ a connected compact set of $\cW^{cs}$
which intersects each strong stable leaf of $\cW^{cs}_x$ in at most one point.

\begin{lema}\label{l.graph-uniform}
If for some point $x_0\in H(p)$,
the set $\cW^{cs}_{x_0}\cap H(p)$ is not totally disconnected,
then for each $x\in H(p)$, there exists a graph $\Gamma_x\subset \cW^{cs}_x\cap H(p)$
containing $x$ which meets all the strong stable plaques of $\cW^{cs}_x$ that intersect a small neighborhood of $x$.
\end{lema}
\begin{proof}
Let us consider a non trivial connected compact set $\Gamma\subset \cW^{cs}_{x_0}$.
By lemma~\ref{l.graph} this is a graph.
Let us consider a point $z\in \Gamma$ which is not an endpoint.
One also chooses a trapped plaque family $\cD$ above $H(p)$ tangent to $E^{cs}$
whose plaques have a small diameter and are contained in the plaques of $\cW^{cs}$.
Consequently the connected component $\Gamma_z$
of $\Gamma\cap \cD_z$ contains $z$ and has its endpoints inside $\overline{\cD_z}\setminus\cD_z$. We are aimed to build at each point $x\in H(p)$ a similar graph $\Gamma_x\subset \cD_x$.
This will imply the conclusion of the lemma.
Let us first choose a periodic point $q$ homoclinically related to $p$ and close to $z$.
By projecting $\Gamma$ inside $\cW^{cs}_{q}$ along the center-unstable holonomy,
one deduces that $\cD_q$ contains a graph $\Gamma_q\subset H(p)$
whose endpoints are inside $\overline{\cD_q}\setminus \cD_q$.
It contains a point close to $q$. Since $W^{ss}_{loc}(q)\setminus \{q\}$
is disjoint from $q$, this proves that $\Gamma_q$ contains $q$.

By the trapping property, for each $n\geq 0$, the connected component $\Gamma_{f^{-n}(q)}$
of $f^{-n}(\Gamma_q)\cap \cD_{f^{-n}(q)}$ has also
its endpoints inside the boundary of $\cD_{f^{-n}(q)}$.
As a consequence $H(p)$ contains a dense set of periodic points $y$ and inside each plaque
$\cD_y$ there exists a graph $\Gamma_y$ containing $y$ whose endpoints belong to
$\overline{\cD_y}\setminus \cD_y$.

For any point $x\in H(p)$ there exists a sequence of periodic points $(y_n)$
converging towards $x$ such that the sequence of graphs $(\overline{\Gamma_{y_n}})$ converges
towards a connected compact set $\overline{\Gamma_x}$: by lemma~\ref{l.graph} this is a graph
and by construction its endpoints belong to $\overline{\cD_y}\setminus \cD_y$
as required.
\end{proof}
\bigskip

We are now able to finish the proof of the theorem.
\begin{proof}[\bf Proof of theorem~\ref{t.tot-discontinuity}]
We assume that the conclusion of the theorem does not holds: in particular, the lemma~\ref{l.graph-uniform} applies.
By theorem~\ref{t.whitney}, there exists two distinct points $x,y\in H(p)$ with $y\in W^{ss}(x)$.
By iterations one may assume that $y$ belongs to the strong stable plaque of $x$
in $\cW^{cs}_x$. By lemma~\ref{l.graph-uniform}, there exists a graph
$\Gamma_x\subset \cW^{cs}_x$ which contains $x$ and meets all the strong stable
plaques of points close to $x$ in $\cW^{cs}_x$.
One now argues as at the end of the proof of lemma~\ref{l.graph}:
if $q$ is a periodic point close to $y$, the projection of $\Gamma_x$
to $\cW^{cs}_q$ has to intersect $W^{ss}_{loc}(q)$ at a point close to $x$, hence different from $q$.
This contradicts the assumptions.
\end{proof}

\subsection{Construction of adapted plaques}

We now give a consequence of theorem~\ref{t.tot-discontinuity} giving plaques
adapted to the geometry of the classes along the center-stable plaques.

Let us consider an invariant compact set $K$ with a dominated splitting $E\oplus F$
and a trapped family tangent to $E$ such that the coherence holds for some constant $10\; \varepsilon>0$
(see lemma~\ref{l.uniqueness-coherence}).
Let $\widetilde \cW$ be another trapped family tangent to $E$ whose plaques have a small diameter
and such that for each $x\in K$ one has $\widetilde \cW_x\subset \cW_x$.
The coherence ensures that any plaque $\widetilde \cW_y$ that intersects the $5\varepsilon$-ball centered at $x$
inside $\cW_x$ is contained in $\cW_x$.

\begin{definition}
In this setting, a set $X\subset K$ that is contained in the $\varepsilon$-ball centered at a point $x\in K$ inside the plaque $\cW_x$
is said to be \emph{$\widetilde \cW$-connected} if the union of the plaques $\cW^{cs}_y$ for $y\in X$
is connected.
\end{definition}
\medskip

When the diameters of the plaques $\widetilde \cW^{cs}$ are small, the $\widetilde \cW^{cs}$-connected sets
have a small diameter.

\begin{prop}\label{p.box}
Let $f_0$ be a diffeomorphism, $H(p_{f_0})$ be a chain-recurrence class which is chain-hyperbolic
such that the bundles $E^{cs},E^{cu}$ are thin trapped and consider some neighborhoods $U$ of $H(p_{f_0})$, $\cU$ of $f_0$ in $\diff^1(M)$ and a plaque family $(\cW_{f,x}^{cs})_{f\in \cU,x\in K_f}$ as provided by lemma~\ref{l.robustness}.

If for each $x\in H(p_{f_0})$, the set $H(p_{f_0})\cap \cW^{cs}_{x}$ is totally disconnected, then
for any $\eta>0$ small,
there exist smaller neighborhoods $\widetilde U\subset U$ of $H(p_{f_0})$ and $\widetilde \cU$ of $f_0$
and there are other plaque families $(\widetilde \cW_{f,x}^{cs})_{f\in \widetilde \cU,x\in \widetilde K_f}$
defined on the maximal invariant sets $\widetilde K_f$ in the closure of $\widetilde U$ for $f$,
satisfying the following properties for each $f\in \widetilde \cU$ and $x\in \widetilde K_f$:
\begin{itemize}
\item[--] The plaque $\widetilde \cW^{cs}_{f,x}$ is contained in $\cW^{cs}_{f,x}$.
\item[--] Any $\widetilde \cW^{cs}_f$-connected set of $K_f\cap\cW^{cs}_{f,x}$ which contains $x$
has diameter smaller than $\eta$.
\end{itemize}
\end{prop}
\begin{proof}
One considers a constant $\varepsilon>0$, two neighborhoods $U_\varepsilon$ of $H(p_{f_0})$
and $\cU_\varepsilon$ of $f_0$ which decrease to $H(p_{f_0})$ and $f_{0}$ as $\varepsilon$ goes to zero,
and a continuous collection of plaque family
$(\cW^{cs}_{\varepsilon,f})_{f\in \cU_\varepsilon}$
defined on the maximal invariant set $K_{\varepsilon,f}$ in the closure of $U_\varepsilon$.
We assume that these families are trapped, that each plaque
$\cW^{cs}_{\varepsilon,f,x}$ has diameter smaller than $\varepsilon$
and that for each $x\in K_{\varepsilon,f}$, the plaque $\cW^{cs}_{\varepsilon,f,x}$ is contained
$\cW^{cs}_{\varepsilon,f}$. Such plaque families are given by remark~\ref{r.nested}.

For $f\in \cU_{\varepsilon}$,
one makes the union $\Delta_f$ of the sets $\overline{\cW^{cs}_{\varepsilon,f,x}}$.
We claim that when $\varepsilon$ goes to zero, the supremum of the diameter of the connected components
of $\Delta_f$ (with respect to the center-stable topology) goes to zero.
Indeed, if this is not the case, one finds as limit set a non-trivial connected component of $H(p_{f_0})$
for $f_0$ and the center-stable topology, which contradicts our assumption.
The plaque family $(\widetilde \cW^{cs}_{f})$ is thus chosen to be $(\cW^{cs}_{\varepsilon,f})$
for some $\varepsilon$ small enough.
\end{proof}

%% file: codimension-1110.tex
\section{Uniform hyperbolicity of the extremal bundles: proof of theorem~\ref{t.2D-central}}\label{s.2D-central2}
In this section we end the proof of theorem~\ref{t.2D-central}.
We consider:
\begin{enumerate}
\item\label{a1} a diffeomorphism $f_0$ and a chain-hyperbolic homoclinic class $H(p_{f_0})$ which is a
chain-recurrence class endowed with a dominated splitting
$E^{cs}\oplus E^{cu}$ such that:
\begin{itemize}
\item[1a.] $E^{cu}$ is one-dimensional and $E^{cs},E^{cu}$ are thin trapped by $f$ and $f^{-1}$ respectively.
\item[1b.] The intersection of $H(p)$ with the center-stable plaques is totally disconnected.
\end{itemize}
\item\label{a2} a $C^2$-diffeomorphism $f$ that is $C^1$-close to $f_0$,
\item\label{a3} a chain-recurrence class $K$ for $f$ contained in a small neighborhood of $H(p_{f_0})$ such that:
\begin{itemize}
\item[3a.] All the periodic points of $K$ are hyperbolic.
\item[3b.] $K$ does not contain a sink, nor a closed curve $\gamma$ tangent to $E^{cu}$,
invariant by some iterate $f^n$, $n\geq 1$,
such that $f^n_{|\gamma}$ is conjugated to an irrational rotation.
\end{itemize}
\item\label{a4} a transitive invariant compact set $\Lambda\subset K$ for $f$ such that
the bundle $E^{cu}$ is uniformly expanded on any proper invariant compact subset of $\Lambda$.
\end{enumerate}
We prove here the following proposition.

\begin{propo}\label{propoE^{cu}ishyp}
Let us consider some diffeomorphisms $f_0$, $f$, some chain-recurrence classes
$H(p_{f_0})$, $K$ and a subset $\Lambda\subset K$ satisfying the assumptions~\ref{a1})-\ref{a4}) above. Then the bundle $E^{cu}$ is uniformly expanded on any proper invariant compact subset of $\Lambda$.
\end{propo}

Let us explain how to conclude the proof of the theorem~\ref{t.2D-central}.

\paragraph{Proof of theorem~\ref{t.2D-central}}
Under the hypothesis of the theorem, the assumptions~\ref{a1}) and~\ref{a2}) above are clearly satisfied.
Note that since $K$ is contained in a small neighborhood of $H(p_{f_0})$, the same holds for any
chain recurrence class $K'$ which meets $K$.
If for any such chain-recurrence class $K'$, the bundle $E^{cu}$ is uniformly expanded,
the same holds for $K$, hence the conclusion of the theorem holds.
Note that if $K'$ contains a curve $\gamma$ tangent to $E^{cu}$ such that
$f^n$ preserves $\gamma$ and is conjugated to an irrational rotation for some $n\geq 1$,
then from the domination $E^{cs}$ is uniformly contracted on the union $X$ of the iterates of $\gamma$
and consequently $X$ is an attractor. Since $K'$ is chain-transitive,
$K'$ coincides with $X$ and is contained in $K$; this gives theorem~\ref{t.2D-central} in this case.
The same holds if $K'$ contains a sink.
We will now assume by contradiction that the conclusion of the theorem does not hold and
hence that $K'$ satisfies~\ref{a3}) and that the bundle $E^{cu}$ is not uniformly expanded by $f$ on $K'$.

One can then consider an invariant compact set $\Lambda\subset K'$ whose bundle $E^{cu}$ is not uniformly expanded and that is minimal for this property. Such a set exists by Zorn's lemma
since if $\{\Lambda_\alpha\}_{\alpha\in \cA}$ is a family of invariant compact sets totally ordered by the inclusion and if $E^{cu}$ is uniformly expanded on the intersection $\cap_{\alpha\in \cA}\Lambda_\alpha$, then the same holds on the $\Lambda_{\alpha}$ for $\alpha$ large enough.
By minimality, for any proper invariant compact set of $\Lambda$, the bundle $E^{cu}$ is uniformly expanded.

Since $E^{cu}$ is one-dimensional and not uniformly expanded on $\Lambda$, there exists
an invariant measure $\mu$ supported on $\Lambda$ and whose Lyapunov exponent along $E^{cu}$
is non-positive. One can assume that $\mu$ is ergodic and by minimality of $\Lambda$
its support coincides with $\Lambda$. This implies that $\Lambda$ is transitive and
satisfies~\ref{a4}).

By applying proposition~\ref{propoE^{cu}ishyp} to $f,\Lambda,K'$, the bundle $E^{cu}$
is uniformly expanded on $\Lambda$ which is a contradiction.
Consequently the conclusion of theorem~\ref{t.2D-central} holds.
\qed

\bigskip

In the following we are in the setting of proposition~\ref{propoE^{cu}ishyp}
and prove that $E^{cu}$ is uniformly expanded on $\Lambda$. The proof follows the strategy
of~\cite{PS1} (see also~\cite{PS3,PS4,Pu1} for more general contexts). The new difficulty is to work with a
non-uniformly contracted bundle $E^{cs}$ having dimension larger than $1$;
the summability arguments and the construction of Markovian rectangles become more delicate.

\paragraph{Strategy.}
Our goal is to find a non-empty open set $B$ of $\Lambda$
which satisfy:
\begin{itemize}
\item[(E)] \emph{For any $x\in B$ and $n\geq 1$ such that $f^{-n}(x)\in B$
we have $\|Df^{-n}_{|E^{cu}}(x)\|<\frac 1 2$.}
\end{itemize}
This concludes the proof of the proposition~\ref{propoE^{cu}ishyp}.
Indeed if one considers any point $x\in \Lambda$, then:
\begin{itemize}
\item[--] either its backward orbit intersects $B$ and property (E) applies,
\item[--] or the $\alpha$-limit set of $x$ is a proper invariant compact subset of $\Lambda$
whose bundle $E^{cu}$ is uniformly contracted by $f^{-1}$.
\end{itemize}
In both cases, the point $x$ has a backward iterate $f^{-n}(x)$ such that
$\|Df^{-n}_{|E^{cu}}(x)\|<1$.
By compactness one deduces that there is some $k\geq 1$ such that
for any $x\in \Lambda$ the derivative $\|Df^{-k}_{|E^{cu}}(x)\|$ is smaller than $1/2$,
concluding the proof that $E^{cu}$ is uniformly expanded on $\Lambda$.
\medskip

\subsection{Topological hyperbolicity on $\Lambda$}\label{ss.topo}
We begin with preliminary constructions and recall some results from~\cite{PS1}
which only use the one-codimensional domination $E^{cs}\oplus E^{cu}$
and the fact that $f$ is $C^2$.
We introduce (in this order) the following objects satisfying several properties stated
in this section:
\begin{itemize}
\item[--] some constants $\kappa,\lambda,\mu,\chi$ related to the domination,
\item[--] two transverse cone fields $\cC^{cs},\cC^{cu}$ on a neighborhood of $H(p_{f_0})$: they are thin neighborhoods
of the bundles $E^{cs}$ and $E^{cu}$ over $H(p_{f_0})$ and they are invariant by $f^{-1}_0$ and $f_0$ respectively.
\item[--] two continuous trapped $C^1$-plaques families
$(\cW^{cs}_f)$, $(\cW^{cu}_f)$ provided by the lemma~\ref{l.robustness},
defined for diffeomorphisms $f$
that are $C^1$-close to $f_0$ and tangent to
the bundles $(E^{cs}_f)$ and $(E^{cu}_f)$ over the maximal invariant sets
in a small neighborhood of $H(p_{f_0})$: the plaques are small and
tangent to $\cC^{cs},\cC^{cu}$.
\item[--] some constants $\varepsilon,\widetilde \varepsilon$
which control the geometry of the center-stable plaques under backward iterations,
their coherence and their intersections,
\item[--] some small neighbourhood $U$ of $H(p_{f_0})$: 
for any diffeomorphism $f$ we then denotes $K_f$ the maximal invariant set of $f$
in $U$.
\item[--] a continuous family of trapped $C^1$-plaques $(\widehat \cW^{cs}_f)$
tangent to $E^{cs}$ over the maximal invariant set in a small neighborhood of $H(p_{f_0})$:
they have a small diameter so that $\widehat \cW^{cs}_x$ is contained in $U$
for each $x\in K$; moreover for each $x\in K$, the plaque
$\widehat \cW^{cs}_x$ is contained in $\cW^{cs}_x$.
This family is obtained by remark~\ref{r.nested}.
It will be used in order to define holes at section~\ref{ss.rectangle}.
\item[--] a scale $\rho$ smaller than the diameter of the plaques $\widehat \cW^{cs}$
and which control the size of Markovian rectangles,
\item[--] a $C^2$-diffeomorphism $f$, a chain-recurrence class $K$ and a chain-transitive set $\Lambda$
satisfying the conditions of the proposition~\ref{propoE^{cu}ishyp}:
the $C^1$-distance between $f$ and $f_0$
and the size of the neighborhood of $H(p_{f_0})$ containing $K$
are chosen small enough in order to satisfy further conditions
that will appear in section~\ref{ss.construction}.
\item[--] a scale $r>0$ which depends on the $C^2$-diffeomorphism $f$ and on the set $\Lambda$, 
where the plaques $\cW^{cu}$ are nicely controled.
\end{itemize}

Now we list a series of properties that are used (and refered to) in the proof of proposition \ref{propoE^{cu}ishyp}.

\paragraph{a) Dominated splitting.}
We first state some consequences of the domination $E^{cs}\oplus E^{cu}$.
To simplify the presentation, one can change the Riemannian metric (see~\cite{gourmelon}) and
find $\kappa\in (0,1)$ such that for each point $x\in H(p_{f_0})$,
and each unitary vectors $u\in E^{cs}_x$ and $v\in E^{cu}_x$, one has
$\|Df_0.u\|\leq \kappa \|Df_0.v\|$.
One then chooses some $\lambda,\mu\in (0,1)$ such that
$\lambda\mu>\kappa$. This implies that:
\begin{itemize}
\item[] \emph{For any $x\in K_f$ one has
$$\|D{f}_{|E^{cs}}(x)\|\geq \lambda \;\; \Rightarrow \;\; \|D{f}_{|E^{cu}}(x)\|> \mu^{-1}.$$}
\end{itemize}
\medskip

Since $E^{cu}$ is not uniformly expanded on $\Lambda$, there exists $\zeta\in \Lambda$
such that $\|Df^n_{|E^{cu}}(\zeta)\|\leq 1$ for all $n\geq 1$.
Note that since $E^{cu}$ is uniformly expanded on any invariant compact subset, the forward orbit of $\zeta$ is dense in $\Lambda$.
With the domination $E^{cs}\oplus E^{cu}$ one deduces:
\begin{itemize}
\item[(i)] \emph{There exists a point $\zeta$ with dense forward orbit in $\Lambda$
such that for each $n\geq 1$ one has
$$\prod_{i=0}^{n-1}\|Df_{|E^{cs}}(f^i(\zeta))\|\leq \kappa^n.$$}
\end{itemize}
\medskip

We fix some small constant $\chi>0$ such that $(1+\chi)\kappa<\lambda$.
Choosing $\cC^{cs}$ thin enough one gets:
\begin{itemize}
\item[(ii)] \emph{For any points $x,y$ that are close and contained in a small neighborhood of $H(p_{f_0})$ and for any unitary vector $u\in \cC^{cs}_x$, one has
$$\|Df_x.u\|\leq (1+\chi)\; \sup\left\{\|Df_y.v\|,\; v\in \cC^{cs}_y, \|v\|=1\right\}.$$}
\end{itemize}

\paragraph{b) Center stable and unstable plaques.}
Assuming that the plaques are small and the cones thin, one deduces from
our choice of $\lambda,\mu$:
\begin{itemize}
\item[(iii)] \emph{If for some point $x\in K_f$ and any $n\geq 0$ one has
$$\prod_{i=0}^{n-1}\|Df_{|E^{cs}}(f^{i}(x)\|\leq \lambda^n,$$
then $\cW^{cs}_x$ is contained in the stable set of $x$, i.e.
the diameter of $f^n(\cW^{cs}_x)$ goes to $0$ as $n\to +\infty$.}
\item[(iv)] \emph{If for some point $x\in K_f$ and some $n\geq 0$ one has
$$\prod_{i=0}^{n-1}\|Df_{|E^{cs}}(f^{i}(x)\|\geq \lambda^n,$$
then the norm of the derivative of $f^{-n}$ along
the plaque $\cW^{cu}_{f^n(x)}$ is smaller than $\mu^n$.}
\end{itemize}
\medskip

The center-stable discs do not degenerate under
backward iterations: let us fix $\varepsilon>0$ small; then
there is $\widetilde \varepsilon>0$ small such that
choosing $f$ close to $f_0$ and $U$ small the following holds.
\begin{itemize}
\item[(v)]\emph{Consider any segment of orbit $(z,\dots,f^{n}(z))$ in $U$ and any disc $D$ tangent to $\cC^{cs}$,
containing a ball centered at $f^n(z)$ of radius $\widetilde \varepsilon$.
Then the preimage $f^{-n}(D)$ contains a ball $B$ centered at $z$ and of radius $\varepsilon$,
whose iterates $f^i(B)$, $i\in \{0,\dots,n\}$, have radius bounded by $\widetilde \varepsilon$.}
\end{itemize}
Indeed each point $f_0^{i}(z)$ is close to a point $x_i\in H(p_{f_0})$.
Each disc $D$ in the plaque $\cW^{cs}_{f_0^n(z)}$ at $f_0^n(z)$ can be viewed as the graph of a Lipschitz map above a domain $\Delta_n$
of $\cW^{cs}_{x_n}$. The invariance of the cones $\cC^{cs},\cC^{cu}$
and the fact that the bundle $E^{cs}$ is thin trapped shows that
$f_0^{-k}(D)$, for $k\in \{0,\dots,n\}$ contains the graph of a Lipschitz map
above a domain $\Delta_{n-k}$ of $\cW^{cs}_{x_{n-k}}$ whose radius is uniformly bounded
from below. The property extends to any diffeomorphism $f$ that is $C^1$-close.
\medskip

The coherence of the plaques (lemma~\ref{l.uniqueness-coherence}) gives:
\begin{itemize}
\item[(vi)] \emph{For any points $x,y\in K_f$ that are $\varepsilon$-close,
if $\cW^{cs}_x\cap \cW^{cs}_y\neq \emptyset$ then $f(\cW^{cs}_{y})\subset \cW^{cs}_{f(x)}$.
If $\cW^{cu}_x\cap \cW^{cu}_y\neq \emptyset$ then $f^{-1}(\cW^{cu}_{y})\subset \cW^{cu}_{f^{-1}(x)}$.}
\end{itemize}
\medskip

The holonomy along the center-stable plaques can be chosen to ``preserve the order":
\begin{itemize}
\item[(vii)] \emph{For any points $x,y\in K_f$ that are $\varepsilon$-close, the plaques $\cW^{cs}_{x}$ and
$\cW^{cu}_{y}$ intersect in a unique point.}
\item[(viii)] \emph{For any points $x^-,x^+,y,z\in K_f$ that are $\varepsilon$-close,
assume that $y$ belongs to a subinterval of $\cW^{cu}_{y}$ bounded by $x^-,x^+$
and denote $\tilde x^-,\tilde x^+,\tilde y$ the intersections of
the plaques $\cW^{cs}_{x^-}, \cW^{cs}_{x^+}, \cW^{cs}_{y}$ with $\cW^{cu}_z$.
Then $\tilde y$ belongs to the subinterval of $\cW^{cu}_z$ bounded by $\tilde x^-,\tilde x^+$.}

(This is a consequence of the coherence of the $\cW^{cs}$-plaques given by the property (vi).)
\end{itemize}
\medskip

\paragraph{c) Smoothness and stability of the center-unstable plaques.} We now use the following result which is based on a Denjoy argument.
(The proof in~\cite{PS1} is written for surface diffeomorphisms but as it is noticed in~\cite{PS3}
this does not make any difference.)

\begin{lemma}[\cite{PS1}, lemma 3.3.2, item1)]\label{l.smoothness}
Let $f$ be a $C^2$-diffeomorphism and $K$ be an invariant compact set endowed with a dominated
splitting $E^{cs}\oplus E^{cu}$ such that $E^{cu}$ is one-dimensional, $K$ does not contain sinks and all its periodic points hyperbolic. Then, there exists
a locally invariant plaque family $\gamma$ tangent to $E^{cu}$ such that
\begin{itemize}
\item[--] the maps $\gamma_x\colon E^{cu}_x\to M$, $x\in K$,
define a continuous family of $C^2$-embeddings;
\item[--] for any $r_0>0$, there exists $r_1>0$ such that for any $x\in K$ and $n\geq 0$
the image of the curve $\gamma_{x,r_1}:=\gamma_x(B(0,r_1))$ by $f^{-n}$
is contained in $\gamma_{f^{-n},r_0}$.
\end{itemize}
\end{lemma}
For the $C^2$-diffeomorphism $f$ and the chain-recurrence class $K$ one deduces that
the plaques $\cW^{cu}$ are $C^2$ in a neighborhood of the section $0\in E^{cu}$
which remains small by backward iterations. Indeed, the
coherence (lemma~\ref{l.uniqueness-coherence}), gives $r>0$ such that
$\cW^{cu}_x(B(0,r))$ is contained in $\gamma_{x}$ for any $x\in K$.

\paragraph{d) Topological expansion along the center-unstable plaques.}
The following result, whose proof is identical to the surface case~\cite{PS1}, asserts that the center-unstable curves
$\gamma$ in the center-unstable direction are unstable manifolds.

\begin{lemma}[\cite{PS1}, lemma 3.5.2]\label{l.top-expansion}
Under the setting of lemma~\ref{l.smoothness}, for any transitive invariant compact set
$\Lambda\subset K$ such that on any proper invariant compact sets the bundle $E^{cu}$ is uniformly expanded, there exists $r>0$ such that
\begin{itemize}
\item[] for any $x\in \Lambda$, the length of
$f^{-n}(\gamma_{x,r})$ decreases uniformly to $0$ as $n\to+\infty$.
\end{itemize}
\end{lemma}
\medskip

In the following we fix $r>0$ small and depending on $\Lambda$,
as given by the previous lemma, and we denote by
$W^{cu}_{loc}(x)$ the $C^2$-curve $\gamma_{x,r}$ for $x\in K$.
By lemma~\ref{l.smoothness}, the family of unstable curves $(W^{cu}_{loc}(x))_{x\in K}$
is continous for the $C^2$ topology.
For points $x\in \Lambda$ we sometimes write $W^{u}_{loc}(x)=W^{cu}_{loc}(x)$.

\subsection{Adapted rectangles}\label{ss.rectangle}
\paragraph{a) Rectangles.}
The set $B$ in condition (E) will be obtained from a geometry adapted to the splitting $E^{cs}\oplus E^{cu}$.
A rectangle\footnote{The name refers to the rectangles of Markov partitions.
For general hyperbolic sets $K$ the rectangles are subsets of $K$ but on surfaces one can also consider
geometrical Markov partitions~\cite[Appendix 2]{PT} whose rectangles are subsets of the surface diffeomorphic to $[0,1]^2$.
In higher dimensions, when the unstable bundle is one-dimensional,
one can build rectangles that are laminated by curves as in definition~\ref{d.rectangle}.}
of $\Lambda$ will be a union of local unstable leaves of points of $K$.

\begin{defi}\label{d.rectangle} A \emph{rectangle} $R$ is a union $\bigcup_{x\in X}\gamma_x$ with $X\subset K$
such that for each $x\in X$ the set $\gamma_x$ is an open interval of $W^{cu}_{loc}(x)$ bounded by two distinct points
$x^-_R,x^+_R$ in $K_f$ and such that the following properties hold:
\begin{enumerate}
\item\label{rectangle1} $R$ has diameter smaller than $\rho$,
\item $R\cap \Lambda$ is open in $\Lambda$,
\item\label{rectangle3} for any $x,y\in X$, the point $y^-_R$ belongs to $\cW^{cs}_{x^-_R}$
and the point $y^+_R$ belongs to $\cW^{cs}_{x^+_R}$.
\end{enumerate}
The sets $\{x^-_R,\; x\in X\}$ and $\{x^+_R,\; x\in X\}$ are called the \emph{boundaries} of $R$.
\medskip

By item~\ref{rectangle3}) and the property (vi),
any two curves $\gamma_x,\gamma_{x'}$ with $x,x'\in X$ are either disjoint or coincide.
For any $z\in X$ or $z\in R\cap\Lambda$, one can thus denote by
$W^{cu}_R(z)$ the curve $\gamma_x$ containing $z$.
\end{defi}
\medskip

\begin{defi}
A rectangle $S$ is a \emph{subrectangle} of $R=\bigcup_{x\in X}\gamma_x$ if it is a union
$\bigcup_{x\in X}\gamma'_x$ over the same set $X$ as $R$ and if one has $\gamma_x'\subset \gamma_x$ for each $x\in X$.
\end{defi}

\begin{remark}\label{r.uniqueness}
Note that if $S,T$ are two subrectangles of $R$ and if
$x^-_S=x^-_T$ for some $x\in X$, then it holds for all $x$.
Indeed for any $y\in X$, the point $y^-_{T}$ is the intersection of
$\cW^{cs}_{x^{-}_T}=\cW^{cs}_{x^{-}_S}$ with $W^{cu}_{loc}(y)$.
In particular if $W^{cu}_{S}(x)=W^{cu}_{T}(x)$ for some $x\in X$, then $S=T$.
\end{remark}

\paragraph{b) Adapted rectangles.}
We introduce for rectangles a kind of Markov property.

\begin{defi}\label{d.adapted} A rectangle $R$ is \emph{adapted} if for any $x,y\in X$ and $n\geq 0$,
\begin{enumerate}
\item[--] the curve $W^{cu}_R(y)$ is either disjoint from or contained in $f^n(W^{cu}_R(x))$,
\item[--] in the case $W^{cu}_R(y)\subset f^n(W^{cu}_R(x))$ there exists a subrectangle $S$ of $R$
such that for each $z\in X$ the image $f^n(W^{cu}_S(z))$ is an unstable curve of $R$
and such that $f^n(S)$ contains $W^{cu}_R(y)$.
\end{enumerate}
This subrectangle $S$ is called a \emph{return} and $n$ is called a \emph{return time} of $R$.
In the case $f^k(S)$ is disjoint from $R$ for any $0<k<n$, one says that $S$ and $n$
are a \emph{first return} and a \emph{first return time} of $R$.
\end{defi}

The next lemma shows that returns of adapted rectangles are adapted (take $S=R$).
\begin{lemma}\label{l.adapted}
Let $R$ be an adapted rectangle and $S$ be a subrectangle of $R$.
Let also $R'$ be a return of $R$ with return time $n$.
Then $S'=R'\cap f^{-n}(S)$ is a subrectangle of $R'$. If $S$ is adapted, $S'$ is adapted.
\end{lemma}
\begin{proof}
Note that $S'$ has diameter smaller than $\rho$ and $S'\cap \Lambda$ is open in $\Lambda$.
For $x'\in X$, we consider the point $x\in X$ such that $f^n(W^{cu}_{R'}(x'))=W^{cu}_{R}(x)$ and we define $\gamma_{x'}=f^{-n}(W^{cu}_{S}(x))$. By construction and since $R$ is adapted, $S'$ is the union
$\bigcup_{x'\in X}\gamma_{x'}$.
In order to prove that $S'$ is a rectangle it remains to check the item~\ref{rectangle3} of the definition.

For $x',y'\in X$ , we consider $x,y\in X$ such that $f^n(W^{cu}_{R'}(x'))=W^{cu}_{R}(x)$
and $f^n(W^{cu}_{R'}(y'))=W^{cu}_{R}(y)$.
We then denote $x_{S'}^-=f^{-n}(x_S^-)$ and $y_{S'}^-=f^{-n}(y_S^-)$.
We have to prove that $y^-_{S'}$ belongs to $\cW^{cs}_{x^-_{S'}}$.
Let $z$ by the intersection
between $\cW^{cs}_{x_{S'}^-}$ and $W^{cu}_{R'}(y_{S'}^-)$.
The image $f^{n}(z)$ is the intersection between $\cW^{cs}_{x_{S}^-}$ and $W^{cu}_{R}(y_{S}^-)$.
Since $S$ is a subrectangle of $R$, $f^{n}(z)$ and $y_{S}^-$ coincide, hence $z$ and $y_{S'}^-$ coincide, as required.
\medskip

We now assume that $S$ is adapted and prove that $S'$ is adapted too.
Let us suppose that $f^m(W^{cu}_{S'}(x'))$ intersects $W^{cu}_{S'}(y')$ for some $m\geq 0$.
Taking the image by $f^m$,
one deduces that $f^m(W^{cu}_{S}(x))$ intersects $W^{cu}_{S}(y)$.
Since $S$ is adapted, one has $W^{cu}_{S}(y)\subset f^m(W^{cu}_{S}(x))$.
This implies that $W^{cu}_{S'}(y')$ is contained in $f^m(W^{cu}_{S'}(x'))$,
proving the first item of definition~\ref{d.adapted}.

Since $R$ is adapted, there exists a subrectangle $R''$ of $R$
such that, for each $z'\in X$, the image
$f^m(W^{cu}_{R''}(z'))$ is an unstable curve of $R$ and such that
$f^m(W^{cu}_{R''}(x'))=W^{cu}_{R}(y')$. By the first part of the lemma, the intersection
$T'=R''\cap f^{-m}(S')$ is a subrectangle of $R''$.
Note that $W^{cu}_{T'}(x')$ is contained in $W^{cu}_{S'}(x')$. By property (viii)
this implies that for any $z\in X$ one has $W^{cu}_{T'}(x')\subset W^{cu}_{S'}(x')$
proving that $T'$ is a subrectangle of $S'$ such that
$W^{cu}_{T'}(x')$ is mapped on $W^{cu}_{S'}(y')$. Hence $S'$ is adapted.
\end{proof}
\bigskip

\paragraph{c) Holes.} In general, $\La\cap R$ is smaller than $R$ and one can introduce the notion of hole.
\begin{defi}
A \emph{hole} of a rectangle $R$ is a subrectangle that is disjoint from
$\Lambda$ and that is maximal for the inclusion and these properties.

A hole has \emph{aperiodic boundary} if
its boundary $\bigcup_{x\in X}\{x^-_S,x^+_S\}$ is disjoint from its forward iterates.
\end{defi}

\begin{lemma}\label{l.hole}
\emph{1.} If $S$ is a hole of $R$ then
either for any unstable curve $W^{cu}_R(x)$ of $R$ one has $x^-_S=x^-_R$
or there exists a sequence $(x_n)$ in $R\cap \Lambda$ such that
$d(x_n,x_{n,S}^-)$ goes to zero as $n\to +\infty$.
\smallskip

\noindent
\emph{2.} Holes of adapted rectangles are adapted.
\smallskip

\noindent
\emph{3.} For any adapted rectangle $R$, any hole $S$ with aperiodic boundary and any $\tau\geq 1$, there exists $N\geq 1$
such that for any $x\in \Lambda\cap R$ and any $n\geq N$ satisfying
$f^{-n}(W^{cu}_S(x))\subset S$, the iterates $f^{-n-k}(W^{cu}_S(x))$
for $k\in\{1,\dots,\tau\}$ are disjoint from $S$.
\end{lemma}
\begin{proof}
Let $S$ be a hole of $R$ and $W^{cu}_R(x)$ be an unstable curve.
We suppose that $x^-_S\neq x^+_R$.
The points $y\in K_f\cap \overline{R}$ can be ordered by considering the projections
$\cW^{cs}_y\cap \overline{W^{cu}_R(x)}$ on $\overline{W^{cu}_R(x)}$ in such a way that $x^-_S<x^+_S$.
The union of the curves $\gamma'_y\subset W^{cu}_R(y)$ for $y\in X$,
bounded by $y^-_R$ and $y^+_S$, is a rectangle. Thus, since $S$ is a hole
and $x^-_R<x^-_S$, there exists points $y\in \Lambda\cap R$
such that $x^-_R< y \leq x^-_S$.

If there exists an increasing sequence $(x_n)\in \Lambda\cap R$
whose projections on $W^{cu}_R(x)$ converge towards $x^-_S$,
then the distance $d(x_n,x^-_{n,S})$ goes to zero and we are done.
So we assume by contradiction that this is not the case.
There exists a point $\bar x\in \Lambda\cap \overline R$
which is the limit of points $y\in \Lambda\cap R$ and such that
there is no point $y\in \Lambda\cap R$ satisfying $\bar x<y\leq x^-_S$.
Since $R$ has diameter smaller than $\rho$, which has been chosen smaller than
the size of the plaque $\widehat \cW^{cs}$,
the plaque $\widehat \cW^{cs}_{\bar x}$ intersects each curve
$W^{cu}_R(y)$ at a point $y^-_T$.
The union of the curves $\gamma'_y\subset W^{cu}_R(y)$ for $y\in X$,
bounded by $y^-_T$ and $y^+_R$ is a rectangle
whose intersection with $\Lambda$ is empty. This contradicts
the maximality of $S$. We have thus proved the first item of the lemma.
\medskip

Let us assume that $R$ is adapted and that $W^{cu}_S(y)$ intersects $f^n(W^{cu}_S(x))$
for some $n>0$ and some $x,y\in X$.
We have to show that $f^n(x^-_S)$ and $f^n(x^+_S)$ do not belong to the open curve
$W^{cu}_S(y)$.
Since $R$ is adapted, there exists a return $T$ of $R$
such that $f^n(W^{cu}_T(x))=W^{cu}_R(y)$.
By property (viii), the rectangle $T$ contains $S$.
In the case $z^-_S$ and $z^-_R$ coincide for $z\in X$, the point
$f^n(z^-_S)=f^n(z^-_T)$ does not belong to the interior of the curves of $R$, as required.
Otherwise, there exists by the first item a sequence $(x_k)$ in $\Lambda\cap R$
such that $d(x_k,x_{k,S^-})$ goes to $0$ as $k$ goes to $+\infty$.
Hence $f^n(x_k)$ is close to $f^n(x_{k,S}^-)$ and belongs to $R$.
We have thus proved that $\cW^{cs}_{f^n(x^-_S)}$ is accumulted by points of $\Lambda\cap R$.
As a consequence, $f^n(x^-_S)$ can not belong to the interior of an unstable curve of $S$.
This gives the second item of the lemma.
\medskip

Note that $S$ has only finitely many returns with return time smaller or equal to $\tau$.
If $S$ has aperiodic boundary, its boundary is disjoint from the boundary of each of its returns:
there exists $\delta>0$ such that for any return $T$ with return time smaller or equal to $\tau$,
one has $d(x^-_S,x^-_T)>\delta$ and $d(x^+_S, x^+_T)>\delta$.
For $n$ larger than some constant $N$, the unstable curves $f^{-n}(W^{u}_{loc}(x))$ of points $x\in \Lambda$
have a size smaller than $\delta$. If $x\in R$ and $f^{-n}(W^{cu}_S(x))\subset S$, then
the iterate $f^{-n}(x)\in \Lambda$ belongs to $R\setminus S$. One deduces that $f^{-n}(W^{cu}_S(x))$
belongs to a return of $S$ with return time larger than $N$.
This gives the third item of the lemma.
\end{proof}

\subsection{Construction of adapted rectangles}\label{ss.construction}
The assumptions~\ref{a1}) and~\ref{a2}) are now used for the construction of adapted rectangles.
The proof is strongly based on proposition~\ref{p.box}.
\begin{proposition}\label{p.construction}
There exists an adapted rectangle $R$ such that $R\cap \Lambda$ is non-empty.

Moreover one can choose $R$ in such a way that one of the following cases occur.
\begin{enumerate}
\item For any $\tau\geq 0$, there is a first return $S$ of $R$ with return time larger than $\tau$
such that $\Lambda\cap S\neq \emptyset$.
\item There exists a hole $S$ of $R$ with aperiodic boundary.
\end{enumerate}
\end{proposition}
\noindent
The section continues with the proof of this proposition.

\paragraph{a) The construction.}
We have to require further assumptions on $f$ and $\Lambda$
needed to perform the following construction.
We first choose $\eta>0$ small. In particular one has $\eta<\rho<\varepsilon$
and the $10\; \eta$-neighborhood of $H(p_{f_0})$ is contained in $U$.

Let us apply the proposition~\ref{p.box}: one gets a smaller open neighborhood
$\widetilde U$ of $H(p_{f_0})$ such that for any diffeomorphism $f$ that is close enough
to $f_0$ in $\diff^1(M)$, there exists a continuous family
of $C^1$-plaques $\widetilde \cW^{cs}$ tangent to $E^{cs}$
over the maximal invariant set $\widetilde K_f$ of $f$ in $\widetilde U$
which satisfies the following properties:
\begin{itemize}
\item[--] If two plaques $\cW^{cs}_x$ and $\widetilde \cW^{cs}_y$ have an intersection in the $\rho$-ball centered at $x$
then $\widetilde \cW^{cs}_y\subset \cW^{cs}_x$.
\item[--] Any $\widetilde\cW^{cs}$-connected set of $K\cap \cW^{cs}_x$ containing $x$ has radius smaller than $\eta$.
\end{itemize}

Since $f$ is close to $f_0$, if the chain-transitive set $\Lambda$ for $f$ is contained in a small neighborhood
of $H(p_{f_0})$, then the chain-recurrence class $K$ that contains $\Lambda$
is also contained in $\widetilde U$.
We thus have the inclusions $\Lambda\subset K\subset \widetilde K_f\subset K_f$.

\paragraph{\it Approximation by periodic orbits.}
We build a sequence of periodic points $(p_k)$ in $K$ such that
\begin{itemize}
\item the orbit of $p_k$ converges toward $\Lambda$ for the Hausdorff topology,
\item for each iterate $f^n(p_k)$, the plaque $\cW^{cs}_{f^n(p_k)}$ is contained in
the stable manifold of $f^n(p_k)$.
\end{itemize}
Let us fix $\zeta\in \Lambda$ satisfying the property (i).
With the property (iii), the plaque $\cW^{cs}_\zeta$ is contained in the stable set of $\zeta$.
Note that $\zeta$ is not periodic since otherwise $\Lambda$ would be a periodic orbit: by our assumptions, either it would be a sink or the bundle $E^{cu}$ would be uniformly expanded,
contradicting our assumptions.
This ensures that all the plaques $\cW^{cs}_{f^n(p_k)}$ and $\cW^{cs}_\zeta$ are disjoint.

\begin{claim}\label{l.shadowing}
For any $\alpha>0$, there exists $\delta>0$ with the following property.
For any forward return $y=f^n(\zeta)$ that is $\delta$-close to $\zeta$,
there exists $x\in W^{cu}_{loc}(\zeta)\cap K$ such that
$d(f^k(x),f^k(\zeta))$ is smaller than $\alpha$ for each $0\leq k\leq n$ and
the image $f^n(\cW^{cs}_{x})$ is contained in $\cW^{cs}_{x}$.

In particular for any $k\geq 0$ one has
$$\prod_{i=0}^{k-1}\|Df_{|E^{cs}}(f^i(x))\|\leq \lambda^k.$$
\end{claim}
\begin{proof}
From lemma~\ref{l.top-expansion}, there exists $r_0$ such that for any
point $z\in \Lambda$, the backward iterates of the ball centered at $z$ and of radius $r_0$ in
$W^{u}_{loc}(z)$ have a length smaller than $\alpha$.
For $\delta_0$ small enough and any point $y,x\in K$ that are $\delta_0$-close to $\zeta$,
the plaque $\cW^{cs}_{x}$ intersects $W^{cu}_{loc}(y)$ at a point $y'$ which belongs
to the ball centered at $y$ and of radius less than $r_0$ in $W^{cu}_{loc}(y)$.
For $n$ large enough, the length of any curve $f^{-n}(W^{cu}_{loc}(y))$ with $y\in \Lambda$
is smaller than $\delta_0$.
We choose $\delta\in (0,\delta_0)$ so that the returns $f^n(\zeta)$
that are $\delta$-close to $\zeta$ occur for $n$ large.

We define inductively a sequence of points $x_i\in K\cap W^{cu}_{loc}(\zeta)$
that are $\delta_0$-close to $\zeta$ and satisfy:
\begin{itemize}
\item[--] $d(f^{k}(x_i),f^k(\zeta))<\alpha$ for any $0\leq k\leq n$,
\item[--] $f^n(\cW^{cs}_{x_{i+1}})$ is contained in $\cW^{cs}_{x_{i}}$ and $x_0=\zeta$.
\end{itemize}
With properties (i) and (ii), this implies that
\begin{itemize}
\item[--] For any $k\geq 0$ one has $\prod_{j=0}^{k-1}\|Df_{|E^{cs}}(f^j(x_i))\|\leq \lambda^k$.
\end{itemize}
The construction is done in the following way.
Let us assume that $x_i$ has been defined.
Then the plaque $\cW^{cs}_{x_i}$ intersects $W^{u}_{loc}(y)$ in a point
$y_i$. By property (iii) the point $y_i$ belongs to the stable and the unstable set of $\Lambda$,
hence belongs to $K$. Moreover the distances $d(f^{-k}(y_i),y)$
are smaller than $\alpha$ for any $k\geq 0$.
We then define $x_{i+1}=f^{-n}(y_i)$ and by construction
this point is $\delta_0$-close to $\zeta$ and belongs to $W^{u}_{loc}(\zeta)$.

The map $x_i\mapsto x_{i+1}$ is continuous and monotonous, hence converges to
a fixed point $x\in W^{u}_{loc}(\zeta)\in K$. The construction and properties (i), (ii)
give the announced conclusions on $x$.
\end{proof}
\medskip

Since $\zeta$ is recurrent, the lemma~\ref{l.shadowing}
gives a sequence of points $(x_k)$ in $W^{u}_{loc}(\zeta)\cap K$
which converges toward $\zeta$ and such that each plaque $\cW^{cs}_{x_k}$
is mapped into itself by an iterate $f^{n_k}$.
The contraction along the bundle $E^{cs}$ at $x_k$ shows that
each plaque $\cW^{cs}_{x_k}$ is contained in the stable manifold of a periodic point
$p_k\in \cW^{cs}_{x_k}\cap K$.

By construction, the orbit $(x_k,f(x_k),\dots,f^{n_k-1}(x_k))$ is contained in an arbitrarily
small neighborhood of $\Lambda$. With the contraction along the bundle $E^{cs}$ at $x_k$
and the fact that $f^{n_k}(\cW^{cs}_{x_k})\subset \cW^{cs}_{x_k}$, one deduces that the whole
forward orbit of $x_k$ and the orbit of $p_k$ are close to $\Lambda$ for the Hausdorff
topology. Since the plaques $\cW^{cs}$ are trapped, each plaque $\cW^{cs}_{f^n(p_k)}$
is contained in the stable set of the orbits of $p_k$.

\paragraph{\it The boundary $\cW^{cs}_{p^-},\cW^{cs}_{p^+}$.}
We fix some periodic point $p_k$ for $k$ large and consider the set $P$ of all iterates
$p'$ of $p_k$ such that $d(p',x)<5\; \eta$.
\medskip

We choose $x_0\in \Lambda$ close to $\zeta$ and $p^-,p^+\in P$
so that the open interval $I\subset W^{u}_{loc}(\zeta)$ bounded by
$\cW^{cs}_{p^-}$ and $\cW^{cs}_{p^+}$ has the following properties:
\begin{itemize}
\item[--] for any point $p'\in P$ the intersection of
$\cW^{cs}_{p'}$ with $W^{u}_{loc}(\zeta)$ does not belong to $I$,
\item[--] $\cW^{cs}_{x_0}$ intersects $I$.
\end{itemize}

The plaques $\cW^{cs}_{f^n(p^\pm)}$ close to $\zeta$ are controled:
\begin{claim}\label{c.iterate}
For any $n\geq 0$, either the iterate $f^n(\cW^{cs}_{p^+})$ does not meet the ball centered at $\zeta$ of radius $2\eta$,
or $\cW^{cs}_{f^n(p^+)}$ does not intersect $I$. The same holds with the iterates of $\cW^{cs}_{p^-}$.
\end{claim}
\begin{proof} Let us fix a large integer $N$.
Since $\zeta$ is non-periodic and $\cW^{cs}_\zeta$ is contained in its
stable set, the iterates $f^{n}(\cW^{cs}_\zeta)$ are pairewise disjoint.
From the construction, having chosen $\cW^{cs}_{p^+}$ close to $\cW^{cs}_\zeta$
and $I$ close to $\zeta$, the plaques $f^n(\cW^{cs}_{p^+})$ do not meet $I$ for $n\leq N$.

When $n=N$, the radius of the plaque $f^n(\cW^{cs}_{p^+})$ is small, and the plaque is contained
in $\widetilde \cW^{cs}_{f^{n}(p^+)}$. By the trapping property, any iterate
$f^n(\cW^{cs}_{p^+})$ with $n\geq N$ is thus contained in $\widetilde \cW^{cs}_{f^{n}(p^+)}$
and has a radius smaller than $\eta$.
One deduces that if $f^n(\cW^{cs}_{p^+})$ meets the ball centered at $\zeta$ of radius $2\eta$,
then the distance between $f^n(p^+)$ and $\zeta$ is smaller than $3\eta$.
Consequently, $f^n(p^+)$ belongs to $P$ and $\cW^{cs}_{f^n(p^+)}$ does not meet $I$.
\end{proof}
\medskip

\paragraph{\it The rectangle $R$.}
Let us consider in the $2\eta$-neighborhood of $\zeta$ the set $X_0$ of points $z\in \cW^{cs}_\zeta$ that belong
to a forward iterate $f^j(W^{u}_{loc}(y))$ associated to a point $y\in \Lambda$. Then we define $X$ as
the largest $\widetilde \cW^{cs}$-connected subset of $X_0$ containing $\zeta$. By the choice of $\widetilde \cW^{cs}$,
the set $X$ has diameter bounded by $\eta$.
We define $R$ as the union of curves
$\gamma_z\subset W^{cu}_{loc}(z)$, $z\in X$, bounded by
the intersections $z^-,z^+$ between $W^{cu}_{loc}(z)$
and $\cW^{cs}_{p^-},\cW^{cs}_{p^+}$.

By the choice of $\eta$ and the construction, the points $z^-,z^+$ belong to $K_f$.
With property (vi), one deduces that the items~\ref{rectangle1})
and~\ref{rectangle3}) of the definition~\ref{d.rectangle} are satisfied.

Consider any close points $x,y\in \Lambda$ with $x\in R$.
The intersections of $W^{u}_{loc}(x)$ and $W^{u}_{loc}(y)$ with $\cW^{cs}_\zeta$
are close, hence belong to the same $\widetilde \cW^{cs}$-component of $X_0$.
As a consequence, $W^{u}_{loc}(y)\cap \cW^{cs}_\zeta$ belongs to $X$.
This shows that $y$ belongs to $R$. We have proved that $\Lambda\cap R$ is open in $\Lambda$ and that $R$ is a rectangle. By construction it contains the point $x_0$ and the intersection $R\cap \Lambda$
is non-empty.
\medskip

\paragraph{b) $R$ is adapted.}
Let us assume that for some $x,y\in X$,
a forward iterate $f^n(\gamma_{x})$ intersects $\gamma_y$.
Considering a large backward iterate, the two curves $f^{n-m}(\gamma_x)$
and $f^{-m}(\gamma_y)$ are small and contained in local unstable curves
$W^{u}_{loc}(x')$ and $W^{u}_{loc}(y')$ for some points $x',y'\in \Lambda$.
By property (vi), one deduces that $f^{n-m}(\gamma_x)$
and $f^{-m}(\gamma_y)$ are contained in a same unstable curve $W^{u}_{loc}(x')$.
In particular, if $f^n(\gamma_x)$ intersects $\gamma_y$ but does not contain $\gamma_x$,
then the image of an endpoint $f^n(x^-)$ (or $f^n(x^+)$) of $\gamma_x$
is contained inside $\gamma_y$. One deduces that $\cW^{cs}_{f^n(p^-)}$
intersects $I$. Since $f^n(x^-)$ is $2\eta$-close to $x$, this contradicts the claim~\ref{c.iterate} above.
We have thus proved the first item of definition~\ref{d.adapted}.
\medskip

Assume now that $f^n(\gamma_x)$ contains $\gamma_y$.
One can define the subrectangle $S$ of $R$ whose unstable curves are bounded by
$\cW^{cs}_{x^-_S}$ and $\cW^{cs}_{x^+_S}$, with $x^\pm_S=f^{-n}(y^\pm_R)$.
It remains to prove that $f^n(S)$ is contained in $R$.
Let us consider the set $X^+_S$ of points $z^+_S$ for $z\in X$, i.e. the intersection
of $\cW^{cs}_{x^+_S}$ with the unstable curves $W^{cu}_{loc}(z)$.
Since $z^+_S$ and $z$ are close, the set $X^+_S$ is connected for the larger plaque family
$f^{-1}(\widetilde \cW^{cs})$ containing the plaques $f^{-1}(\widetilde\cW^{cs}_{f(x)})$ for $x\in \widetilde K_f$.
Note that $n$ is larger than $2$.
As a consequence, the set $f^n(X^+_S)$ is $f(\widetilde \cW^{cs})$-connected.
One thus deduces that the set $X'$ of intersections of the curves $W^{cu}_{loc}(z)$, $z\in X^+_S$,
with $\cW^{cs}_\zeta$ is $\widetilde \cW^{cs}$-connected. Since it contains $y\in X$, this set is contained
in the $\widetilde \cW^{cs}$-component $X$. This proves the second item of definition~\ref{d.adapted}
and $R$ is adapted.

\paragraph{c) Periodic center-stable plaques.}
Let us assume that there exist $x\in \Lambda$ and $n\geq 1$ such that
the plaque $\cW^{cs}_x$ is mapped into itself by $f^n$.
The set $\Lambda$ is not contained in the orbit of the plaque $\cW^{cs}_x$:
otherwise the property (i) would imply that $\zeta$ is a sink of $\cW^{cs}_x$,
contradicting the fact that $\Lambda$ is non-periodic.
Since $\cW^{cs}_\zeta$ is contained in the stable manifold of $\zeta$, the closure of $\cW^{cs}_\zeta$
and of $\cW^{cs}_x$ are disjoint.

Note that the rectangle $R$ can have been constructed arbitrarily thin in the center unstable direction,
hence it is contained in an arbitrarily small
neighborhood of $\cW^{cs}_\zeta$. In particular, the closure of $R$
and the closure of the orbit of $x$ are disjoint.
Since $\Lambda$ is transitive the first return time on $\Lambda\cap R$ is unbounded,
giving the first case of proposition~\ref{p.construction}.

\paragraph{d) Non-periodic center-stable plaques.}
Let us assume the opposite case: there does not exist $x\in \Lambda$ and $n\geq 1$ such that
the plaque $\cW^{cs}_x$ is mapped into itself by $f^n$.
Let $R_0$ be a rectangle as obtained in paragraphs a) and b).
One can assume also that the first case of the proposition~\ref{p.construction} does not hold.
Since $R_0\cap \Lambda$ is non empty and since $\Lambda$ is the Hausdorff limit of periodic points,
there exists a periodic point $p\in K_f$ whose plaque $\cW^{cs}_p$ intersects
all the unstable leaves $W^{cu}_{R_0}(z)$ of $R_0$ at a point $z_p$ which is not in $\Lambda$.

As in the proof of lemma~\ref{l.hole}, the points $K_f\cap R_0$ are ordered by
their projection on an unstable curve of $R_0$.
There exist two points $x^-,x^+\in K_f\cap \overline R_0$
such that $x^-<z_p<x^+$, any point $y\in \Lambda\cap R_0$ satisfies $y\leq x^-$
or $y\geq x^+$ and such that there is no $\bar x^-<x^-$ or
$\bar x^+>x^+$ with the same properties.
One checks easily that the collection of curves
$\gamma'_z\subset W^{u}_{R_0}(z)$ bounded by points
in $\cW^{cs}_{x^-}$ and $\cW^{cs}_{x^+}$ is a rectangle and a hole $S_0$ of $R_0$.
\medskip

We then explain how to modify $R_0$ in order to obtain a hole with aperiodic boundary.
Since $R_0\cap \Lambda$ is non-empty, one can assume by lemma~\ref{l.hole} that
there exists a sequence $x_n\in \Lambda\cap R_0$ such that $d(x_n,x^-_{n,S_0})$ goes to zero
as $n$ goes to $+\infty$. Let us denote by $X^-=\{z^-_{S_0},\;z\in X\}$
one of the boundaries of $S_0$. By construction there exists $x^-\in\overline{X^-}\cap \Lambda$
such that the plaque $\cW^{cs}_{x^-}$ contains $X^-$.
One deduces that $X^-$ is disjoint from its forward iterates.

One can choose the points $x_n$ to have a dense forward orbit.
In particular they return to $R_0$. Since the return time is bounded,
one can also assume that they all have the same return time, hence belong
to a same return $T$ of $R_0$. The set $X^-$ belong to $T$: otherwise it would be
contained in the boundary of $T$ and mapped by a forward iterate into the boundary of $R_0$;
since the closure of $X^-$ meets $\Lambda$ and since the boundary of $R_0$ is contained
in the stable set of a periodic orbit, this would imply that the $\cW^{cs}$-plaque of a point of $\Lambda$
is mapped into itself, contradicting our assumption.

Note that the set $X^-$ is still one of the boundaries of a hole of $T$
and that the boundary of $T$ is still contained in the stable set of periodic orbits.
One can thus replace $R_0$ by $T$ and repeat the same argument.
Doing that several times, one gets a deeper return $R$ of $R_0$
which contains the set $X^-$.
The rectangle $R$ is arbitrarily thin in the unstable direction,
hence it contains a hole $S$ whose boundaries are $X^-$ and a boundary of $R$.
By construction the boundaries of $R$ are disjoint from their iterates.
This implies that $S$ has aperiodic boundaries.
\bigskip

The proof of the proposition~\ref{p.construction} is now complete.
\qed

\subsection{Summability}
For any point $x\in K$ we denote by
$\ell(J)$ the length of any interval $J$ contained in its local unstable manifold $W^{cu}_{loc}(x)$.
This section is devoted to the proof of the next proposition.
\begin{proposition}\label{p.summability}
For any adapted rectangles $S,R$, where $S$ is a subrectangle of $R$,
there exists $K(S)>0$ satisfying the following: for any $x\in \Lambda\cap R$,
and any $n\geq 0$ such that the curves $f^{-k}(W^{cu}_{S}(x))$, $0<k<n$,
are disjoint from $S$, we have
$$\sum_{k=0}^{n-1}\ell(f^{-k}(W^{cu}_{S}(x)))<K(S).$$
Moreover, there is $K_0>0$ which only depends on $R$ such that
$K(S)<K_0$ when $S$ is a return of $R$.
\end{proposition}
\bigskip

\paragraph{a) Summability for first returns.}
The first case corresponds to~\cite[lemma 3.7.1]{PS1}.
\begin{lemma}\label{ll.sum}
For any adapted rectangle $R$ with $R\cap \Lambda\neq \emptyset$,
there are $C_1>0$, $\sigma_1\in (0,1)$ as follows.

For any unstable curve $W^{cu}_{R}(x)$ of $R$ with $W^{cu}_{loc}(x)\cap \Lambda\neq \emptyset$,
any interval $J\subset W^{cu}_R(x)$
and any $n\geq 0$ such that the iterates $f^{-j}(W^{cu}_{R}(x))$ with $0<j<n$
are disjoint from $R$ we have
$$\ell(f^{-n}(J))\leq C_1\; \sigma_1^n\; \ell(J).$$
\end{lemma}
\begin{proof}
Let us consider a point $z\in \Lambda\cap R$.
Since $\Lambda\cap R$ is open, one can choose a small open neighborhood $V$ of $z$.
The maximal invariant set
$$\La_1= \bigcap_{n\in \ZZ}{f^n(\La - V)}$$
in $\Lambda\cap (M\setminus V)$ is compact and proper in $\Lambda$.
By assumption $E^{cu}$ is expansive on $\La_1$.
It is thus possible to get a neighborhood of $\La_1$ such that while the iterates remain
in this neighborhood the subbundle $E^{cu}$ is uniformly expanded by $Df$.
Moreover, the number of iterates that an orbit of $\Lambda$
remains in the complement of the mentioned neighborhood of $\La_1$
and $V$ is uniformly bounded.

Since $W^{cu}_{loc}(x)\cap \Lambda\neq \emptyset$,
one can always assume that $x$ belongs to $\Lambda$.
By lemma~\ref{l.top-expansion}, choosing $n_0$ large enough (and independant from $x,J,j$),  the curves $f^{-j}(W^{u}_{loc}(x))$
for $j\geq n_0$ are small. If $j<n$, the segment $f^{-j}(J)$ is disjoint from $R$,
hence $f^{-j}(x)$ is disjoint from $V$. Moreover $x$ belongs to $\Lambda$.
One deduces that there exist uniform constants $\sigma\in (0,1)$ and $C>0$
such that $\|Df^{-j}_{|E^{cu}}(x)\|< C\sigma^{j}$ for all $0<j<n$.
Since for $n_0$ large enough the curves $f^{-j}(W^{u}_{loc}(x))$
are small, the derivatives $\|Df_{E^{cu}}(f^{-j}(x))\|$
and $\|Df_{E^{cu}}(f^{-j}(y))\|$ for $y\in W^{u}_{loc}(x)$ are close.

One deduces that for any $0<j<n$ and $y\in W^{u}_{loc}(x)$ one also has
$\|Df^{-n}_{|E^{cs}}(y)\|< C_1\sigma_1^{n}$ for other constants $C_1>0$, $\sigma_1\in (0,1)$.
The conclusion of the lemma follows.
\end{proof}
\medskip

\paragraph{b) Distortion along center-stable holonomies and contracting returns.}
We will need to compare the unstable curves.
\begin{defi}
A rectangle $R$ has \emph{distortion bounded by $\Delta>0$} if for any unstable
curves $W^{cu}_{R}(x)$, $W^{cu}_{R}(y)$ one has:
$$\frac 1 {\Delta}\leq \frac{\ell(W^{cu}_R(x))}{\ell(W^{cu}_R(y))}\leq {\Delta}.$$
\end{defi}
We will also need to consider returns that contract along the center-stable bundle.

\begin{defi}\label{d.contracting}
Let us fix $N\geq 0$. A point $z_0\in K_f$
is \emph{$N$-contracting} in time $n$ if there exists
$m\leq N$ in $\{0,\dots,n\}$ such that for each
$i\in \{m,\dots,n\}$ one has
$$\prod_{k=m}^{i}\|Df_{|E^{cs}}(f^{k}(z_0))\|\leq \lambda^{i-m}.$$
A return $S$ of a rectangle $R$ with returning time $n$
is \emph{$N$-contracting} if there $z_0\in K_f\cap \overline S$
which is $N$-contracting in time $n$.
\end{defi}

The following lemma is similar to~\cite[lemma 3.4.1]{PS1}.
\begin{lemma}\label{l.distortion-holo}
For any adapted rectangle $R$ and any $N\geq 0$,
there is $\Delta_1>0$ such that
any $N$-contracting return of $R$ has distortion bounded by $\Delta_1$.
\end{lemma}
\begin{proof}
One chooses a $C^1$-foliation $\cF$ tangent to the cone field $\cC^{cs}$
and containing the plaques $\cW^{cs}_{x^-}$, $\cW^{cs}_{x^+}$ of $R$.
For any unstable curves of $R$ with basepoints $x,y\in X$,
one gets a diffeomorphism $\Pi_{x,y}\colon W^{cu}_{R}(x)\to W^{cu}_{R}(y)$,
whose derivative is bounded from above and below uniformly in $x,y\in X$.

Let $S$ be a $N$-contracting return of $R$.
For any unstable curves of $S$ with basepoints $x',y'$, their images by $f^n$
coincide with some curves $W^{cu}_{R}(x)$, $W^{cu}_{R}(y)$ of $R$.
This allows us to define a diffeomorphism $\Pi^n_{x',y'}\colon W^{cu}_{S}(x')\to W^{cu}_{S}(y')$ by
$$\Pi^n_{x',y'}=f^{-n}\circ\Pi_{x,y}\circ f^n.$$
The distortion of $S$ is thus controled by the following quantity, for any $z\in W^{cu}_{S}(x')$:
$$\|Df^n_{|TW^{cu}_{S}(z)}\|/\|Df^n_{|TW^{cu}_{S}(\Pi^n_{x',y'}(z))}\|.$$
Using that the forward iterates of any vector tangent to $\cC^{cu}$ in $U$
converge towards $E^{cu}$ (uniformly) exponentially fast and that the bundle $E^{cu}$
is H\"older (see~lemma~\ref{l.bundle}), one can argue as in~\cite[lemma 3.4.1]{PS1}
and show that there exist some uniform constants $C>0$ and $\alpha\in (0,1)$ such that
$$\frac{\|Df^n_{|TW^{cu}_{S}(z)}\|}{\|Df^n_{|TW^{cu}_{S}(\Pi^n_{x',y'}(z))}\|}
\leq \exp\left(C+C\sum_{i=0}^{n-1}d(f^i(z),f^i(\Pi^n_{x',y'}(z)))^\alpha\right).$$

It remains to estimate $d(f^i(z),f^i(\Pi^n_{x',y'}(z)))$
and it is clearly enough to consider the case $i\geq N$.
Using the property (v) stated in section~\ref{ss.topo}, there exists a disc $B$
centered at $z$ tangent to $\cC^{cs}$ of radius larger than $\varepsilon$,
whose iterates $f^{i}(B)$, $i\in \{0,\dots,n\}$ have a radius smaller than $\widetilde \varepsilon$
and such that $f^{n}(B)$ is contained in a leaf of the foliation $\cF$.
One deduces that $B$ contains the point $\Pi^n_{x',y'}(z)$.
From property (ii), the distance $d(f^i(z),f^i(\Pi^n_{x',y'}(z))$ is thus bounded
by
\begin{equation*}
\begin{split}
d(f^i(z),f^i(\Pi^n_{x',y'}(z)))&\leq
d(f^m(z),f^m(\Pi^n_{x',y'}(z)))\;
(1+\chi)^{i-m}\;\prod_{k=m}^{i}\|Df_{|E^{cs}}(f^k(z_0))\|\\
&\leq \widetilde\varepsilon\;(1+\chi)^{i-m}\;\lambda^{i-m},
\end{split}
\end{equation*}
where $z_0\in K_f\cap \overline S$ is a point which satisfies the definition~\ref{d.contracting}
for some interger $m\geq N$.
We have assumed that $(1+\chi)\lambda<1$ (recall section~\ref{ss.topo}),
hence the sum $\sum_{i=0}^{n-1}d(f^i(z),f^i(\Pi^n_{x',y'}(z)))^\alpha$
is bounded uniformly. This concludes the proof of the lemma.
\end{proof} 

With the same proof,
the lemma~\ref{l.distortion-holo} generalizes to the following setting.

\begin{lemma}\label{l.distortion-holo2}
For any adapted rectangles $S,R$ such that $S$ is a subrectangle of $R$ and for any $N\geq 0$,
there exists $\Delta_1(S)$ such that for any $N$-contracting return $R'$ of $R$ with return time $n$,
the subrectangle $S'=R'\cap f^{-n}(S)$ has distortion bounded by $\Delta_1(S)$.
\end{lemma}

\paragraph{c) Summability between contracting returns}
One now obtains the summability for returns which do not satisfy lemma~\ref{l.distortion-holo}.

\begin{lemma}\label{l.deep}
For any adapted rectangle $R$ and any $N\geq 1$ large enough, there is $K_1>0$ as follows.

Consider $x\in \Lambda\cap R$ and $0\leq k<l$ such that:
\begin{itemize}
\item[--] $f^{-k}(W^{cu}_R(x))\subset R$ and $f^{-k}(x)$ is $N$-contracting in time $k$,
\item[--] for any $k<j<l$, either $f^{-j}(W^{cu}_R(x))\cap R=\emptyset$
or $f^{-j}(x)$ is not $N$-contracting in time $j$.
\end{itemize}
Then for any curve $J\subset W^{cu}_R(x)$ one has
$$\sum_{j=k}^{l}\ell(f^{-j}(J))\leq K_1\; \ell(f^{-k}(J)).$$
\end{lemma}
\begin{proof}
We let $\cR\subset \{k,\dots,l\}$ be the set of integers $n$ such that $f^{-n}(W^{u}_R(x))\subset R$. Since $R$ is adapted the other integers satisfy $f^{-n}(W^{u}_R(x))\cap R=\emptyset$.
The lemma~\ref{ll.sum} can be restated as:

\begin{claim}\label{l.sum1} There exists $K_2>0$ satisfying the following.

For any integers $r<p$ in $\{k,\dots,l\}$ such that
$r\in \cR$ and $\{r+1,r+2,\dots p-1\}\cap \cR=\emptyset$, one has
$$\sum_{j=r}^{p}\ell(f^{-j}(J))<K_2\; \ell(f^{-r}(J)).$$
\end{claim}
\medskip

We also introduce the set $\cP\subset \{0,\dots,l\}$
of integers $n$ such that for each
$0\leq i < n$ one has
$$\prod_{j=i+1}^n\|Df_{|E^{cs}}(f^{-j}(x))\|\leq \lambda^{n-i}.$$
The summability between iterates in $\cP$ is ensured by the next classical argument.
\begin{claim}\label{l.sum2}
For any integers $p<r$ in $\{k,\dots,l\}$
such that $p\in \cP$ and $\{p+1,p+2,\dots,r-1\}\cap \cP=\emptyset$, one has
$$\ell(f^{-r+1}(J))<\mu^{r-p-1}\ell(f^{-p}(J)).$$
\end{claim}
\begin{proof}
Using that the integers $n\in \{p+1,\dots,r-1\}$ are not in $\cP$,
one proves inductively that
\begin{equation}\label{e.pliss}
\prod_{j=p+1}^n\|Df_{|E^{cs}}(f^{-j}(x))\|> \lambda^{n-p}.
\end{equation}
Indeed, if one has $\|Df_{|E^{cs}}(f^{-p-1}(x))\|\leq \lambda$,
then using that $p$ belongs to $\cP$, one deduces that $p+1$ also, which is a
contradiction. Moreover if the inequatlity~\eqref{e.pliss}
holds for all the integers $p+1,\dots,n-1$
and is not satisfied for $n$, then for all $i\in\{p,\dots,n-1\}$ one gets
$$\prod_{j=i+1}^{n}\|Df_{|E^{cs}}(f^{-j}(x))\|\leq \lambda^{n-i}.$$
Since $p$ belongs to $\cP$ this
implies that $n$ also which is a contradiction.
This proves that~\eqref{e.pliss} holds.

The property~\eqref{e.pliss} for $n=r-1$ together with (iv) in section~\ref{ss.topo}
imply that the norm of $Df^{p-r+1}_{|W^{u}_S(f^{-p}(x))}$
along the plaque $W^{cu}_S(f^{-p}(x))$ is smaller than $\mu^{r-p-1}$,
giving the required conclusion.
\end{proof}
\medskip

We can now prove the lemma.
Let $C_f>1$ be an upper bound of $\|Df\|$.
We choose $N$ large enough so that one has $\mu^N K_2 C_f<\frac 1 2$.

Let us consider $p_s<p_{s-1}<\dots<p_0$ in $\cP$
and $k\leq r_s<r_{s-1}<\dots<r_0\leq l$ in $\cR$
which satisfy:
\begin{enumerate}
\item[--] For each $i\in \{0,\dots,s\}$ one has $p_i\leq r_i$ and
for $i\in \{1,\dots,s\}$ one has $r_{i}\leq p_{i-1}$.
\item[--] There is no $r\in \cR$ such that $r_{i}<r<p_i$.
There is no $p\in \cP$ such that $p_i<p<r_{i-1}$.
\item[--] $p_s\leq k$ and when $s\geq 1$ one has $k<p_{s-1}$. There is no $r\in \cR$ such that $r_0<r\leq l$.
\end{enumerate}
These sequences are defined inductively:
$r_0$ is the largest integer in $\cR$ smaller or equal to $l$ and
$p_0$ is the largest integer in $\cP$ smaller or equal to $r_0$.
Assume that $p_i\leq r_i$ have been constructed.
If $p_i\leq k$ we set $s=i$ and the construction stops.
Otherwise we let $r_{i+1}$ be the largest integer in $\cR$ that is smaller or equal
to $p_i$ and smaller than $r_{i}$.
By assumption $p_{i}$ is larger than $n$, hence $r_{i+1}$ is larger or equal to $n$.
Then $p_{i+1}$ be the largest integer in $\cP$ smaller or equal to $r_{i+1}$ and smaller than $p_i$.
\medskip

Since $f^{-k}(x)$ is $N$-contracting in time $k$, we have $p_s\geq k-N$.
One deduces
$$\ell(f^{-p_s}(J))\leq C_f^{N}\; \ell(f^{-k}(J)).$$
Using claims~\ref{l.sum1} and~\ref{l.sum2}, for each $i\in\{1,\dots,s\}$ one has
$$\sum^{p_{i-1}}_{k=p_i} \ell(f^{-k}(J))\leq ((1-\mu)^{-1}+K_2C_f)\; \ell(f^{-p_i}(J)),$$
$$\sum_{k=p_0}^{l} \ell(f^{-k}(J))\leq ((1-\mu)^{-1}+K_2C_f)\; \ell(f^{-p_0}(J)),$$
$$\ell(f^{-p_{i-1}}(J))\leq \mu^{r_{i}-p_i}K_2C_f\; \ell(f^{-p_{i}}(J)).$$
By our assumptions, when $i$ satisfies $0<i<s$ the point $f^{-r_i}(x)\in R$
is not $N$-contracting. As a consequence $r_i-p_i\geq N$. This implies by our choice of $N$,
$$\ell(f^{-p_{i-1}}(J))\leq \mu^N K_2 C_f\; \ell(f^{-p_i}(J)) \leq \frac 1 2 \ell(f^{-p_i}(J)).$$
Putting all these estimates together one gets the conclusion:
$$\sum_{j=k}^{l}\ell(f^{-j}(J))<((1-\mu)^{-1}+K_2C_f)(1+2K_2C_f)C_f^N\; \ell(f^{-k}(J)).$$
\end{proof}

\bigskip

\begin{proof}[\bf d) Proof of the proposition~\ref{p.summability}]
Let us choose $N\geq 1$ large and consider the constant $K_1$ given by lemma~\ref{l.deep}.
The lemma~\ref{l.distortion-holo2} applied to the rectangle $S$ gives a bound $\Delta_1(S)$.
We fix an unstable curve $W^{cu}_R(x_0)$ of $R$.
We set $K(S)=2\Delta_1(S) K_1\ell(W^{cu}_R(x_0))$.
We also set $n_S=0$ (in the case $S$ is a return we will obtain a better result taking
$n_S$ equal to the return time).

Let $x\in \Lambda\cap R$ and $J=W^{u}_S(x)$.
We introduce the set $\cR_{\cP}\subset \{-n_S,\dots,n\}$
of integers $i$ such that
$f^{-i}(J)\subset R$ and $f^{-i}(x)$ is $N$-contracting in time $i+n_S$.
Since $R$ is adapted, the lemma~\ref{l.adapted} shows that for each $i\in \cR_{\cP}$, there exists a subrectangle $S_i$ of $R$ such that
\begin{itemize}
\item[--] $f^{-i}(J)$ is an unstable curve of $S_i$,
\item[--] for each unstable curve $W^{cu}_{S_i}(z)$ of $S_i$ the image
$f^{i}(W^{cu}_{S_i}(z))$ is an unstable curve of $S$.
\end{itemize}

\begin{lemma}\label{l.sum4}
For any $i'< i$ in $\cR_\cP\cap \{1,\dots,n\}$, the rectangles $S_i,S_{i'}$ are disjoint.
\end{lemma}
\begin{proof}
Assume by contradiction that some unstable curves $f^{-i}(W^{cu}_{S}(y))$ and $f^{-i'}(W^{cu}_{S}(y'))$ of $S_i$ and $S_{i'}$ intersect.
Then $W^{cu}_{S}(y')$ intersects $f^{i'-i}(W^{cu}_{S}(y))$ and since $S$ is adapted,
there exists a return $T$ of $S$ with returning time $i-i'$ such that
$f^{i'-i}(W^{cu}_{S}(y))$ is an unstable curve of $T$.
One deduces from remark~\ref{r.uniqueness} and lemma~\ref{l.adapted}
that $f^{i'}(S_{i})$ is contained in $T$, hence in $S$.
This contradicts the assumption that $f^{i'-i}(W^{cu}_S(x))$ is disjoint from $S$.
\end{proof}

Let $i_0$ be the largest integer in $\cR_\cP$ which is smaller or equal to $n_S$.
(When $n_S=0$, one has $i_0=0$).
We now end the proof of the proposition~\ref{p.summability}.
The lemma~\ref{l.deep} implies that
\begin{equation*}\begin{split}
\sum_{k=0}^{n-1}\ell(f^{-k}(W^{cu}_{S}(x)))
&\;\leq\; \sum_{k=i_0-n_S}^{n-1}\ell(f^{-k}(W^{u}_{S}(x)))\\
&\;\leq\;K_1 \left(\ell(f^{-i_0}(J))+\sum_{i\in \cR_\cP,\; i> n_S}\ell(f^{-i}(J))\right).
\end{split}
\end{equation*}

Since $f^{-i}(x)$ is $N$-contracting in time $i+n_S$,
the lemma~\ref{l.distortion-holo} implies that for each $i\in \cR_\cP$
$$\ell(f^{-i}(J))\leq \Delta_1(S)\; \ell(W^{cu}_{S_i}(x_0)).$$
The lemma~\ref{l.sum4} implies that the intervals $W^{cu}_{S_i}(x_0)$ for $i\in \cR_\cP$
with $i> n_S$ are disjoint.
As a consequence
$$\sum_{i\in \cR_\cP,\; i> n_S}\ell(W^{cu}_{S_i}(x))\leq \ell(W^{cu}_R(x_0)).$$
Putting together these last three inequalities, one concludes the proof of the proposition~\ref{p.summability}
in the general case $S$ is an adapted subrectangle:
$$\sum_{k=0}^{n-1}\ell(f^{-k}(W^{cu}_{S}(x)))\leq 2\Delta_1(S) K_1\; \ell(W^{cu}_R(x_0))=K(S).$$
When $S$ is a return, we take $n_S$ equal to the return time so that $f^{n_S}(J)$
is an unstable curve of $R$.
The constant $\Delta_1$ is given by lemma~\ref{l.distortion-holo2} and as before we set
$K_0=2\Delta_1 K_1\ell(W^{cu}_R(x_0))$.
We repeat the same proof, noting that the subrectangles $S_i$ are returns of $R$, so that
for each $i\in \cR_\cP$ we have the better estimate
$$\ell(f^{-i}(J))\leq \Delta_1\; \ell(W^{cu}_{S_i}(x_0)).$$
The conclusion of the proposition~\ref{p.summability} thus holds with the uniform constant $K_0$.
\end{proof}

\subsection{Proof of the proposition~\ref{propoE^{cu}ishyp}}
In order to conclude the proof of proposition~\ref{propoE^{cu}ishyp} we
consider a rectangle $R$ as given by the section~\ref{ss.construction} and we
distinguish between two cases described by the proposition~\ref{p.construction}.

\paragraph{a) Distortion along unstable curves.}
Since by lemma~\ref{l.smoothness}, the unstable curves  the set $K$ are contained in a continuous $C^2$-plaque family,
the classical distortion estimates hold (see for instance~\cite[lemma 3.5.1]{PS1}).
\begin{itemize}
\item[(D)] \emph{There is $\Delta_2>0$ such that for any $z\in K$, any $x,y$ in an interval $J\subset W^{cu}_{loc}(z)$,
and any $n\geq 0$,
$$\frac{\|Df^{-n}_{|E^{cu}}(x)\|}{\|Df^{-n}_{|E^{cu}}(y)\|}\leq \exp\left(\Delta_2\sum_{k=0}^{n-1}\ell(f^{-k}(J))\right).$$
In particular,
\begin{equation}\label{e.disto}
\|Df^{-n}_{|E^{cu}}(x)\|\leq \frac{\ell(f^{-n}(J))}{\ell(J)}\; \exp\left(\Delta_2\sum_{k=0}^{n-1}\ell(f^{-k}(J))\right).
\end{equation}}
\end{itemize}

As a consequence we also get the following.

\begin{itemize}
\item[(D')] \emph{For any $C>0$ there is $\eta>0$ such that for any $z\in K$, for any
intervals $J\subset \widehat J\subset W^{cu}_{loc}(z)$ and for any $n\geq 0$ satisfying
$\ell(\widehat J)\leq (1+\eta)\; \ell(J)$ and $\sum_{k=0}^{n-1}\ell(f^{-k}(J))\leq K$, then one has
$$\sum_{k=0}^{n-1}\ell(f^{-k}(J))\leq 2\; C.$$
In particular for any $x\in \widehat J$ one has
\begin{equation*}
\|Df^{-n}_{|E^{cu}}(x)\|\leq \frac{\ell(f^{-n}(J))}{\ell(J)}\; \exp\left(2\; \Delta_2\sum_{k=0}^{n-1}\ell(f^{-k}(J))\right).
\end{equation*}}
\end{itemize}
\bigskip

\paragraph{b) Adapted rectangles with unbounded first returns.}
We conclude proposition~\ref{propoE^{cu}ishyp} in the first case
of the proposition~\ref{p.construction}.
(The end of the proof corresponds to~\cite[lemma 3.7.4]{PS1}.)
\begin{lemma}\label{l.unbounded2}
For any adapted rectangle $R$, there exists $\tau\geq 0$ as follows.

If there exists a first return $S_0$ of $R$ with return time larger than $\tau$
and such that $S_0\cap \Lambda\neq \emptyset$, then, there also exists a return $S$
of $R$ such that $S\cap \Lambda\neq \emptyset$ which has the following property:
for any $x\in S\cap \Lambda$ and $n\geq 1$ such that
$f^{-n}(x)\in S$ we have $\|Df^{-n}_{|E^{cu}}(x)\|<\frac 1 2$.
\end{lemma}
In particular the property (E) holds with $B=S\cap \Lambda$.
\begin{proof}
Let $K_0,K_1,N,\Delta_2$ be some constants associated to $R$ so that proposition~\ref{p.summability} and lemmas~\ref{l.distortion-holo} and~\ref{l.deep} hold.
Let $L$ be a lower bound for the length of unstable curves $W^{cu}_R(z)$ of $R$
and $l$ be an upper bound for all the backward iterates $f^{-j}(W^{cu}_R(z))$ with $j\geq 0$.
Recall that $\Delta_2>0$ is a constant such that~\eqref{e.disto} holds.
We also set
$$\delta=\frac L {\Delta_1}\exp(-\Delta_2\;(K_0+K_1\;l))/3$$
and choose $\tau\geq 1$ so that for any $z\in \Lambda$ the backward iterates $f^{-k}(W^{cu}_{loc}(z))$
with $k\geq \tau$ have a length smaller than $\delta$
(see lemma~\ref{l.top-expansion}).
We then consider a return $S_0$ of $R$ with return time $n_0$
larger than $\tau$ such that $S_0\cap \Lambda\neq \emptyset$. Two cases occur.

\paragraph{\it Case 1: no contracting backward iterate.}
We assume first that for any $x\in S_0\cap \Lambda$, there is no backward iterate $f^{-j}(x)$ with
$j\geq 0$ which belongs to a $N$-contracting return of $R$ with return time $j$.
In this case, we set $S=S_0$. For any point $x\in S\cap \Lambda$
and any $j\geq 1$ such that $f^{-j}(x)\in S$ one can apply
the lemma~\ref{l.deep} to $x$ and the integers $k=0$ and $l=j$.
One deduces that one has
$$\sum_{i=0}^j \ell(f^{-j}(W^{cu}_R(x)))\leq K_1\ell(W^{cu}_R(x))\leq K_1\;l.$$
Note that $j\geq n_0\geq \tau$. Since $z$ belongs to $\Lambda$, one deduces that
$f^{-j}(W^{cu}_R(z))$ is smaller than $\delta$. With property (D), one gets
\begin{equation*}
\begin{split}
\|Df^{-j}_{|E^{cs}}(x)\|&\leq \frac{\ell(f^{-j}(W^{cu}_R(z)))}{\ell(W^{cu}_R(z))}\exp(\Delta_2\;K_1\;l)\\
&\leq \frac{\delta}{L}\exp(\Delta_2\;K_1\;l)<1/2.
\end{split}
\end{equation*}
The lemma is thus proved in this case.
\medskip

\paragraph{\it Case 2: contracting backward iterates exist.}
We first build the return $S$.
\begin{claim}\label{c.select} There exists a $N$-contracting return $S$ of $R$ with return time $n_1\geq \tau$
such that $\Lambda\cap S\neq \emptyset$
and such that for each $z\in \Lambda\cap S$ one has
$$\sum_{j=0}^{n_1}\ell(f^j(W^{cu}_{S}(z)))<K_1\;\ell(W^{cu}_{R}(f^{n_1}(z))).$$
\end{claim}
\begin{proof}
There exists a point $x\in f^{n_0}(S_0)\cap \Lambda$ and a backward iterate $f^{-n_1}(x)$ with $n_1> n_0$ which belongs to a $N$-contracting $S$ return of $R$ with return time $n_1$. One can assume that
$n_1$ is minimal: consequently for any $i\in \{1,\dots,n_1-n_2\}$ the iterate $f^{i}(S)$ does not intersect a
$N$-contracting return of $R$ with return time $n_1-i$.
Since $S_0$ is a first return, the iterates $f^i(S)$ for $i\in \{n_1-n_0+1,\dots,n_1-1\}$
do not intersect $R$. The lemma~\ref{l.deep} can thus be applied to the points $z\in\Lambda\cap f^{n_1}(S)$ and the integers
$k=0$ and $l=n_1$. In particular, for any $z\in\Lambda\cap S$ one gets the announced inequality.
\end{proof}

We now prove that the return $S$ given by the claim~\ref{c.select}
satisfies the conclusions of the lemma~\ref{l.unbounded2}.
It is enough to consider a point $x\in S\cap \Lambda$ and $n\geq 1$ such that
$f^{-n}(x)\in S$ and $f^{-k}(x)\notin S$ for $0<k<n$. By lemma~\ref{l.adapted}, the rectangle $S$ is adapted,
hence $f^{-k}(W^{cu}_S(x))$ is disjoint from $S$ for any $0<k<n$.
One deduces by proposition~\ref{p.summability} that
$$\sum_{k=0}^{n-1} \ell(f^{-k}(W^{cu}_{S}(x)))<K_0.$$
By our choice of $S$ one has
$$\sum_{j=0}^{n_1} \ell(f^{j}(W^{cu}_{S}(f^{-n}(x))))<K_1 \ell(f^{n_1}(W^{cu}_{S}(f^{-n}(x)))=
K_1\ell(W^{cu}_{R}(f^{n_1-n}(x))).$$
In particular, the property (D) gives
$$\|Df^{n_1-n}_{|E^{cu}}(x)\|\leq \frac{\ell(f^{n_1-n}(W^{cu}_S(x)))}{\ell(W^{cu}_S(x))}\; \exp\left(\Delta_2\;K_0\right).$$
$$\|Df^{-n_1}_{|E^{cu}}(f^{n_1-n}(x))\|\leq \frac{\ell(W^{cu}_S(f^{-n}(x)))}{\ell(W^{cu}_R(f^{n_1-n}(x)))}\; \exp\left(\Delta_2\;K_1\;l\right).$$
Since $S$ is an $N$-contracting return of $R$, the lemma~\ref{l.distortion-holo} gives
$$\frac{\ell(W^{cu}_S(f^{-n}(x)))}{\ell(W^{cu}_S(x))}\leq \Delta_1.$$
We also have
$$\ell(f^{n_1-n}(W^{cu}_S(x)))=\ell(f^{-n}(W^{cu}_R(f^{n_1}(x))))<\delta.$$
Combining these inequalities, one gets the required estimate:
$$\|Df^{-n}_{|E^{cu}}(f(x))\|\leq \frac{\delta}{L\; \Delta_1}\exp\left(\Delta_2\;(K_0+K_1\;l)\right)<1/2.$$
\end{proof}

\paragraph{c) Adapted rectangles with holes.} We obtain a stronger summability result for holes.
This is similar to~\cite[lemma 3.7.7]{PS1}.
\begin{lemma}\label{l.sum-hole}
Let $R$ be an adapted rectangle and $S$ be a hole of $R$ with aperiodic boundary.
Then, there exists $K_3>0$ such that for any $x\in R\cap \Lambda$, we have
$$\sum_{k\geq 0} \ell(f^{-k}(W^{cu}_{S}(x)))<K_3.$$
\end{lemma}
\begin{proof}
By lemma~\ref{l.hole}, $S$ is an adapted rectangle.
Let $(n_i)$ be the sequence of returns of $W^{cu}_S(x)$ into $S$, that is the integers such that
$f^{-n_i}(W^{cu}_S(x))\subset S$. 
Let us consider two consecutive returns $n_i,n_{i+1}$. By the proposition~\ref{p.summability}, we have
$$\sum_{k=n_i}^{n_{i+1}}\ell((f^{-k}(W^{cu}_{S}(x)))<K(S).$$
It is enough to bound uniformly the sum $\sum_{i\geq 0} \ell(f^{-n_i}(W^{cu}_{S}(x)))$.

From (D) we have
$$\frac{\ell(f^{-n_{i+1}}(W^{cu}_{S}(x)))}{\ell(f^{-n_{i}}(W^{cu}_{S}(x)))}
\leq \frac{\ell(f^{n_i-n_{i+1}}(W^{cu}_{S}(f^{-n_i}(x))))}{\ell(W^{cu}_{S}(f^{-n_i}(x)))}
\exp(\Delta_2K(S)).$$
By lemma~\ref{l.hole},
there exists $N\geq 1$ such that for $n_i\geq N$ the difference $n_{i+1}-n_i$
is large and by lemma~\ref{l.top-expansion}, the length $\ell(f^{n_i-n_{i+1}}(W^{cu}_S(f^{-n_i}(x))))$ is smaller than $\ell(W^{cu}_S(f^{-n_i}(x)))\exp(-\Delta_2K(S))/2$.
In particular $\ell(f^{-n_{i+1}}(W^{cu}_{S}(x)))$
is smaller than $\ell(f^{-n_{i}}(W^{cu}_{S}(x)))/2$ for any $n_i\geq N$.
The corollary follows.
\end{proof}
\medskip

It remains to conclude proposition~\ref{propoE^{cu}ishyp} in the second case
of the proposition~\ref{p.construction}.
\begin{lemma}\label{l.hole-conclusion}
For any adapted rectangle $R$ having a hole $S$ with aperiodic boundary and such that
$R\cap \Lambda\neq \emptyset$, there exists a non-empty open subset
$B\subset R$ of $\Lambda$ such that property (E) holds.
\end{lemma}
\begin{proof}
Let $K_3,\Delta_2$ be the constants given by lemma~\ref{l.sum-hole}
and the property (D)
and let $\eta$ be the constant given by the property (D') and associated to $C=K_3$.
Since $R\cap \Lambda$ is non-empty $S$ is proper in $R$.
Up to exchange the boundaries $x^-,x^+$ of $R$ and $S$, one deduces
by lemma~\ref{l.hole} that
there exists a sequence $(x_n)$ in $R\cap \Lambda$ such that
$d(x_n,x^-_{n,S})$ goes to zero as $n\to +\infty$.
Since $\Lambda$ is transitive and is not a single periodic orbit,
one can assume that the points $x_n$ are not periodic.
We fix such a point $x$ so that
$d(x,x^-_{S})<\eta\; \ell(W^{cu}_S(x))$.

Let $L$ be a lower bound for the length of the curves $W^{cu}_S(z)$ of $S$
and let
$\delta=L\exp(-2\;\Delta_2\;K_3)/3$. We choose $\tau$ large enough such that for any
$z\in \Lambda$ the curves $f^{-n}(W^{cu}_{loc}(z))$ for $n\geq \tau$
have a length smaller than $\delta$.
Since $x$ is not periodic, one can find a small neighborhood $B$
of $x$ in $\Lambda$ such that $B$ is disjoint from its first $\tau$ iterates
and for any $y\in B$ one has $d(y,y^-_{S})<\eta\; \ell(W^{cu}_S(y))$.

For any return $f^{-n}(y)$ in $B$ one has $n\geq \tau$.
Lemma~\ref{l.sum-hole} and property (D') thus give:
\begin{equation*}
\begin{split}
\|Df^{-n}_{|E^{cu}}(y)\|&\leq \frac{\ell(f^{-n}(W^{cu}_S(y)))}{\ell(W^{cu}_S(y))}\; \exp\left(2\; \Delta_2\sum_{k=0}^{n-1}\ell(f^{-k}(W^{cu}_S(y)))\right)\\
&\leq \frac{\delta}{L}\exp(2\;\Delta_2\; K_3)<1/2.
\end{split}
\end{equation*}

\end{proof}
\bigskip

The proof of the proposition~\ref{propoE^{cu}ishyp} is now complete.